\documentclass[11pt]{amsart}
\usepackage[margin=1in]{geometry}
\usepackage{amsmath,amssymb,amsthm,mathtools}
\usepackage{mathrsfs}
\usepackage[T1]{fontenc}
\usepackage{lmodern}
\usepackage{microtype}
\usepackage{tikz}
\usepackage{enumitem}
\usepackage{booktabs}
\usepackage{needspace}
\usepackage{titlesec,etoolbox}
\usepackage[colorlinks=true,linkcolor=blue,citecolor=blue,urlcolor=blue,pagebackref=true]{hyperref}
\patchcmd{\abstract}{\bfseries}{\normalfont}{}{}
\titleformat{\section}{\large\normalfont}{\thesection}{1em}{}
\titleformat{\subsection}{\normalsize\normalfont}{\thesubsection}{1em}{}
\titleformat{\part}[display]{\centering\Large\normalfont}{\partname{} \thepart}{1em}{}
\newtheoremstyle{plainunbold}{6pt}{6pt}{\itshape}{}{\normalfont}{.}{.5em}{}
\newtheoremstyle{remarkunbold}{6pt}{6pt}{\normalfont}{}{\normalfont}{.}{.5em}{}
\theoremstyle{plainunbold}
\usepackage[colorlinks=true,linkcolor=blue,citecolor=blue,urlcolor=blue]{hyperref}
\hypersetup{pdftitle={Improved dimension bounds for Falconer’s
distance problem in the plane},pdfauthor={Alex Iosevich and Alexia Yavicoli},bookmarksnumbered=true}
\setlist{nosep}
\numberwithin{equation}{section}
\newtheorem{theorem}{Theorem}[section]
\newtheorem{lemma}[theorem]{Lemma}
\newtheorem{proposition}[theorem]{Proposition}
\newtheorem{corollary}[theorem]{Corollary}

\theoremstyle{remarkunbold}
\newtheorem{remark}[theorem]{Remark}
\newtheorem{question}[theorem]{Question}
\theoremstyle{plainunbold}
\title{Improved dimension bounds for Falconer's\\ distance problem in the plane}

\author{A. Iosevich}
\address{Department of Mathematics, University of Rochester, Rochester, NY, USA}
\email{iosevich@gmail.com}

\author{A. Yavicoli}
\address{Department of Mathematics, University of British Columbia}
\email{yavicoli@math.ubc.ca}

\thanks{A.~I. was supported in part by National Science Foundation grant
DMS-2154232. A.~Y. was supported in part by the Natural Sciences and Engineering Research
Council of Canada (NSERC), grants GR030571 and GR030540.}

\begin{document}
\maketitle

\begin{abstract}
Let $E\subset\mathbb R^2$ be analytic with $t=\dim_{\mathrm H}E>1$.
We prove that its distance set satisfies
$\dim_{\mathrm H}\Delta E\ge\frac{4}{5}+\eta_0$ for an absolute
constant $\eta_0>0$, obtained by compactness without an explicit
numerical value. We also obtain the explicit estimate
\[
\dim_{\mathrm H}\Delta E\ge
\frac{t^2-2t+5}{2t^2-7t+10}
\quad\text{for }1<t\le\frac{11-\sqrt{41}}{4},
\]
which reaches $\frac{7}{8}$ at the right endpoint.
For the pinned distance sets, we prove
$\dim_{\mathrm H}\Delta_xE\ge q_*=0.7813\ldots$ for every
$x\in\mathbb R^2$ outside a set of Hausdorff dimension at most
one, with stronger estimates depending on $t$ and further
bounds on the exceptional sets.
If $E$ is compact and supports an $s$-Frostman measure $\mu$,
where $s>1$, then
\[
\underline{\dim}_{\mathrm B}\Delta_xE\ge
\max\left\{\frac{4}{5}+\eta_0,
\min\left\{\frac{3s+1}{5},\frac{7}{8}\right\}\right\}
\quad\text{for $\mu$-almost every }x.
\]
The same lower bound holds for pinned Hausdorff dimension
under a common dyadic mass law, which may oscillate with
the scale.

The proof retains information about both endpoints in the
Kolmogorov complexity estimates of Stull and Fiedler--Stull.
Applying the partition estimates with the roles of the
endpoints reversed yields the bound $\frac{4}{5}$ and the
rational estimate above. To pass beyond $\frac{4}{5}$, we
examine the constraints imposed by equality. Both information
profiles must have excess information on a common interval
of scales. Encoding a single source measure for all pins
converts this constraint into a configuration of points and
circles with many rich perpendicular bisectors. The lower
bounds supplied by Katz--Tardos and Ren--Wang contradict
the upper bound of Fu--Gan--Ren, forcing the fixed gain
$\eta_0$. This passage from rigidity to a strict improvement
has a precedent in the work of H\'era, Shmerkin and Yavicoli
on Furstenberg sets, where configurations close to the
elementary bound are excluded using Bourgain's projection
theorem. A separate selection argument at each finite
scale establishes the pinned lower box estimate for
every $s$-Frostman measure with $s>1$.

Together with recent work of Stull, Fiedler, Cs\"ornyei
and their collaborators, these results illustrate how
effective dimension and Kolmogorov complexity, in combination
with Fourier analytic and combinatorial techniques, can yield
new bounds in classical geometric measure theory.
\end{abstract}

\tableofcontents 

\section{Introduction}\label{sec:introduction}

The Falconer distance problem is one of the most intriguing problems in modern geometric measure theory. Let $E\subset\mathbb R^2$, and define its distance set and its distance set pinned at
$x\in\mathbb R^2$, respectively, by 
\[
\Delta E=\{|y-z|:y,z\in E\},\qquad \Delta_xE=\{|x-y|:y\in E\}, 
\] where $|\cdot|$ denotes the Euclidean norm. 

We call $x$ the pin and $E$ the source set. A source point is a point
$y\in E$ from which the distance to the pin is measured. For $x\in E$,
$\Delta_xE\subseteq\Delta E$.

For $s\ge0$, define
\[
\mathcal H^s(E)=\lim_{r\searrow0}
\inf\left\{\sum_i(\operatorname{diam}U_i)^s:
U_i\subseteq\mathbb R^2,\ E\subseteq\bigcup_iU_i, \operatorname{diam}U_i\le r\right\},
\]
where the covers are countable and $\operatorname{diam}$ denotes
diameter. The Hausdorff dimension is
$\dim_{\mathrm H}E=\inf\{s\in[0,\infty):\mathcal H^s(E)=0\}$.
The infimum in the display before taking the limit is the
$s$-dimensional Hausdorff content at scale $r$. Omitting the
restriction $\operatorname{diam}U_i\le r$ defines the Hausdorff content
without a scale restriction.

An analytic set is a continuous image
of a complete separable metric space.
Every Borel set is analytic.
The question is how much of the size of $E$ must survive under the
distance map. Falconer conjectured that $\Delta E$ has positive Lebesgue
measure whenever $\dim_{\mathrm H}E>1$~\cite{Falconer}. Even the weaker
conclusion $\dim_{\mathrm H}\Delta E=1$ is open when
$\dim_{\mathrm H}E$ is only slightly larger than one.

For $\dim_{\mathrm H}E$ slightly larger than one, the known lower bounds
approach $\frac{3}{4}$. Stull~\cite[Theorem~1]{S} proved that every analytic
$E\subset\mathbb R^2$ with $\dim_{\mathrm H}E>1$ satisfies
\[
\dim_{\mathrm H}\Bigl\{x \in \mathbb{R}^2:\dim_{\mathrm H}\Delta_xE<\frac{1}{2}+\frac{\dim_{\mathrm H}E}{4}\Bigr\}\le1.
\]
Since this exceptional set cannot contain $E$, also
$\dim_{\mathrm H}\Delta E\ge\frac{1}{2}+\frac{\dim_{\mathrm H}E}{4}$. Both bounds
decrease to $\frac{3}{4}$ as $\dim_{\mathrm H}E$ decreases to one. We obtain
improvements that remain positive in this limit.

\begin{theorem}[A universal pinned bound]\label{thm:universal-main}
Let $\alpha_*=\frac{10-3\mathrm e}{24-7\mathrm e}=0.371107\ldots$, where $\mathrm e$ is the base of the natural logarithm, and 
\[
q_*=\frac{3}{4}+\frac{3\alpha_*}{40-12\alpha_*}=0.78131\ldots.
\]
Every analytic $E\subset\mathbb R^2$ with $\dim_{\mathrm H}E>1$
satisfies
\[
\dim_{\mathrm H}\left\{x\in\mathbb R^2:
\dim_{\mathrm H}\Delta_xE<Q(\dim_{\mathrm H}E)\right\}\le1,
\]
where
\[
Q(t)=1-\frac{(1-q_*)(3-2t)}{1+(6-8q_*)(t-1)} \quad \text{for } 1\le t\le\frac{3}{2},
\]
and $Q(t)=1$ for $t\ge\frac{3}{2}$. The function $Q$ is increasing on
$[1,\frac{3}{2}]$, with $Q(1)=q_*$ and $Q(\frac{3}{2})=1$.
\end{theorem}

The exceptional set consists of the pins where the stated bound
fails.
We denote by $B(z,r)$ the open ball
of radius $r$ about $z$. A finite Borel measure $\nu$ is $s$-Frostman if
$\nu(B(z,r))\le Cr^s$ for every $z$ and $0<r\le1$, with $C$
independent of $z,r$. Such a measure assigns mass zero to sets of
Hausdorff dimension less than $s$ \cite[Chapter~8]{Mattila}.
Thus the conclusion holds for $\nu$-almost every pin whenever
$s>1$.
The constant $\alpha_*$ comes from the estimate of Katz and Tardos 
for pairwise sums within rows of a matrix whose entries are all
distinct~\cite[Corollary~5]{KT}. Katz and Tardos proved this estimate on the way to improved bounds for the Erdős distinct distances problem in the plane. Their explicit five-entry
and nine-entry estimates give, by the same proof, the thresholds
$\frac{132}{169}$ and $\frac{368}{471}$ in place of $q_*$, and
\[
Q_5(t)=\frac{32t+100}{211-42t}, \qquad 
Q_9(t)=\frac{88t+280}{589-118t}
\]
in place of $Q(t)$ for $1\le t\le\frac{3}{2}$. The proof keeps the
threshold $q_*$ independent of the source dimension, while the scale
at which its estimates apply may depend on the source and pin
measures. The source dimension enters through the estimate in
Section~\ref{sec:asymmetric} and its iteration in
Section~\ref{sec:bootstrap}. The proof outline below describes which information about the two points these estimates retain.

Allowing either endpoint to serve as the pin yields a larger gain in the exponent. 

\begin{theorem}[Unpinned distance sets]\label{thm:main}
There is an absolute constant $0<\eta_0<\frac{3}{40}$ such that
every analytic $E\subset\mathbb R^2$ with $t=\dim_{\mathrm H}E>1$
satisfies
\begin{equation}\label{eq:target}
\dim_{\mathrm H}\Delta E\ge
\max\left\{\frac{4}{5}+\eta_0,Y(t),R(t)\right\},
\end{equation}
where $S_0=\frac{11-\sqrt{41}}{4}=1.149218\ldots$ and
\[
Y(t)=\begin{cases}
\displaystyle\frac{t^2-2t+5}{2t^2-7t+10},&1\le t\le S_0,\\[5pt]
\displaystyle\frac{7}{8},&t\ge S_0,
\end{cases}
\qquad
R(t)=\min\left\{1,\frac{t(4-t)}{5-t}\right\}.
\]
In particular, $\dim_{\mathrm H}\Delta E\ge\frac{7}{8}$ for $t\ge S_0$.
\end{theorem}

The function $Y$ is increasing on $[1,S_0]$, with
$Y(1)=\frac{4}{5}$ and $Y(S_0)=\frac{7}{8}$.
For every fixed $1<t\le S_0$, its gain over $\frac{4}{5}$ is explicit:
\[
Y(t)-\frac{4}{5}
=\frac{3(t-1)(5-t)}{5(2t^2-7t+10)}>0.
\]
This gain tends to zero as $t\searrow1$. The constant $\eta_0$
provides a uniform positive gain throughout $t>1$.
The bound $R(t)$ follows from Fiedler--Stull~\cite[Theorem~1
and Corollary~2]{FS}; it is also recovered by letting $p$ increase
to $t$ in Theorem~\ref{thm:pinnedmain} below. It exceeds $Y(t)$
when $t>\frac{39-\sqrt{401}}{16}=1.1859\ldots$.
Table~\ref{tab:intro-comparison} records the comparison with earlier
unpinned estimates.

Two arguments enter Theorem~\ref{thm:main}. The reciprocal
partition estimate in Section~\ref{sec:reciprocal}, combined with
the source-dimension iteration, proves $Y(t)$. The fixed gain
$\eta_0$ requires more. At the scalar endpoint, both profiles have
a common interval of excess information. Section~\ref{sec:new-rigidity}
proves this rigidity statement, and Section~\ref{sec:new-geometry}
uses the common source measure to rule out the resulting geometric
configuration. The compactness argument does not assign a numerical
value to $\eta_0$. This constant is chosen independently of the set
and its measures. The starting precision in the finite argument
may depend on the measures, as explained in
Remark~\ref{rem:threshold-dependence}.
The elementary arguments in Appendices~\ref{sec:thirteen-seventeen}
and~\ref{sec:refined} illustrate the first interaction between the
two endpoint estimates.

There is a stronger pinned conclusion if Hausdorff dimension of
the image is replaced by lower box dimension. For a bounded set
$F$, let $N_r(F)$ be the least number of open balls of radius $r$
covering $F$, and define
\[
\underline{\dim}_{\mathrm B}F
=\liminf_{n\to\infty}\frac{\log_2N_{2^{-n}}(F)}{n}.
\]
Put
\begin{equation}\label{eq:lower-box-functions}
L(s)=\min\left\{\frac{3s+1}{5},\frac{7}{8}\right\},\qquad
L_*(s)=\max\left\{\frac{4}{5}+\eta_0,L(s)\right\}.
\end{equation}

\begin{theorem}[Pinned lower box dimension]\label{thm:new-box}
Let $E\subset\mathbb R^2$ be compact, and let $\mu$ be any
nonzero finite $s$-Frostman measure supported on $E$, with $s>1$.
Then
\begin{equation}\label{eq:new-box}
\underline{\dim}_{\mathrm B}\Delta_xE\ge L_*(s)
\quad\text{for $\mu$-almost every }x.
\end{equation}
After normalizing $\mu$ to have mass one, for every $q<L_*(s)$
and $M>0$ one has
\begin{equation}\label{eq:new-tail}
\mu\{x\in E:N_{2^{-n}}(\Delta_xE)<2^{qn}\}\le n^{-M}
\end{equation}
for all sufficiently large $n$. The starting precision may depend
on $E,\mu,s,q,M$.
\end{theorem}

The exponent $s$ belongs to the specified measure. No lower
mass bound or equality of Hausdorff and packing dimensions is
assumed. The explicit bound $L(s)$ reaches $\frac{7}{8}$ at
$s=\frac{9}{8}$. The proof chooses two bad pins whose complexity
profiles nearly agree at the precision under consideration.
The pair may change with the precision; Section~\ref{sec:new-finite}
establishes the uniform estimates needed for this selection.

\begin{corollary}[Internal lower-box exceptional sets]\label{cor:new-exceptional}
For every bounded $E\subset\mathbb R^2$,
\begin{equation}\label{eq:new-exceptional}
\dim_{\mathrm H}\{x\in E:
\underline{\dim}_{\mathrm B}\Delta_xE<\frac{4}{5}+\eta_0\}\le1.
\end{equation}
For $\frac{4}{5}\le q<\frac{7}{8}$, one also has
\[
\dim_{\mathrm H}\{x\in E:
\underline{\dim}_{\mathrm B}\Delta_xE\le q\}\le\frac{5q-1}{3},
\qquad
\dim_{\mathrm H}\{x\in E:
\underline{\dim}_{\mathrm B}\Delta_xE<\frac{7}{8}\}\le\frac{9}{8}.
\]
No measurability assumption on $E$ is required.
\end{corollary}

These are internal lower-box statements. Theorem~\ref{thm:universal-main}
concerns Hausdorff dimension and pins anywhere in $\mathbb R^2$.
Under the common dyadic mass law of Theorem~\ref{thm:uniform-main},
matching profiles persist along all sequences of precisions, and
the bound $L_*(s)$ also holds for pinned Hausdorff dimension.
Appendix~\ref{sec:matching-realization} treats that case.

Fiedler and Stull combine the dimension parameters of the two endpoints
into a common minimum in part of their argument, and observe in
footnote~7 that keeping them separate could improve the bound at the
cost of more calculations~\cite[\S4.3, footnote~7]{FS}. The next theorem
carries out this refinement, keeping the source dimension
$\dim_{\mathrm H}E$ and the exceptional-set parameter $p$ separate.

\begin{theorem}[Pinned exceptional sets]\label{thm:pinnedmain}
Let $E\subset\mathbb R^2$ be analytic with $\dim_{\mathrm H}E>1$,
and let $1\le p<\dim_{\mathrm H}E$. Put
\[
B(p,t)=\min\left\{1,\frac{p+t(3-p)}{5-p}\right\}.
\]
Then
\[
\dim_{\mathrm H}\{x\in\mathbb R^2:
\dim_{\mathrm H}\Delta_xE<B(p,\dim_{\mathrm H}E)\}\le p.
\]
In particular, $\dim_{\mathrm H}\Delta_xE\ge
\min\{1,\frac{1}{4}+\frac{\dim_{\mathrm H}E}{2}\}$ outside an exceptional set of dimension
at most one.
\end{theorem}

The proof of Theorem~\ref{thm:universal-main} controls the directions of the
source points as seen from each pin. Controlling also the directions of the
pins as seen from each source point leads to a second estimate, which is better
when $S_{\mathrm{rel}}<\dim_{\mathrm H}E<S_f$. Let
\begin{align*}
 S_{\mathrm{rel}}&=1+\frac{5-8\alpha_*-3\alpha_*^2}{6\alpha_*^2-48\alpha_*+30}
 =1.124334\ldots,\\
 S_{\mathrm{switch}}&=1+\frac{5\alpha_*^2-\alpha_*-1}{10\alpha_*^2+2\alpha_*-6}
 =1.175876\ldots,\\
 S_f&=1+\frac{6+10\alpha_*-\sqrt{(6+10\alpha_*)^2-48\alpha_*(1+2\alpha_*)}}{24\alpha_*}
 =1.197246\ldots.
\end{align*}
For $S_{\mathrm{rel}}\le t\le S_f$, define
\begin{equation}\label{rel-new-pinned-function}
 Q_{\mathrm{rel}}(t)=\min\Biggl\{
 1-\frac{5(1-\alpha_*)(3-2t)}{6\{1+2\alpha_*+2(1-2\alpha_*)t\}},\quad
 1-\frac{3-2t}{4+6\alpha_*-4\alpha_*t}\Biggr\}.
\end{equation}
The first expression is the smaller one up to $S_{\mathrm{switch}}$.
The second is the smaller one after that point.

\begin{theorem}[A second pinned bound]\label{rel-new-pinned-theorem}
Let $E\subset\mathbb R^2$ be analytic with
$\dim_{\mathrm H}E\in[S_{\mathrm{rel}},S_f]$. Then
\[
\dim_{\mathrm H}\{x\in\mathbb R^2:
\dim_{\mathrm H}\Delta_xE<Q_{\mathrm{rel}}(\dim_{\mathrm H}E)\}\le1.
\]
For $S_{\mathrm{rel}}<\dim_{\mathrm H}E<S_f$, the bound $Q_{\mathrm{rel}}(\dim_{\mathrm H}E)$ exceeds
$Q(\dim_{\mathrm H}E)$, Liu's bound $\dim_{\mathrm H}E-\frac{1}{3}$ \cite[Theorem 1.1]{Li}, and $\frac{1}{4}+\frac{\dim_{\mathrm H}E}{2}$ of Theorem \ref{thm:pinnedmain}.
\end{theorem}

At $t=S_{\mathrm{rel}}$, $Q_{\mathrm{rel}}(t)=Q(t)$. At $t=S_f$,
$Q_{\mathrm{rel}}(t)=t-\frac{1}{3}$, Liu's bound.
For example, $Q_{\mathrm{rel}}(1.18)=0.856983\ldots$, while
$Q(1.18)=0.853434\ldots$. It does not change
the universal pinned constant $q_*$.

Above the threshold $q_*$ we also bound the exceptional set.

\begin{theorem}[Exceptional sets above $q_*$]
\label{thm:exceptional-quantitative}
Let $E\subset\mathbb R^2$ be analytic with $\dim_{\mathrm H}E>1$.
\begin{enumerate}[label=(\alph*)]
\item For $q_*\le q\le\frac{19}{24}$,
\[
\dim_{\mathrm H}\{x\in\mathbb R^2:\dim_{\mathrm H}\Delta_xE<q\}
\le21-6\alpha_*-\frac{15-3\alpha_*}{q}.
\]
\item For $0<k\le\frac{1}{6}$ with $k<\dim_{\mathrm H}E-1$,
\[
\dim_{\mathrm H}\{x\in\mathbb R^2:\dim_{\mathrm H}\Delta_xE<P(1+k)\}\le1+k,
\]
where
\[
P(t)=\frac{t^{3}+(3\alpha_*-2)t^{2}-(21\alpha_*+4)t+15\alpha_*+20}{9t^{2}-(6\alpha_*+39)t+50}.
\]
\end{enumerate}
The right side in (a) equals $1$ at $q=q_*$, increases with $q$, and
equals $\frac{39-42\alpha_*}{19}=1.2322\ldots$ at $q=\frac{19}{24}$.
The function $P$ is increasing on $[1,\frac{7}{6}]$, with $P(1)=q_*$ and
$P(\frac{7}{6})=0.8612\ldots$. In particular, if $1<\dim_{\mathrm H}E\le\frac{7}{6}$,
then $\dim_{\mathrm H}\{x\in E:\dim_{\mathrm H}\Delta_xE\ge q\}=\dim_{\mathrm H}E$ for
every $q<P(\dim_{\mathrm H}E)$.
\end{theorem}

Part (a) does not depend on the source dimension. Part (b) uses the source dimension
through the same iteration as $Q(\dim_{\mathrm H}E)$, and $P(\dim_{\mathrm H}E)>Q(\dim_{\mathrm H}E)$ for
$1<\dim_{\mathrm H}E\le\frac{7}{6}$.
For $\dim_{\mathrm H}E>1+k$ part (b) improves part (a), whose
threshold at exceptional dimension $1+k$ is smaller than $P(1+k)$.
At exceptional dimension $1+k$, part~(b) improves
Theorem~\ref{thm:pinnedmain} precisely when
\[
\dim_{\mathrm H}E<\frac{(4-k)P(1+k)-1-k}{2-k}.
\]
For $1<\dim_{\mathrm H}E<1.13$, its last assertion improves both
$R(\dim_{\mathrm H}E)$ of Fiedler--Stull~\cite[Corollary~2]{FS}
and Liu's bound $\frac{4\dim_{\mathrm H}E-2}{3}$ for
$\sup_{x\in E}\dim_{\mathrm H}\Delta_xE$~\cite[Theorem~1.1]{Li}.

Table~\ref{tab:intro-comparison} records the principal comparisons.
In its pinned rows, each estimate holds outside a set of pins in
$\mathbb R^2$ of Hausdorff dimension at most one. For the table,
write $t=\dim_{\mathrm H}E$ and
\[
U(t)=\max\left\{\frac{4}{5}+\eta_0,Y(t),R(t)\right\}.
\]

\begin{table}[!htbp]
\centering
\caption{Selected lower bounds obtained here and their gains over the
best previously known bounds in the corresponding settings, with
$t=\dim_{\mathrm H}E$. All gains are strictly positive throughout
the indicated ranges.}
\label{tab:intro-comparison}
\begin{tabular}{@{}p{2.45in}p{1.3in}p{1.8in}@{}}
\toprule
Setting & Earlier bound & Bound here and gain\\
\midrule
Unpinned, uniformly for $t>1$ &
$\frac{3}{4}$, Stull &
$\frac{4}{5}+\eta_0$\par Gain $\frac{1}{20}+\eta_0$\\[9pt]

Pinned, uniformly for $t>1$ &
$\frac{3}{4}$, Stull &
$q_*$\par Gain $>\frac{1}{32}$\\[9pt]

Unpinned, $1<t<5-\sqrt{15}$ &
$\frac{t(4-t)}{5-t}$\par Fiedler--Stull &
$U(t)$\par Gain $U(t)-R(t)$\\[9pt]

Unpinned, $5-\sqrt{15}\le t<\frac{37}{32}$ &
$\frac{4t-2}{3}$, Liu &
$U(t)$\par Gain $U(t)-\frac{4t-2}{3}$\\[9pt]

Pinned, $1<t\le\frac{10}{9}$ &
$\frac{1}{2}+\frac{t}{4}$, Stull &
$Q(t)$\par Gain $Q(t)-\frac{1}{2}-\frac{t}{4}$\\[9pt]

Pinned, $\frac{10}{9}<t\le S_{\mathrm{rel}}$ &
$t-\frac{1}{3}$, Liu &
$Q(t)$\par Gain $Q(t)-t+\frac{1}{3}$\\[9pt]

Pinned, $S_{\mathrm{rel}}<t<S_f$ &
$t-\frac{1}{3}$, Liu &
$Q_{\mathrm{rel}}(t)$\par
Gain $Q_{\mathrm{rel}}(t)-t+\frac{1}{3}$\\[5pt]
\bottomrule
\end{tabular}
\end{table}

The earlier estimates are due to Stull~\cite[Theorem~1]{S},
Fiedler--Stull~\cite[Corollary~2]{FS}, and
Liu~\cite[Theorem~1.1]{Li}.
The first two rows compare bounds that are uniform as $t$ decreases
to one. The remaining rows compare bounds at a fixed $t$.
For unpinned distances, the earlier bound is the larger of
$R(t)$ and $\frac{4t-2}{3}$ on the displayed ranges.
These bounds cross at $t=5-\sqrt{15}=1.127016\ldots$.
Liu's bound reaches $\frac{7}{8}$ at $t=\frac{37}{32}$,
whereas $Y(t)$ reaches it at $t=\frac{11-\sqrt{41}}{4}$.
For pinned distances at exceptional dimension one, the earlier
bound is the larger of Stull's $\frac{1}{2}+\frac{t}{4}$ and
Liu's $t-\frac{1}{3}$, which cross at $t=\frac{10}{9}$.
Fiedler--Stull's stronger dimension-dependent estimate in
\cite[Corollary~2]{FS} does not give this exceptional-dimension-one
conclusion.
Let
\[
\tau=\frac{5+14\alpha_*-\sqrt{25-8\alpha_*^2}}{12\alpha_*}
=1.191685897\ldots.
\]
$Q(t)$ exceeds Liu's bound for $1<t<\tau$.
Together, $Q(t)$ and $Q_{\mathrm{rel}}(t)$ extend this to $1<t<S_f$, and beyond $S_f$
Liu's bound is stronger and implies full dimension for
$t\ge\frac{4}{3}$.
For example, at $t=1.1$ the unpinned bound increases from
$R(1.1)=0.817948\ldots$ to $Y(1.1)=0.849576\ldots$,
and at $t=1.18$ the pinned bound increases from
$0.846666\ldots$ to $Q_{\mathrm{rel}}(1.18)=0.856983\ldots$.

\subsection{Earlier results and the inputs to this paper}

Falconer proved positive Lebesgue measure of the planar distance set
above dimension $\frac{3}{2}$~\cite{Falconer}.
Wolff lowered this
threshold to $\frac{4}{3}$~\cite{Wolff}. Guth, Iosevich, Ou and Wang
proved the positive-measure conclusion above $\frac{5}{4}$, including
a pinned conclusion with a pin in $E$~\cite{GIOW}. The present paper
concerns dimensions close to one, where the dimension problem remains
open without additional assumptions.

Bourgain's work, following Katz--Tao, gave a uniform improvement over
the elementary exponent $\frac{1}{2}$ at source dimension one
\cite{Bourgain,KatzTao}. For sources of dimension greater than one,
Keleti--Shmerkin obtained the universal unpinned bound $\frac{37}{54}$
and the pinned bound $\frac{2}{3}\dim_{\mathrm H}E$ outside a set of
pins of dimension at most one when
$1<\dim_{\mathrm H}E\le\frac{3}{2}$~\cite[Theorems~1.4 and~1.1]{KS}.
Shmerkin improved the universal unpinned exponent to $\frac{40}{57}$
and obtained the pinned exponent $\frac{29}{42}$ outside a set of
pins of dimension at most one~\cite[Theorems~1.1 and~1.2]{Sh}. Their treatment of information across scales provides
a measure-theoretic precursor to our use of information at
intermediate scales.

The point-to-set principle of J. H. Lutz and N. Lutz expresses
Hausdorff dimension through the description lengths of individual
points~\cite{LL}. Lutz--Stull developed estimates for these
description lengths under orthogonal projection, the map
$z\mapsto z\cdot v$ for a unit vector $v$~\cite{LS}.
Stull used this approach to obtain the pinned bound quoted
above~\cite{S}. With the covering number $N_r(F)$ defined above, the upper box dimension of $F$
is $\overline{\dim}_{\mathrm B}F=\limsup_{r\searrow0}
\frac{\log N_r(F)}{\log(\frac{1}{r})}$. The packing dimension is
\[
\dim_{\mathrm P}E=\inf_{E\subseteq\bigcup_iE_i}
\sup_i\overline{\dim}_{\mathrm B}E_i,
\]
where the infimum is over countable covers by bounded sets
\cite[Chapter~5]{Mattila}. Fiedler--Stull refined the scale
partitions and retained both Hausdorff and packing dimension
parameters~\cite[Theorems~1 and~5]{FS}. At planar
packing dimension two, their expression becomes
$\min\{1,\frac{d(4-d)}{5-d}\}$ for $1<d<\dim_{\mathrm H}E$, with
an exceptional-pin dimension bound $d$. Sections~\ref{sec:borrowed}
and~\ref{sec:transfer} specify the information estimates retained from
these proofs. References to their equation and lemma numbers use the
listed arXiv versions: \cite{S}, v2, and \cite{FS}, v1.
Appendix~\ref{app:retained-prefix}
contains the retained prefix calculation and its endpoint case.

Liu's $L^2$ method provides additional dimension-dependent estimates
\cite{Li}. At exceptional-pin dimension one, it yields the threshold
$\min\{1,\dim_{\mathrm H}E-\frac{1}{3}\}$. It also implies
\[
\sup_{x\in E}\dim_{\mathrm H}\Delta_xE
\ge\min \left\{\frac{4\dim_{\mathrm H}E-2}{3},1 \right\}.
\]
The supremum is needed because of the correction to the pinned endpoint
in~\cite{Li}. The pin achieving an exponent close to the right side may
depend on that exponent. Liu states these results for compact $E$.
They extend to analytic $E$ by inner approximation with compact subsets.

Regularity can lead to stronger conclusions. A compact set $E$ of positive diameter is
$s$-Ahlfors--David regular if it supports a finite Borel measure
$\mu$ with support equal to the set and
$cr^s\le\mu(B(z,r))\le Cr^s$ for all points $z$ of the set and
$0<r\le\operatorname{diam}E$.
The constants are positive and
independent of $z,r$. Orponen proved that the distance sets of such
planar sets have packing dimension one when $s\ge1$~\cite{OrponenRegular}. Shmerkin obtained
full pinned Hausdorff dimension when
$\dim_{\mathrm H}E=\dim_{\mathrm P}E>1$~\cite{ShmerkinRegular}.
Fiedler--Stull cover further ranges of these two dimensions
\cite{FS,FSLow}. Shmerkin--Wang proved that for every $d\ge2$ and every Borel
$E\subseteq\mathbb R^d$ with
$\dim_{\mathrm H}E=\dim_{\mathrm P}E=\frac{d}{2}$,
$\sup_{x\in E}\dim_{\mathrm H}\Delta_xE=1$
\cite[Theorem~1.4]{SW}. Liu recently proved that if $E,F\subset\mathbb R^2$ are Borel, $\dim_{\mathrm H}E>1$, $\dim_{\mathrm H}F=\dim_{\mathrm P}F$ and $\dim_{\mathrm H}E+\dim_{\mathrm H}F>2$, then $\Delta_xE$ has positive Lebesgue measure for some $x\in F$~\cite[Theorem~1.1]{Liu2026}. At source dimension at most one, Cholak, Cs\"ornyei, N. Lutz,
P. Lutz, Mayordomo and Stull proved the pinned supremum bound
$\frac{3}{4}\dim_{\mathrm H}E$ for analytic $E$~\cite[Theorem~6]{CCLLMS}.

A closely related combination of algorithmic information and
incidence geometry appears in recent work of Cs\"ornyei and
Stull~\cite{CSHyperplanes}. They extend the Ren--Wang incidence
estimate to hyperplanes in higher dimensions and prove that,
for every Borel set $E\subset\mathbb R^d$ with
$\dim_{\mathrm H}E>\frac{d}{2}$, where $d=3$ or $4$,
\[
\sup_{x\in E}\dim_{\mathrm H}\Delta_xE\ge\frac{2}{3}.
\]
Their proof converts incidence estimates into conditional
Kolmogorov complexity bounds and combines them over intervals
of scales. These intervals are selected by a decomposition of
complexity functions based on Jones's traveling salesman theorem
\cite[Section~5]{CSHyperplanes}. In dimension two, their
incidence theorem specializes to the Ren--Wang estimate used
here. Our planar improvement combines information retained
at both endpoints with a common source measure to obtain
incompatible lower and upper counts for rich perpendicular
bisectors.

Our geometric argument uses further inputs. The radial projection
about $x$ sends $y\ne x$ to its direction
$\pi_x(y)=\frac{y-x}{|y-x|}$ on the unit circle
$S^1=\{z\in\mathbb R^2:|z|=1\}$. Orponen's theorem controls the
distribution of these directions~\cite{O}. The crossing lemma, in the circle-drawing tradition
of Sz\'ekely, turns many circle incidences into repeated pairs
\cite{Szekely}. The restricted-sum estimate of Katz--Tardos bounds from below
the number of bisector directions produced by rows, meaning
ordered lists of source points assigned to a common circle~\cite{KT}, and the max-flow min-cut theorem of
Ford--Fulkerson turns this into a simultaneous choice of one
direction per row~\cite{FF}. The Furstenberg estimate of Ren--Wang
bounds the number of rich bisector tubes from below~\cite{RW}, and
the incidence estimate of Fu--Gan--Ren for rectangularly spaced
balls bounds it from above~\cite{FGR}. We state the precise forms
needed when they enter the proof.

A related use of rigidity occurs in the work of H\'era, Shmerkin
and Yavicoli on Furstenberg sets~\cite[Section~4.1]{HSY}.
There, a configuration close to the elementary bound has a
restricted incidence structure, which is incompatible with
Bourgain's projection theorem. Here the starting constraint is
equality in the reciprocal information estimates. It forces a
common interval of excess information, from which we obtain
rich bisectors. The final contradiction uses the geometric
estimates specified above.

Several features of the proof also have counterparts in work on the
Erd\H{o}s similarity conjecture. This conjecture asks whether, for
every infinite $A\subset\mathbb R$, there is a measurable set of
positive Lebesgue measure containing no set $x+\lambda A$ with
$x\in\mathbb R$ and $\lambda\ne0$.
Shmerkin and Yavicoli~\cite[Section~3.2]{SYUniversality} use
probabilistic selection to construct one small set whose sums with
every member of a prescribed family cover a fixed region. The cost
of treating the family simultaneously is logarithmic in the number
of discrete tests. The same requirement of a common choice enters
Lemmas~\ref{lem:source-thinning} and~\ref{lem:pin-coarsening} below.
The source sample must serve every retained pin, and the pin sample
must preserve the witnesses for every retained source pair.
Separate choices would lose the shared incidence structure.

A second connection concerns approximate coincidences of sums.
Mora Cuellar, Iosevich, Kulkarni, Rojas Aravena and
Yavicoli~\cite[Section~3.2]{MIKRYSumsets} bound their number for two
lacunary sequences, meaning positive sequences $(a_j)$ satisfying
$a_{j+1}\le q a_j$ for some fixed $0<q<1$.
Their estimate allows many separated sums to be extracted without
requiring the two sequences to occupy different ranges of scales.
Here sums of angular coordinates determine bisector directions,
so a related counting question enters Section~\ref{sec:finite-incidence}.
Our angular rows need not be lacunary, and we use the restricted-sum
estimate of Katz--Tardos under the hypotheses available to us.

The lattice construction of Iosevich and
Yavicoli~\cite[Section~2]{IYLattice} illustrates a further role for
information across scales. Sufficient branching between successive
lattice scales allows the construction of an infinite set $A$ whose
triple sumset $A+A+A=\{a_1+a_2+a_3:a_1,a_2,a_3\in A\}$ lies in the
limiting set. In the present argument, excess information over a
common interval of scales produces the mass bounds from which the
rich bisectors are extracted. Both arguments use how branching is
distributed across scales. The arithmetic alignment of the lattice
construction, however, is additional structure that a general
Frostman measure need not possess.

There is also a comparison with the work of Iosevich, Kulkarni,
Mora Cu\'ellar, Rojas Aravena and Yavicoli on Rajchman
measures~\cite[Lemma~2.3 and Proposition~3.2]{IKMRYRajchman}.
A Rajchman probability measure is a Borel probability measure whose
Fourier--Stieltjes transform tends to zero at infinity. For such a
measure with compact support, that qualitative decay yields a density
statement modulo one for all sufficiently large dilates of the
support. This leads to a single set excluding all nontrivial affine
copies of that support. Our endpoint stability argument likewise
separates a qualitative limiting statement from the finite-scale
geometry, with the measure affecting the starting precision.
Here the required input is supplied by information profiles and
incidence estimates. No Rajchman hypothesis is imposed.
These comparisons describe related uses of selection, additive
structure and scale information. The estimates needed for the
distance problem are proved below.

\subsection{The bisector results as separate statements}
For distinct $y,z\in\mathbb R^2$, their perpendicular bisector is
\[
\mathcal B(y,z)=\{x\in\mathbb R^2:|x-y|=|x-z|\}.
\]

The perpendicular bisectors $\mathcal B(y,z)$ defined above will
be treated as a family of lines.
Write a line as $L_{n,c}=\{z\in\mathbb R^2:n\cdot z=c\}$, where $n\in S^1$
is a unit normal and $c\in\mathbb R$ is its offset. The parameters
$(n,c)$ and $(-n,-c)$ represent the same unoriented line. We give
this space the metric
\[
\operatorname{dist}(L_{n,c},L_{n',c'})
=\min\{|n-n'|+|c-c'|,|n+n'|+|c+c'|\}.
\]
The dimension of a family of lines is taken in this metric. A tube is a neighborhood of a line
of specified width. Relative to a finite pin set, a tube is called
$k$-rich if it contains at least $k$ pins.

Several geometric statements below are independent of the information
theory. Here the Hausdorff dimension of a family of lines is computed
in the space of affine lines equipped with the metric above,
not as the dimension of their union in $\mathbb R^2$. Proposition~\ref{prop:bisector-dimension} establishes the elementary
identity
\[
\dim_{\mathrm H}\{\mathcal B(y,z):z\in F\}
=\dim_{\mathrm H}F \qquad \text{for } y\notin F.
\]
It follows that the set of all bisectors determined by $E$ has
Hausdorff dimension at least $\dim_{\mathrm H}E$ when this dimension
is positive. The reflection map proves this identity directly.
It supplies a Hausdorff-dimension statement to accompany the finite
results.

The finite geometric estimates used below have their own hypotheses.
They concern approximate circle drawings, rows of points with rich
bisectors, the distribution of their directions, and rich tubes for
point sets with controlled spacing. These statements can also be
used independently of the description-length estimates below.

\subsection{How the proofs work}\label{sec:intro-proof-outline}

A short description of a distance places constraints on both
endpoints. The estimates of Stull and Fiedler--Stull express these
constraints at each scale. We retain their dependence on the
information at both endpoints. In the unpinned problem, the resulting
inequalities can be applied in both directions. They force the
scalar bound $\frac{4}{5}$ and identify the profiles near equality. For a prescribed pin they need not do so. They instead
force a particular distribution of pin and source mass, which we
use to construct a finite configuration of circles and points.

The geometric contradiction concerns one family of perpendicular
bisectors. Each bisector has many pins near it. The source points
on the circles provide many possible bisector directions, and
Katz--Tardos and Ren--Wang turn those choices into a lower bound
for the number of distinct bisectors. The distribution of the pins
provides an upper bound for the same number. The improvement occurs
when these bounds are incompatible. Thus information theory determines
the configuration that a small distance set would require, and
incidence geometry rules out that configuration.

There are two reasons for keeping more information than a single
dimension exponent. The first is that information may accumulate
unevenly across scales. Information missing on one interval has to
appear on another, and the projection estimates can use its location.
The second is that all pins see the same source set. A mass bound
for one source measure can be used simultaneously in the circle
drawings at different pins. Choosing a separate source measure for
each pin would lose the shared vertices needed for the counting
argument. We now explain both points and how the estimates fit
together.

Let $x,y\in\mathbb R^2$ be distinct. The line segment joining them is
\[
 [x,y]=\{(1-t')x+t'y:t'\in[0,1]\},
\]
and its length is $|x-y|$. We measure information by program length.
Fix a universal binary machine with prefix-free programs, meaning
that no admissible program is a proper initial segment of another.
An oracle $A\subseteq\mathbb N$ supplies background information.
The machine may ask whether an integer belongs to $A$ without
including the answer in its program. The oracle is fixed before
the points and the sequence of precisions are chosen.

For $z\in\mathbb R^m$, let $K_r^A(z)$ be the length of a shortest
program, using $A$, that locates $z$ within distance $2^{-r}$.
The notation $K_r^{A,x}(|x-y|)$ also supplies the complete pin $x$,
so the program describes only the distance. Section~\ref{sec:complexity}
specifies the encodings and the distinction between supplying a
whole point and supplying only a coarse approximation. These are
the conventions of~\cite{LiVitanyi,S,FS}.

To record where the information lies, choose terminal precisions
$r_j\to\infty$ and write $u\in[0,1]$ for a fraction of terminal
precision. Along a common subsequence we have
\begin{align*}
 f(u)&=\lim_{j\to\infty}\frac{K_{\lfloor ur_j\rfloor}^A(y)}{r_j},&
 g(u)&=\lim_{j\to\infty}\frac{K_{\lfloor ur_j\rfloor}^A(x)}{r_j},\\[6pt]
 d_x(u)&=\lim_{j\to\infty}\frac{K_{\lfloor ur_j\rfloor}^{A,x}(|x-y|)}{r_j},&
 d_y(u)&=\lim_{j\to\infty}\frac{K_{\lfloor ur_j\rfloor}^{A,y}(|x-y|)}{r_j}.
\end{align*}
Here $\lfloor v\rfloor$ and $\lceil v\rceil$ denote integer rounding
down and up, respectively, and the complexities at precision zero
are set equal to zero. These functions are the information profiles.
The source profile is $f$, the pin profile is $g$, and the distance
has one profile for each choice of pin. The points remain fixed.
Section~\ref{sec:profiles} proves uniform convergence of all four
profiles along one subsequence.

The pairs used in the proof have directions requiring nearly one bit
per unit of precision, even after either endpoint is supplied.
Also, supplying one endpoint does not substantially shorten a
description of the other. Orponen's radial projection theorem and
the measure-coding argument allow us to select such pairs. Under
these hypotheses, the prefix estimate reads
\[
 d_x(u)\ge\frac{u}{2}+\frac{f(u)}{4}.
\]
It says how much information already accumulated at the source must
have reached the distance by scale $u$. At the lowest profile
$f(u)=u$, this is $d_x(u)\ge\frac{3u}{4}$. At the largest planar
profile $f(u)=2u$, it is $d_x(u)\ge u$. To improve the first bound,
we must use how the profiles change between scales.

This is the purpose of the interval transfer estimate. Suppose a
coarse location of $y$ and its distance from $x$ are known. Another
possible source point $z$ must be near that location and satisfy
$|x-z|=|x-y|$. Thus $x$ lies on the perpendicular bisector of
$[y,z]$. The points $y,z$ determine a projection of $x$, at a
precision governed by their separation. The projection estimates
of Stull and Fiedler--Stull therefore put a lower bound on the
description length of every competitor $z$. If all competitors
require many bits, the distance and the coarse location must
together contain many bits as well.

The transfer threshold $H_{f,g}(a,b)$ in~\eqref{eq:H} is the least
of these competitor costs, capped by the source information at
the terminal scale. Lemma~\ref{lem:transfer} states the resulting bound
\[
 d_x(b)-d_x(a)\ge[H_{f,g}(a,b)-f(a)]_+.
\]
The subtraction of $f(a)$ accounts for the coarse source information
already supplied. Partitioning the projection calculation into
intervals with different patterns of information growth strengthens
$H$ to $\widehat H$ in Lemma~\ref{lem:fullpartition}. Both estimates
retain the two endpoint profiles.

For the unpinned problem, a short description of $|x-y|$ remains
short after either endpoint is supplied. We may therefore assume
that both $d_x(1)$ and $d_y(1)$ are small and use the transfer
inequalities in both orientations. A common cutoff is a scale
after which all competitors in either orientation cost at least
a specified amount of information, called the transfer level.
Prefix and transfer then bound both endpoint profiles from below
at the cutoff. Their Lipschitz bounds carry some of this information
to earlier scales, where it strengthens the next transfer estimate.

The final contradiction uses the least possible common cutoff.
At this scale one orientation must have a competitor whose cost
equals the transfer level. Otherwise the cutoff could be moved
slightly earlier. We retain the two separate endpoint excesses,
apply the partition estimates in both directions, and return to
this equality. The resulting lower bound exceeds the value allowed
by the equality. Section~\ref{sec:reciprocal} proves in this way that
$d_x(1),d_y(1)$ cannot both be smaller than $\frac{4}{5}$.
Keeping the source dimension in the transfer level also yields
$Y(\dim_{\mathrm H}E)$ in Section~\ref{sec:unpinned}. The fixed-pair
construction and the point-to-set principle in
Section~\ref{sec:classical} convert these conclusions into bounds
for $\dim_{\mathrm H}\Delta E$. Proposition~\ref{prop:direct-terminal-flat}
shows why the stated scalar system stops at $\frac{4}{5}$.
At equality, however, the profiles have much more structure.
They grow at the maximal rate before $\frac{1}{5}$ and remain
above $u+\frac{1}{5}$ until $\frac{3}{5}$. This interval of
excess information allows the geometric count to pass the scalar
endpoint. Compactness makes the same contradiction valid in a
fixed neighborhood of $\frac{4}{5}$.

For a prescribed pin, this argument has a specific failure. A
description that is short when $x$ is supplied may be long when
only $y$ is supplied. Thus $d_x(1)$ may be small while $d_y(1)$ is
large. Proposition~\ref{prop:pinned-obstruction} exhibits abstract
profiles satisfying all the scalar inequalities with $d_x(1)$
arbitrarily close to $\frac{3}{4}$ and $d_y(1)=1$. The pinned
improvement must therefore use a relation among many pairs.

Suppose that the set of pins where the proposed bound fails has
Hausdorff dimension greater than one. We choose a source probability
measure $\mu$ and a pin probability measure $\nu$, supported on
separated compact sets, both with Frostman exponent greater than
one. A relation of pairs is a Borel set $G$ in the product of these
supports. Its fiber at a pin is $G_x=\{y:(x,y)\in G\}$. We retain
one such relation with
\[
 (\nu\times\mu)(G)=\int\mu(G_x)d\nu(x)>0.
\]
The radial projection theorem supplies bounded angular densities
on this relation after restriction. It is these restrictions of
the same original measures that are used throughout the geometry.

Fix a proposed distance bound $q$. For every pair in the retained
relation, the inequality $K_n^{A,x}(|x-y|)\le qn$ holds at infinitely
many precisions. These precisions may differ from pair to pair.
Lemma~\ref{lem:seventh-classical-recurrence} finds infinitely many
common precisions at which the pairs satisfying this inequality
have total mass at least $n^{-2}$. At one such precision there are
at most $2^{qn+O(\log n)}$ occupied radius bins per pin. Indeed,
there are only that many short programs, and an approximation to
a radius identifies its bin up to a bounded choice. Thus the small
distance set produces many source incidences on relatively few
circles at each pin.

The profile argument supplies two further facts at these precisions.
First, $g(u)-u\ge W$ throughout an interval $[u_-,u_+]$ of controlled
length, with $W>0$. This is the pin plateau. The profile need not be
constant on that interval. Coding the pin measure turns the plateau
into upper bounds for pin mass over a range of scales. Second, the
source reaches a prescribed amount of information early enough
that a lower bound for the radius complexity remains after a
coarse source location has been supplied. This is the source-parent
estimate in Section~\ref{sec:parent-cap}.

The cancellation in that second fact is worth displaying. Set
$\delta=2^{-n}$, and let $C$ be a dyadic parent square containing
$y$. A dyadic square has the form
$2^{-k}([i,i+1)\times[j,j+1))$, and a parent is a containing square
at a coarser level. If $J$ is the radius interval of length $\delta$
containing $|x-y|$, put
\[
 p_{C,J}=\mu\{z\in C:|x-z|\in J\}.
\]
The transfer estimate, the source coding identity, and the code
for the radial distribution of $\frac{\mu|_C}{\mu(C)}$ imply,
up to the selectable small errors,
\[
 \lambda n+\log_2\mu(C)
 \le K_n^{A,x}(|x-y|\mid C)
 \le-\log_2p_{C,J}+\log_2\mu(C).
\]
The parent mass cancels, leaving $p_{C,J}\le\delta^\lambda$ at
leading exponent order. Proposition~\ref{prop:conditional-parent-cap}
states the precise bound with its error tolerance. It holds for
the original source measure in every occupied parent and annulus.
Thus many incidences cannot all be hidden in one small part of
a circle. If the source dimension is larger, its stronger profile
lower bound permits a larger transfer level $\lambda$, as proved
in Section~\ref{sec:bootstrap}.

We now select source squares once for all pins and use their centers
as vertices. An assigned incidence records a retained pair, its
source square, and its radius bin. The vertex lies within
$O(\delta)$ of the associated circle. The source-parent mass bound
controls how many assigned vertices can occur in one parent on
that circle. The common-scale mass lower bound allows the selections
to lose fixed powers of $n$. Their logarithms are $o(n)$, so they
do not change the limiting exponents.

Join consecutive assigned vertices along the circles. A drawing
with many edges but controlled crossings must repeat many source
pairs. A repeated pair $y,z$ is seen from many circle centers.
Those centers lie within distance $O(\frac{\delta}{|y-z|})$ of
the perpendicular bisector of $[y,z]$. A witnessing pin is a pin
to which both source vertices were assigned on the same circle.
We have therefore constructed bisector tubes with many witnesses,
as illustrated in Figure~\ref{fig:shared-pair-bisector}.
The drawing and pruning arguments also provide short lists of
vertices on a circle, called rows, from which pairs can be chosen.

\begin{figure}[htbp]
\centering
\begin{tikzpicture}[x=1.05cm,y=1.05cm,>=stealth,
every node/.style={font=\small}]
\fill[blue!9] (-.13,-.55) rectangle (.13,3.45);
\draw[blue!60!black,dashed] (0,-.55)--(0,3.45);
\coordinate (y) at (-1,0);
\coordinate (z) at (1,0);
\coordinate (xone) at (0,1.4);
\coordinate (xtwo) at (0,2.7);
\draw[blue!65!black,thin] (y)--(xone)--(z);
\draw[orange!85!black,thin,dashed] (y)--(xtwo)--(z);
\draw[gray!70] (y)--(z);
\draw[blue!65!black,thick]
 plot[domain=215:325,samples=80] ({sqrt(2.96)*cos(\x)},{1.4+sqrt(2.96)*sin(\x)});
\draw[orange!85!black,thick]
 plot[domain=238:302,samples=80] ({sqrt(8.29)*cos(\x)},{2.7+sqrt(8.29)*sin(\x)});
\fill (y) circle (.045) node[above left] {$y$};
\fill (z) circle (.045) node[above right] {$z$};
\fill[blue!65!black] (xone) circle (.045) node[right=4pt] {$x_1$};
\fill[orange!85!black] (xtwo) circle (.045) node[right=4pt] {$x_2$};
\node[anchor=west,align=left] at (2.45,2.85)
{The same source pair lies\\on circles about different pins.};
\draw[->,gray!70] (2.28,2.5)--(.35,2.3);
\node[anchor=west,align=left] at (2.45,1.45)
{The circle centers lie on\\the perpendicular bisector.};
\draw[->,gray!70] (2.28,1.15)--(.18,.75);
\node[anchor=west,align=left] at (2.45,-.1)
{With incidence error $O(\delta)$,\\the centers lie in a tube of width\\$O(\frac{\delta}{|y-z|})$.};
\node[align=center] at (0,-.92) {One source pair, many witnessing pins};
\end{tikzpicture}
\caption{Why repeated circle incidences produce rich bisectors.
The two circles shown pass exactly through $y,z$. The finite
construction has errors of order $\delta$, so its witnesses lie
near the bisector. The shaded tube is enlarged for visibility.
The same source vertices are used at every pin.}
\label{fig:shared-pair-bisector}
\end{figure}
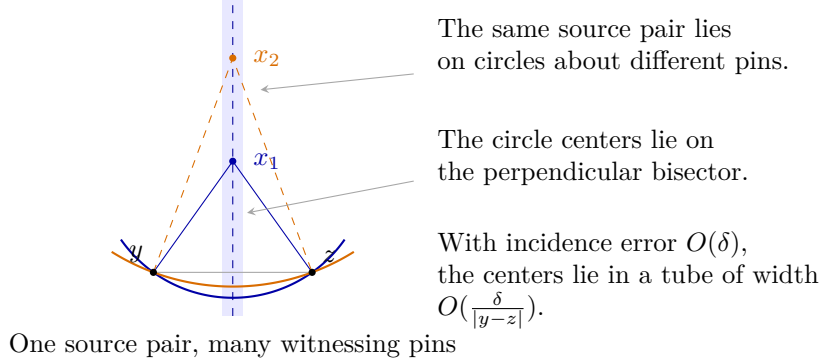

The lower count uses the freedom in this choice. On an exact circle,
if two source points have angles $\theta_i,\theta_j$ about the pin,
their bisector direction is $\frac{\theta_i+\theta_j}{2}$ modulo
$\pi$. Our approximate circles obey the corresponding relation
with controlled errors. After removing repeated angular entries,
the Katz--Tardos restricted-sum estimate bounds the number of
available directions for collections of rows. A large total number
of directions would not by itself suffice, since they might cluster in
a short angular interval. The matching argument in
Section~\ref{sec:finite-incidence} chooses one pair per row while
controlling the directions at every intermediate scale. Ren--Wang's
Furstenberg theorem can then be applied to the pins and their
selected tube directions. It forces many distinct bisector tubes.

The first upper count uses the same tubes and their witnessing pins.
The pin plateau controls concentration in rectangles. The sampling
argument of Section~\ref{sec:rectangular-sampling} converts that
control into the rectangular occupancy required by Fu--Gan--Ren.
Their incidence theorem limits how many distinct tubes can each
contain the required number of pins. Comparing this upper count
with the Furstenberg lower count proves
$Q(\dim_{\mathrm H}E)$, including the universal bound $q_*$.

For the second upper count, keep the angular bound as seen from
the source endpoint as well as from the generating pin. For a
source pair separated by $\ell$, a length-$D$ portion of its
bisector subtends an angle of order $\ell D$ from that endpoint,
in our separated regions. The reverse angular bound therefore
limits how many witnesses of this pair can lie in one small
square. A rich pair must have witnesses in many such squares.
Lemma~\ref{rel-lem:RW-linear-C} retains the dependence on their
coarse tube occupancy and turns this additional distribution
information into a second upper count. It improves the pinned
bound to $Q_{\mathrm{rel}}(\dim_{\mathrm H}E)$ on the interval
specified in Theorem~\ref{rel-new-pinned-theorem}.

The distinction between these counts is useful. Katz--Tardos
controls the choices of directions, Ren--Wang turns distributed
directions into a lower count, and the upper incidence estimates
use the distribution of witnesses. All three refer to the same
finite configuration. Figure~\ref{fig:proof-outline} records these
dependencies.

\begin{figure}[!htbp]
\centering
\begin{tikzpicture}[x=1cm,y=1cm,>=stealth,line width=.55pt,
every node/.style={align=center,font=\small,inner sep=6pt}]
\node[draw,text width=12cm] (radial) at (0,0)
{Orponen's radial projection theorem\\
Good directions and bounded angular densities on one relation};
\node[draw,text width=12cm] (input) at (0,-1.45)
{Information estimates of Stull and Fiedler--Stull\\
Retain both endpoint profiles in the prefix and transfer bounds};
\node[draw,text width=5.4cm] (unpin) at (-3.3,-3.2)
{Reciprocal estimates\\Scalar bounds $\frac{4}{5}$ and $Y(t)$\\
Rigidity near the endpoint};
\node[draw,text width=5.4cm] (pin) at (3.3,-3.2)
{Prescribed pin\\Pin plateau and\\source-parent annular cap};
\node[draw,text width=12cm] (geometry) at (0,-5.15)
{One common source measure\\Circle rows whose pairs have rich bisectors};
\node[draw,text width=12cm] (directions) at (0,-6.7)
{Katz--Tardos and simultaneous direction selection\\
Bisector directions distributed over all relevant scales};
\node[draw,text width=5.4cm] (upper) at (-3.3,-8.5)
{Upper counts\\Fu--Gan--Ren and the pin plateau\\
Further count from reverse angles};
\node[draw,text width=5.4cm] (lower) at (3.3,-8.5)
{Lower count\\Ren--Wang\\Furstenberg estimate};
\node[draw,text width=12cm] (result) at (0,-10.35)
{Incompatible bisector counts\\
Unpinned $\frac{4}{5}+\eta_0$; pinned $Q(t)$ and $Q_{\mathrm{rel}}(t)$};
\draw[->] (radial) -- (input);
\draw[->] (input.south) -- (0,-2.2) -| (unpin.north);
\draw[->] (0,-2.2) -| (pin.north);
\draw[->,dashed] (unpin.south) -- (unpin.south |- geometry.north);
\draw[->] (pin.south) -- (pin.south |- geometry.north);
\draw[->] (geometry) -- (directions);
\draw[->] (directions.south) -- (0,-7.45) -| (upper.north);
\draw[->] (0,-7.45) -| (lower.north);
\draw[->] (upper.south) -- (upper.south |- result.north);
\draw[->] (lower.south) -- (lower.south |- result.north);
\end{tikzpicture}
\caption{The information and incidence arguments. The reciprocal
inequalities prove the scalar bounds. The dashed arrow records the
further use of endpoint rigidity in the geometric argument. In every
geometric application, the same bisectors are counted from below
through their directions and from above through the distribution of
their witnesses.}
\label{fig:proof-outline}
\end{figure}
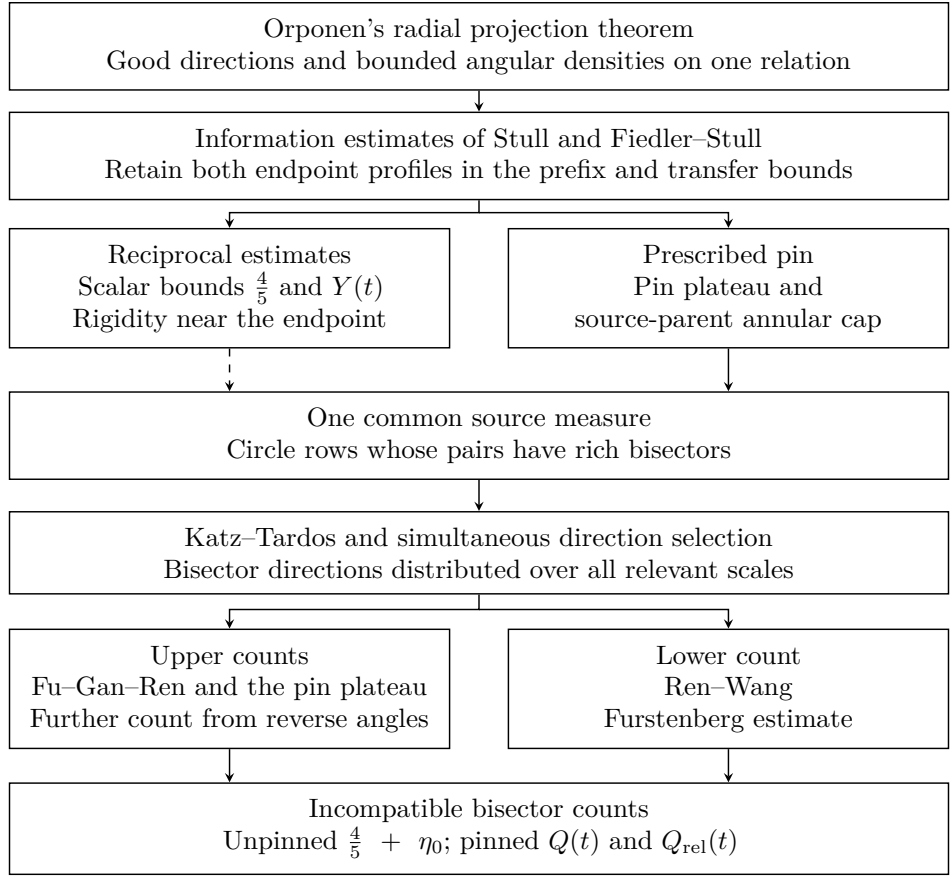

Keeping the pin Frostman exponent in the same lower count yields
the quantitative exceptional-set estimates of
Section~\ref{sec:exceptional}. Appendix~\ref{sec:matching-realization}
combines the matching scalar estimate with the geometric argument
when almost every pair has matching information profiles. The elementary
scalar proofs are retained in
Appendices~\ref{sec:thirteen-seventeen} and~\ref{sec:refined}.
Section~\ref{sec:gain} identifies the limitations of the scalar
inequalities and the separate geometric comparisons.
Appendix~\ref{sec:joint-count-obstruction} treats their joint
optimization.

\section{Preliminaries}\label{sec:preliminaries}

We collect the geometric conventions and the definitions of information
used in the proofs.
The basic object is a shortest program producing a rational point
within a prescribed error. All programs below use the same optimal
machine. The pair $(x,y)$ and the background oracle $A$ remain fixed
as the error tends to zero. This is what will allow a statement about
individual points to imply a Hausdorff-dimension bound.

\subsection{Sets, measures, and geometric scales}

The Hausdorff measure $\mathcal H^s$, Hausdorff and packing
dimensions, and Frostman and Ahlfors--David bounds were defined in
the introduction. See also~\cite[Chapters~4, 5, 8]{Mattila}. We will use countable stability of Hausdorff dimension. The dimension of a countable union is
the supremum of the dimensions of its members \cite[Chapter~4]{Mattila}.

The restriction $\mu|_F$ is the
measure $B\mapsto\mu(B\cap F)$, and its normalization is
$\frac{\mu|_F}{\mu(F)}$ when $0<\mu(F)<\infty$.

Frostman's lemma says that every analytic set of dimension greater
than $s$ contains a compact subset carrying an $s$-Frostman
probability measure~\cite[Chapter~8]{Mattila}. Its energy
\[
I_t(\mu)=\iint |x-y|^{-t}d\mu(x)d\mu(y)
\]
is finite for $0<t<s$. We apply the radial projection $\pi_x$ only on source regions
separated from $x$. For a measurable map $T$, its pushforward
$T_*\mu$ is defined by $(T_*\mu)(B)=\mu(T^{-1}(B))$. A measure is
absolutely continuous with respect to another if it assigns zero
mass to every set of zero mass for the other. A density of $\nu$
with respect to $\mu$ is a nonnegative measurable function $h$
such that $\nu(B)=\int_B h d\mu$ for every measurable $B$.
This is also called the Radon--Nikodym derivative. Angular
Lebesgue measure means arclength
on $S^1$. A bound $(\pi_x)_*\mu\le L\mathcal H^1$ means that every
arc $J$ has preimage mass at most $L|J|$, where $|J|$ is its length.
For a Borel relation $G$ we use the fibers $G_x=\{y\in\mathbb R^2:(x,y)\in G\}$
and $G^y=\{x\in\mathbb R^2:(x,y)\in G\}$ and the pair mass $(\nu\times\mu)(G)$.

For a set $F\subseteq\mathbb R^m$,
$\operatorname{dist}(z,F)=\inf_{w\in F}|z-w|$ and
$F^{(r)}=\{z\in\mathbb R^m:\operatorname{dist}(z,F)<r\}$. A set is
$r$-separated if distinct points have distance at least $r$.
A fixed enlargement of a square is its concentric dilation by a
fixed factor. We use the Hausdorff metric on nonempty compact sets,
\[
d_{\mathrm H}(K,K')=\max\left\{
\sup_{z\in K}\operatorname{dist}(z,K'), \quad
\sup_{w\in K'}\operatorname{dist}(w,K)\right\}.
\]
A map is $L$-Lipschitz if it increases distances by at most a factor
$L$.
It is bi-Lipschitz if it is one-to-one and it and its inverse
are Lipschitz. Such maps preserve Hausdorff dimension
\cite[Chapter~7]{Mattila}. A lower envelope for a function is
another function lying below it at every argument. We also call a
lower bound a barrier, and a constant lower bound a floor.
For functions on a fixed compact interval,
$\|f-g\|_\infty=\sup_u|f(u)-g(u)|$ is their uniform distance.
A functional is lower semicontinuous if its value at a limit is
at most the lower limit of its values along every convergent
sequence.
Upper semicontinuity reverses this inequality and uses
the upper limit. For sets $R_n$, the set
$\limsup_nR_n=\bigcap_N\bigcup_{n\ge N}R_n$ consists of points
belonging to infinitely many $R_n$.

We write $[t]_+=\max\{t,0\}$, $\#F$ for cardinality, and $1_G$ for
the indicator of $G$, equal to one on $G$ and zero outside it.
Logarithms are to base two. Dyadic cubes in $\mathbb R^m$ are the
half-open products $\prod_{j=1}^m[2^{-n}k_j,2^{-n}(k_j+1))$.
In the plane they are squares. We write $\mathcal D_r$ for the grid
of dyadic squares of side $r$, and $Q_n(z)$ for the unique
level-$n$ cube containing $z$. A dyadic parent or ancestor is a
containing cube at a specified coarser level. A dyadic interval
has the analogous one-dimensional meaning. For $z\ne0$,
$\arg z$ denotes its angle modulo $2\pi$.
An angular chart chooses
a real representative on an arc shorter than a full circle.
The notation $A\lesssim B$ means $A\le CB$ with the stated
parameters fixed, and $A\gtrsim B$ means $B\lesssim A$, and
$A\asymp B$ or $A\sim B$ means both inequalities. The notation
$A\ll B$ requires a sufficiently small ratio $\frac{A}{B}$, or a
ratio tending to zero when a limit is specified. As usual,
$O(B)$ denotes a quantity bounded in absolute value by $CB$ and
$o(B)$ one whose ratio to $B$ tends to zero in the indicated limit.
A loss
$\delta^{-\tau}$, where $\tau>0$ is fixed and small, is called a
small power loss. It is distinguished from a fixed multiplicative
constant, which does not change a dimension exponent.

We use the line-space metric defined in the introduction.
In local coordinates, write a line as
\[
L_{\theta,c}
=\{x\in\mathbb R^2:(\cos\theta,\sin\theta)\cdot x=c\},
\]
where $\theta$ is the angle of its unit normal and $c$ is its
signed offset from the origin. If $\theta$ ranges over a fixed
interval of length less than $\pi$ and $c$ remains bounded, then
\[
\operatorname{dist}(L_{\theta,c},L_{\theta',c'})
\asymp |\theta-\theta'|+|c-c'|.
\]
The comparison constants depend only on these parameter ranges.
A $\rho$-tube is a neighborhood of width
$C\rho$ of a line, with $C$ fixed. Tube separation means separation
of their line parameters. A family of $\rho$-tubes is
parameter-separated if the defining lines of distinct members have
line-space distance at least $c\rho$, where $c>0$ is fixed.
These conventions make the covering numbers of line families precise. Dyadic slope--intercept tubes
used in the Ren--Wang estimate are introduced in
Section~\ref{sec:finite-incidence}. The overlap condition for the
tubes and balls in the Fu--Gan--Ren estimate is given in
Section~\ref{sec:rectangular-sampling}.

Several letters are used with different meanings in different parts
of the paper.
We list them here, with the places where they are defined. The terminal resolution is $\delta=2^{-n}$, except in
Lemma~\ref{rel-lem:RW-linear-C} and its proof, where $\delta<\Delta$
are two auxiliary dyadic scales.
In Part~\ref{part:geometry}, $\eta$ denotes the side length
of the dyadic squares, called source parents, into which
the source is grouped. In Section~\ref{sec:relation-upper} it also
denotes a small exponent loss. The dyadic scales $\ell_j$ are used
to classify the angular separations of source vertices
about a pin. They determine the candidate widths
$\rho_j=\frac{\delta}{\ell_j}$ of the tubes about perpendicular
bisectors. We write $\rho$ for the selected width
(Sections~\ref{sec:circle-rows} and~\ref{sec:finite-incidence}).
In the proof of Lemma~\ref{lem:transfer}, $\ell_j$ instead denotes a
precision. The pins are selected at resolution $\delta$ and later thinned to
resolution $\rho$. The letters $\beta,\beta_0$ denote deficits of
lower envelopes in Sections~\ref{sec:refined},
\ref{sec:unpinned} and~\ref{sec:reciprocal}, while in Part~\ref{part:geometry} the
letter $\beta$ is the exponent loss in the source mass cap
$b=\delta^{\lambda-\beta}$ (Section~\ref{sec:circle-rows}). The letter
$\lambda$ denotes the transfer level $1+\Delta$ of Sections~\ref{sec:plateau} and~\ref{sec:unpinned}. Written without a set
argument, $\Delta$ is this excess or the auxiliary scale above, never a
distance set. The letters $b$, $k$, $B$ and $Q$ also have local
meanings, fixed where they are introduced. We use $\varrho$ for a
distance value, such as $|x-y|$, in the finite-precision complexities
defined below relative to the oracle $(A,x)$. The superscript $A$ denotes the
fixed oracle. Scalar cutoffs such as $A_0$ are local to the
statements defining them. The constant $\eta_0$ is reserved for
the uniform gain in Theorem~\ref{thm:main}.
The plateau parameters
$u_-,u_+,W,\varsigma,\chi,\mathfrak e$ are defined in
Proposition~\ref{prop:pin-plateau}. The subscripts in $H_{f,g}$ and $\widehat H_{f,g}$ identify
the source profile $f$ and the pin profile $g$, respectively
(Sections~\ref{sec:transfer} and~\ref{sec:pinned-obstruction}).

\subsection{Complexity at finite precision}\label{sec:complexity}
Use the fixed universal prefix-free oracle machine $U$ of
Section~\ref{sec:intro-proof-outline}, with the usual optimality
convention. Simulating another prefix-free machine adds at most a
constant depending on that machine. The following are standard
prefix-free complexity conventions~\cite{LiVitanyi}.
Their
finite-precision versions are as in~\cite[\S2]{S,FS}.
Write $\{0,1\}^*$ for the set of finite binary strings.
For finite binary strings
$\sigma,\tau$, the conditional prefix-free Kolmogorov complexity of
$\sigma$ given $\tau$ is
\[
K(\sigma\mid\tau)=\min\{|\pi|:\pi\in\{0,1\}^*,\ U(\pi,\tau)=\sigma\},
\]
the length of a shortest program that outputs $\sigma$ when $\tau$ is
supplied as auxiliary input, and $K(\sigma)$ is the same minimum with
the empty string as auxiliary input. For an oracle $A\subseteq\mathbb N$, the
quantities $K^A(\sigma)$ and $K^A(\sigma\mid\tau)$ are defined in the
same way for a machine that may query membership in $A$. Omitting the
oracle superscript means that the oracle is empty. Rational
points $q\in\mathbb Q^m$ are identified with their binary encodings
under a fixed computable bijection, so that $K^A(q)$ and
$K^A(p\mid q)$ are defined. Thus the finite-precision complexity defined in the introduction is
$K_r^A(z)=\min\{K^A(q):q\in\mathbb Q^m\cap B(z,2^{-r})\}$.
The convention $K_0^A(z)=0$ is used only at the initial profile scale.
The join $A\oplus B$ of two oracles encodes $A$ in the even
positions and $B$ in the odd positions. A real point $x$ used as
an oracle is encoded by the binary expansions of its coordinates,
with terminating expansions chosen at dyadic rationals and with
signs and integer parts included. Thus $K^{A,x}$ uses the join of
$A$ with this encoding and supplies arbitrarily precise rational
approximations to $x$. An oracle extends $A$ if it computes $A$.
Its set need not contain $A$. For $z\in\mathbb R^m$, $w\in\mathbb R^n$,
define finite conditioning by
\[
K_{r,s}^A(z\mid w)
=\max_{q\in\mathbb Q^n\cap B(w,2^{-s})}
\min_{p\in\mathbb Q^m\cap B(z,2^{-r})}K^A(p\mid q).
\]
Equal-precision conditioning is abbreviated by
$K_r^A(z\mid w):=K_{r,r}^A(z\mid w)$. For conditioning points
$w_i\in\mathbb R^{n_i}$, simultaneous finite conditioning means
\[
K_r^A(z\mid w_1,\ldots,w_k)
=\max_{\substack{q_i\in\mathbb Q^{n_i}\cap B(w_i,2^{-r})\\1\le i\le k}}
\min_{p\in\mathbb Q^m\cap B(z,2^{-r})}
K^A(p\mid\langle q_1,\ldots,q_k\rangle),
\]
where $\langle q_1,\ldots,q_k\rangle$ is a fixed computable encoding
of the ordered tuple as a finite string. The maximum makes this
complexity a bound for every admissible collection of conditioning
approximations. In contrast, real points in the oracle superscript
supply all their binary digits. The quantity $K_{r,s}^A(z\mid z)$ is
the number of bits needed to pass from precision $s$ to precision $r$
at the point $z$. A dyadic square $Q$ is identified with the finite string that
encodes its level and position.
Thus $K^A(Q)$ is defined, and
$K_n^A(z\mid Q)$ denotes the minimum of $K^A(p\mid Q)$ over rational
$p$ within $2^{-n}$ of $z$. Joins of several oracles, as in
$K^{A,G,x}$, are formed in the same way as $K^{A,x}$. We use the standard
finite-precision chain rules, symmetry of information, and oracle
monotonicity (adding an oracle can only decrease a complexity, up to a
fixed simulation cost). The notation $K_r^A(z,w)$ treats the pair as
a point in the product Euclidean space. In particular, for fixed points
and $1\le s\le r$, the chain rules used below include
\begin{align*}
K_r^A(z,w)&=K_r^A(z)+K_r^A(w\mid z)+O(\log r),\\
K_{r,s}^A(z\mid z)&=K_r^A(z)-K_s^A(z)+O(\log r).
\end{align*}
Interchanging $z,w$ in the first identity is the
symmetry-of-information relation. The errors have the same order
with any fixed additional oracles.
See~\cite[\S2]{S,FS}. Finite conditioning on an
approximation to a point is a weaker thing than full access to that point
as an oracle, and we will need both.

The effective Hausdorff dimension of $z$ relative to $A$ is
\[
\dim^A z=\liminf_{r\to\infty}\frac{K_r^A(z)}{r}.
\]
Its upper effective dimension, or effective packing dimension, is
$\operatorname{Dim}^A z=\limsup_{r\to\infty}\frac{K_r^A(z)}{r}$
\cite[\S2]{FS}.
The same definitions apply to joined oracles. In particular,
\[
\dim^{A,x}z=\liminf_{r\to\infty}\frac{K_r^{A,x}(z)}{r},
\qquad
\operatorname{Dim}^{A,x}z=\limsup_{r\to\infty}\frac{K_r^{A,x}(z)}{r}.
\]
Here the machine has access to both $A$ and the full binary encoding
of $x$. The same convention applies to any finite list of additional
oracles.
A quantity is computable relative to $A$ if an oracle algorithm approximates it to every prescribed rational
accuracy. It is lower semicomputable if such an algorithm produces
an increasing rational sequence converging to it. A collection is
computably enumerable relative to $A$ if an oracle algorithm lists
its elements, with no requirement to decide nonmembership.
The point-to-set principle~\cite{LL} states that, for every $F\subseteq\mathbb R^m$,
\[
\dim_{\mathrm H} F=\min_{A\subseteq\mathbb N}\sup_{z\in F}\dim^A z.
\]
A compact set $F\subseteq\mathbb R^m$ is effectively compact relative
to $A$ if the collection of finite families of rational open balls
whose union covers $F$ is computably enumerable relative to $A$.
We also use the effectively compact form of the principle. If a
compact set $F$ is effectively compact relative to $A$, then
$\dim_{\mathrm H} F=\sup_{z\in F}\dim^A z$.
See~\cite[Theorem~5]{S}. If $F$ is
effectively compact relative to $A$, then $\Delta_xF$ is effectively
compact relative to $(A,x)$~\cite[Proposition~7]{S}. Compact sets and
their measures are always encoded into the oracle before the
points to which information estimates are applied are selected.

For a unit vector $v$, write $p_vz=z\cdot v$. Effective projection
estimates are developed in~\cite{LS,S,FS}. A nonintegral subscript
of $K$ is rounded down to the nearest integer. A bounded change in
precision changes complexity by at most $O(\log(r+2))$ at terminal
precision $r$. The convention at precision zero is as above.

\subsection{The pointwise hypotheses}\label{sec:pointwise-hypotheses}
For fixed distinct points $x,y$, put $\mathbf e=\frac{y-x}{|y-x|}$. The pointwise
hypotheses are the conditions of~\cite[\S4]{FS} for both choices of pin.
\begin{enumerate}[label=(H\arabic*)]
\item $\dim^A x>1$ and $\dim^A y>1$.
\item $K_r^{A,x}(\mathbf e)\ge r-O(\log r)$ and $K_r^{A,y}(\mathbf e)\ge r-O(\log r)$.
\item $K_r^{A,x}(y)\ge K_r^A(y)-O(\log r)$ and $K_r^{A,y}(x)\ge K_r^A(x)-O(\log r)$.
\item $K_r^A(\mathbf e\mid x)\ge r-o(r)$ and $K_r^A(\mathbf e\mid y)\ge r-o(r)$. These
finite-precision bounds follow from (H2), up to logarithmic precision costs.
\end{enumerate}
The constants may depend on the fixed pair and base oracle.
The directional complexity deficiency at precision $r$ relative
to $(A,x)$ means $r-K_r^{A,x}(\mathbf e)$. Thus (H2) says that this
quantity is at most logarithmic in $r$.
This is the directional
randomness referred to below. Condition (H3) is the full-oracle
independence condition. Supplying the other endpoint saves at
most logarithmically many bits.

To obtain the bound $Y(\dim_{\mathrm H}E)$ in \eqref{eq:target},
where $Y$ and $S_0$ are defined in Theorem~\ref{thm:main}, we prove
the following pointwise estimate. Let $0\le\kappa<\kappa_0:=S_0-1$.
If $x,y$ satisfy (H1)--(H4) and $\dim^Ax,\dim^Ay>1+\kappa$, then
\begin{equation}\label{eq:pointtarget}
\liminf_{r\to\infty}\frac{K_r^A(|x-y|)}{r}\ge Y(1+\kappa).
\end{equation}
Here $K_r^A(|x-y|)$ is computed using only the background oracle $A$,
with neither $x$ nor $y$ supplied as an additional oracle.
Supplying either endpoint can only shorten a description, up to
a fixed simulation cost. Consequently
\begin{equation}\label{eq:pin-conditioning}
\max\{K_r^{A,x}(|x-y|),K_r^{A,y}(|x-y|)\}
\le K_r^A(|x-y|)+O(1).
\end{equation}
Thus $K_r^A(|x-y|)$ bounds both pin-conditioned distance complexities
from above, up to a fixed additive constant. Which of these two conditioned
complexities is larger may depend on the precision. Our unpinned
argument allows this dependence.
It does not require the same choice
of pin to give the stronger lower bound at every scale.

\subsection{Information profiles and the two choices of pin}\label{sec:profiles}

A point in a bounded part of the plane has
$K_r^A(z)\le2r+O(\log r)$. More generally, increasing the precision
from $m$ to $n$ costs at most $2(n-m)+O(\log n)$ bits.
These bounds allow us to retain the information at intermediate
scales, as follows.

Fix $A,x,y$ satisfying the pointwise hypotheses and a sequence
$r_j\to\infty$. The four normalized functions used to define
$f,g,d_x,d_y$ in Section~\ref{sec:intro-proof-outline} have a common
subsequence converging uniformly on $[0,1]$. The initial value is
zero, and logarithmic errors disappear after division by $r_j$.
The endpoint profiles are nondecreasing and $2$-Lipschitz, and
\begin{equation}\label{eq:profileclass}
u\le f(u),\,g(u)\le2u\qquad\text{for } 0\le u\le1.
\end{equation}
The distance profiles are nondecreasing and $1$-Lipschitz.
In
particular $0\le d_x(u),d_y(u)\le u$. When these functions are
considered only through their displayed inequalities, without
requiring a realization by points, we refer to them as scalar
profiles.
The compactness assertion follows from the preceding precision
increment bounds by linear interpolation and the usual uniform
compactness argument for functions with a common Lipschitz bound,
allowing errors tending to zero.

\begin{lemma}[Joint extraction]\label{lem:joint-extraction}
Every sequence of terminal precisions tending to infinity has one
subsequence along which the four profiles $f,g,d_x,d_y$ above
converge uniformly. Along that same subsequence, the prefix and
interval-transfer inequalities hold for both choices of pin.
\end{lemma}
\begin{proof}
Apply the compactness argument just given to the four functions
successively, retaining each preceding convergence. Lemma~\ref{lem:prefix} holds at every precision, and
Lemma~\ref{lem:transfer} holds along every subsequence on which the
profiles converge, so both conclusions hold along this one.
\end{proof}

\begin{table}[htbp]
\centering
\begin{tabular}{@{}lllll@{}}
\toprule
Pin & Source point & Source profile & Pin profile & Distance profile\\
\midrule
$x$ & $y$ & $f$ & $g$ & $d_x$\\
$y$ & $x$ & $g$ & $f$ & $d_y$\\
\bottomrule
\end{tabular}
\caption{The profiles retain their definitions when the endpoints exchange roles.}
\label{tab:roles}
\end{table}
Table~\ref{tab:roles} records how the profiles are relabelled when the
endpoints exchange roles.

By \eqref{eq:pin-conditioning}, an upper bound for $K_r^A(|x-y|)$ controls both distance
profiles. For a prescribed pin only $d_x$ is controlled, so the unpinned
argument cannot be used without change. Throughout the paper the points and oracle
are chosen before the precision varies.

In the arguments below, we partition intervals of precision scales
into smaller intervals. Let $F$ denote either a complexity function,
extended to real precision by linear interpolation between successive
precision levels, or a limiting profile. Put $h(u)=F(u)-u$.
Thus
\[
h(b)-h(a)=F(b)-F(a)-(b-a)
\]
measures the growth of $F$ on $[a,b]$ relative to growth at rate one.

Following the color conventions of~\cite[\S3]{S}
and~\cite[\S\S3--4]{FS}, we call an interval $[a,b]$:
\begin{itemize}
\item yellow if $h$ attains its minimum at the left endpoint:
      $h(u)\ge h(a)$ for every $u\in[a,b]$;
\item teal if $h$ attains its minimum at the right endpoint:
      $h(u)\ge h(b)$ for every $u\in[a,b]$;
\item green if both conditions hold:
      $h(a)=h(b)=\min_{u\in[a,b]}h(u)$.
\end{itemize}
These conditions specify where the minimum is attained. They do not
require $h$ to be monotone. In particular, a green interval need not
be an interval on which $h$ is constant.

A labelled yellow--teal partition assigns each interval one of the
labels yellow or teal, according to the conditions above. A green
interval may receive either label. The yellow loss is
\[
\sum_{\substack{[a,b]\text{ in the partition}\\
                [a,b]\text{ labelled yellow}}}
\bigl(h(b)-h(a)\bigr).
\]
Every summand is nonnegative, and green intervals contribute zero.

For the auxiliary piecewise linear constructions, we also call a
piece red when $h$ is increasing on it and blue when
$h$ is decreasing on it. A red--green--blue sequence consists of
pieces with these colors occurring in that order as the precision
increases. In each application, the partition intervals must also
satisfy the length restrictions stated there.

\subsection{From dyadic masses to descriptions}

The following elementary coding fact will be used both in the geometric
reduction and in the optional application in the appendix. A measure
is computable relative to $A$ here when its dyadic masses can be
approximated by oracle programs with any prescribed rational error.

\begin{lemma}[Coding a computable measure]\label{lem:measure-coding}
Let $A\subseteq\mathbb N$ be an oracle, and let $\mu$ be a
compactly supported probability measure on $\mathbb R^2$ whose
dyadic masses are uniformly computable relative to $A$. That is,
a single algorithm with access to $A$, given a dyadic square $Q$
and an integer $m\ge1$, approximates $\mu(Q)$ with error at most
$2^{-m}$.

Then, for $\mu$-almost every $z$,
\begin{equation}\label{eq:measure-coding}
K_n^A(z)=-\log_2\mu(Q_n(z))+O_z(\log(n+2)).
\end{equation}
Moreover, for each fixed oracle $B$ such that $A$ is computable
relative to $B$, one has
\[
K_n^B(z)\ge-\log_2\mu(Q_n(z))-O_{B,z}(\log(n+2))
\]
for $\mu$-almost every $z$. The exceptional null set may depend
on $B$.

In particular, for each fixed $x\in\mathbb R^2$,
\[
K_n^{A,x}(y)\ge K_n^A(y)-O_{x,y}(\log(n+2))
\]
for $\mu$-almost every $y$, with an exceptional null set that may
depend on $x$.
\end{lemma}

\begin{proof}
Kraft's inequality states that a prefix-free set of binary strings
$w$ satisfies $\sum_w2^{-|w|}\le1$, and conversely that lengths
$\ell_i$ with $\sum_i2^{-\ell_i}\le1$ are the codeword lengths of a
prefix-free code~\cite[Chapter~5]{CoverThomas}. We use its effective
form, the Kraft--Chaitin theorem~\cite[Chapter~4]{LiVitanyi}, as
follows. Given $n$, for each level-$n$ square $Q$ and integer
$k\ge0$, request a codeword of length $k+1$ for $Q$ once a rational
mass approximation proves $\mu(Q)>2^{-k}$. Each request is made
once. These requests can be enumerated relative to $A$, and
\[
\sum_Q\sum_{k:2^{-k}<\mu(Q)}2^{-(k+1)}
\le\sum_Q\mu(Q)=1.
\]
For every square of positive mass the shortest requested length
is at most $-\log_2\mu(Q)+2$. The theorem therefore produces a
prefix-free code with this bound by one procedure relative to $A$
taking $n$ as input. This avoids having to decide whether a
computable mass equals a dyadic threshold. Given the level $n$ and a codeword assigned to a square $Q$,
the decoding algorithm, with access to $A$, outputs the label
of $Q$. Thus the codeword, together with a description of $n$,
provides a description of the square. Since describing $n$
costs $O(\log(n+2))$ bits, we obtain
\[
K^A(Q_n(z))\le-\log_2\mu(Q_n(z))+O(\log(n+2)).
\]

For the reverse inequality, the shortest programs of distinct
level-$n$ squares relative to $B$ form a prefix-free set, so the
direct form of Kraft's inequality implies $\sum_Q2^{-K^B(Q)}\le1$. Every
square in the set below has $\mu(Q)<(n+2)^{-4}2^{-K^B(Q)}$, hence,
for each fixed $B$,
\[
\mu\{z\in\mathbb R^2:K^B(Q_n(z))< -\log_2\mu(Q_n(z))-4\log_2(n+2)\}
\le(n+2)^{-4}.
\]
The sum of these upper bounds is finite. The Borel--Cantelli
lemma therefore says that almost every $z$ belongs to only
finitely many of the displayed exceptional sets, proving the
reverse bound.
A precision-$n$ approximation to $z$ identifies $Q_n(z)$ among a
bounded number of squares. Conversely, the label of $Q_n(z)$
together with an index among its fixed-depth descendants specifies
a center within $2^{-n}$ of $z$. The index has bounded length.
It
need not be determined by the square alone. Thus point and square
complexities differ by $O(\log(n+2))$. This proves both assertions.
The last follows by taking $B=(A,x)$ and using the first assertion.
The exceptional null set may depend on $B$. In particular, we do
not intersect the full-measure sets over all real pins. We first
fix a pin, apply the almost-everywhere assertion in the source
variable, and then use Fubini for the two fixed measures.
\end{proof}

\part{Information at the two endpoints}\label{part:exchanging-pin-source}

\section{Retaining information in the projection estimates}
\label{sec:borrowed}

We derive the estimates below by retaining intermediate complexity
quantities in the arguments of Stull and Fiedler--Stull.
One involves the source, and the other involves the pin. Keeping
these quantities allows the two endpoint profiles to interact.
We identify the relevant inequalities in the earlier proofs and
check uniformity in the projection direction. Appendix~\ref{app:retained-prefix}
gives the source-increment summation and the modifications needed
at upper effective dimension two. Directional uniformity matters because, in
Section~\ref{sec:transfer}, that direction is determined by a
competitor, a point $z$ at the same distance from the pin as $y$,
which a short description of the distance cannot distinguish from
$y$.

\begin{lemma}[Prefix estimate]\label{lem:prefix}
Under (H1)--(H4), with $x$ as the pin,
\begin{equation}\label{eq:prefix}
K_r^{A,x}(|x-y|)\ge\frac{r}{2}+\frac{1}{4}K_r^A(y)-o(r).
\end{equation}
The same bound holds with $x$ and $y$ exchanged.
\end{lemma}
\begin{proof}
The argument of \cite[\S4.3]{FS} partitions the range of precisions
$[1,r]$ into blocks of two kinds, called yellow and teal. On a
yellow block the complexity of the distance grows at least as fast
as the precision, and on a teal block it grows at least as fast as the
complexity of the source. Let $L_r$ be the total length of the
yellow blocks. Proposition~\ref{prop:appendix-prefix}, proved in
Appendix~\ref{app:retained-prefix}, establishes, for this $L_r$, the two
retained inequalities
\[
K_r^{A,x}(|x-y|)\ge\frac{2r+L_r}{3}-o(r), \qquad 
K_r^{A,x}(|x-y|)\ge K_r^A(y)-L_r-o(r).
\]
Their weighted average with coefficients $\frac{3}{4}$ and
$\frac{1}{4}$ is the asserted bound. The appendix constructs the
partition, proves the transfer on its teal blocks at $D=2$, and
sums the errors before removing the pin from the source complexity.
\end{proof}

The point of retaining the source complexity is visible in the
weighted average because the unknown yellow length cancels exactly.
\[
\frac{3}{4}\frac{2r+L}{3}+\frac{1}{4}(K_r^A(y)-L)
=\frac{r}{2}+\frac{K_r^A(y)}{4}.
\]

The second borrowed estimate is Stull's projection theorem, read so as to
keep the terminal complexity of the projected point.
Here $K_R^A(x\mid p_vx,v)$ uses precision-$R$ rational approximations
to both $p_vx=x\cdot v$ and $v$ as finite auxiliary input, taking the
worst case over these approximations as in Section~\ref{sec:complexity}.
The background oracle is $A$.

\begin{lemma}[Projection estimate retaining terminal complexity]\label{lem:projection}
Let $x$ satisfy $\dim^A x>1$. Fix $\tau>0$. Suppose $\tau R\le t\le R$ and $v\in S^1$ satisfies
\[
K_s^{A,x}(v)\ge s-C\log R \quad \text{ for } 1\le s\le t
\]
with a fixed constant $C$. This is the directional-randomness input in
\cite[Theorem~13]{S}, in its local uniform form. Then
\begin{equation}\label{eq:projection}
K_R^A(x\mid p_vx,v)
\le\max\left\{K_R^A(x)-R,\frac{K_R^A(x)-t}{2}\right\}+o(R).
\end{equation}
The difference $K_R^A(x)-K_R^A(x\mid p_vx,v)$ is the amount of
information about $x$ that its projection determines.
We call it the
retained projection information. Equivalently, \eqref{eq:projection}
says that it is at least
\begin{equation}\label{eq:gain}
\min\left\{R,\frac{K_R^A(x)+t}{2}\right\}-o(R).
\end{equation}
More precisely, for each $\xi>0$ the error is at most $\xi R$ once $R\ge R_0(A,x,\tau,C,\xi)$, where $C$ is the common directional-deficiency constant. This threshold is independent of $v$ and $t$.
\end{lemma}
\begin{proof}
We read the proof of \cite[Theorem~13]{S} before its last simplification. In the case with no red--green--blue sequence, equation~(16) and the line immediately following it give
\[
K_R^A(x\mid p_vx,v)\le K_R^A(x)-R+o(R).
\]
In the other case, the printed equation~(19) yields, for a nonnegative partition parameter $B$,
\[
K_R^A(x\mid p_vx,v)
\le \min\{K_R^A(x)-B-t,B\}+o(R).
\]
Since $\min\{c-B,B\}\le \frac{c}{2}$, this is at most $\frac{K_R^A(x)-t}{2}+o(R)$. Combining the two cases proves the assertion. To justify the uniform threshold, the remark after \cite[Lemma~14]{S} makes its dependence on the tolerance, deficiency constant, and $|x|$ explicit, with no direction dependence. The local lemmas use the same uniform geometric and coding estimates. The partition depends on the interpolated complexity of $x$,
the terminal precision $R$, and the length comparison $t$.
It does
not depend on $v$. Write $T=t-O(\log R)$ for the length comparison after rounding,
chosen so that $T\ge\frac{t}{2}$ for all sufficiently large $R$.
In the greedy construction of \cite[Lemma~18]{S}, every two
successive pieces advance by more than $T$, except at the terminal
endpoint. Each complementary range therefore uses at most
$2\lceil\frac{R}{T}\rceil+2$ pieces. The proof of its Theorem~13
keeps one green interval, or two adjacent green intervals, and
applies this greedy construction on the two complementary ranges.
Including the green pieces leaves at most
\[
4\lceil\frac{R}{T}\rceil+6
\le8\lceil\frac{1}{\tau}\rceil+6.
\]
The case without a red--green--blue sequence has no larger count.
Thus a common upper bound is
\[
m_\tau=20\lceil\frac{1}{\tau}\rceil+20.
\]
Use the fixed local tolerance $\frac{\xi}{4m_\tau}$ in each estimate.
Discard partition pieces whose right endpoints are below
$h_R=\lceil\sqrt R\rceil$. Their total length is at most $h_R$ and
their full complexity costs $o(R)$. For every remaining piece with
right endpoint $b\ge h_R$, $C\log R\le2C\log b$. Thus the local
lemmas apply with the same deficiency constant $2C$ and the same
fixed tolerance, independently of $v,t$. Summing over the bounded
number of pieces costs at most $\frac{\xi R}{4}$ plus
$O_{C,m_\tau}(\log R)$. The remaining initial cutoff from
$\dim^A x>1$ depends only on $A,x$. This provides one
$R_0(A,x,\tau,C,\xi)$ for all the permitted $v,t$.
\end{proof}

There is one endpoint issue to address before using these estimates.
In~\cite{FS}, the parameter $D$ bounds the upper effective dimensions
of the two points. Some partition counts there contain
$\frac{1}{2-D}$. We allow $D=2$, so those counts do not apply.
What we actually need is a positive lower bound for the ratio of two
precisions. The next lemma supplies it, both here and in
Section~\ref{sec:asymmetric}.

\begin{lemma}[Bounded precision ratio at the endpoint $D=2$]\label{lem:endpoint-ratio}
Let $1<d<\frac{3}{2}$ and let $[a,b]$ be a teal interval of the partition of
\cite[\S4.3]{FS} with lower parameter $d$ and upper parameter $D=2$. Then
for $b$ sufficiently large,
\[
 \frac{a}{b}\ge\frac{d}{5-d}-o(1)>0.
\]
For a competing point $z\ne y$ with $|x-z|=|x-y|$, put
$t=-\log_2|y-z|$. At terminal precision $b$, the non-tangent range
means $t<\frac{b}{2}$, and the tangent range means $t\ge \frac{b}{2}$.
Consequently, in every non-tangent application of the collision
estimate \cite[Lemma~30]{FS} in the proof of \cite[Lemma~31]{FS}
on such an interval, where $t<\frac{b}{2}$,
the directional precision $t$
satisfies $t\ge a-O(1)$ and the projection precision is $R=b-t$, so that
$\frac{t}{R}$ is bounded below by a positive constant depending only on $d$.
The
projection partition of Lemma~\ref{lem:projection} has uniformly bounded
size.
Lemma~\ref{lem:projection} applies with the information gain
$\min\{R,\frac{dR+t}{2}\}-o(R)$. No division by $2-D$ occurs.
\end{lemma}
\begin{proof}
The lower bound on $\frac{a}{b}$ is the first estimate of \cite[Lemma~26]{FS},
which remains valid throughout the stated range. The second estimate
of that lemma is not needed here. The
bound $S<\frac{2}{d(2-D)}$ on the number of red--green--blue sequences
in the proof of \cite[Theorem~16]{FS} is not used at $D=2$.
Instead the bounded ratio $\frac{t}{R}$ bounds the size of the
projection partition, and the retained projection estimate
Lemma~\ref{lem:projection} supplies exactly the input required in the
proof of \cite[Lemma~31]{FS}. Thus that proof applies with uniform errors
in the actual applications. Appendix~\ref{app:retained-prefix}
writes out the selected teal transfer and lists each replacement
of a bound involving $2-D$.
\end{proof}

\section{The interval transfer lemma}\label{sec:transfer}

Suppose the distance from $y$ to $x$ is known, as is a coarse
approximation to $y$. How many further bits are needed to locate
$y$ accurately? A competitor $z$ must lie near that coarse location
and on the same circle about $x$. If every such competitor requires
a long description, a short description of the distance cannot hide
much source information. This is the transfer principle used by
Stull and Fiedler--Stull.

We retain the dependence on $|z-y|$. This separation determines the
precision at which the competitor reveals a projection of the pin,
and hence the part of the pin profile that enters the estimate.

Suppose $|x-z|=|x-y|$ and $s=-\log_2|y-z|$. We consider
$0<s\le r$, the separation range required by the transfer lemma
below. The exact identity
\[
\left|\mathbf e\cdot\frac{y-z}{|y-z|}\right|=\frac{|y-z|}{2|x-y|}
\]
shows that the unit normal $v=\frac{y-z}{|y-z|}$ to the perpendicular bisector lies within $C_{x,y}2^{-s}$ of one of the two perpendiculars to $\mathbf e$. The sign costs at most one bit. Thus (H2) supplies a common logarithmic directional-deficiency bound through precision $s-c_0$, where $c_0$ depends only on the fixed pair. 

The equal-distance condition also places $x$ on the perpendicular
bisector of $[y,z]$. Consequently,
\[
p_vx=p_v\left(\frac{y+z}{2}\right)
=\frac{|y|^2-|z|^2}{2|y-z|}.
\]
Thus the two source points determine a projection of the pin.
Since $|y-z|=2^{-s}$, precision-$r$ approximations to $y$ and $z$
determine both $v$ and $p_vx$ at precision $r-s-O(1)$. Indeed,
normalizing the difference $y-z$ loses approximately $s$ bits
of precision.

The complexity bound below combines two contributions.
The term $K_s^A(y)$ accounts for the coarse position information
shared by $y$ and $z$. The minimum term accounts for the additional
information about the pin revealed by the perpendicular-bisector
constraint, using Lemma~\ref{lem:projection} at precision $r-s$
and the directional bound established above.
For every fixed $\theta\in(0,\frac{1}{2})$, uniformly for competitors with
$\theta r\le s<\frac{r}{2}-O(\log r)$, combine the
complexity bound for a competing point in \cite[Lemma~30]{FS}
with Lemma~\ref{lem:projection} to obtain

\begin{equation}\label{eq:collision}
K_r^A(z)\ge K_s^A(y)+
\min\left\{r-s,\frac{K_{r-s}^A(x)+s}{2}\right\}-o(r).
\end{equation}
For $s\ge\frac{r}{2}-O(\log r)$, put $r'=r-O(\log r)$
so that the identity
\[
|p_{\mathbf e}z-p_{\mathbf e}y|=\frac{|z-y|^2}{2|x-y|}
\]
makes the two projections indistinguishable at precision $r'$.
Lemma~13 of \cite{FS} and the directional information in (H2)--(H4)
then give
\begin{align*}
K_r^A(z)&\ge K_s^A(y)+K_{r'-s,r'}^A(\mathbf e\mid y)-O(\log r)\\
&\ge K_s^A(y)+K_{r'-s}^{A,y}(\mathbf e)-O(\log r)\\
&\ge K_s^A(y)+r-s-o(r).
\end{align*}
Here finite conditioning supplies no more information than the
full oracle $y$. If $r-s=O(\log r)$, closeness to $y$ alone
implies the same bound with logarithmic error. Thus the terminal
range does not require a direction-dependent threshold.

The transfer argument uses an auxiliary device from
\cite[Lemma~14]{FS}. For a rational $\eta$ with
$\eta r\le K_r^A(y)$, that lemma constructs from
$y$, $r$, and $\eta$ an oracle $G=D(r,y,\eta)$, which we call the
lowering oracle, with the following properties, all up to
$O(\log r)$ errors. Relative to $(A,G)$ the complexity of $y$ at
precision $r$ equals $\eta r$, its complexity at each precision
$s\le r$ is at least $\min\{\eta r,K_s^A(y)\}$, and complexities
conditioned on $y$ are unchanged. Thus $G$ caps the
source complexity at $\eta r$. This lets us test whether a proposed
amount of source information must pass to the distance. With this
oracle the positive source term becomes
$\min\{\eta r,K_s^A(y)\}$. Item~(ii) of that lemma preserves the pin information in the
complexity bound for a competing point.
\[
K_{r-s,r}^{A,G}(x\mid y)=K_{r-s,r}^A(x\mid y)+O(\log r),
\]
whereas
\[
K_{r-s}^{A,G}(x\mid p_vx,v)
\le K_{r-s}^A(x\mid p_vx,v)+O(1).
\]
More explicitly, with $R=r-s$, item~(ii), finite versus full
conditioning, and (H3) give
\[
K_{R,r}^{A,G}(x\mid y)
=K_{R,r}^A(x\mid y)+O(\log r)
\ge K_R^{A,y}(x)-O(\log r)
\ge K_R^A(x)-O(\log r).
\]
The unconditional pin information is preserved as well.
\[
K_R^{A,G}(x)\ge K_R^{A,G,y}(x)-O(1)
\ge K_R^{A,y}(x)-O(\log r)
\ge K_R^A(x)-O(\log r).
\]
These are preservation properties of the lowering oracle.
No
algorithm computing a shortest-description oracle from $y$ is
needed. For the tangent bound, the conditioning precision in Lemma~13
of \cite{FS} is $r'$, whereas the lowering oracle was constructed at
precision $r$. We therefore use the full-oracle assertion of
Lemma~14(ii), rather than identify these two finite conditioning
precisions. When $r'-s\ge1$, it implies
\begin{align*}
K_{r'-s,r'}^{A,G}(\mathbf e\mid y)
&\ge K_{r'-s}^{A,G,y}(\mathbf e)-O(\log r)\\
&=K_{r'-s}^{A,y}(\mathbf e)+O(\log r)\\
&\ge r'-s-O(\log r).
\end{align*}
Together with the capped value of $K_s^{A,G}(y)$, this proves
\[
K_r^{A,G}(z)
\ge\min\{\eta r,K_s^A(y)\}+r-s-O(\log r)
\]
through the tangent range. If $s\ge r'$, then $r-s=O(\log r)$,
and the same conclusion follows by using the closeness of $z$ to
$y$ at precision $s$. Thus the terminal end of this range is
covered as well.
The negative projection residual can only decrease after $G$
is supplied. Consequently the capped complexity bound for a competing point uses
the projection theorem relative to the original oracle $A$,
without assuming new directional randomness relative to $G$.

We now use the joint limiting profiles defined in
Section~\ref{sec:preliminaries}. All interval endpoints below are
fixed proportions of the terminal precision.
No error estimate is
optimized over a growing collection of microscopic intervals.
By Lemma~\ref{lem:prefix},
\begin{equation}\label{eq:limitprefix}
d_x(u)\ge \frac{u}{2}+\frac{f(u)}{4},\qquad  d_y(u)\ge \frac{u}{2}+\frac{g(u)}{4}.
\end{equation}

Put
\begin{align}
p_g(s;b)&=\min\left\{b-s,\, \frac{g(b-s)+s}{2}\right\},\label{eq:p}\\
H_{f,g}(a,b)&=\min\left\{f(b),\, \inf_{a\le s\le b}\big(f(s)+p_g(s;b)\big)\right\}.
\label{eq:H}
\end{align}
Because of \eqref{eq:profileclass}, $p_g(s;b)\ge\min\{\frac{b}{2},b-s\}$.
For $s\ge \frac{b}{2}$ it equals $b-s$. We call $p_g$ the closed projection gain, referring to its
explicit formula, and $H_{f,g}$ the transfer threshold. For a
separation parameter $s$, the quantity $f(s)+p_g(s;b)$ is the
normalized competitor cost. The other entry $f(b)$ is the terminal
cap. Thus $H_{f,g}(a,b)$ is the least of these costs on $[a,b]$.

\begin{lemma}[Transfer in the limit]\label{lem:transfer}
For every fixed $0<a<b\le1$,
\begin{equation}\label{eq:transfer}
d_x(b)-d_x(a)\ge\left[H_{f,g}(a,b)-f(a)\right]_+.
\end{equation}
With $y$ as the pin and $x$ as the source, the corresponding inequality is
\[
d_y(b)-d_y(a)\ge[H_{g,f}(a,b)-g(a)]_+.
\]
\end{lemma}
\begin{proof}
If $H_{f,g}(a,b)\le f(a)$, monotonicity proves the claim.
Otherwise choose
\[
f(a)<c<H=H_{f,g}(a,b),\qquad  h=H-c>0,
\]
and put
\[
n_j=\lfloor br_j\rfloor,\qquad 
\ell_j=\lceil ar_j\rceil,\qquad  m_j=\lfloor cr_j\rfloor.
\]
We cap the source complexity at $m_j$ using
$G_j=D(n_j,y,\frac{m_j}{n_j})$ relative to $A$.
The strict margin $h$ will make the competitor estimates uniform.
The rational cap parameter costs $O(\log r_j)$ bits to describe.
it is specified by the two integers $m_j,n_j=O(r_j)$.
All uses of the lowering lemma and the chain rule are uniform in
the auxiliary oracle. The machine constructions relativize to
the same universal oracle machine, and the cap parameters are
included in the description. Thus the error bounds remain valid
although $G_j$ changes with $j$.
We suppress integer roundings in the following estimates.

For competitors at separations $s r_j$ with $a\le s<\frac{b}{2}$, the capped complexity bound for a competing point is
\[
K_{br_j}^{A,G}(z)\ge r_j\min\{c,f(s)\}+r_jp_g(s;b)-o(r_j).
\]
Together with the tangent estimate for $s\ge \frac{b}{2}$, this is bounded below uniformly by
\[
cr_j+\min\{p_g(s;b),h\}r_j-o(r_j).
\]
Choose any fixed rational $0<\varepsilon<\min\{\frac{1}{4},\frac{h}{4b}\}$, and then take $j$ large enough that the total error is at most $\frac{\varepsilon br_j}{2}$. For $s<\frac{b}{2}$, $p_g(s;b)\ge \frac{b}{2}$ ensures the required $\varepsilon br_j$ surplus. For $s\ge \frac{b}{2}$, $p_g(s;b)=b-s$, and
\[
\min\{b-s, \, h\}-\frac{\varepsilon b}{2}
\ge\min\{\varepsilon b, \, b-s-\varepsilon b\}.
\]
This is exactly the terminal allowance in \cite[Lemma~15]{FS}, including $s\nearrow b$. Its hypothesis only requires separations $S=-\log_2|y-z|\le n_j$.
Points closer than terminal precision need not be considered. The $O(\log r_j)$ transition about $\frac{b}{2}$ is absorbed by the fixed margin. Uniform convergence and the fixed ratio $\frac{a}{b}>0$ supply common error bounds over all required competitors. Its other hypothesis holds since $K_{n_j}^{A,G_j}(y)=m_j+O(\log r_j)\le m_j+\frac{\varepsilon n_j}{2}$.

Explicitly, \cite[Lemma~14(iii)]{FS}, (H3), and the chain rule give
\begin{align*}
K_{n_j}^{A,G_j,x}(y)&\ge m_j-O(\log r_j),\\
K_{\ell_j}^{A,G_j,x}(y)&\le K_{\ell_j}^A(y)+O(1),\\
K_{n_j,\ell_j}^{A,G_j,x}(y\mid y)&\ge m_j-K_{\ell_j}^A(y)-O(\log r_j).
\end{align*}
Lemma~15 of \cite{FS} transfers this increment with loss $3\varepsilon n_j+O(\log r_j)$, including the description of its rational parameters. Removing $G_j$ only increases the conditional distance complexity up to a fixed simulation cost. With $\varrho=|x-y|$ and a fixed integer $c_1$, the worst-case definition of finite conditioning implies
\[
K_{n,\ell}^{A,x}(\varrho\mid y)
\le K_{n,\ell-c_1}^{A,x}(\varrho\mid\varrho)+O(\log n):
\]
every precision-$\ell$ approximation to $y$, with oracle $x$, computes a precision-$(\ell-c_1)$ approximation to $\varrho$. The chain rule and the bounded scalar precision shift identify the expression on the right with
\[
K_{br_j}^{A,x}(|x-y|)-K_{ar_j}^{A,x}(|x-y|)+O(\log r_j).
\]
Let $j\to\infty$ to obtain $d_x(b)-d_x(a)\ge c-f(a)-3\varepsilon b$. This holds for every sufficiently small fixed $\varepsilon>0$.
Let $\varepsilon\searrow0$, and then $c\nearrow H$. This proves \eqref{eq:transfer}. The partition count is bounded because $\frac{a}{b}$ is fixed.
No precision-dependent collection of intervals is summed.
\end{proof}

The geometric and information-theoretic input is now contained in
three groups of inequalities: the profile bounds
\eqref{eq:profileclass}, the prefix estimates \eqref{eq:limitprefix},
and the interval transfers \eqref{eq:transfer}. The next section shows how to choose a fixed pair satisfying these
hypotheses. The elementary consequences in
Appendices~\ref{sec:thirteen-seventeen} and~\ref{sec:refined} give a
short introduction to the scalar argument. The stronger estimates
are proved in Sections~\ref{sec:unpinned} and~\ref{sec:reciprocal}.

\section{From a pair of points to a set}
\label{sec:classical}

The scalar inequalities concern one fixed pair of points. To use
it for a distance set, we must choose a pair that has enough
complexity and enough independence in both directions.
The point-to-set principle requires this for every initial oracle.
We first enlarge that oracle to include the fixed geometric data,
and only then choose the pair. All precisions will use this same
choice.

For a compact set, we can encode a sequence of finite rational sets
whose Hausdorff distances from it are at most $2^{-j}$ at stage $j$.
This supplies approximations with a known error, and makes the compact
set effectively compact relative to the encoding oracle. For a finite
measure, enumerate the dyadic cubes and encode rational approximations
to every cube mass, each with error at most $2^{-j}$ at stage $j$.
Include the total mass and a bound for the support. These are countably
many finite strings, so a fixed coding of finite strings and indices
places all the data in one oracle $A_0$. Positive total masses can
then be divided out to compute the normalized measures. The coordinate
boundaries have measure zero for the measures used below, since their
Frostman exponents exceed one.

Given an arbitrary oracle $B$, use the join
$A=B\oplus A_0$: even positions encode $B$ and odd positions encode
$A_0$. Thus a computation with oracle $A$ has access to all the fixed
geometric data and can simulate a computation with oracle $B$.
No effective procedure for constructing $A_0$ from an arbitrary
analytic set is required. The point-to-set principle permits every
oracle, and the data are fixed before any point is selected.

\begin{proposition}[A fixed pair for every oracle]\label{prop:fixed-pair}
Let $E\subset\mathbb R^2$ be analytic with $\dim_{\mathrm H}E>1$, and let
$1<s<\dim_{\mathrm H}E$. For every oracle $B$ there are an oracle
$A=B\oplus A_0$ and points $x,y\in E$ satisfying (H1)--(H4) relative
to $A$, with $\dim^Ax\ge s$ and $\dim^Ay\ge s$.
\end{proposition}
\begin{proof}
Choose separated compact pieces $E_1,E_2\subset E$ with positive finite $s$-dimensional Hausdorff measure.
By the Hausdorff density theorem~\cite[Theorem~6.2]{Mattila},
the upper $s$-density of $\mathcal H^s|_{E_i}$ is finite almost
everywhere. A countable restriction followed by inner regularity
provides compact positive-measure subsets $E_i'\subset E_i$ and
common constants $C,r_0>0$ such that
$\mathcal H^s(E_i\cap B(z,r))\le Cr^s$ for $z\in E_i'$ and $0<r<r_0$.
If $B(w,r)$ meets $E_i'$ at $z$, it lies in $B(z,2r)$.
Enlarging $C$ for $r\ge\frac{r_0}{2}$ proves that
$\mathcal H^s|_{E_i'}$ is $s$-Frostman. Let $\mu_i$ be its
normalization and fix $1<t<s$.
Both measures have finite $t$-energy.

Start with an arbitrary oracle $B$, and use the preceding construction
to form $A=B\oplus A_0$, encoding both compact pieces, their finite
Hausdorff restrictions, the total masses, and all dyadic cell masses.
This supplies the hypotheses of \cite[Lemma~36(2)]{S}, also after
normalization. For each fixed endpoint $x$, that lemma establishes
full-oracle independence for almost every opposite endpoint $y$.
Apply it for both choices of pin and then use Fubini to obtain (H3)
for product-almost every pair. For every fixed $A$, the set
$\{z\in\mathbb R^2:\dim^A z<s\}$ has zero $\mathcal H^s$ measure by the usual
low-complexity covers. Thus (H1) also holds almost everywhere, even
though the background oracle contains the measure data.

Orponen's radial absolute-continuity theorem \cite[Theorem~1.11]{O} has an exceptional pin set of dimension at most $2-t<1$. The opposite $s$-dimensional measure assigns this set zero mass. For each fixed pin, the angular complexity deficiency set relative to that pin and $A$ has zero arclength by \cite[Lemma~36(1)]{S}. Radial absolute continuity and Fubini give (H2) for product-almost every pair, with either endpoint as the pin. The bad-pair relations here are Borel because the finite-precision oracle computation conditions are countable combinations of finite-query conditions. The finite conditional directional conditions follow directly from the full-oracle bounds. Fix one pair in the intersection of these full-measure sets.

The pair satisfies $\dim^Ax,\dim^Ay\ge s$, since the set
$\{z\in\mathbb R^2:\dim^Az<s\}$ is $\mathcal H^s$-null. Thus the quantifier order is
\begin{align*}
&\forall B \, \exists A=B\oplus A_0 \, \exists(x,y)\in E_1\times E_2 \\
& \forall\text{ sufficiently large precisions, with (H1)--(H4) satisfied by this same pair. }
\end{align*}
Choosing different pairs
at different scales would not suffice. The scalar bounds for this pair,
$Y$ and the scalar bound $\frac{4}{5}$, are proved in Sections~\ref{sec:unpinned}
and~\ref{sec:reciprocal}, after the partition estimates of
Part~\ref{part:prescribed-pin}. Together with the point-to-set
principle, these bounds establish the $Y(t)$ term and the lower bound
$\frac{4}{5}$ in Theorem~\ref{thm:main}. The uniform gain $\eta_0$ is
proved in Sections~\ref{sec:new-geometry} and~\ref{sec:new-universal}. The term
$R(\dim_{\mathrm H}E)$ follows from Theorem~\ref{thm:pinnedmain}.
\end{proof}

There is no pinned conclusion from this argument alone. A low
unconditioned distance complexity forces both endpoint-conditioned
complexities to be low simultaneously. A low complexity conditioned on the prescribed pin $x$ alone does
not provide \eqref{eq:pinenvelope}, whose derivation used the upper
bound for the distance complexity conditioned on $y$. Part~\ref{part:prescribed-pin} identifies the resulting scalar obstruction
and the critical profile structure used in the geometric proof.

\section{The asymmetric pinned refinement}
\label{sec:asymmetric}

What can the information estimates prove when the pin is prescribed?
We first separate the lower dimension of the pin from that of the
source. Fiedler and Stull use one parameter for both, and note in
their footnote~7 that this restriction can be relaxed. Keeping the
two parameters separate yields an explicit exceptional-set estimate.
The proof uses the partition argument alone.

\begin{proposition}[Separate pin and source parameters]
\label{prop:asymmetric-pinned}
Let $1<p\le s<\frac{3}{2}$, let $A$ be an oracle, and let $x\ne y\in\mathbb R^2$.
Suppose $\dim^A x>p$, $\dim^A y>s$, and, with
$\mathbf e=\frac{y-x}{|y-x|}$, suppose the Fiedler--Stull conditions with $x$ as the pin hold:
\[
K_r^{A,x}(\mathbf e)\ge r-O(\log r),\qquad 
K_r^{A,x}(y)\ge K_r^A(y)-O(\log r),\qquad 
K_r^A(\mathbf e\mid y)\ge r-o(r).
\]
Then
\begin{equation}\label{eq:asymmetric-pinned-bound}
\dim^{A,x}|x-y|\ge
B(p,s):=\min\left\{1,\, \frac{p+s(3-p)}{5-p}\right\}.
\end{equation}
\end{proposition}

\begin{proof}
Use the partition of \cite[\S4.3]{FS} with its auxiliary lower parameter
$d=p$ and upper parameter $D=2$. These parameters satisfy all the lower and
upper dimension assumptions used in that construction. Let $f$ be its
piecewise linear interpolation of $K_u^A(y)$, and let $L$ be the total
yellow length in the partition of $[1,r]$.

Consider a teal interval $[a,b]$. In the green case in the construction,
\[
f(b)-f(a)=b-a.
\]
Otherwise, $a$ is the root selected by equations~(10)--(11) of \cite{FS}.
With $d=p,D=2$, its exact defining identity is
\begin{equation}\label{eq:asymmetric-teal-identity}
f(b)-f(a)=\frac{p}{2} b-\frac{p-1}{2}a.
\end{equation}
Retain the stronger source bound $f(b)\ge sb-o(b)$, rather than replacing
it by $pb-o(b)$. The planar upper bound $f(a)\le2a+O(\log b)$ implies
\[
2a+O(\log b)
\ge \left(s-\frac{p}{2}\right)b+\frac{p-1}{2}a-o(b),
\]
and hence
\begin{equation}\label{eq:asymmetric-left-ratio}
\frac{a}{b}\ge\frac{2s-p}{5-p}-o(1).
\end{equation}
For $k=\frac{a}{b}$, the slope in \eqref{eq:asymmetric-teal-identity} is
\[
C_p(k)=\frac{p-(p-1)k}{2(1-k)},\qquad 
C_p'(k)=\frac{1}{2(1-k)^2}>0.
\]
Consequently, a non-green teal has source growth at least
\[
\left(c(p,s)-o(1)\right)(b-a),\qquad 
c(p,s)=\frac{p(2-s)+s}{5-2s}.
\]
A green interval has slope one. If $c(p,s)>1$, a sufficiently large
non-green teal would contradict its teal property.
If $c(p,s)=1$, the
same calculation yields slope at least $1-o(1)$. Thus every teal satisfies
\begin{equation}\label{eq:asymmetric-teal-rate}
K_{b,a}^A(y\mid y)\ge
\min\{1, \, c(p,s)\}(b-a)-o(b).
\end{equation}
Independence and the finite-precision chain rule allow the same estimate
relative to $(A,x)$, at logarithmic cost per interval.

The transfer on each teal remains exactly \cite[Lemma~31]{FS}.
The
partition itself has not been changed. At the endpoint $D=2$ the bound
$S<\frac{2}{d(2-D)}$ on the number of red--green--blue sequences
in the proof of \cite[Theorem~16]{FS} is not usable.
Lemma~\ref{lem:endpoint-ratio} supplies the replacement, with
$d=p$, and shows that the proof of Lemma~31 applies with uniform errors.

Write $Q_r=K_r^{A,x}(|x-y|)$ and $c_0=\min\{1,c(p,s)\}$.
The first summation in \cite[\S4.4]{FS}, using
\eqref{eq:asymmetric-teal-rate}, reads
\[
Q_r\ge L+c_0(r-L)-o(r).
\]
For its second summation retain the source lower bound $sr$:
the yellow source increments total at most $2L+o(r)$, whereas the
yellow distance increments total at least $L-o(r)$ and all teal
source increments transfer to distance increments. Therefore
\[
sr\le K_r^{A,x}(y)+o(r)\le L+Q_r+o(r).
\]
Combining,
\begin{equation}\label{eq:asymmetric-two-bounds}
Q_r\ge\max\{L+c_0(r-L), sr-L\}-o(r).
\end{equation}
If $c_0=1$, this proves full dimension. Otherwise the two affine functions
of $L$ meet in $[0,r]$, and their common value is
\[
\frac{s+c(p,s)(1-s)}{2-c(p,s)} r
=\frac{p+s(3-p)}{5-p} r.
\]
The condition $c(p,s)\ge1$ is equivalent to the last fraction being at
least one, proving the stated cap.

Fix an arbitrarily small error tolerance before taking
$r$ large. Successive teal intervals at least halve the precision, and
yellow blocks are maximal, so there are $O(\log r)$ blocks and the sum
of their upper endpoints is $O(r)$. The interval errors therefore total
$O(\varepsilon r)+O((\log r)^2)$. A fixed initial range has negligible
cost. The positive lower precision ratio of Lemma~\ref{lem:endpoint-ratio}
ensures the projection estimates are uniform. Letting $r\to\infty$ and then
$\varepsilon\searrow0$ proves \eqref{eq:asymmetric-pinned-bound}.
\end{proof}

The two inequalities \eqref{eq:asymmetric-two-bounds} contain more
than the cap $B(p,s)$. Keeping the terminal source complexity in the
second one, instead of the lower bound $sr$, yields a prefix estimate
that improves Lemma~\ref{lem:prefix} whenever an endpoint has
dimension larger than one.

\begin{corollary}[Prefix estimate with the retained teal rate]
\label{cor:prefix-rate}
Under the hypotheses of Proposition~\ref{prop:asymmetric-pinned},
suppose $c=c(p,s)<1$. Then
\begin{equation}\label{eq:prefix-rate-finite}
K_r^{A,x}(|x-y|)\ge\frac{cr+(1-c)K_r^A(y)}{2-c}-o(r),
\end{equation}
and the limiting profiles $f,d_x$ of Section~\ref{sec:preliminaries},
along any common sequence of precisions, satisfy
\begin{equation}\label{eq:prefix-rate}
d_x(u)\ge(1-2b)u+bf(u) \quad \text{for } 0\le u\le1,\qquad 
b=b(p,s)=\frac{1-c(p,s)}{2-c(p,s)}\in\Bigl(0,\frac{1}{4}\Bigr].
\end{equation}
\end{corollary}
\begin{proof}
In the proof of the proposition, $Q_r\ge L+c(r-L)-o(r)$, and the
second summation shows that $K_r^{A,x}(y)\le L+Q_r+o(r)$ before the source
lower bound $sr$ is inserted.
The independence hypothesis
$K_r^{A,x}(y)\ge K_r^A(y)-O(\log r)$ then implies
$K_r^A(y)\le L+Q_r+o(r)$. Multiply the first inequality by
$\frac{1}{2-c}$ and the second by $\frac{1-c}{2-c}$ and add.
The
terms in $L$ cancel, which is \eqref{eq:prefix-rate-finite}. For
$u>0$ apply \eqref{eq:prefix-rate-finite} at the precision
$\lfloor ur_j\rfloor$, divide by $r_j$, and let $j\to\infty$.
The
identity $\frac{c}{2-c}=1-2b$ proves \eqref{eq:prefix-rate}, and at
$u=0$ both sides vanish. Since $c(p,s)$ increases in $p$ and in $s$,
$c\ge c(1,1)=\frac{2}{3}$, so $b\le\frac{1}{4}$, and $b>0$ because
$c<1$.
\end{proof}

At $p=s=1$, \eqref{eq:prefix-rate} is the prefix estimate
\eqref{eq:limitprefix}.
Since $b\le\frac{1}{4}$ and $f(u)\le2u$,
\eqref{eq:prefix-rate} implies \eqref{eq:limitprefix} in every case.
With $p=1+v$ and $s=1+\kappa$, $0\le v\le\kappa$,
\begin{equation}\label{eq:b-general}
b(1+v,1+\kappa)=\frac{1-2\kappa-v+\kappa v}{(1-\kappa)(4-v)},
\end{equation}
which is positive exactly when $c<1$. We will use the following two cases.
\begin{equation}\label{eq:b-pin-sym}
\begin{gathered}
b_{\mathrm{pin}}(\kappa)=\frac{1-2\kappa}{4(1-\kappa)}
 \quad \text{for } 0\le\kappa<\frac{1}{2},\\
b_{\mathrm{sym}}(\kappa)=b(1+\kappa,1+\kappa)
=\frac{1-3\kappa+\kappa^{2}}{(1-\kappa)(4-\kappa)}
 \quad \text{for } 0\le\kappa<\frac{3-\sqrt5}{2}.
\end{gathered}
\end{equation}
The first is $b(1,1+\kappa)$. For $\kappa>0$, if the pin has
dimension larger than one and the source has dimension larger than
$1+\kappa$, choose $p>1$ sufficiently close to one and
$s<1+\kappa$ sufficiently close to $1+\kappa$, with
$p\le s$ and $p<\dim^A x$. The corollary applies, and continuity
establishes \eqref{eq:prefix-rate} with $b=b_{\mathrm{pin}}(\kappa)$
after letting $p\searrow1$ and $s\nearrow1+\kappa$.
For $\kappa=0$, choose $p=s>1$ below both endpoint dimensions
and let $p=s\searrow1$, obtaining $b=\frac{1}{4}$.
The second coefficient is obtained, for $\kappa>0$, by taking
$p=s\nearrow1+\kappa$ when both endpoint dimensions are larger
than $1+\kappa$. Its case $\kappa=0$ follows from the same
$p=s\searrow1$ argument. It holds for either choice of pin.

\begin{proof}[Proof of Theorem~\ref{thm:pinnedmain}]
First suppose $1<\dim_{\mathrm H}E\le\frac{3}{2}$ and
$1\le p<\dim_{\mathrm H}E$. Write $t=\dim_{\mathrm H}E$.
Suppose, to obtain a contradiction, that the exceptional set in the
statement has dimension greater than $p$. Since
\[
 \{x\in\mathbb R^2:\dim_{\mathrm H}\Delta_xE<B(p,t)\}
 =\bigcup_{\substack{q\in\mathbb Q\\0<q<B(p,t)}}
   \{x\in\mathbb R^2:\dim_{\mathrm H}\Delta_xE<q\},
\]
countable stability supplies $q<B(p,t)$ for which the latter set
has dimension greater than $p$.

Choose $p'>p$, with $p'>1$, smaller than both this exceptional
dimension and $t$. Since $B$ is nondecreasing in its first parameter,
$B(p',t)\ge B(p,t)>q$. By continuity in its second parameter,
choose $p'<s<t$ such that $q<B(p',s)$. Here $s<\frac{3}{2}$.
Choose a compact source subset $K\subset E$ with
$\dim_{\mathrm H}K>s$. Every selected bad pin for $E$ is also a
bad pin for $K$, so the Borel set
\[
 B_q(K)=\{x\in\mathbb R^2:\dim_{\mathrm H}\Delta_xK<q\}
\]
has dimension greater than $p'$. Its Borel measurability follows
by expressing small Hausdorff content through finite rational
open covers of the compact set $\Delta_xK$.

Take a source Frostman measure $\mu$ of exponent greater than $s$
on a compact subset of $K$, and a pin Frostman measure $\nu$ of
exponent greater than $p'$ on a compact subset of $B_q(K)$.
Positive-mass restrictions arrange separated compact supports.
Fix one oracle $A$ encoding the supports and both measures,
including their dyadic masses, as in Section~\ref{sec:classical}.
Lemma~\ref{lem:measure-coding} and Fubini's theorem imply full-oracle
independence for product-almost every pair in both orientations.
The Frostman bounds also imply
$\dim^A x>p'$ and $\dim^A y>s$ for product-almost every pair.

Apply Orponen's radial absolute-continuity theorem~\cite{O} to
each measure. In either orientation its exceptional set has
dimension less than one, and hence has zero mass for the opposite
measure. The angular complexity exceptional sets have zero
arclength. Radial absolute continuity and Fubini therefore imply
\[
 K_r^{A,x}(\mathbf e)\ge r-O(\log r),\qquad
 K_r^{A,y}(\mathbf e)\ge r-O(\log r)
\]
for product-almost every pair, where $\mathbf e=\frac{y-x}{|y-x|}$.
Changing the sign of the direction in the second orientation
costs only a constant. The second inequality implies the finite
conditional bound required in
Proposition~\ref{prop:asymmetric-pinned}.
Select one pair satisfying these conclusions. That proposition yields
\[
 \dim^{A,x}|x-y|\ge B(p',s)>q.
\]
On the other hand, the compact source support is effectively
compact relative to $A$. Its distance image is effectively compact
relative to $(A,x)$, so the restricted point-to-set principle implies
\[
 \dim^{A,x}|x-y|\le\dim_{\mathrm H}\Delta_xK<q.
\]
This is a contradiction. Notice that the two directional estimates
were obtained separately. Randomness relative to $x$ alone would
not supply the estimate conditioned on $y$.

If $\dim_{\mathrm H}E\ge\frac{3}{2}$, choose compact source subsets
with dimensions below $\frac{3}{2}$ and increasing to it. The
$p=1$ conclusion just proved supplies lower bounds increasing to
one, outside a countable union of exceptional sets of dimension at
most one. Thus the pinned distance sets have full dimension there.
This also proves the assertion for every $1\le p<\dim_{\mathrm H}E$,
since $B(p,\dim_{\mathrm H}E)=1$ in this range.
\end{proof}

\part{Partition estimates and the two endpoint profiles}\label{part:prescribed-pin}

\section{Why the scalar method stops at three quarters}
\label{sec:pinned-obstruction}

The estimates just proved still approach $\frac{3}{4}$ as the
source dimension decreases to one. Is this only a weakness in how
we optimized them? The following construction shows that it is a
limitation of the pointwise scalar system without a common-source
mass bound. The later annular cap records information about one
measure across many pairs, which is an additional kind of constraint.

We allow every finite yellow--teal partition from Stull's projection
argument, so the obstruction persists even after the projection
gain is strengthened. The resulting profiles satisfy all the
displayed inequalities, have endpoint dimensions greater than one,
and have $d_x(1)$ arbitrarily close to $\frac{3}{4}$. The opposite
distance profile has $d_y(1)=1$, which explains why the unpinned
argument leads to no contradiction.

The stronger partition functional has a second use. Near the
barrier, it forces the pin profile to contain a plateau.
That plateau will provide the pin spacing estimates in
Part~\ref{part:geometry}.

For a profile $g$, put $h_g(u)=g(u)-u$. For a labelled
yellow--teal partition $\mathcal P$, define its loss by
\[
\operatorname{loss}_g(\mathcal P)=
 \sum_{[v,w]\in\mathcal P\text{ yellow}}
       \bigl(h_g(w)-h_g(v)\bigr).
\]
Let $\mathcal L_t(g;R)$ be the infimum of this loss over finite
partitions of $[0,R]$ with piece lengths at most $t$. The colors
have the endpoint-minimum meanings given in
Section~\ref{sec:profiles}. A partition template records its
ordered endpoints and color labels.
It can be transported to a
nearby interval as in the next lemma. Define
\[
 P_g(t;R)=g(R)-\mathcal L_t(g;R),\qquad 
 \widehat p_g(t;b)=\max\left\{
   \min\left\{b-t,\frac{g(b-t)+t}{2}\right\}, 
   P_g(t;b-t)\right\},
\]
and set $P_g(t;0)=0$, so $\widehat p_g(b;b)=0$.
For $0<a\le b\le1$, define
\[
\widehat H_{f,g}(a,b)=
\min\left\{f(b),
\inf_{a\le t\le b}\bigl(f(t)+\widehat p_g(t;b)\bigr)\right\}.
\]
Thus $\widehat H_{f,g}$ is $H_{f,g}$ with $p_g$ replaced by
$\widehat p_g$. Telescoping the increments of $h_g$ shows
that $\mathcal L_t(g;R)\ge h_g(R)$, because teal increments are
nonpositive. Consequently
\begin{equation}\label{eq:pinned-gain-cap}
 \widehat p_g(t;b)\le b-t.
\end{equation}
Lemma~\ref{lem:partition-endpoints} below shows that
$P_g(t;b-t)\ge p_g(t;b)$, so in fact $\widehat p_g(t;b)=P_g(t;b-t)$
for $t<b$. We keep the closed gain in the notation because $p_g$,
unlike $P_g$ (Remark~\ref{rem:partition-nonmonotonicity}), is
monotone in $g$, which the envelope arguments of
Sections~\ref{sec:thirteen-seventeen}, \ref{sec:refined},
and~\ref{sec:unpinned} use.

\begin{lemma}[Uniform finite witnesses for the partition gain]
\label{lem:uniform-partition-witnesses}
The map $(g,t,R)\mapsto P_g(t;R)$ is lower semicontinuous for the
uniform topology on profiles and for $t,R>0$.
Fix $0<a<\frac{b}{2}\le\frac{1}{2}$, profiles $f,g$, and
$c<\widehat H_{f,g}(a,b)$. There exist $\xi>0$, an integer $m$,
and finitely many finite partition templates with the following
property. If $\|f'-f\|_\infty+\|g'-g\|_\infty<\xi$, then
$f'(b)>c+\xi$ and, for every $t\in[a,\frac{b}{2}]$, either the
closed projection gain or a partition obtained from one of the
templates yields
\[
f'(t)+\max\left\{
\min\left\{b-t,\frac{g'(b-t)+t}{2}\right\},
g'(b-t)-\operatorname{loss}_{g'}(\mathcal P_t)
\right\}>c+\xi.
\]
The partition $\mathcal P_t$ has at most $m$ pieces, covers
$[0,b-t]$, and has all piece lengths at most $t$.
Templates are transported by linear rescaling, trimming, and
splitting at minima. Thus the number of pieces and the strict
margin are uniform in the competitor separation.
\end{lemma}
\begin{proof}
Put $h=g-u$. Take a witnessing partition of $[0,R]$ with $k$
pieces and loss $B$. If $(R',t',g')$ is close to $(R,t,g)$, rescale
its endpoints by $\frac{R'}{R}$. The transported function and
$h'=g'-u$ differ by at most
\[
\|g'-g\|_\infty+2|R'-R|
\]
on corresponding points. Split each transported piece at a minimum
of $h'$. For an originally yellow piece, its new loss is at most
its old loss plus twice this error.
For an originally teal piece
the new loss is at most twice this error. Indeed, if $v$ is the
new minimum on $[u,w]$, the repaired loss is $h'(w)-h'(v)$.

A transported piece exceeds $t'$ by at most
$e=\frac{t}{R}|R'-R|+|t'-t|$. In a sufficiently small
neighborhood $e<t'$. For a yellow piece, split off its terminal
excess and repair that short piece at a minimum. For a teal
piece, keep its terminal interval of length $t'$ and repair the
short initial piece. The $1$-Lipschitz bound charges at most
$2e$ per piece. Hence a repaired partition has at most $6k$
pieces and loss
\[
B+O_{k,R,t}\bigl(\|g'-g\|_\infty+|R'-R|+|t'-t|\bigr).
\]
Taking a partition within any positive tolerance of the infimum
proves upper semicontinuity of $\mathcal L_t(g;R)$, and therefore
lower semicontinuity of $P_g(t;R)$.

Now put $\eta_{\mathrm{gap}}=\widehat H_{f,g}(a,b)-c>0$. At every
$t\in[a,\frac{b}{2}]$, the closed gain or one finite partition
witnesses a sum greater than $c+\frac{3\eta_{\mathrm{gap}}}{4}$.
The terminal cap
satisfies $f(b)\ge c+\eta_{\mathrm{gap}}$. The preceding repair argument and
continuity of the closed gain preserve a margin greater than
$\frac{\eta_{\mathrm{gap}}}{2}$ in a neighborhood of that separation and of the
two profiles. A finite subcover of the compact separation interval
provides finitely many templates. Take the maximum of their repaired
piece counts and the minimum of their neighborhood radii, and
decrease $\xi$ below $\frac{\eta_{\mathrm{gap}}}{2}$. These choices prove the claim.
\end{proof}

\begin{lemma}[The full partition transfer]\label{lem:fullpartition}
Let $x,y$ satisfy (H1)--(H4) relative to $A$. Along every sequence
of precisions on which the four profiles converge uniformly,
the limiting profiles satisfy
\[
d_x(b)-d_x(a)\ge[\widehat H_{f,g}(a,b)-f(a)]_+
 \quad \text{for } 0<a<b\le1,
\]
and, with $y$ as the pin,
\[
d_y(b)-d_y(a)\ge[\widehat H_{g,f}(a,b)-g(a)]_+.
\]
\end{lemma}
\begin{proof}
Fix a finite yellow--teal partition of $[0,R]$ with interval lengths at
most $t$. Uniform convergence of the normalized pin complexity makes
its intervals approximately yellow or teal, with errors tending to zero.
The hypotheses in \cite[Lemmas~14--15]{S} already allow an additive
error proportional to the terminal precision of the interval. If the
uniform profile error is $\varepsilon_N\to0$, the yellow or teal
condition has slack at most $2\varepsilon_NN+O(\log N)$ after
rounding and applying the chain rule. On pieces whose terminal
precision is at least $\kappa N$, this slack is $o(1)$ times that
precision, for each fixed $\kappa>0$. Thus those lemmas apply.
A yellow interval costs its increment of $g-u$, and a teal interval costs
$o(1)$ after normalization. The chain rule sums these costs over the
fixed finite partition, as in \cite[Lemma~17]{S}. Directional randomness
through $tN-O(\log N)$ suffices.
At a length boundary, trim an $O(\log N)$ precision segment from
the right endpoint of a yellow interval or the left endpoint of a
teal interval, preserving its color condition, and charge the
removed segment by its full complexity.
Thus this partition retains $g(R)$ minus its yellow loss. Take the
infimum over finite partitions after the large-precision limit.
No effective procedure for recovering a limiting profile or an
optimal partition is needed. Once the cap and its strict margin
are fixed, the next paragraph selects finitely many templates
with a fixed bound on their numbers of pieces. At precision $N$,
encode their rounded integer endpoints and color labels as part
of the description. This costs $O(m\log(N+2))$ bits for at most
$m$ pieces, and hence disappears after division by $N$.
The template count is fixed before $N$ tends to infinity.

Lemma~\ref{lem:uniform-partition-witnesses} supplies finitely many
partition templates and a uniform strict margin for a cap below
$\widehat H_{f,g}(a,b)$ on the early separation interval
$[a,\frac{b}{2}]$. If that interval is empty, only the tangent
estimate is needed. To avoid
any threshold issue from tiny pieces, discard an initial precision
range $[0,\kappa N]$, costing at most $2\kappa N+o(N)$, and charge newly
trimmed tiny pieces by their full complexity. Hold $\kappa>0$ fixed as
$N\to\infty$, then let $\kappa\searrow0$. Every remaining local
projection estimate is applied at a uniformly positive macroscopic
precision. The direction-uniform bounds in \cite[Lemmas~14--15]{S}
therefore apply to all competitors.

For $t\ge \frac{b}{2}$, the tangent estimate yields $b-t$ as before. Substituting
these uniform projection bounds into the capped collision argument of
Lemma~\ref{lem:transfer} completes the proof. The positive pin information
is preserved by the lowering oracle and its negative projection
residual can only decrease.
No randomness relative to that extra oracle
is assumed.
\end{proof}

\begin{proposition}[A scalar obstruction for a prescribed pin]
\label{prop:pinned-obstruction}
For every $0<\theta<\frac{1}{4}$ there are nondecreasing, $2$-Lipschitz
profiles $f_\theta,g_\theta$ vanishing at zero, with
\[
 (1+\theta)u\le f_\theta(u),g_\theta(u)\le2u,
\]
and nondecreasing, $1$-Lipschitz functions $d_x,d_y$ vanishing at zero,
such that both prefix inequalities, the refined prefix inequalities
\eqref{eq:prefix-rate} for every coefficient $b(p,s)$ with
$1<p\le s\le1+\theta$, and both strengthened transfer
inequalities with $\widehat H$ hold, but
\[
 d_x(1)=\frac{3}{4}+\theta,\qquad  d_y(1)=1.
\]
Thus these scalar constraints alone cannot force
$d_x(1)\ge\frac{3}{4}+\varepsilon$ for any fixed universal
$\varepsilon>0$, when $x$ is the prescribed pin.
\end{proposition}

\begin{proof}
Put $F=1+\theta$, $Q=\frac{3}{4}+\theta$, and define
\[
 h_0(u)=
 \begin{cases}
 u,&0\le u\le\frac{1}{4},\\
 \frac{1}{4},&\frac{1}{4}\le u\le\frac{1}{2},\\
 \frac{3}{4}-u,&\frac{1}{2}\le u\le\frac{3}{4},\\
 0,&\frac{3}{4}\le u\le1,
 \end{cases}
   g_0(u)=u+h_0(u).
\]
Take
\[
 f_\theta(u)=\min\{2u,F\},\qquad 
 g_\theta(u)=(1-\theta)g_0(u)+2\theta u,\qquad
   d_x(u)=\min\{u,Q\},\qquad  d_y(u)=u.
\]
The stated monotonicity, Lipschitz and dimension bounds are immediate.
The prefix inequality for $d_x$ holds when $u\le Q$ since $f_\theta(u)\le2u$,
and when $u\ge Q$ since
\[
 \frac{u}{2}+\frac{f_\theta(u)}{4}
 \le\frac{1}{2}+\frac{F}{4}=\frac{3}{4}+\frac{\theta}{4}\le Q.
\]
The prefix inequality for $d_y$ follows from $g_\theta(u)\le2u$.
For the refined prefixes, $c(p,s)$ increases in both parameters,
so $b(p,s)\ge b_{\mathrm{sym}}(\theta)>0$. Their right side is
$u-b(2u-f_\theta(u))\le u$ when $u\le Q$. When $u\ge Q$,
$f_\theta(u)=1+\theta$, and that right side is at most
\[
1-b_{\mathrm{sym}}(\theta)(1-\theta)
=\frac{(1+\theta)(3-\theta)}{4-\theta}<Q,
\qquad Q-\frac{(1+\theta)(3-\theta)}{4-\theta}
=\frac{5\theta}{4(4-\theta)}>0.
\]
For $d_y$, the refined prefixes again follow from $g_\theta(u)\le2u$.

We first record the relevant obstruction to projection partitions.  Write
$h_\theta=g_\theta-u$.  It equals $u$ on $[0,\frac{1}{4}]$ and is
nondecreasing on $[0,\frac{1}{2}]$.  If $0<t<\frac{1}{4}$ and $R>\frac{1}{2}$, every
partition interval whose interior meets $[0,\frac{1}{4}]$ ends before $\frac{1}{2}$.
It cannot be teal while containing a positive part of that initial
increase.  Therefore its yellow costs, summed over the initial rise,
are at least $\frac{1}{4}$.  Hence
\begin{equation}\label{eq:pinned-loss-barrier}
 \mathcal L_t(g_\theta;R)\ge\frac{1}{4}
  \quad \text{for } 0<t<\frac{1}{4}, R>\frac{1}{2}.
\end{equation}
Also $g_\theta(R)\le g_0(R)+\theta$ for $0\le R\le1$.

Consider the transfer inequality for $d_x$ on $0<a<b\le1$.  If $b\le Q$, the
distance increment is $b-a$, and \eqref{eq:pinned-gain-cap} proves the
required inequality by evaluating the infimum at $a$.
Suppose $b>Q$.  If $a\ge\frac{1}{4}$, the terminal cap implies
$\widehat H_{f_\theta,g_\theta}(a,b)\le F$.  When $a<\frac{F}{2}$,
\[
 F-f_\theta(a)=F-2a\le Q-a;
\]
when $a\ge \frac{F}{2}$, this source increment is zero.  These statements
give the required transfer, including the case $a\ge Q$.

It remains to consider $a<\frac{1}{4}$ and $b>Q$.  Set $R=b-a>\frac{1}{2}$.
By \eqref{eq:pinned-loss-barrier},
\[
 P_{g_\theta}(a;R)\le g_0(R)+\theta-\frac{1}{4}\le Q-a.
\]
For the last inequality, use $g_0(R)=R$ when $R\ge\frac{3}{4}$, and
$g_0(R)=\frac{3}{4}$ when $\frac{1}{2}<R<\frac{3}{4}$.
The other gain in the maximum defining $\widehat p$ is also at most
$Q-a$: when $R\ge\frac{3}{4}$ it is at most $\frac{b+\theta}{2}$, and otherwise
it is at most $\frac{\frac{3}{4}+\theta+a}{2}$.
Both are bounded by $Q-a$ for
$a<\frac{1}{4}$ and $b\le1$.
Evaluating the competitor infimum at $a$ now yields
\[
 \widehat H_{f_\theta,g_\theta}(a,b)
 \le f_\theta(a)+Q-a.
\]
This is the required transfer inequality for $d_x$, because $d_x(a)=a$ and $d_x(b)=Q$.

Finally, the transfer inequality for $d_y(u)=u$ always holds.
\[
 \widehat H_{g_\theta,f_\theta}(a,b)
 \le g_\theta(a)+\widehat p_{f_\theta}(a;b)
 \le g_\theta(a)+b-a.
\]
This completes the verification of every interval for both choices of pin.
\end{proof}

\begin{remark}[The partition gain is not monotone]
\label{rem:partition-nonmonotonicity}
Let $g_1(u)=u$ and $g_2(u)=u+h_0(u)$, with $h_0$ as in the preceding
proof. Then $g_1\le g_2$, but
\[
P_{g_1}\left(\frac{1}{4};\frac{3}{4}\right)=\frac{3}{4},\qquad
P_{g_2}\left(\frac{1}{4};\frac{3}{4}\right)=\frac{1}{2}.
\]
For $g_1$, every partition has zero loss. For $g_2$, a partition
piece of length at most $\frac{1}{4}$ meeting the initial increase in a
positive-length interval ends before $\frac{1}{2}$ and cannot be teal.
The yellow costs therefore sum to at least $\frac{1}{4}$. The partition
with endpoints $0,\frac{1}{4},\frac{1}{2},\frac{3}{4}$ attains that loss.
Thus a pointwise pin lower envelope cannot be substituted inside
$P_g$. The envelope updates in Sections~\ref{sec:thirteen-seventeen}
and~\ref{sec:refined} use only the closed gain $p_g$.
The plateau
argument below and the proof of Theorem~\ref{thm:main} in
Section~\ref{sec:unpinned} work directly with partition losses.

An affine pin profile $g(u)=\sigma u$, $1\le\sigma\le2$, illustrates
the information recovered by the full partition gain. Its loss is
$(\sigma-1)R$ for every partition, so $P_g(t;R)=R$.
The full transfer then implies $\widehat H_{f,g}(a,1)\ge1$, since
$f(u)\ge u$. Consequently $d_x(1)\ge1-f(a)\ge1-2a$ for every
$a>0$, and $d_x(1)=1$ follows by letting $a$ tend to zero.
\end{remark}

The countermodel is a statement about the displayed pointwise
inequalities, without the common-source mass bound of
Proposition~\ref{prop:conditional-parent-cap}.
It does not assert that a Euclidean pair realizes all four profiles.
The next section derives the common structure of nearly critical
profiles.
The geometric argument will use information shared by
several actual pairs.

\section{The plateau forced by a small pinned distance profile}
\label{sec:plateau}

The scalar obstruction does more than explain the limitation of the
previous argument. It suggests what must happen if a pinned distance
profile is nearly as small as $\frac{3}{4}$. We now prove that the
source must acquire information quickly at the early scales, and
that the pin must retain an excess of nearly $\frac{1}{4}$ over the
trivial profile on an interval of length nearly $\frac{1}{4}$. Both
statements are proved at a general transfer level, because a source
of dimension larger than one admits a larger level
(Section~\ref{sec:bootstrap}), and the level enters the final
exponent comparison.

\subsection{The certificate and the first hitting scale}

Throughout this section $f,g,d_x$ satisfy the profile bounds
\eqref{eq:profileclass}, the prefix estimate \eqref{eq:limitprefix}
with $x$ as the pin, and the transfer estimate for $d_x$ with
$\widehat H$ (Lemma~\ref{lem:fullpartition}). Fix
$0\le\Delta<\frac{1}{2}$ and put
\begin{equation}\label{eq:level-notation}
\lambda=1+\Delta,\qquad  F=\frac{1}{2}+\Delta,\qquad 
a_c=\frac{1}{4}+\frac{\Delta}{2},\qquad 
W_0=\frac{1-2\Delta}{4},\qquad 
q=\frac{3}{4}+\frac{\Delta}{2}+\omega,
\end{equation}
where $\omega\ge0$. Let $a_*$ be the first point
with $f(a_*)=F$.
It exists when $f(1)\ge F$. We say that the pair
has the certificate at level $\lambda$ if
\begin{equation}\label{eq:certificate}
H_{f,g}(a_*,1)\ge\lambda,
\end{equation}
where $H_{f,g}$ is the closed transfer functional \eqref{eq:H}. In
particular $f(1)\ge\lambda$. For $\Delta=0$ the certificate holds
for every admissible pair. For $a_*\le s\le\frac{1}{2}$ one has
$f(s)\ge\frac{1}{2}$ and $p_g(s;1)\ge\frac{1}{2}$, for
$s\ge\frac{1}{2}$ one has $f(s)\ge s$ and $p_g(s;1)=1-s$, and
$f(1)\ge1$. Section~\ref{sec:bootstrap} obtains $\Delta>0$ from a
source dimension larger than one.

\begin{lemma}[The first hitting scale]\label{lem:first-hit}
If the pair has the certificate at level $\lambda$ and
$d_x(1)\le q$, then
\[
a_c\le a_*\le a_c+2\omega .
\]
\end{lemma}
\begin{proof}
The bound $f(u)\le2u$ implies $a_*\ge\frac{F}{2}=a_c$. The prefix
estimate at $a_*$, the transfer estimate on $[a_*,1]$, and
\eqref{eq:certificate} give
\[
q\ge d_x(1)\ge\frac{a_*}{2}+\frac{F}{4}+\lambda-F
=\frac{a_*}{2}+\frac{F}{4}+\frac{1}{2},
\]
which is $a_*\le2q-\frac{F}{2}-1=a_c+2\omega$.
\end{proof}

This first-hitting estimate holds for every $\omega\ge0$. In
particular, suppose both orientations satisfy the prefix and closed
transfer estimates, and $d_x(1),d_y(1)\le Q$ with
$\frac{3}{4}\le Q<1$. At level $\lambda=1$, both closed transfer
thresholds then have height at least one at
\[
A_{\mathrm{init}}=\frac{1}{4}+2\left(Q-\frac{3}{4}\right)
=2Q-\frac{5}{4}.
\]
Indeed, apply the lemma with either endpoint as the pin and use
the monotonicity of $H$ in its lower endpoint. This is the form
needed in Section~\ref{sec:reciprocal}.

For the plateau arguments that follow, we impose the additional
restriction
\[
0<\omega<\frac{1-2\Delta}{24}.
\]

\subsection{Partition bounds}

For a continuous function $h$ on a compact interval $I=[u,v]$
and $t>0$, define
\[
L_t(h;I)=\inf_{\mathcal P}
\sum_{\substack{[a,b]\in\mathcal P\\
                 [a,b]\text{ labelled yellow}}}
\bigl(h(b)-h(a)\bigr),
\]
where $\mathcal P$ ranges over finite yellow--teal partitions of $I$
with piece lengths at most $t$, as in
Section~\ref{sec:pinned-obstruction}. A yellow piece has its minimum at
its left endpoint, and a teal piece has its minimum at its right
endpoint. Either label is allowed when both endpoints minimize $h$.
Set $L_t(h;[u,u])=0$. Thus
$\mathcal L_t(g;R)=L_t(h_g;[0,R])$. The next lemma is the
elementary partition bound behind Stull's projection argument.

\begin{lemma}[A terminal-minimum partition bound]
\label{lem:sixth-terminal-minimum}
Let $h:[0,L]\to[0,\infty)$ be $1$-Lipschitz, with $h(L)=0$ and
$h(0)=p$.  For every $t>0$,
\begin{equation}\label{eq:sixth-terminal-minimum}
 L_t(h;[0,L])\le \max\{0,\frac{L-t-p}{2}\}.
\end{equation}
The right side is attained as an upper bound by a partition with at most
$\lceil\frac{2L}{t}\rceil+8$ intervals.
\end{lemma}
\begin{proof}
When $L\le t$ the whole interval is teal.  First suppose that $h$ is
piecewise linear with slopes $1$ and $-1$.  Apply the red--green--blue
construction preceding \cite[Observation~20]{S} to $u+h(u)$.
In terms of $h$, red means increasing, blue means decreasing, and green
means an interval of length at most $t$ whose two endpoints are equal
and minimize $h$.  This construction uses only the finite polygonal
profile.
The point-dimension assumption used later in
\cite[Lemma~21]{S} is not involved here.

If there is a red--green--blue sequence, the construction supplies one
or two adjacent green intervals of combined length at least $t$.
Keep these intervals and partition the two complementary intervals by
the ordinary greedy yellow--teal rule. This rule
chooses the furthest admissible next endpoint at distance at most $t$,
labelling the resulting interval teal whenever possible.  A non-teal
yellow interval has length $t$, unless it reaches the terminal endpoint.
otherwise continuity permits a further yellow extension.  Two successive
teal intervals can be merged whenever their combined length is at most
$t$.  Consequently successive greedy endpoints obey
$a_{i+2}>a_i+t$, except at the last endpoint, and a subinterval of length
$\ell$ needs at most $\lceil\frac{2\ell}{t}\rceil+2$ pieces.  The partition
retaining the green interval or pair therefore has at most
$\lceil\frac{2L}{t}\rceil+8$ pieces.

Let $U$ be its yellow loss and let $V$ be minus the sum of its teal
increments, including the green intervals with increment zero.  Then
$U-V=-p$.  The $1$-Lipschitz bound, and the green length of at least
$t$, give $U+V\le L-t$.  Thus $U\le\frac{L-t-p}{2}$.

If there is no red--green--blue sequence, there can be no red interval
in the auxiliary partition.  Indeed, after a red interval a later blue
interval would give a red--green--blue sequence by choosing the last
red preceding the first subsequent blue.  The absence of later blue
intervals would make the final height strictly greater than a preceding
height, since green intervals have zero endpoint increment.  This
contradicts $h(L)=\min h$.  The auxiliary partition thus consists only
of blue and green intervals.  It can be coarsened to an all-teal
partition. Greedily choose the furthest teal endpoint at distance at
most $t$ lying outside the interiors of the auxiliary green intervals.
Such a next endpoint always exists, because each green interval has
length at most $t$ and the heights at all eligible endpoints are
nonincreasing.  Two consecutive resulting intervals could be merged
if their combined length were at most $t$.  The same endpoint spacing
therefore bounds their number by $\lceil\frac{2L}{t}\rceil+2$.  Their loss is zero.

Finally approximate a general $h$ uniformly by nonnegative finite
polygonal functions $h_j$ with slopes $1$ and $-1$, preserving both
endpoint values.  One construction replaces each linear interpolation
segment by a rising-then-falling roof with the same endpoints.  The
$1$-Lipschitz assumption makes that roof possible.
Its height above the
linear segment is at most the mesh size.  The preceding partitions
have a uniform cardinality bound.  After padding by zero-length
intervals, pass to a subsequence on which every endpoint and every
colour converges.  The endpoint-minimum properties are closed under
uniform convergence, as are the length constraints, and the total
yellow loss converges.  Remove zero-length intervals.  The resulting
finite partition proves both assertions.
\end{proof}

\begin{lemma}[Endpoint estimates for the partition loss]
\label{lem:partition-endpoints}
Let $h:[0,L]\to[0,\infty)$ be $1$-Lipschitz with $h(0)=p_0$ and
$h(L)=p_1$. For every $t>0$,
\begin{equation}\label{eq:partition-positive-endpoint}
 L_t(h;[0,L])\le p_1+\frac{1}{2}\left[L-t-p_0-p_1\right]_+.
\end{equation}
If moreover $p_0=0$ and
\[
 L_t(h;[0,L])\ge p_1+W,\qquad  W>\frac{L-p_1-2t}{2},
\]
then, for every $u\in[0,L]$,
\begin{equation}\label{eq:partition-inversion}
 h(u)\ge\min\{W, u+2W-L+p_1+t, 2W+t-u\}.
\end{equation}
\end{lemma}
\begin{proof}
Reversal $u\mapsto L-u$ exchanges yellow and teal intervals and is a
bijection between admissible partitions. For any partition, the
forward loss minus the reversed loss is the sum of all signed
increments, namely $p_1-p_0$.
Hence the same identity holds for the
two infima. Partition losses are also subadditive under
concatenation, since partitions with losses arbitrarily close to
their infima can be concatenated.

Choose $z$ where $h$ attains its minimum $m$. Apply
\eqref{eq:sixth-terminal-minimum} to $h-m$ on $[0,z]$ and to the
reversal of $h-m$ on $[z,L]$. Combining this with the reversal identity, we obtain
\[
 L_t(h;[0,L])\le p_1-m+\frac{1}{2}[z-t-p_0+m]_+
 +\frac{1}{2}[L-z-t-p_1+m]_+.
\]
If both positive parts vanish, the right side is at most $p_1$.
If both are positive, it equals $\frac{1}{2}(L-2t-p_0+p_1)$.
If only the first is positive, it equals
$\frac{1}{2}(z-t-p_0+2p_1-m)\le\frac{1}{2}(L-t-p_0+p_1)$, since
$p_1-m\le L-z$. If only the second is positive, it equals
$\frac{1}{2}(L-z-t+p_1-m)\le\frac{1}{2}(L-t-p_0+p_1)$, since
$p_0-m\le z$. Thus the loss is at most
$\max\{p_1,\frac{1}{2}(L-t-p_0+p_1)\}$, which is
\eqref{eq:partition-positive-endpoint}.

For the second assertion fix $u$ and put $p=h(u)$. The reversal
identity and \eqref{eq:sixth-terminal-minimum} on $[0,u]$, and
\eqref{eq:partition-positive-endpoint} on $[u,L]$, give
\begin{align*}
 L_t(h;[0,L])-p_1
 &\le p+\frac{1}{2}[u-t-p]_++\frac{1}{2}[L-p_1-u-t-p]_+\\
 &\le\max\left\{p, \frac{p+u-t}{2}, \frac{p+L-p_1-u-t}{2},
 \frac{L-p_1-2t}{2}\right\}.
\end{align*}
The last entry is strictly below $W$, so one of the first three is
at least $W$. Solving each for $p$ results in
\eqref{eq:partition-inversion}.
\end{proof}

\begin{lemma}[A long interval above a prescribed height]
\label{lem:plateau-abstract}
Let $h:[0,L]\to[0,\infty)$ be $1$-Lipschitz with $h(0)=0$ and
$h(L)=p_1$. Suppose $t>0$ and
\[
 L_t(h;[0,L])\ge p_1+W,\qquad 
 W>\max\left\{p_1,\frac{L-p_1-2t}{2}\right\}.
\]
Then there is a $u_-$ with
\begin{equation}\label{eq:plateau-abstract}
 W\le u_-\le L-p_1-t-W,\qquad  h(u)\ge W \quad \text{for } u_-\le u\le u_-+t.
\end{equation}
\end{lemma}
\begin{proof}
Put $c=\frac{L-p_1}{2}$. The last two terms in
\eqref{eq:partition-inversion} coincide at $u=c$ and equal
$2W+t-c\ge W$, so $h(c)\ge W$. For $p_1<m<W$, the same estimate
shows that
\[
 (L-p_1-t-2W+m, 2W+t-m)\cap[0,L]\subseteq\{h>m\}.
\]
This interval contains $c$, and its endpoints lie in $[0,L]$, since
otherwise the estimate at $0$ or at $L$ would contradict
$h(0)=0<m$ or $h(L)=p_1<m$. Let $(a_m,b_m)$ be the component of
$\{h>m\}$ containing $c$. Then $h(a_m)=h(b_m)=m$ and
\[
 a_m\le L-p_1-t-2W+m,\qquad  b_m\ge2W+t-m.
\]
Suppose $b_m-a_m\le t$. Then $[a_m,b_m]$ is an admissible teal
interval with loss zero. On the two complementary intervals,
\eqref{eq:partition-positive-endpoint} implies
\[
 L_t(h;[0,a_m])\le m+\frac{1}{2}[a_m-t-m]_+=m,\qquad 
 L_t(h;[b_m,L])\le p_1+\frac{1}{2}[L-b_m-t-m-p_1]_+=p_1,
\]
because both positive-part arguments are at most
$L-p_1-2t-2W<0$. Concatenation would imply
$L_t(h;[0,L])\le m+p_1<W+p_1$, a contradiction. Thus $b_m-a_m>t$.

As $m$ increases to $W$, the components $(a_m,b_m)$ are nested,
since all contain $c$. Their closures intersect in an interval
$[u_-,u_+]$ with $u_+-u_-\ge t$, and $h\ge W$ on it by continuity. The
bound for $a_m$ implies $u_-\le L-p_1-t-W$, and $h(0)=0$ with the
Lipschitz bound implies $u_-\ge W$. The interval $[u_-,u_-+t]$ has the
asserted properties.
\end{proof}

\subsection{The pin plateau and its excess}

The plateau proposition records, besides the plateau, the amount
$\mathfrak e\omega$ by which the full transfer threshold
$\widehat H_{f,g}(a,1)$ falls short of $\lambda$ for lower
endpoints $a$ after the critical competitor. This excess enters
the finite comparison in two ways, and Lemma~\ref{lem:source-envelope}
turns it into a lower envelope for the source profile.

\begin{proposition}[The pin plateau]\label{prop:pin-plateau}
Suppose the pair has the certificate at level $\lambda$ and
$d_x(1)\le q$. Then there are numbers
$\chi,\varsigma,\mathfrak e$ and $u_-,u_+,W$ with
\begin{equation}\label{eq:plateau-parameters}
\begin{gathered}
 -1\le\chi\le2,\qquad  \varsigma\ge\max\{0,-2\chi\},\qquad 
 3\varsigma+2\chi\le4,\qquad  0\le\mathfrak e\le\varsigma,\\
 u_+-u_-=a_c+\chi\omega\le a_*,\qquad  f(u_+-u_-)=F-\varsigma\omega,\\
 W=W_0-(\varsigma+\chi-\mathfrak e)\omega,\qquad 
 0\le u_--W\le(2\varsigma-2\mathfrak e)\omega,
\end{gathered}
\end{equation}
such that
\begin{equation}\label{eq:plateau}
 g(u)-u\ge W \quad \text{for } u_-\le u\le u_+,
\end{equation}
and such that, for every $a>u_+-u_-$,
\begin{equation}\label{eq:excess-threshold}
 \widehat H_{f,g}(a,1)\ge\lambda-\mathfrak e\omega .
\end{equation}
In particular
\begin{equation}\label{eq:plateau-consequences}
 W\ge W_0-2\omega,\qquad 
 (u_+-u_-)+2W=\frac{3}{4}-\frac{\Delta}{2}-(2\varsigma+\chi-2\mathfrak e)\omega,\qquad
   u_+\le\frac{1}{2}+(\varsigma-\mathfrak e)\omega .
\end{equation}
\end{proposition}
\begin{proof}
Write $h=g-u$. By Lemma~\ref{lem:first-hit},
$a_c\le a_*\le a_c+2\omega$, and by \eqref{eq:certificate} every
competitor $s\ge a_*$ satisfies $f(s)+\widehat p_g(s;1)\ge\lambda$.

For every $s$ and $R=1-s$, the partition gain
is at least the closed gain. Indeed
\eqref{eq:partition-positive-endpoint} with $p_0=0$ implies
$L_s(h;[0,R])\le h(R)+\frac{1}{2}[R-s-h(R)]_+$, hence
\[
 P_g(s;R)=g(R)-L_s(h;[0,R])
 \ge\min\left\{R,\frac{g(R)+s}{2}\right\}=p_g(s;1),
\]
so that $\widehat p_g(s;1)=P_g(s;1-s)$.

Fix $0<\vartheta\le\frac{1}{8}$ and put
$\Phi(u)=\frac{3f(u)}{4}-\frac{u}{2}$ and
$\ell=a_c-(1+\vartheta)\omega$. Since $\Phi(0)=0<\ell$ and
$\Phi(a_*)=\frac{3F}{4}-\frac{a_*}{2}\ge a_c-\omega>\ell$, there
is a last point $a_0<a_*$ with $\Phi(a_0)=\ell$, and $\Phi\ge\ell$
on $[a_0,a_*]$.
Since $\Phi(u)\le u$, also $a_0\ge\ell>0$. The
prefix estimate implies $f(a_0)-d_x(a_0)\le\Phi(a_0)=\ell$, so the
transfer estimate yields
\[
 \Lambda_\vartheta:=\widehat H_{f,g}(a_0,1)
 \le d_x(1)+f(a_0)-d_x(a_0)\le q+\ell=\lambda-\vartheta\omega .
\]
As $f(1)\ge\lambda$ and every competitor at or after $a_*$ costs at
least $\lambda$, there is a competitor $s\in[a_0,a_*)$ with
\[
 f(s)+P_g(s;1-s)=\Lambda_\vartheta+e,\qquad 
 0\le e\le\vartheta^2\omega .
\]
Write
\[
 s=a_c+\chi\omega,\qquad  f(s)=F-\varsigma\omega,\qquad 
 R=1-s,\qquad  h(R)=\varphi\omega,\qquad 
 \mathfrak e=\frac{\lambda-\Lambda_\vartheta}{\omega}\ge\vartheta,
\]
where $\chi,\varsigma,\varphi,\mathfrak e$ depend on $\vartheta$.
The bounds $a_0\le s<a_*$, $f(s)\le F$, $f(s)\le2s$, and
$\Phi(s)\ge\ell$ give
\begin{equation}\label{eq:plateau-prelimit}
 -1-\vartheta\le\chi\le2,\qquad  \varsigma\ge\max\{0,-2\chi\},\qquad 
 3\varsigma+2\chi\le4+4\vartheta,\qquad  \varsigma\le2+2\vartheta .
\end{equation}
Consequently $\varsigma+\chi\le\frac{7}{3}$,
$\varsigma-\frac{\chi}{2}\le\frac{5}{2}(1+\vartheta)\le\frac{45}{16}$,
$3\varsigma+\chi\le5(1+\vartheta)\le\frac{45}{8}$, and
$\widehat p_g(s;1)=\Lambda_\vartheta+e-f(s)
=\frac{1}{2}+(\varsigma-\mathfrak e)\omega+e$ is smaller than the
length cap $R=\frac{3}{4}-\frac{\Delta}{2}-\chi\omega$, because
$(\varsigma+\chi)\omega<W_0$ and
$e\le\vartheta^2\omega<\mathfrak e\omega$. Hence the closed gain
$p_g(s;1)$ equals $\frac{g(R)+s}{2}=\frac{1+\varphi\omega}{2}$, and
$p_g\le\widehat p_g$ implies
\[
 \mathfrak e\le\varsigma-\frac{\varphi}{2}+\frac{e}{\omega}.
\]
The identity $\widehat p_g(s;1)=P_g(s;R)=R+h(R)-L_s(h;[0,R])$ now
reads
\[
 L_s(h;[0,R])=h(R)+W_\vartheta,\qquad 
 W_\vartheta=W_0-(\varsigma+\chi-\mathfrak e)\omega-e .
\]
The hypotheses of Lemma~\ref{lem:plateau-abstract} hold for $h$ on
$[0,R]$ with $t=s$ and $p_1=h(R)$:
\begin{align*}
 W_\vartheta-\frac{R-h(R)-2s}{2}
 &=\frac{1+2\Delta}{8}
 -\left(\varsigma-\frac{\chi}{2}-\frac{\varphi}{2}-\mathfrak e\right)\omega-e
 \ge\frac{1}{8}-\frac{45}{16}\omega-e>0,\\
 W_\vartheta-h(R)
 &=W_0-(\varsigma+\chi+\varphi-\mathfrak e)\omega-e
 \ge W_0-\frac{45}{8}\omega-3e>0,
\end{align*}
using $\varphi\ge0$ in the first line and
$\varphi\le2\varsigma-2\mathfrak e+\frac{2e}{\omega}$ in the second,
together with $\omega<\frac{W_0}{6}\le\frac{1}{24}$ and
$e\le\vartheta^2\omega\le\frac{\omega}{64}$. The lemma supplies
$h\ge W_\vartheta$ on $[u_{-,\vartheta},u_{-,\vartheta}+s]$ with
\[
 W_\vartheta\le u_{-,\vartheta}\le R-h(R)-s-W_\vartheta
 =W_\vartheta+(2\varsigma-\varphi-2\mathfrak e)\omega+2e .
\]
Moreover, for every $a>s$ we have $a>a_0$, so
$\widehat H_{f,g}(a,1)\ge\Lambda_\vartheta=\lambda-\mathfrak e\omega$.

Let $\vartheta\searrow0$ along a sequence on which
$(\chi,\varsigma,\varphi,\mathfrak e,u_{-,\vartheta})$ converges, and
denote the limits by the same letters. Put $t=a_c+\chi\omega$,
$W=W_0-(\varsigma+\chi-\mathfrak e)\omega$, $u_-=\lim u_{-,\vartheta}$,
and $u_+=u_-+t$. Continuity of $f$ and $h$ ensures $f(t)=F-\varsigma\omega$,
$t\le a_*$, and $h\ge W$ on $[u_-,u_+]$.
The limiting bounds give
\eqref{eq:plateau-parameters}, with
$0\le\mathfrak e\le\varsigma-\frac{\varphi}{2}\le\varsigma$.
The monotonicity of $\widehat H_{f,g}$ in its lower endpoint implies
\eqref{eq:excess-threshold}. Finally
$\varsigma+\chi\le\frac{4+\chi}{3}\le2$ gives $W\ge W_0-2\omega$, the
formula for $(u_+-u_-)+2W$ is a substitution, and
$u_+\le W+(2\varsigma-2\mathfrak e)\omega+a_c+\chi\omega
=\frac{1}{2}+(\varsigma-\mathfrak e)\omega$.
\end{proof}

We write
\begin{equation}\label{eq:u0-varpi}
 u_0=\varsigma+2\chi\ge0,\qquad 
 \varpi=\min\{2,1+u_0\},\qquad 
 e_0=\mathfrak e+[u_0-1]_+
\end{equation}
for the combinations of the plateau parameters that enter the
geometry. The inequalities defining the polytope \eqref{eq:plateau-parameters} imply
\begin{equation}\label{eq:plateau-polytope}
 \varsigma\le2-\frac{u_0}{2},\qquad 
 2\varsigma+\chi\le3-\frac{u_0}{4},\qquad 
 \varsigma+\chi\le2,\qquad 
 \varpi+\varsigma\le\frac{7}{2},\qquad 
 2\varsigma+\varpi\le5 .
\end{equation}
Indeed $3\varsigma+2\chi=2\varsigma+u_0\le4$ proves the first two.
For the last two use $\varpi=1+u_0$ when $u_0\le1$ and
$2\varsigma\le4-u_0\le3$ when $u_0\ge1$.

\subsection{The source profile near the first hitting scale}

The source parents of the finite argument will have side
$\delta^{a}$ with $a$ slightly larger than $a_*$. The next two
lemmas locate $a_*$ in terms of the plateau parameters and bound the
source profile at the scales between $a_c-\omega$ and $a_*$.

\begin{lemma}[The first hitting scale and the plateau]
\label{lem:compatible-hit}
Under the hypotheses of Proposition~\ref{prop:pin-plateau},
\[
 a_*\le a_c+\varpi\omega .
\]
\end{lemma}
\begin{proof}
Write $s=u_+-u_-=a_c+\chi\omega$ and $D_0=2s-f(s)=u_0\omega$, and put
$b_0=\frac{a_*}{2}$. Since $s\ge a_c-\omega$ and
$b_0\le\frac{a_c}{2}+\omega$, we have $b_0<s$. For $b_0\le t\le s$
the $2$-Lipschitz bound implies $f(t)-t\ge f(s)-2s+t=t-D_0\ge b_0-D_0$.
For $s\le t\le a_*$, monotonicity and the Lipschitz bound give
$f(t)\ge\max\{F-\varsigma\omega,F-2(a_*-t)\}$, and the mean of the
two affine bounds yields
$f(t)-t\ge F-a_*-\frac{\varsigma\omega}{2}$. On $[b_0,a_*]$ every
closed gain is saturated because $t\ge\frac{a_*}{2}$ and $g(R)\ge R$ give
$p_g(t;a_*)=a_*-t$. Hence
\[
 H_{f,g}(b_0,a_*)\ge\min\left\{a_*+b_0-D_0, F-\frac{\varsigma\omega}{2}\right\}.
\]
Prefix at $b_0$, the transfer on $[b_0,a_*]$, and $f(b_0)\le2b_0$
give
\[
 d_x(a_*)\ge H_{f,g}(b_0,a_*)-b_0
 \ge\min\left\{a_*-D_0, F-\frac{a_*}{2}-\frac{\varsigma\omega}{2}\right\}.
\]
The second branch exceeds $a_c+\omega$, because $a_*\le a_c+2\omega$,
$\varsigma\le2$, and $\omega<\frac{a_c}{6}$ give
$F-\frac{a_*}{2}-\frac{\varsigma\omega}{2}-a_c-\omega
\ge\frac{a_c}{2}-3\omega>0$. On the other hand the transfer on
$[a_*,1]$ and \eqref{eq:certificate} give
$d_x(a_*)\le d_x(1)-(\lambda-F)\le q-\frac{1}{2}=a_c+\omega$. Hence
the first branch is at most $a_c+\omega$, that is,
$a_*\le a_c+(1+u_0)\omega$. Together with
Lemma~\ref{lem:first-hit} this is the claim.
\end{proof}

\begin{lemma}[A source envelope at the row scales]
\label{lem:source-envelope}
Under the hypotheses of Proposition~\ref{prop:pin-plateau}, if
$\mathfrak e\le\frac{1}{8}$, then
\begin{equation}\label{eq:source-envelope}
 f(a)\ge a+a_c-(1+u_0+\mathfrak e)\omega
  \quad \text{for } a_c-\omega\le a\le a_c+\varpi\omega.
\end{equation}
\end{lemma}
\begin{proof}
Let $s=u_+-u_-$ and $D_0=u_0\omega$ as above. For $a\le s$ the
Lipschitz bound implies $f(a)\ge2a-D_0$, and
$(2a-D_0)-[a+a_c-(1+u_0+\mathfrak e)\omega]
=a-a_c+(1+\mathfrak e)\omega\ge\mathfrak e\omega\ge0$. For
$a\ge a_*$ we have $f(a)\ge F$, while the right side of
\eqref{eq:source-envelope} is at most $F$ because
$a\le a_c+\varpi\omega$ and $\varpi\le1+u_0$.

Let $s<a<a_*$ and put $b=\frac{a}{2}\le s$. As in the preceding
proof, $f(t)-t\ge b-D_0$ on $[b,s]$,
$f(t)-t\ge\frac{f(s)+f(a)}{2}-a$ on $[s,a]$, and the closed gains
on $[b,a]$ are saturated, so
\[
 H_{f,g}(b,a)\ge\min\left\{a+b-D_0, \frac{f(s)+f(a)}{2}\right\},\qquad
 d_x(a)\ge\min\left\{a-D_0, \frac{f(s)+f(a)-a}{2}\right\}.
\]
On the other hand \eqref{eq:excess-threshold} and the transfer on
$[a,1]$ give $d_x(a)\le q-\lambda+\mathfrak e\omega+f(a)
=f(a)-a_c+(1+\mathfrak e)\omega$. If the second branch of the lower
bound were at most this upper bound, then
$f(a)\ge f(s)+2a_c-a-2(1+\mathfrak e)\omega
\ge3a_c-(\varpi+\varsigma+2+2\mathfrak e)\omega$, so
$f(a)-F\ge a_c-\frac{23}{4}\omega>0$ by
\eqref{eq:plateau-polytope} and $\mathfrak e\le\frac{1}{8}$,
contradicting $f(a)\le F$. Hence the first branch is at most the
upper bound, which is \eqref{eq:source-envelope}.
\end{proof}

\subsection{Larger certificates from the source dimension}
\label{sec:bootstrap}

If the source profile satisfies $f(u)\ge(1+\kappa)u$, which is the
case for a source measure with Frostman exponent $1+\kappa$, the
certificate holds at a level above one, and the level can be
improved by iterating the prefix and transfer estimates. The
iteration is carried out with the prefix estimate in the form
\begin{equation}\label{eq:general-prefix}
 d_x(u)\ge(1-2b)u+bf(u) \quad \text{for } 0\le u\le1,
\end{equation}
where $0<b\le\frac{1}{4}$: the case $b=\frac{1}{4}$ is
\eqref{eq:limitprefix}, and Corollary~\ref{cor:prefix-rate} yields
smaller coefficients when the endpoints have dimension larger than
one.

\begin{lemma}[Bootstrapped certificate]\label{lem:bootstrap}
Suppose $f(u)\ge(1+\kappa)u$ on $[0,1]$ with $0<\kappa<1$, that
\eqref{eq:general-prefix} holds for some $0<b\le\frac{1}{4}$, and
that $d_x(1)\le q$ with
\begin{equation}\label{eq:bootstrap-range}
 1-b(1-\kappa)\le q<1-\frac{b(1-\kappa)}{2}.
\end{equation}
Put
\[
 \mathcal D=\frac{\kappa}{b(1-\kappa)},\qquad 
 \lambda_*=1+\mathcal D(1-q),\qquad 
 \Delta_*=\lambda_*-1 .
\]
Then the first point $a_*$ with $f(a_*)=\lambda_*-\frac{1}{2}$
exists, and
\[
 H_{f,g}(a,1)\ge\lambda_* \quad \text{for } a_*\le a\le1,\qquad 
 a_*\le\frac{q-\frac{1}{2}-b(\lambda_*-\frac{1}{2})}{1-2b},\qquad 
 a_*\le\frac{1}{4}+2\Bigl(q-\frac{3}{4}\Bigr)-\frac{\Delta_*}{2}.
\]
\end{lemma}
\begin{proof}
Put $\lambda_0=1+\frac{\kappa}{2}$ and let $a_0$ be the first point
with $f(a_0)=\lambda_0-\frac{1}{2}=\frac{1+\kappa}{2}$.
It lies at
or before $\frac{1}{2}$. For $a_0\le s\le\frac{1}{2}$,
$p_g(s;1)\ge\frac{1}{2}$ implies $f(s)+p_g(s;1)\ge\lambda_0$.
For
$s\ge\frac{1}{2}$, $f(s)+1-s\ge1+\kappa s\ge\lambda_0$, and $f(1)\ge1+\kappa\ge\lambda_0$. Thus $H_{f,g}(a,1)\ge\lambda_0$ for
$a\ge a_0$.

Suppose $H_{f,g}(a,1)\ge\lambda_j$ for $a\ge a_j$, where $a_j$ is
the first point with $f(a_j)=\lambda_j-\frac{1}{2}$, $a_j\le\frac{1}{2}$,
and $\lambda_0\le\lambda_j\le\lambda_*$. The prefix estimate
\eqref{eq:general-prefix} at $s\ge a_j$ and the transfer on $[s,1]$
give $q\ge(1-2b)s+bf(s)+\lambda_j-f(s)$, that is,
$(1-b)f(s)\ge\lambda_j-q+(1-2b)s$. Hence, for
$\frac{1}{2}\le s\le1$,
\[
 f(s)+1-s\ge1+\max\left\{\kappa s, \frac{\lambda_j-q-bs}{1-b}\right\}
 \ge\lambda_{j+1}:=1+\frac{\kappa}{b+(1-b)\kappa}(\lambda_j-q),
\]
the minimum over $s$ of the maximum being attained where the two
affine functions meet, at $s_j=\frac{\lambda_j-q}{b+(1-b)\kappa}$.
Here $s_j\ge\frac{1}{2}$ is the upper inequality in
\eqref{eq:bootstrap-range} at $\lambda_j=\lambda_0$, and
$s_j\le1$ follows from $\lambda_j\le\lambda_*\le1+\kappa$, which is
the lower inequality. At $s=\frac{1}{2}$ this shows that the first
point $a_{j+1}$ with $f(a_{j+1})=\lambda_{j+1}-\frac{1}{2}$ is at
or before $\frac{1}{2}$.
For $a_{j+1}\le s\le\frac{1}{2}$ the gain
$p_g(s;1)\ge\frac{1}{2}$ gives $f(s)+p_g(s;1)\ge\lambda_{j+1}$, and $f(1)\ge1+\kappa\ge\lambda_{j+1}$. Thus $H_{f,g}(a,1)\ge\lambda_{j+1}$
for $a\ge a_{j+1}$.

The map $\lambda\mapsto1+\frac{\kappa}{b+(1-b)\kappa}(\lambda-q)$ has
slope less than one and fixed point $\lambda_*$, and
$\lambda_0\le\lambda_*$ is the upper inequality in
\eqref{eq:bootstrap-range}. Hence
$\lambda_j\nearrow\lambda_*$, the points $a_j$ increase to some
$a_*\le\frac{1}{2}$ with $f(a_*)=\lambda_*-\frac{1}{2}$ by
continuity, and the inequalities $H_{f,g}(a,1)\ge\lambda_j$ for
$a\ge a_*$ pass to the limit. The prefix estimate
\eqref{eq:general-prefix} and the transfer at $a_*$ give
$q\ge(1-2b)a_*+b(\lambda_*-\frac{1}{2})+\frac{1}{2}$, which is the
first bound for $a_*$.
The second follows in the same way from
\eqref{eq:limitprefix}, which is implied by
\eqref{eq:general-prefix} because $b\le\frac{1}{4}$ and $f(u)\le2u$.
\end{proof}

At $b=\frac{1}{4}$ the lemma says that if $\frac{\kappa}{4}\le\varepsilon<\frac{1+\kappa}{8}$
and $q=\frac{3}{4}+\varepsilon$, then $\lambda_*=\frac{1-4\kappa\varepsilon}{1-\kappa}$.
With $b=b_{\mathrm{pin}}(\kappa)$ from \eqref{eq:b-pin-sym}, the
range is $\frac{\kappa}{2}\le\varepsilon<\frac{1+2\kappa}{8}$ and
\begin{equation}\label{eq:D-pin-sym}
 \mathcal D=\frac{4\kappa}{1-2\kappa},\qquad 
 \Delta_*=\frac{\kappa(1-4\varepsilon)}{1-2\kappa};
  \text{with } b=b_{\mathrm{sym}}(\kappa),\qquad 
 \mathcal D=\frac{\kappa(4-\kappa)}{1-3\kappa+\kappa^{2}} .
\end{equation}
With $\lambda=\lambda_*$ and $\Delta=\Delta_*$, the pair has the
certificate \eqref{eq:certificate}, and the target
$q=\frac{3}{4}+\varepsilon$ has the form \eqref{eq:level-notation}
with $\omega=\varepsilon-\frac{\Delta_*}{2}$. The two bounds for
$a_*$ imply $\Delta_*\le2\varepsilon$. Since $q<1$, we obtain
$0\le\Delta_*<\frac{1}{2}$ and $\omega\ge0$. This is how the source
dimension enters Theorems~\ref{thm:universal-main}
and~\ref{thm:main}: through the prefix coefficient $b$ and the
lower bound $1+\kappa$, it raises the level of the certificate, and
the arguments of Section~\ref{sec:unpinned} and of
Part~\ref{part:geometry} are carried out at that level.

\section{Barriers for the reciprocal estimates}
\label{sec:reciprocal-barriers}

A small distance profile forces both endpoints to carry information
before a common cutoff. We first follow that information to later
scales. The two barriers proved here will allow us to move the cutoff
backwards. The reciprocal argument in the next section retains the
separate endpoint values when this common lower bound is no longer enough.

Throughout this section $f,g,d_x,d_y$ satisfy \eqref{eq:profileclass},
the prefix estimate \eqref{eq:limitprefix}, the transfer estimate of
Lemma~\ref{lem:transfer}, and the full partition transfer of
Lemma~\ref{lem:fullpartition}, for both choices of pin, and
$d_x(1),d_y(1)\le Q$. We write $h_g=g-u$ as in
Section~\ref{sec:pinned-obstruction}, and $\ell=\min\{f,g\}$.

\subsection{The correlated prefix--transfer estimate}

For $0<a\le1$ and $z\ge0$ put
\begin{equation}\label{eq:phi}
 \phi_a(z)=z-\frac{3}{4}\min\{z,2a\}=\max\Bigl\{\frac{z}{4}, z-\frac{3a}{2}\Bigr\},
\end{equation}
and, for a continuous nondecreasing $\ell_0:[0,1]\to[0,\infty)$,
\begin{equation}\label{eq:G}
 G_{\ell_0}(a,b)=\frac{a}{2}+\inf_{a\le t\le b}
 \bigl\{\phi_a(\ell_0(t))+p_{\ell_0}(t;b)\bigr\}
  \quad \text{for } 0<a\le b\le1;
\end{equation}
the endpoint $t=b$ contributes $\frac{a}{2}+\phi_a(\ell_0(b))$, since
$p_{\ell_0}(b;b)=0$. The right side depends only on $\ell_0$ on
$[0,b]$ and is nondecreasing in $\ell_0$. The condition $\ell_0\le f,g$
is used in the following lemma.

\begin{lemma}[Correlated prefix and transfer]\label{lem:correlated}
If $\ell_0\le f$ and $\ell_0\le g$, then $d_x(b),d_y(b)\ge G_{\ell_0}(a,b)$
for $0<a\le b\le1$.
\end{lemma}
\begin{proof}
For $a<b$, the transfer estimate without its positive part and the
prefix estimate give
\[
 d_x(b)\ge\frac{a}{2}+\frac{f(a)}{4}+H_{f,g}(a,b)-f(a)
 =\frac{a}{2}+\inf_{a\le t\le b}\bigl\{f(t)-\frac{3}{4}f(a)+p_g(t;b)\bigr\},
\]
where the endpoint $t=b$ accounts for the term $f(b)$ in
\eqref{eq:H}. Since $f$ is nondecreasing and $f(a)\le2a$,
\[
 f(t)-\frac{3}{4}f(a)\ge f(t)-\frac{3}{4}\min\{f(t),2a\}
 =\phi_a(f(t))\ge\phi_a(\ell_0(t)),
\]
because $\phi_a$ is nondecreasing, and $p_g(t;b)\ge p_{\ell_0}(t;b)$.
For $a=b$ the claim is the prefix estimate, as $\ell_0(b)\le2b$.
Exchanging the roles of $x$ and $y$ proves the bound for $d_y$.
\end{proof}

\subsection{Two barriers}

Fix $\lambda\ge1$ and put $k=\lambda-Q$. We say that the tails have
height $\lambda$ at $A$ if $H_{f,g}(A,1)\ge\lambda$ and
$H_{g,f}(A,1)\ge\lambda$. If the tails have height $\lambda$ at $A$,
then for $b\ge A$ the transfer estimate on $[b,1]$ implies
$Q\ge d_x(1)\ge d_x(b)+\lambda-f(b)$, and likewise for $g$, so
Lemma~\ref{lem:correlated} with $\ell_0=\ell$ implies
\begin{equation}\label{eq:tail-common}
 \ell(b)\ge k+G_\ell(a,b) \quad \text{for } b\ge A, 0<a\le b.
\end{equation}
Moreover, if $H_{f,g}(a,1)\ge\lambda$, then
$Q\ge\frac{a}{2}+\frac{f(a)}{4}+\lambda-f(a)$, that is,
\begin{equation}\label{eq:cutoff-defect}
 f(a)\ge\frac{4}{3}\Bigl(k+\frac{a}{2}\Bigr),\qquad \text{hence}\quad
 a\ge k\quad\text{and}\quad f(u)\ge2u-\frac{4}{3}(a-k) \quad \text{for } 0\le u\le a,
\end{equation}
by $f(a)\le2a$ and the $2$-Lipschitz bound. We call $\frac{4}{3}(a-k)$
the defect of the cutoff $a$.

\begin{lemma}[Two barriers]\label{lem:two-barriers}
Let $k>0$, $0\le\beta<k$, $4k\le1$, and $A=k+\frac{3}{4}\beta$.
Suppose that $\ell(t)\ge2t-\beta$ for $0\le t\le A$ and that
\eqref{eq:tail-common} holds. Then
\begin{align}
 \ell(t)&\ge t+k-\frac{\beta}{2}&&\text{for } A\le t\le3k,
 \label{eq:first-barrier}\\
 \ell(t)&\ge\frac{t}{2}+\frac{5k}{2}-\frac{\beta}{2}&&\text{for } 3k\le t\le4k.
 \label{eq:second-barrier}
\end{align}
\end{lemma}
\begin{proof}
Put $m=k-\frac{\beta}{2}$.

We first prove the lower bound on the middle interval. For $0\le\mathfrak d<k$ put $a=k-\mathfrak d$
and
\[
 L(t)=\max\bigl\{t, \min\{2t-\beta, t+a-\frac{\beta}{2}\}\bigr\}.
\]
We claim that
\begin{equation}\label{eq:first-auxiliary}
 G_L(a,b)\ge b-\frac{\beta}{2}
  \quad\text{if }\beta\le2a\text{ and }a\le b\le3a .
\end{equation}
The two affine pieces of $L$ meet at $t_0=a+\frac{\beta}{2}$, and
$L(t)=t+a-\frac{\beta}{2}\ge2a$ for $t\ge t_0$. Put
$F(t)=\frac{a}{2}+\phi_a(L(t))$. For $t\ge t_0$,
$\phi_a(L(t))=L(t)-\frac{3a}{2}$, so $F(t)=t-\frac{\beta}{2}$. For
$a\le t\le t_0$, $L(t)\ge2t-\beta$ implies
$F(t)\ge\frac{a+t}{2}-\frac{\beta}{4}$. The branch $b-t$ of
$p_L(t;b)$ therefore yields $F(t)+b-t\ge b-\frac{\beta}{2}$ in both
cases, using $t\le t_0$ in the second. For the other branch put
$u=b-t$ and consider
$F(t)+\frac{L(u)+t}{2}-(b-\frac{\beta}{2})$. If $t\le t_0$ this is at
least $\frac{a}{2}+\frac{\beta}{4}-u+\frac{L(u)}{2}$, which is
nonnegative for $u\le t_0$ by $L(u)\ge2u-\beta$ and $\beta\le2a$, and
for $u\ge t_0$ by $L(u)=u+a-\frac{\beta}{2}$ and $u\le b-a\le2a$. If
$t\ge t_0$, twice the difference is $b-3u+L(u)$, which is at least
$t-\beta\ge0$ for $u\le t_0$ and at least $3a+\frac{\beta}{2}-b\ge0$
for $u\ge t_0$, by $u\le b-t_0$. This proves the claim.

For $A\le B<3k$ let
$\mathfrak d(B)=\max_{A\le t\le B}[t+m-\ell(t)]_+$. This is
continuous in $B$, and $\mathfrak d(A)=0$ because
$\ell(A)\ge2A-\beta=A+m+\frac{\beta}{4}$. Suppose
$\mathfrak d=\mathfrak d(B)>0$ is so small that $B<3(k-\mathfrak d)$
and $\beta\le2(k-\mathfrak d)$, and put $a=k-\mathfrak d$. Then
$\ell\ge L$ on $[0,B]$: on $[0,A]$ by hypothesis, and on $[A,B]$
because $\ell(t)\ge t+m-\mathfrak d=t+a-\frac{\beta}{2}$. At a point
$b\in[A,B]$ where the maximum $\mathfrak d$ is attained,
\eqref{eq:tail-common}, the monotonicity of $G$ in its envelope, and
\eqref{eq:first-auxiliary} give $\ell(b)\ge k+b-\frac{\beta}{2}=b+m$,
contradicting $\ell(b)=b+m-\mathfrak d$. Hence the set of
$B\in[A,3k)$ with $\mathfrak d(B)=0$ is open and closed in $[A,3k)$,
so it is all of $[A,3k)$, and continuity covers the endpoint $t=3k$.

We now extend the lower bound to the next interval. Put
$L_0(t)=\max\{t,\min\{2t-\beta,t+m\}\}$ and
$P(t)=\frac{t}{2}+\frac{5k}{2}-\frac{\beta}{2}$. By the hypothesis and
the first barrier, $\ell\ge L_0$ on $[0,3k]$, and
$\ell(3k)\ge3k+m=P(3k)$. Fix $3k\le B<4k$, let
$\mathfrak D=\max_{3k\le t\le B}[P(t)-\ell(t)]_+$, and suppose
$\mathfrak D>0$ is so small that $a=k-\frac{\mathfrak D}{2}$ satisfies
$a>\beta$ and $B-a<3k$. Let $b\in[3k,B]$ attain the maximum. We claim
\begin{equation}\label{eq:second-auxiliary}
 G_\ell(a,b)\ge\frac{b}{2}+\frac{3k}{2}-\frac{\beta}{2}-\frac{\mathfrak D}{2}.
\end{equation}
Write $e=\frac{\mathfrak D}{2}$ and $t_0=k+\frac{\beta}{2}$, where the
two affine pieces of $L_0$ meet. For $a\le t\le3k$ we use
$\ell\ge L_0$.
With $F(t)=\frac{a}{2}+\phi_a(L_0(t))$,
\[
 F(t)=
 \begin{cases}
 \frac{a+t}{2}-\frac{\beta}{4},& a\le t\le a+\frac{\beta}{2},\\[2pt]
 2t-\beta-a,& a+\frac{\beta}{2}\le t\le t_0,\\[2pt]
 t+m-a,& t_0\le t\le3k,
 \end{cases}
\]
because $L_0(t)=2t-\beta$ on $[a,t_0]$, $L_0(t)=t+m\ge2a$ on
$[t_0,3k]$, and $2t-\beta\ge2a$ exactly when $t\ge a+\frac{\beta}{2}$.
In each case $F(t)\ge t-\frac{\beta}{2}$, so the branch $b-t$ gives at
least $b-\frac{\beta}{2}$, which suffices since $b\ge3k$. For the
other branch put $u=b-t<3k$.
After subtracting $\frac{b}{2}$ the
quantity to bound is $F(t)+\frac{L_0(u)-u}{2}$. If $t\le t_0$, then
$u\ge3k-t_0\ge t_0$, so $L_0(u)=u+m$ and the quantity is at least
$F(a)+\frac{m}{2}=\frac{3k}{2}-\frac{\beta}{2}-e$. If $t\ge t_0$ and
$u\ge t_0$, it is at least $t_0+m-a+\frac{m}{2}=\frac{3k}{2}-\frac{\beta}{4}+e$.
If $t\ge t_0$ and $\beta\le u\le t_0$, then $L_0(u)=2u-\beta$ and the
quantity equals $b-\frac{u}{2}+m-a-\frac{\beta}{2}\ge b-\frac{k}{2}-\frac{5\beta}{4}+e$,
which suffices since $b\ge3k$ and $\beta<k$. If $u\le\beta$, then
$L_0(u)\ge u$ and the quantity is at least
$b-\beta+m-a=b-\frac{3\beta}{2}+e$, which suffices as well. Finally let
$3k\le t\le b$. Then $\ell(t)\ge P(t)-\mathfrak D\ge2a$, so
$\frac{a}{2}+\phi_a(\ell(t))=\ell(t)-a\ge\frac{t}{2}+\frac{3k}{2}-\frac{\beta}{2}-e$.
Adding $b-t\ge\frac{b}{2}-\frac{t}{2}$ for the first branch, or
$\frac{\ell(u)+t}{2}\ge\frac{u+t}{2}=\frac{b}{2}$ for the other, yields
at least $\frac{b}{2}+\frac{3k}{2}-\frac{\beta}{2}-e$. This proves the
claim. Now \eqref{eq:tail-common} implies
$\ell(b)\ge P(b)-\frac{\mathfrak D}{2}$, contradicting
$\ell(b)=P(b)-\mathfrak D$. The same open-closed argument propagates
$\mathfrak D(B)=0$ from $B=3k$ to every $B<4k$, and continuity covers the endpoint
$t=4k$.
\end{proof}

\begin{lemma}[Updating a common cutoff]\label{lem:cutoff-update}
Suppose $1\le\lambda<\frac{9}{8}$, $\frac{1}{5}<k<\frac{1}{4}$,
$k<\frac{4\lambda-1}{14}$, and $0\le v<k$.
If the tails have height $\lambda$ at
$A'=k+\frac{3v}{4}<\frac{1}{2}$, they also have that height at
\[
a(v)=\max\left\{\lambda-1+v,
\frac{2\lambda-1-k}{4}+\frac{5v}{8},
\frac{4\lambda-1-5k}{9}+\frac{5v}{9}\right\}.
\]
The new defect is
\begin{equation}\label{eq:general-cutoff-update}
\frac{4}{3}(a(v)-k)=\max\left\{
\frac{4v}{3}-\frac{4(k-\lambda+1)}{3},
\frac{5v}{6}+\frac{2\lambda-1-5k}{3},
\frac{20v}{27}+\frac{4(4\lambda-1-14k)}{27}\right\}.
\end{equation}
\end{lemma}
\begin{proof}
The cutoff defect \eqref{eq:cutoff-defect} and
Lemma~\ref{lem:two-barriers} give the early bound and both barriers,
with $v$ in place of $\beta$. The third entry of $a(v)$ implies
$a(v)-k\ge\frac{4\lambda-1-14k+5v}{9}>0$.
If $a(v)\ge A'$ there is nothing to prove, so
let $a(v)<A'$. Competitors $t\ge A'$ and the terminal cap
are covered by the tails at $A'$. For $a(v)\le t\le A'$, the branch
$1-t$ of $p_g(t;1)$ gives $f(t)+1-t\ge1+t-v\ge\lambda$, since
$t\ge\lambda-1+v$. For the other branch, $1-t\in[A',4k]$, because
$A'<\frac{1}{2}$ and $a(v)\ge k>1-4k$.
So the two barriers for $g$
and the early bound for $f$ give
\[
 f(t)+\frac{g(1-t)+t}{2}
 \ge\min\Bigl\{\frac{1}{2}+2t+\frac{k}{2}-\frac{5v}{4},
 \frac{1}{4}+\frac{9t}{4}+\frac{5k}{4}-\frac{5v}{4}\Bigr\}\ge\lambda,
\]
the last inequality being $t\ge a(v)$ for the second and third
entries of $a(v)$. Thus $H_{f,g}(a(v),1)\ge\lambda$, and
$H_{g,f}(a(v),1)\ge\lambda$ by symmetry.

The formula for the new defect follows by subtracting $k$ from
each entry and multiplying by $\frac{4}{3}$.
\end{proof}

\section{Reciprocal estimates at a general transfer level}
\label{sec:reciprocal}\label{sec:new-level}

We now pass the estimates from one endpoint to the other and back
again. The level-one conclusion is the scalar bound $\frac{4}{5}$.
Keeping the transfer level as a parameter will also account for the
source dimension. The essential point is to retain the separate
excesses at the early cutoff until they return to the same contact
equation at the end of the proof.

In this and the following scalar arguments, a scalar system means
four profiles satisfying \eqref{eq:profileclass}, the prefix bounds
\eqref{eq:limitprefix}, and the closed and full partition transfers
of Lemmas~\ref{lem:transfer} and~\ref{lem:fullpartition} in both
orientations. In particular, the endpoint profiles are nondecreasing
and $2$-Lipschitz; the distance profiles are nondecreasing and
$1$-Lipschitz; all four vanish at zero. The notation
\begin{equation}\label{eq:new-system}
d_x(u)\ge\frac{u}{2}+\frac{f(u)}{4},\qquad
d_x(b)-d_x(a)\ge[\widehat H_{f,g}(a,b)-f(a)]_+
\end{equation}
will refer to these inequalities together with their reverse orientation.
We write $\mathcal L_v(g;[a,b])=L_v(g-u;[a,b])$ for the
partition loss of Section~\ref{sec:pinned-obstruction} and abbreviate
$\mathcal L_v(g;R)=\mathcal L_v(g;[0,R])$.
The full transfer implies the closed transfer. Lemma~\ref{lem:uniform-partition-witnesses}
allows strict caps to pass to uniformly converging systems; no
continuity of the partition infimum is needed.

Restriction to $[0,r]$ followed by division of both axes by $r$
preserves this system. At level one its certificate is automatic.
At the first hit $f(a)=\frac{1}{2}$ we have $a\le\frac{1}{2}$.
For competitors below $\frac{1}{2}$, use $f\ge\frac{1}{2}$ and
$p_g\ge\frac{1}{2}$; for later competitors, use
$f(u)+1-u\ge1$. Thus $H_{f,g}(a,1)\ge1$, and likewise
with the endpoints exchanged.

\begin{proposition}[Reciprocal estimate at level $\lambda$]\label{newlev:prop:level}
Assume the scalar system \eqref{eq:new-system} in both orientations. Let $\lambda=1+\Delta$ with
$0\le\Delta<\frac{1}{8}$. Suppose there are cutoffs $a_f,a_g\le\frac{1}{2}$ such that
\begin{align*}
 f(a_f)&=\lambda-\frac{1}{2},&H_{f,g}(a_f,1)&\ge\lambda,\\
 g(a_g)&=\lambda-\frac{1}{2},&H_{g,f}(a_g,1)&\ge\lambda.
\end{align*}
Then
\begin{equation}\label{newlev:eq:level-result}
 \max\{d_x(1),d_y(1)\}\ge\frac{4}{5}+\frac{3}{5}\Delta.
\end{equation}
\end{proposition}

\begin{proof}
Suppose the conclusion fails. Choose
\[
 \max\left\{d_x(1),d_y(1),\frac{11+10\Delta}{14}\right\}
 <Q<\frac{4}{5}+\frac{3}{5}\Delta.
\]
Such a choice is possible because $\Delta<\frac{1}{8}$. Put $k=\lambda-Q$.
Then
\begin{equation}\label{newlev:eq:krange}
 \frac{1+2\Delta}{5}<k<\frac{3+4\Delta}{14}<\frac{1}{4}.
\end{equation}
We say that a cutoff is common if both closed tails have height
at least $\lambda$ there.

We begin by locating the least common cutoff.
The prefixes and transfer at $a_f,a_g$ show that both cutoffs are
at most
\[
 A_{\rm init}=2Q-\frac{\lambda}{2}-\frac{3}{4}
             =\frac{3}{4}+\frac{3\Delta}{2}-2k.
\]
The initial defect is $v_0=1+2\Delta-4k$. Lemma~\ref{lem:cutoff-update} has defect maps
\[
 \frac{4}{3}v-\frac{4}{3}(k-\Delta),\quad
 \frac{5}{6}v+\frac{1+2\Delta-5k}{3},\quad
 \frac{20}{27}v+\frac{4(3+4\Delta-14k)}{27}.
\]
Their fixed points are $4(k-\Delta)$, $2+4\Delta-10k<0$, and
\[
 \beta_{\max}=\frac{12+16\Delta}{7}-8k.
\]
The range \eqref{newlev:eq:krange} ensures
$0<\beta_{\max}<v_0<k$ and $v_0<4(k-\Delta)$.
Iteration therefore yields a common cutoff
$k+\frac{3\beta_{\max}}{4}$.
The common cutoffs form a nonempty closed interval. Its least element has the form
\begin{equation}\label{newlev:eq:cutoff}
 A_*=k+\frac{3\beta}{4},\qquad 0\le\beta\le\beta_{\max}.
\end{equation}
\eqref{eq:cutoff-defect} and Lemma~\ref{lem:two-barriers} give,
with $\ell=\min(f,g)$, $x_0=k+\frac{\beta}{2}$, and $c=k-\frac{\beta}{2}$,
\begin{align}
 \ell(u)&\ge2u-\beta&& (0\le u\le A_*),\label{newlev:eq:early}\\
 \ell(u)&\ge u+c&& (x_0\le u\le3k),\label{newlev:eq:middle}\\
 \ell(u)&\ge \frac{u}{2}+\frac{5k}{2}-\frac{\beta}{2}&& (3k\le u\le4k).
 \label{newlev:eq:late}
\end{align}
The orderings implied by \eqref{newlev:eq:krange} are
\[
 k\le x_0\le A_*<2k<T:=1-2k<3k<1-A_*\le b_*:=1-k<4k<1.
\]
In addition $\beta<k$ and $k-\beta>\Delta$.

We next improve the bounds on $[T,b_*]$.
For $b\ge A_*$, the oriented form of Lemma~\ref{lem:correlated}
and the transfer from $b$ to $1$ imply
\begin{equation}\label{rec-eq:correlated}
f(b)\ge k+\frac{a}{2}+\inf_{a\le u\le b}
\{\phi_a(f(u))+p_g(u;b)\}.
\end{equation}
Only $k=\lambda-Q$ enters this estimate. We use it to prove the
following bounds without replacing the two endpoint excesses by their minimum.

Put
\[
 P=f(x_0)-2k,\qquad  p=g(x_0)-2k.
\]
Then $0\le P,p\le\beta$, and the $2$-Lipschitz bound implies
\begin{equation}\label{rec-eq:backwardexcess}
 f(t)\ge2t-\beta+P \quad \text{for } t\le x_0,\qquad
   g(t)\ge2t-\beta+p \quad \text{for } t\le x_0.
\end{equation}
Set $F=f(T)$ and $G=g(T)$, and define
\begin{equation}\label{rec-eq:cones}
 \begin{split}
 C_f&=\min\{c,a_0+\frac{P}{4}+\frac{G}{2}\},\\
 C_g&=\min\{c,a_0+\frac{p}{4}+\frac{F}{2}\},\qquad 
 a_0=\frac{7k}{2}-1-\frac{\beta}{4}.
 \end{split}
\end{equation}

\begin{equation}\label{newlev:eq:cones} f(u)\ge u+C_f,\quad g(u)\ge u+C_g\quad(T\le u\le b_*).\end{equation}

We prove the first assertion. Monotonicity and Lipschitz continuity
give $g(u)\ge G-2[T-u]_+$. Let $d_f=a_0+\frac{P}{4}+\frac{G}{2}$, and put
$M_f=b_*+d_f$.
For $T\le b\le b_*$ and $k\le t\le T$ we claim
\begin{equation}\label{rec-eq:earlyg}
 k+\frac{k}{2}+\phi_k(f(t))+p_g(t;b)\ge b+C_f.
\end{equation}
On $[k,x_0]$ use \eqref{rec-eq:backwardexcess} and split at
$y_0=k+\frac{\beta-P}{2}$. The term before the projection gain is bounded below by
$\frac{3k}{2}+\frac{t}{2}-\frac{\beta-P}{4}$ for $t\le y_0$, and by
$2t-\beta+P$ for $t\ge y_0$.
The branch $b-t$ contributes at least $b+c$.
For the other branch, the cone for $g$ yields at least
$\min\{M_f,b+d_f\}$ on the first interval. On the second interval
it yields at least
$\min\{M_f+\frac{\beta-P}{2},b+d_f\}$.
Since $b+d_f\le M_f$, both are at least $b+d_f$.

For $x_0\le t\le T$ we have
$f(t)\ge\max\{t+c,2k+P\}$. The first gain branch again contributes
$b+c$. The other gain branch uses $b-t\le T$ and is at least
\[
 b+\frac{G}{2}-T+f(t)-\frac{t}{2}
 \ge b+\frac{G}{2}-T+\frac{3k}{2}-\frac{\beta}{4}+\frac{P}{2}
 \ge b+d_f.
\]
This proves \eqref{rec-eq:earlyg}.

Replace $k$ in the cutoff by $a=k-\epsilon>0$. The cost on
$[a,T]$ is at least $b+C_f-3\epsilon$. On $[k,T]$, changing $a$
loses at most $\frac{\epsilon}{2}$, since $\phi_a$ increases as $a$
decreases. On $[a,k]$, compare with $t=k$. The $2$-Lipschitz source
profile loses at most $2\epsilon$ inside the $1$-Lipschitz function
$\phi_k$, the additive cutoff loses $\frac{\epsilon}{2}$, and monotonicity
of $g$ implies $p_g(t;b)\ge p_g(k;b)-\frac{\epsilon}{2}$.

Let $B_f(u)=\max\{F,u+C_f\}$ on $[T,b_*]$. It is nondecreasing
and $1$-Lipschitz. The middle barrier \eqref{newlev:eq:middle} implies
$F\ge T+c>2k$, and $C_f\le c$, so $B_f(T)=F$.
If $D_0=\max[B_f-f]_+>0$, choose a maximizer $b$ and
$0<\epsilon\le\min\{\frac{D_0}{4},\frac{k}{2}\}$. At that maximizer
$B_f(b)=b+C_f$. Early competitors have cost at least
$B_f(b)-3\epsilon>f(b)$.
For $T\le t\le b$, we have $t\ge \frac{b}{2}$, so $p_g(t;b)=b-t$.
Also $f(t)\ge F>2k>2a$. Their cost in \eqref{rec-eq:correlated} is
\[
 f(t)+\epsilon+b-t
 \ge B_f(b)-D_0+\epsilon>f(b).
\]
This contradicts \eqref{rec-eq:correlated}. Hence $f\ge B_f$, proving
the assertion. Exchange the orientations for $g$.

The full partition estimate in the reverse orientation improves
the cone bounds. Each of the two possible losses requires a
correction by $\Delta$.

The endpoint values satisfy
\begin{equation}\label{rec-eq:feedback}
 \begin{split}
 G\ge\min\{&2C_g+3k+\frac{\beta}{2}-p+\frac{P}{2}-2\Delta,\\
             &C_g+1-k-\frac{\beta}{2}-p+\frac{P}{2}-\Delta\},
 \end{split}
\end{equation}
and the symmetric inequality for $F$.

Let $a<k$ tend to $k$, put $\epsilon=k-a$, and
$e_t=2t-f(t)$. Prefix and full transfer give a competitor
$s\in[a,A_*)$, with $R=1-s$, for which
\begin{equation}\label{rec-eq:losslower}
 L_s(h_g;[0,R])\ge
 h_g(R)+s+(k-a)-e_s+\frac{3}{4}e_a-\Delta-o(1).
\end{equation}
Indeed their threshold is at most
$Q+a-\frac{3e_a}{4}<\lambda$.
Competitors at or above $A_*$ and the terminal
cap have cost at least $\lambda$. The full partition gain is
$g(R)-L_s(h_g;[0,R])$, which proves \eqref{rec-eq:losslower}.

On $[0,x_0]$, \eqref{eq:partition-positive-endpoint} bounds the loss by $h_g(x_0)$, whose
positive part vanishes because $s>\beta$ for $a$ close to $k$.
On $[x_0,T]$, subtract the floor $c$ from \eqref{newlev:eq:middle} and
use \eqref{eq:partition-positive-endpoint}. Adding results in
\begin{equation}\label{rec-eq:twoloss}
 L_s(h_g;[0,x_0])+L_s(h_g;[x_0,T])
 \le\max\left\{G-T+p,\frac{G-s-\beta+p}{2}\right\}.
\end{equation}
On $[T,R]$ use $h_g\ge C_g$ from \eqref{newlev:eq:cones}.
If $R>b_*$, monotonicity extends this to
$h_g\ge C_g-\epsilon$ throughout $[T,R]$.
After subtracting that floor, the positive-part argument in
\eqref{eq:partition-positive-endpoint} is at most $2\epsilon$, because
$R-T-s=2k-2s\le2\epsilon$. Consequently
\begin{equation}\label{rec-eq:tailloss}
 L_s(h_g;[T,R])\le h_g(R)-C_g+2\epsilon.
\end{equation}
Thus passing slightly beyond $b_*$ produces a vanishing error only.

Write $\mathfrak v=s-a$. The source excess sharpens both elementary
optimizations needed to combine the losses. We claim
\begin{equation}\label{rec-eq:optimizations}
 3\mathfrak v-2e_s+\frac{3}{2}e_a\ge-\frac{\beta-P}{2},\qquad 
 \mathfrak v-e_s+\frac{3}{4}e_a\ge-\frac{\beta-P}{2}.
\end{equation}
For $s\le x_0$, put $B=\beta-P$. Then
$0\le e_a\le B$ and $e_s\le\min\{B,e_a+2\mathfrak v\}$.
Each expression is minimized at the intersection
$\mathfrak v=\frac{B-e_a}{2}$. The resulting values are
$-\frac{B}{2}$ and $-\frac{B}{2}+\frac{e_a}{4}$, respectively.
For $s\ge x_0$, use
$e_s\le\min\{\beta,2\mathfrak v-2\epsilon-P\}$ and $e_a\ge0$.
The minimum over $\mathfrak v\ge0$ is therefore bounded below by
$-\frac{\beta}{2}+3\epsilon+\frac{3P}{2}$ and
$-\frac{\beta}{2}+\epsilon+\frac{P}{2}$.
These are at least $-\frac{\beta-P}{2}$, proving both bounds.

If the second entry of \eqref{rec-eq:twoloss} is used, combining
\eqref{rec-eq:losslower}--\eqref{rec-eq:tailloss} yields
\[
 a+2k+3\mathfrak v-2e_s+\frac{3}{2}e_a-2\Delta
 \le G-\beta+p-2C_g+o(1).
\]
The first bound in \eqref{rec-eq:optimizations} and $a\nearrow k$
give the first entry of \eqref{rec-eq:feedback}.
If the first entry of \eqref{rec-eq:twoloss} is used, the second bound
in \eqref{rec-eq:optimizations} proves the second entry of
\eqref{rec-eq:feedback}. One case occurs along a sequence
$a\nearrow k$. This proves the feedback estimate.

Put $\xi=P-p$, $A=5k-1-\Delta$, and
$B=3k-\frac{1}{2}-\frac{\beta}{2}-\frac{\Delta}{2}$. The cone formulas and
\eqref{rec-eq:feedback} imply
\[
 C_f\ge\min\{c,C_g+A+\frac{\xi}{2},\frac{C_g}{2}+B+\frac{\xi}{2}\},
\]
and its symmetric version with $\xi$ replaced by $-\xi$.
To eliminate $C_g$, define the nondecreasing function
\[
\Phi_\xi(u)=\min\left\{c,u+A+\frac{\xi}{2},\frac{u}{2}+B+\frac{\xi}{2}\right\}.
\]
The two inequalities are $C_f\ge\Phi_\xi(C_g)$ and
$C_g\ge\Phi_{-\xi}(C_f)$. Monotonicity implies
$C_f\ge\Phi_\xi(\Phi_{-\xi}(C_f))$. Expanding this composition results in
\[
\begin{aligned}
\Phi_\xi(\Phi_{-\xi}(u))=\min\Bigl\{&c,c+A+\frac{\xi}{2},
\frac{c}{2}+B+\frac{\xi}{2},u+2A,\\
&\frac{u}{2}+A+B,\frac{u}{2}+\frac{A}{2}+B+\frac{\xi}{4},
\frac{u}{4}+\frac{3B}{2}+\frac{\xi}{4}\Bigr\}.
\end{aligned}
\]
Since $A>0$, the entry $u+2A$ is strictly larger than $u$.
The remaining three nonconstant entries are strictly larger than
$u$ precisely when, respectively,
\[
u<2(A+B),\qquad  u<2B+A+\frac{\xi}{2},\qquad 
u<2B+\frac{\xi}{3}.
\]
Thus, if $C_f$ were below the three constant entries and these
three thresholds, every entry in the composition would exceed
$C_f$, contradicting $C_f\ge\Phi_\xi(\Phi_{-\xi}(C_f))$. We obtain
\begin{equation}\label{newlev:eq:six}
\begin{split}
 C_f\ge\min\Bigl\{&c,c+A+\frac{\xi}{2},\frac{c}{2}+B+\frac{\xi}{2},\\
 &2(A+B),2B+A+\frac{\xi}{2},2B+\frac{\xi}{3}\Bigr\}.
\end{split}
\end{equation}
Exchanging $f$ and $g$ replaces $\xi$ by $-\xi$ and proves the six
bounds for $C_g$. No sign assumption on $\xi$ is used in this elimination.

We now return to the contact at the least cutoff.
Minimality of $A_*$ forces a contact in at least one orientation;
label it so that
\[
 f(A_*)+p_g(A_*;1)=1+\Delta.
\]
Set $K=1+2\Delta-4k$ and $J=K-\beta$. The cone for $g$ at
$1-A_*$ implies
\begin{equation}\label{newlev:eq:weakupper}
 C_g\le\min\{J,K-2P\}.
\end{equation}
Indeed, $p_g(A_*;1)\ge\frac{1+C_g}{2}$ and
$f(A_*)\ge\max\{2k+\frac{\beta}{2},2k+P\}$.

Let
\[
 a_1=5k-1-2\Delta>0,
\]
so $A=a_1+\Delta$. The six entries in \eqref{newlev:eq:six}, minus $J$,
are
\begin{equation}\label{newlev:eq:differences}
 \begin{gathered}
 a_1+\frac{\beta}{2},\qquad
 2a_1+\Delta+\frac{\beta+\xi}{2},\qquad
 \frac{3a_1}{2}+\frac{\Delta}{2}+\frac{\beta}{4}+\frac{\xi}{2},\\
 4a_1+3\Delta,\qquad
 3a_1+2\Delta+\frac{\xi}{2},\qquad
 2a_1+\Delta+\frac{\xi}{3}.
 \end{gathered}
\end{equation}
If $\xi\le0$, the six differences for $C_g$ are all positive,
contrary to \eqref{newlev:eq:weakupper}. Thus $0<\xi\le\beta$.
For the weaker floor $C_g$, its first, second, and fourth entries
are positive. At least one of the other three must be nonpositive.
The third requires $\xi\ge3a_1+\Delta+\frac{\beta}{2}$ and hence
$\xi\ge6a_1+2\Delta$. The fifth and sixth give still stronger
bounds. Therefore
\[
 \beta\ge\xi\ge6a_1+2\Delta.
\]
Substitute this into \eqref{newlev:eq:differences}. Every entry is at
least $4a_1+\Delta$, so
\begin{equation}\label{newlev:eq:stronglower}
 C_f\ge J+4a_1+\Delta.
\end{equation}

The branch $1-A_*$ cannot attain the contact, because
$f(A_*)+1-A_*\ge1+k-\frac{\beta}{4}>1+\Delta$.
Consequently
\[
 g(1-A_*)=2+2\Delta-2f(A_*)-A_*.
\]
Monotonicity and the two lower bounds for $f(A_*)$ yield
\[
 G\le\min\{2+2\Delta-5k-\frac{7\beta}{4},
                 2+2\Delta-5k-\frac{3\beta}{4}-2P\}.
\]
In the formula for $C_f$, these give
\begin{align}
 C_f&\le\min\{k+\Delta-\frac{9\beta}{8}+\frac{P}{4},
                     k+\Delta-\frac{5\beta}{8}-\frac{3P}{4}\}\notag\\
    &\le k+\Delta-\beta.\label{newlev:eq:strongupper}
\end{align}
Use the first entry when $P\le\frac{\beta}{2}$ and the second when
$P\ge\frac{\beta}{2}$. Finally,
\[
 (J+4a_1+\Delta)-(k+\Delta-\beta)=3a_1>0.
\]
This contradicts \eqref{newlev:eq:stronglower}--\eqref{newlev:eq:strongupper}
and proves the proposition.
\end{proof}

\begin{corollary}[The level-one scalar bound]\label{prop:reciprocal-scalar}
Every scalar system satisfies
$\max\{d_x(1),d_y(1)\}\ge\frac{4}{5}$.
\end{corollary}
\begin{proof}
The level-one certificates were verified before
Proposition~\ref{newlev:prop:level}. Apply that proposition with $\Delta=0$.
\end{proof}

\section{The rational Hausdorff bound}\label{sec:new-rational}\label{sec:unpinned}
The source dimension raises the certificate above level one. The
reciprocal estimate converts this higher certificate into an explicit
distance bound by one linear inequality.

\begin{proof}[The bound $Y$ in Theorem~\ref{thm:main}]
Fix $0<\kappa<S_0-1$. Consider the information profiles of a pair satisfying (H1)--(H4), with both effective dimensions greater than $1+\kappa$. Corollary~\ref{cor:prefix-rate}, in its symmetric case, has the coefficient
\[
 b=\frac{1-3\kappa+\kappa^2}{(1-\kappa)(4-\kappa)},
 \qquad d_x(u)\ge(1-2b)u+bf(u),
\]
and the analogous bound for $d_y$. Lemma~\ref{lem:bootstrap} raises the tail level to
\[
 \lambda=1+\Delta,\qquad
 \Delta=D(1-Q),\qquad
 D=\frac{\kappa(4-\kappa)}{1-3\kappa+\kappa^2},
\]
when both terminal distance profiles are at most $Q$ in the
permitted prefix range. Proposition~\ref{newlev:prop:level} then forces
\[
 Q\ge\frac{4}{5}+\frac{3}{5}D(1-Q),\qquad
 Q\ge1-\frac{1}{5+3D}
       =\frac{4+\kappa^2}{5-3\kappa+2\kappa^2}=Y(1+\kappa).
\]

Here is the parameter check needed for this argument by
contradiction. If the proposed scalar bound fails, choose $Q$ just
below $Y(1+\kappa)$ and above both terminal values. The lower end of
the bootstrap range is
$R(1+\kappa)=1-b(1-\kappa)<Y(1+\kappa)$, as follows from
the positive difference computed below. The upper end
$1-\frac{b(1-\kappa)}{2}$ is at least $\frac{7}{8}$, while
$Y(1+\kappa)<\frac{7}{8}$. At the limiting value $Q=Y(1+\kappa)$,
\[
 \Delta=\frac{\kappa(4-\kappa)}{5-3\kappa+2\kappa^2}<\frac{1}{8}
 \quad\Longleftrightarrow\quad
 \kappa<\frac{7-\sqrt{41}}{4}.
\]
The inequalities remain valid when $Q$ is sufficiently close from
below. Thus the certificate and Proposition~\ref{newlev:prop:level} apply.

This proves the scalar improvement. Apply Proposition~\ref{prop:fixed-pair} and the point-to-set argument of Section~\ref{sec:classical}, using an auxiliary Frostman exponent strictly below $t=\dim_{\mathrm H} E$. Let $1+\kappa$ increase to $\min\{t,S_0\}$. Continuity proves the endpoint and the constant continuation; the independent bound $R$ remains available.

For the comparison with the bootstrap range, write $t=1+\kappa$.
Direct subtraction yields
\[
Y(t)-R(t)=\frac{1-5\kappa+9\kappa^2-8\kappa^3+2\kappa^4}{(5-3\kappa+2\kappa^2)(4-\kappa)}>0.
\]
Indeed $0\le\kappa\le S_0-1<\frac{3}{20}$, and the numerator
is at least $1-5\kappa-8\kappa^3>1-\frac{3}{4}-\frac{27}{1000}>0$.
\end{proof}

\section{Matching profiles and the minimum-ratio argument}\label{sec:new-matching}

When the endpoint profiles agree, a single conditioned distance
profile satisfies the scalar inequalities in both orientations.
We can then choose a scale at which its information per unit
precision is nearly smallest. At that scale, every earlier
distance prefix has a useful linear lower bound. This additional
bound lets the tail iteration reach the full source floor.

\begin{proposition}[A stronger bound for a matching pair]\label{newmatch:prop:matched}
Suppose the scalar system has $f=g$ and
$f(u)\ge(1+\kappa)u$, with $0<\kappa<\frac{1}{8}$. Then, separately for
each orientation,
\begin{equation}\label{newmatch:eq:matched}
 d_x(1),d_y(1)\ge\frac{4}{5}+\frac{3}{5}\kappa.
\end{equation}
\end{proposition}

\begin{proof}
Use one of the two distance profiles and call it $d$. Because
$f=g$, the quadruple $(f,f,d,d)$ satisfies every scalar inequality.
The system is invariant under restriction to $[0,r]$ and the
rescaling $F(u)=\frac{f(ru)}{r}$, $D(u)=\frac{d(ru)}{r}$.

Let $q=\inf_{0<r\le1}\frac{d(r)}{r}$. The level-one reciprocal bound,
applied after each rescaling, implies $q\ge\frac{4}{5}$. If $d(1)$ were
below the claimed bound, so would $q$. Choose $r_j$ with
$\frac{d(r_j)}{r_j}\to q$ and extract a uniformly convergent subsequence
of the rescaled profiles. The limiting scalar system still
satisfies all the transfers. For the full transfer, use Lemma~\ref{lem:uniform-partition-witnesses} to pass every strict cap to the approximating
profiles. Rename the limits $f,d$. They satisfy
\begin{equation}\label{newmatch:eq:minratio}
 d(1)=q,\qquad d(u)\ge qu,\qquad f(u)\ge(1+\kappa)u.
\end{equation}

We now raise the tail height to $1+\kappa$.
Initially take $\lambda_0=1+\frac{\kappa}{2}$ and let $a_0$ be the first
point with $f(a_0)=\lambda_0-\frac{1}{2}$. The floor implies $a_0\le\frac{1}{2}$.
For $a_0\le r\le\frac{1}{2}$, use $p_f(r;1)\ge\frac{1}{2}$; for $r\ge\frac{1}{2}$,
use $f(r)+1-r\ge1+\kappa r$. Thus
$H_{f,f}(a,1)\ge\lambda_0$ for $a\ge a_0$.

Suppose the tail has height $\lambda_j$ from its first hitting
point $a_j\le\frac{1}{2}$, where
$\lambda_0\le\lambda_j\le1+\kappa$. The transfer and
\eqref{newmatch:eq:minratio} imply
\[
 f(r)\ge\lambda_j-q+qr\qquad(r\ge a_j).
\]
For $\frac{1}{2}\le r\le1$ it follows that
\begin{align*}
 f(r)+1-r
 &\ge1+\max\{\kappa r,\lambda_j-q-(1-q)r\}\\
 &\ge\lambda_{j+1}:=
 1+\frac{\kappa}{\kappa+1-q}(\lambda_j-q).
\end{align*}
The two affine functions intersect inside $[\frac{1}{2},1]$ because
$\lambda_0\le\lambda_j\le1+\kappa$ and $q<1$.
At $r=\frac{1}{2}$, the inequality puts the next first hitting point at
or before $\frac{1}{2}$. Earlier competitors are handled by
$p_f(r;1)\ge\frac{1}{2}$, and the terminal cap by $f(1)\ge1+\kappa$.
Hence the induction continues.

The affine update has slope less than one and fixed point
$1+\kappa$. Therefore $\lambda_j\uparrow1+\kappa$. The first
hitting points converge, giving the certificate in
Proposition~\ref{newlev:prop:level} at $\Delta=\kappa$ for $(f,f,d,d)$.
That proposition forces $q\ge\frac{4}{5}+\frac{3\kappa}{5}$, the contradiction.
\end{proof}

The linear lower bound in \eqref{newmatch:eq:minratio} is not an additional
assumption on the original distance profile. It is obtained by
choosing a scale where the normalized distance complexity is
nearly minimal and taking a limit. That choice of scale is the reason the tail height reaches the full source floor.

\section{Rigidity at the reciprocal endpoint}\label{sec:new-rigidity}

The scalar inequalities alone admit endpoint models at four fifths
(Proposition~\ref{prop:direct-terminal-flat}). What matters for the
geometric argument is that every such model has a long interval
of excess information. We prove this for two possibly different
endpoint profiles. Figure~\ref{fig:rigid-envelope} shows the forced lower envelope. In this section put
\[
 k=\frac{1}{5},\qquad T=\frac{3}{5},\qquad b_*=\frac{4}{5}.
\]
\begin{proposition}[Two-profile endpoint rigidity]\label{newrig:prop:rigidity}
Every system \eqref{eq:new-system} with $d_x(1),d_y(1)\le\frac{4}{5}$ satisfies
\begin{align}
 f(u)=g(u)&=2u &&(0\le u\le\frac{1}{5}),\label{newrig:eq:initial}\\
 f(u),g(u)&\ge u+\frac{1}{5} &&(\frac{1}{5}\le u\le\frac{3}{5}).\label{newrig:eq:plateau}
\end{align}
Consequently
\begin{equation}\label{newrig:eq:tail03}
 H_{f,g}(\frac{3}{10},1),\ H_{g,f}(\frac{3}{10},1)\ge1.
\end{equation}
\end{proposition}

\begin{figure}[ht]
\centering
\begin{tikzpicture}[x=9cm,y=4.3cm]
 \fill[blue!9] (.2,.2)--(.2,.4)--(.6,.8)--(.6,.6)--cycle;
 \draw[->,gray] (0,0)--(1.06,0) node[right] {$u$};
 \draw[->,gray] (0,0)--(0,1.12) node[above] {profile};
 \draw[gray,dashed] (0,0)--(1,1);
 \draw[very thick,blue!75!black] (0,0)--(.2,.4)--(.6,.8)--(.8,.8)--(1,1);
 \foreach \x/\lab in {.2/{\frac{1}{5}},.6/{\frac{3}{5}},.8/{\frac{4}{5}},1/{1}}
 {\draw (\x,.015)--(\x,-.015) node[below] {$\lab$};}
 \node[blue!70!black] at (.38,.90) {$f,g\ge u+\frac{1}{5}$};
 \node[gray] at (.82,.62) {$f,g\ge u$};
\end{tikzpicture}
\caption{A common lower envelope forced at the endpoint. The shaded
strip has height $\frac{1}{5}$. This envelope is a lower bound, not an assertion
that the profiles coincide or that the displayed function is an extremizer.}
\label{fig:rigid-envelope}
\end{figure}
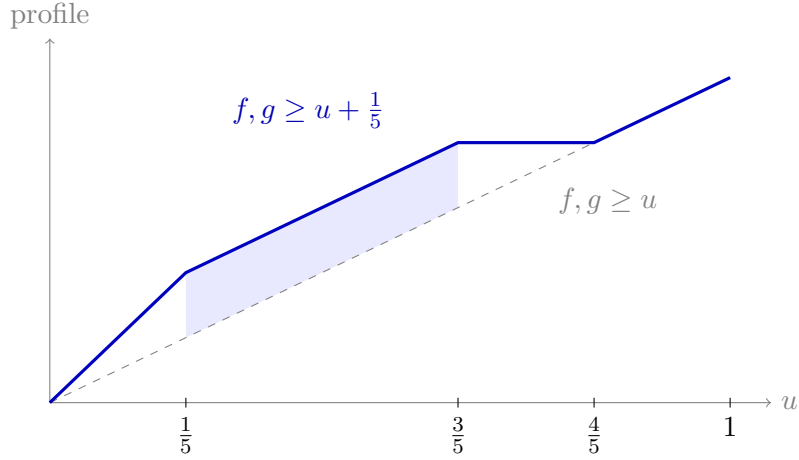

\subsection{Endpoint cutoff data and the exact skew identity}

We first verify the cutoff, cone and feedback estimates at $k=\frac{1}{5}$.
The strict reciprocal proof must be adjusted at this endpoint.
The least common closed-tail cutoff is
\begin{equation}\label{newrig:eq:cutoff}
 A=k+\frac{3\beta}{4},\qquad 0\le\beta\le\frac{4}{35}.
\end{equation}
Set
\begin{gather*}
 x_0=k+\frac{\beta}{2},\quad c=k-\frac{\beta}{2},\quad J=k-\beta,\\
 P=f(x_0)-2k,\quad p=g(x_0)-2k,\quad
 F=f(T),\quad G=g(T),\quad \xi=P-p,\\
 C_f=\min\{c,-\frac{3}{10}-\frac{\beta}{4}+\frac{P}{4}+\frac{G}{2}\},\quad
 C_g=\min\{c,-\frac{3}{10}-\frac{\beta}{4}+\frac{p}{4}+\frac{F}{2}\}.
\end{gather*}
Here $0\le P,p\le\beta$. The barriers and cones are
\begin{align}
 f(u),g(u)&\ge2u-\beta &&(0\le u\le A),\label{newrig:eq:early}\\
 f(u),g(u)&\ge u+c &&(x_0\le u\le T),\label{newrig:eq:middle}\\
 f(u)&\ge u+C_f,\quad g(u)\ge u+C_g &&(T\le u\le b_*).
 \label{newrig:eq:cones}
\end{align}
The reciprocal feedback reads
\begin{align}
 C_f&\ge\min\{c,C_g+\frac{\xi}{2},\frac{C_g+J}{2}+\frac{\xi}{2}\},\label{newrig:eq:coupled}\\
 C_g&\ge\min\{c,C_f-\frac{\xi}{2},\frac{C_f+J}{2}-\frac{\xi}{2}\}.\notag
\end{align}

To justify \eqref{newrig:eq:cutoff} at the endpoint, the first-hit argument
starts at $A_0=\frac{7}{20}$ with defect $\beta_0=k$. The first barrier's
proof still works at that defect. Its requirements are
$\beta_0\le2(k-\mathfrak d)$ and $b<3(k-\mathfrak d)$ for
sufficiently small positive maximum defect $\mathfrak d$.
It yields $\min(f(T),g(T))\ge\frac{7}{10}$. For
$\frac{17}{50}\le u\le\frac{7}{20}$, the early bound and this last estimate give
\[
 f(u)+\frac{g(1-u)+u}{2}\ge\frac{5}{2}u+\frac{3}{20}\ge1,
 \qquad f(u)+1-u\ge1+u-k\ge1,
\]
and similarly in the opposite orientation. Thus $\frac{17}{50}$ is a
common cutoff with defect $\frac{14}{75}<k$. The usual update now applies;
its defect maps at $k=\frac{1}{5}$ are
\[
 v\mapsto\frac{4v}{3}-\frac{4}{15},\qquad v\mapsto\frac{5v}{6},\qquad
 v\mapsto\frac{20v}{27}+\frac{4}{135}.
\]
Their maximum decreases to $\frac{4}{35}$. In the proof of the update,
the strict inequality formerly written $k>1-4k$ is replaced by
$a(v)>k=1-4k$, which follows from its third entry. The other
endpoint equalities, $T=3k$ and $b_*=4k$, cause no change to the
barrier, cone, or feedback proofs.

Minimality of $A$ forces a contact in at least one orientation;
label it so that $f(A)+p_g(A;1)=1$. The length branch cannot
attain the contact, because $f(A)-A\ge k-\frac{\beta}{4}>0$. Hence
\[
 g(1-A)=2-2f(A)-A.
\]
The cone for $g$ at $1-A$ and
$f(A)\ge\max\{2k+\frac{\beta}{2},2k+P\}$ give
\begin{align}
 C_g&\le\min\{J,k-2P\},\label{newrig:eq:upperg}\\
 G&\le\min\{1-\frac{7\beta}{4},1-\frac{3\beta}{4}-2P\}.\notag
\end{align}
Substituting the latter two bounds in the formula for $C_f$ gives
\begin{equation}\label{newrig:eq:upperf}
 C_f\le\min\{k-\frac{9\beta}{8}+\frac{P}{4},k-\frac{5\beta}{8}-\frac{3P}{4}\}\le J.
\end{equation}
For the last inequality use the first entry if $P\le\frac{\beta}{2}$ and
the second otherwise.

If $\beta=0$, \eqref{newrig:eq:early} and \eqref{newrig:eq:middle} already prove
the proposition. Suppose henceforth $\beta>0$. Since
$C_f,C_g\le J<c$, the half-slope branch in the first inequality
of \eqref{newrig:eq:coupled} is no smaller than $C_g+\frac{\xi}{2}$.
The ceiling $c$ exceeds $C_f$. Thus $C_f\ge C_g+\frac{\xi}{2}$.
The second inequality implies the reverse bound. We obtain
the identity
\begin{equation}\label{newrig:eq:skew}
                   C_f-C_g=\frac{P-p}{2}.
\end{equation}
Both cone minima are on their second branches, so also
\begin{equation}\label{newrig:eq:Gidentity}
 G=2C_f+3k+\frac{\beta-P}{2},\qquad
 F=2C_g+3k+\frac{\beta-p}{2}.
\end{equation}

We have not substituted $k=\frac{1}{5}$ in the
six-branch lower bound from the strict reciprocal argument. Its derivation discards a term
$u+2(5k-1)$ using strict positivity, which fails at this endpoint.
The direct identity \eqref{newrig:eq:skew} is what is needed instead.

\subsection{A cutoff inequality that can be used twice}

Write $e_u=2u-f(u)$. This information deficit is nonnegative and
nondecreasing. Euler's number is written $\mathrm e$, and
$\mathfrak e$ denotes the plateau parameter of
Proposition~\ref{prop:pin-plateau}.
The same arguments will be used for $g$.

\begin{lemma}[Cutoff inequality]\label{newrig:lem:push}
Let $k<a<A$ and $\sigma\in\{\frac{3}{4},1\}$. If
\[
                    d_x(a)\ge a-(1-\sigma)e_a,
\]
then $a-k\ge\sigma e_a$.
\end{lemma}
\begin{proof}
Suppose instead that $\theta=a-k<\sigma e_a$. The transfer and
the upper bound on $d_x(1)$ give
\[
 \widehat H_{f,g}(a,1)\le\frac{4}{5}+a-\sigma e_a<1.
\]
Choose competitors $v\in[a,A)$ with full cost below this upper
bound plus an arbitrarily small error. Competitors at or beyond
$A$ and the terminal cap have cost at least one. Put $R=1-v$,
$w=v-a$, and $B=\beta-P$. The partition loss must satisfy
\begin{equation}\label{newrig:eq:losslower}
 \mathcal L_v(g;R)\ge h_g(R)+v+k-a-e_v+\sigma e_a-o(1).
\end{equation}
The endpoint partition bound, using \eqref{newrig:eq:middle}--\eqref{newrig:eq:cones},
gives
\begin{align}
 \mathcal L_v(g;[0,x_0])+\mathcal L_v(g;[x_0,T])
 &\le\max\{G-T+p,\frac{G-v-\beta+p}{2}\},\label{newrig:eq:twoloss}\\
 \mathcal L_v(g;[T,R])&\le h_g(R)-C_g.\notag
\end{align}
Here $v>k>\beta$, $T<R<b_*$, and $R-T-v=2k-2v<0$.
Concatenating these finite partitions bounds the full loss.
Consequently at least one of the inequalities
\begin{align}
 G&\ge C_g+1-k+U_1-p-o(1),\label{newrig:eq:branches}\\
 G&\ge2C_g+a+2k+U_2+\beta-p-o(1),\notag
\end{align}
holds, where $U_1=w-e_v+\sigma e_a$ and
$U_2=3w-2e_v+2\sigma e_a$.

For $v\le x_0$ use $e_v\le\min\{B,e_a+2w\}$, which gives
\[
 U_1\ge-\frac{B}{2}+(\sigma-\frac{1}{2})e_a,\qquad
 U_2\ge-\frac{B}{2}+(2\sigma-\frac{3}{2})e_a.
\]
For $v\ge x_0$ use $e_v\le\min\{\beta,2w+2\theta-P\}$,
which gives
\[
 U_1\ge-\frac{\beta}{2}+\frac{P}{2}-\theta+\sigma e_a,\quad
 U_2\ge-\frac{\beta}{2}+\frac{3P}{2}-3\theta+2\sigma e_a.
\]
These bounds are obtained at the meeting points of the two
affine bounds for $e_v$; allowing all $w\ge0$ only weakens them.

Subtract $G$ from the two right sides of \eqref{newrig:eq:branches}, and
use \eqref{newrig:eq:skew}--\eqref{newrig:eq:Gidentity}. The lower bounds are
\[
\begin{array}{c|cc}
 &\text{first branch}&\text{second branch}\\ \hline
 v\le x_0&J-C_g+(\sigma-\frac{1}{2})e_a
          &\theta+(2\sigma-\frac{3}{2})e_a\\
 v\ge x_0&J-C_g-\theta+\sigma e_a
          &P-2\theta+2\sigma e_a .
\end{array}
\]
All four are strictly positive because $C_g\le J$, $P\ge0$,
$\theta>0$, and $\sigma e_a>\theta$. The errors can be taken
smaller than their minimum, a contradiction. The proof applies
also when $a\ge x_0$, in which case the first row is unnecessary.
\end{proof}

\subsection{Recovering the plateau}

The ordinary prefix is $d_x(a)\ge a-\frac{e_a}{4}$. Apply
Lemma~\ref{newrig:lem:push} with $\sigma=\frac{3}{4}$ for every $k<a<A$:
\begin{equation}\label{newrig:eq:weakdefect}
 e_a\le\frac{4}{3}(a-k),\qquad h_f(a)\ge\frac{4k-a}{3}.
\end{equation}
By continuity $e_k=0$, and monotonicity of $e$ forces $f(u)=2u$
for $0\le u\le k$. The same holds for $g$.

Let $a_0=\frac{A}{2}<k$. On $[a_0,k]$ one has $h_f(u)=u\ge a_0$.
On $[k,A]$, \eqref{newrig:eq:weakdefect} implies
\[
 h_f(u)\ge k-\frac{\beta}{4},\qquad
 k-\frac{\beta}{4}-a_0=\frac{k}{2}-\frac{5\beta}{8}\ge\frac{1}{35}>0.
\]
For $a_0\le u\le A$, the projection gain is saturated because
$p_g(u;A)=A-u$, since $u\ge \frac{A}{2}$ and $g\ge u$.
It follows that
\[
 H_{f,g}(a_0,A)=A+a_0.
\]
Indeed, the cost is $A+h_f(u)\ge A+a_0$, with equality at
$u=a_0$; this also includes the terminal cap. The prefix forces
$d_x(a_0)=a_0$. The closed transfer then implies
\[
 d_x(A)\ge a_0+(A+a_0)-f(a_0)=A.
\]
The $1$-Lipschitz upper bound forces equality and then
\begin{equation}\label{newrig:eq:linear-distance}
                 d_x(u)=u\quad(0\le u\le A).
\end{equation}
The same reasoning applies to $d_y$.

Now use Lemma~\ref{newrig:lem:push} a second time, with $\sigma=1$ and
\eqref{newrig:eq:linear-distance}. We obtain $e_a\le a-k$, hence
\begin{equation}\label{newrig:eq:earlyplateau}
              h_f(u),h_g(u)\ge k\quad(k\le u\le A).
\end{equation}
The first application forced exact information growth before $k$.
That exact growth made the distance profile linear up to $A$, so the
second application can retain the entire source defect.

The plateau propagates to $3k$ as follows.
The oriented version of Lemma~\ref{lem:correlated}, combined with the transfer from $b$ to $1$, at $b\ge A$
and $a=k-\varepsilon$, reduces to
\begin{equation}\label{newrig:eq:correlated}
 f(b)\ge\varepsilon+\inf_{a\le u\le b}\{f(u)+p_g(u;b)\},
\end{equation}
because $f(a)=2a$ and monotonicity makes the linear branch of
$\phi_a(f(u))$ active. If the barrier $f(u)\ge u+k$ fails on
$[k,2k]$, let $D>0$ be its maximum defect, attained at $b>A$.
For $u\ge k\ge \frac{b}{2}$, the gain is $b-u$ and the cost is at least
$b+k-D=f(b)$. For $a\le u\le k$, comparison with $u=k$ loses
at most $3\varepsilon$ and leaves a cost of at least $b+k-3\varepsilon$.
Choose $0<2\varepsilon<D$ in \eqref{newrig:eq:correlated} to get a contradiction.
Exchange the two profiles to establish both barriers on $[k,2k]$.

Next suppose there is a positive maximum defect on $[2k,3k]$.
Late competitors $u\ge \frac{b}{2}$ are handled in the same way.
For $k\le u\le \frac{b}{2}$, both $u$ and $b-u$ belong to $[k,2k]$,
where both barriers now hold. The two gain branches give costs
at least $b+k$ and $\frac{b}{2}+\frac{5k}{2}\ge b+k$, respectively.
The interval $[a,k]$ again loses at most $3\varepsilon$.
This proves \eqref{newrig:eq:plateau} for both profiles.

Finally, if $\frac{3}{10}\le u\le\frac{1}{2}$, then $f(u)\ge\frac{1}{2}$ and
$p_g(u;1)\ge\frac{1}{2}$. If $u\ge\frac{1}{2}$, then $p_g(u;1)=1-u$
and $f(u)\ge u$. The terminal cap is at least one.
This proves \eqref{newrig:eq:tail03} and Proposition~\ref{newrig:prop:rigidity}.

\begin{lemma}[Qualitative stability]\label{lem:new-stability}
For every $\varepsilon>0$ there is $\eta(\varepsilon)>0$ such that
every scalar system \eqref{eq:new-system} with
$d_x(1),d_y(1)\le\frac{4}{5}+\eta(\varepsilon)$ satisfies
\[
 \begin{aligned}
 f(u),g(u)&\ge2u-\varepsilon
       &&(0\le u\le\frac{1}{5}),\\
 f(u),g(u)&\ge u+\frac{1}{5}-\varepsilon
       &&(\frac{1}{5}\le u\le\frac{3}{5}),
 \end{aligned}
\]
and
$H_{f,g}(\frac{3}{10},1),H_{g,f}(\frac{3}{10},1)\ge1-\varepsilon$.
The constant depends only on $\varepsilon$.
\end{lemma}
\begin{proof}
If this failed, choose systems with terminal bounds decreasing
to $\frac{4}{5}$ and violating one of the conclusions at the fixed
tolerance $\varepsilon$. The uniform Lipschitz bounds give a
jointly uniformly convergent subsequence. The prefix and closed
transfers pass to the limit directly, and
Lemma~\ref{lem:uniform-partition-witnesses} passes the full transfer.
The limit is an endpoint system. Proposition~\ref{newrig:prop:rigidity}
satisfies all the asserted bounds with zero error. Uniform convergence,
and continuity of the closed functional $H$, give a contradiction.
\end{proof}

\section{The source parent and the annular cap}\label{sec:parent-cap}

The finite geometry will require a bound on the source mass that
can lie both in a small parent square and near one circle. The
interval-transfer proof already contains the needed information.
Before its last chain rule, the coarse source location is still
available as conditioning data. We keep it.

A lower bound for this conditional complexity becomes an upper bound
for annular mass by measure coding. The cancellation of the parent
mass leaves a bound for every occupied
parent and radius bin with respect to the same original source
measure.

\begin{lemma}[Conditional source-parent transfer]
\label{lem:conditional-parent-transfer}
Let $(x,y)$ satisfy the pointwise hypotheses of the interval-transfer
lemma, relative to $A$, and put $\varrho=|x-y|$.  Along any common limiting
profile sequence $r_j$, for fixed $0<a<b\le1$, write
$C_j=Q_{\lceil ar_j\rceil}(y)$.  Then
\[
 \liminf_{j\to\infty}
 \frac{K_{\lfloor br_j\rfloor}^{A,x}(\varrho\mid C_j)}{r_j}
 \ge[H_{f,g}(a,b)-f(a)]_+.
\]
The notation on the left means ordinary finite-string conditioning on
the encoded dyadic square, while retaining the full oracle $x$.
\end{lemma}
\begin{proof}
Put $H=H_{f,g}(a,b)$. If $H\le f(a)$, the assertion follows
from nonnegativity.  Otherwise
choose $f(a)<c<H$ and put
\[
n_j=\lfloor br_j\rfloor,\qquad
\ell_j=\lceil ar_j\rceil,\qquad m_j=\lfloor cr_j\rfloor.
\]
Use the source-lowering oracle $G_j=D(n_j,y,\frac{m_j}{n_j})$.
For each sufficiently small fixed rational $\varepsilon>0$, the
competitor bounds in the proof of Lemma~\ref{lem:transfer} verify
both hypotheses of \cite[Lemma~15]{FS}, uniformly for all large $j$.
The two finite-precision inequalities being combined are
\begin{align*}
m_j-K_{\ell_j}^A(y)-O(\log r_j)
&\le K_{n_j,\ell_j}^{A,G_j,x}(y\mid y),\\
K_{n_j,\ell_j}^{A,G_j,x}(y\mid y)
&\le K_{n_j,\ell_j}^{A,G_j,x}(\varrho\mid y)
       +3\varepsilon n_j+O_\varepsilon(\log r_j).
\end{align*}
The first follows from Lemma~14(iii) of \cite{FS}, (H3), and the
chain rule. The second is Lemma~15, with the rational cap parameter
included in its description cost. Dropping $G_j$ from the last
conditional complexity increases it up to an absolute simulation
constant. Hence
\[
 K_{n_j,\ell_j}^{A,x}(\varrho\mid y)
 \ge m_j-K_{\ell_j}^A(y)-3\varepsilon n_j
       -O_\varepsilon(\log r_j).
\]
Every allowed precision-$\ell_j$ rational approximation
to $y$ leaves only a bounded number of possibilities for $C_j$.
For each such approximation, specify the correct square by an
index of bounded length. The index may depend on the approximation.
Its length does not. A program supplied with that index and the
precision parameters can therefore simulate a program conditioned
on $C_j$.  Consequently
\[
 K_{\lfloor br_j\rfloor,\lceil ar_j\rceil}^{A,x}(\varrho\mid y)
 \le K_{\lfloor br_j\rfloor}^{A,x}(\varrho\mid C_j)+O(\log r_j).
\]
This is the required direction for the worst-case coarse conditioning.
No maximum is replaced by a minimum.  Take the profile limit, then
$\varepsilon\searrow0$ and $c\nearrow H$.
\end{proof}

\begin{proposition}[An occupied-parent annular mass cap]
\label{prop:conditional-parent-cap}
Let $q>0$, $\lambda\ge1$, and $0<a<1$.
Let $\mu$ be a compactly supported $s$-Frostman probability measure,
$s>1$, whose dyadic masses are computable relative to $A$.
Suppose the source--pin pairs under consideration satisfy the
pointwise prefix and transfer hypotheses and the almost-sure source
coding identity
\[
 K_k^A(y)=-\log_2\mu(Q_k(y))+O_y(\log(k+2)),
\]
and that every limiting profile of the pair along precisions $n$
with $K_n^{A,x}(|x-y|)\le qn$ satisfies $H_{f,g}(a,1)\ge\lambda$.
Then for each such fixed pair $(x,y)$ and every $\zeta>0$, all
sufficiently large $n$ with $K_n^{A,x}(|x-y|)\le qn$ satisfy
\begin{equation}\label{eq:conditional-critical-parent}
 K_n^{A,x}(|x-y|\mid Q_{\lfloor an\rfloor}(y))
 \ge\lambda n-K_{\lfloor an\rfloor}^A(y)-\zeta n.
\end{equation}
In particular, if $C=Q_{\lfloor an\rfloor}(y)$ and $J$ is the
half-open dyadic radius interval of length $2^{-n}$ containing $|x-y|$, then
\begin{equation}\label{eq:occupied-parent-mass-cap}
 \mu\{z\in C:|x-z|\in J\}\le2^{-(\lambda-\zeta)n}
\end{equation}
for all sufficiently large such scales, after decreasing the error
tolerance in~\eqref{eq:conditional-critical-parent} if necessary.
If the original relation is Borel, then for any fixed pin
probability measure $\nu$, the thresholds can
be made uniform on a fixed Borel subrelation, that is, a Borel
subset of the original set of pairs, losing arbitrarily little
$(\nu\times\mu)$-measure.  Therefore, on that subrelation,
\eqref{eq:occupied-parent-mass-cap} holds simultaneously for every
source parent and every terminal radius bin occupied by a retained
low-complexity pair.  The mass on its left is the full source mass,
not the pruned mass.
\end{proposition}
\begin{proof}
Suppose~\eqref{eq:conditional-critical-parent} fails along an infinite
sequence of low-distance scales.  Extract uniformly convergent
endpoint and distance profiles on a subsequence.  By hypothesis
$H_{f,g}(a,1)\ge\lambda$, so
Lemma~\ref{lem:conditional-parent-transfer} contradicts the assumed
failure after division by $n$.  Fixed rounding changes do not affect
this argument.

For the mass conclusion, write
$p_{C,J}=\mu\{z\in C:|x-z|\in J\}$.
Since $s>1$, dyadic coordinate lines and circles have zero source
mass.  For $\mu(C)>0$, its computability allows division by $\mu(C)$
uniformly from the label of $C$. The normalized restriction
$\mu_C=\frac{\mu|_C}{\mu(C)}$ has a radial pushforward about
$x$ whose dyadic interval masses are lower semicomputable relative
to $(A,x,C)$: inner
approximations by dyadic squares converge from below because the
boundaries have zero mass, although no rate of convergence is
available. This is enough to construct the required code. For each
fixed input $(C,n)$ and every level-$n$ radius interval $J$, request
a codeword of length $k+1$ as soon as the enumeration proves
$\frac{p_{C,J}}{\mu(C)}>2^{-k}$. Make each request once. The total
requested weight satisfies
\[
\sum_J\sum_{k:\,2^{-k}<\frac{p_{C,J}}{\mu(C)}}2^{-(k+1)}
\le\sum_J\frac{p_{C,J}}{\mu(C)}=1.
\]
The Kraft--Chaitin theorem \cite[Chapter~4]{LiVitanyi} supplies a
prefix-free code by one algorithm taking $(C,n)$ as input.
The shortest requested length for $J$ is at most
$-\log_2\frac{p_{C,J}}{\mu(C)}+2$. Its decoder outputs the midpoint
of $J$, which is within $2^{-n}$ of $|x-y|$. Describing $n$ costs
$O(\log(n+2))$ bits. Thus, when $p_{C,J}>0$,
\[
 K_n^{A,x}(|x-y|\mid C)
 \le-\log_2\frac{p_{C,J}}{\mu(C)}+O(\log(n+2)).
\]
The coding algorithm takes the square label $C$ as conditional
input, so its program index and its additive constant are uniform
over all occupied parents.
There is no additional $K^A(C)$ term. The zero-mass case is immediate.
Combining this with
\eqref{eq:conditional-critical-parent} and
$K_{\lfloor an\rfloor}^A(y)=-\log_2\mu(C)+O_y(\log n)$
leaves the two inequalities
\[
 \lambda n+\log_2\mu(C)-\zeta_0n-O_y(\log n)
 \le K_n^{A,x}(|x-y|\mid C)
 \le-\log_2p_{C,J}+\log_2\mu(C)+O(\log n),
\]
where $\zeta_0>0$ is the tolerance used in the conditional bound.
Subtracting $\log_2\mu(C)$ from both sides proves
$-\log_2p_{C,J}\ge\lambda n-\zeta_0n-O_y(\log n)$ along the
low-distance scales. Choose $\zeta_0<\zeta$ and then increase the
threshold to absorb the logarithmic error.
This proves the stated error-tolerance formulation.

Finally, the conditions are Borel.  For each fixed error tolerance,
the increasing union of the sets on which the required inequalities
hold for every $n\ge N$ is the full admissible relation.  Choose one
$N$ retaining all but any prescribed amount of its product measure.
Every occupied parent and radial-bin then contains a pair to which the uniform
bound applies, proving the simultaneous assertion. This restriction
changes only the relation of pairs. The oracle $A$ and the original
measures $\mu,\nu$ remain fixed. The pin oracle is never truncated.
\end{proof}

The hypothesis on $H_{f,g}$ is supplied by Section~\ref{sec:plateau}.
If the pair has the certificate \eqref{eq:certificate} at level
$\lambda$ and $d_x(1)\le q$, then $H_{f,g}(a,1)\ge\lambda$ for every
$a\ge a_*$, because $H_{f,g}(a,1)$ is nondecreasing in $a$, and
Lemma~\ref{lem:compatible-hit} shows that $a_*\le a_c+\varpi\omega$. For
$\lambda=1$ the certificate holds for every admissible pair, and
Lemma~\ref{lem:first-hit} implies $a_*\le\frac{1}{4}+2\omega$.
So the
cap holds at every parent scale $a>\frac{1}{4}+2\omega$.

\Needspace{13\baselineskip}
\begin{corollary}[Keeping all parents]
\label{cor:all-parents-kept}
Let $\nu$ be the pin probability measure, let $G$ be one of the
relations with uniform thresholds in
Proposition~\ref{prop:conditional-parent-cap}, and put
\[
 R_n=\{(x,y)\in G:K_n^{A,x}(|x-y|)\le qn\},\qquad
   \mu_{x,n}=\mu|_{R_{n,x}},\qquad
   P_n=\int\mu_{x,n}(\mathbb R^2) d\nu(x).
\]
For all sufficiently large $n$, with $\ell=\lfloor an\rfloor$,
every level-$\ell$ dyadic parent $C$ and terminal
radius bin $J$ satisfy
\[
 \mu_{x,n}\{z\in C:|x-z|\in J\}\le2^{-(\lambda-\zeta)n}.
\]
This uses the original common source measure in all parents
simultaneously.
\end{corollary}
\begin{proof}
An unoccupied cell and radial-bin has zero restricted mass.
An occupied one is
bounded by the full source mass in
\eqref{eq:occupied-parent-mass-cap}.
\end{proof}

At the critical scales, the cap bounds each occupied parent--radius
mass by $\delta^{\lambda}$, up to the stated power loss. Finite selection will
retain only source cells carrying a specified fraction of their
original mass. Dividing the cap by the least retained cell mass then
bounds the number of assigned vertices in each parent on a circle.
This is the parent-incidence hypothesis used in
Section~\ref{sec:circle-rows}. Figure~\ref{fig:parent-annulus}
shows the intersection whose mass is controlled.

\begin{figure}[htbp]
\centering
\begin{tikzpicture}[x=1cm,y=1cm]
\fill[blue!12] (10:4.06)
arc (10:42:4.06)--(42:4.30) arc (42:10:4.30)--cycle;
\draw[blue!65!black] (10:4.06) arc (10:42:4.06);
\draw[blue!65!black] (10:4.30) arc (10:42:4.30);
\fill[gray!6] (3,.5) rectangle (5.2,2.7);
\begin{scope}
\clip (3,.5) rectangle (5.2,2.7);
\fill[blue!25,even odd rule] (0,0) circle (4.30) (0,0) circle (4.06);
\draw[step=.22,gray!45,very thin] (3,.5) grid (5.2,2.7);
\draw[blue!65!black] (0,0) circle (4.06);
\draw[blue!65!black] (0,0) circle (4.30);
\end{scope}
\draw[thick] (3,.5) rectangle (5.2,2.7);
\fill (0,0) circle (.045) node[below] {$x$};
\draw[dashed] (0,0)--(3.88,1.55);
\fill (3.88,1.55) circle (.04) node[above left,fill=white,inner sep=1pt] {$y$};
\node[above] at (4.1,2.9) {source parent $C$};
\draw[<->] (3,.12)--(5.2,.12);
\node[below,font=\small] at (4.1,.06) {$\operatorname{side}(C)=2^{-\lfloor an\rfloor}$};
\node[align=left,anchor=west] at (6,2.45)
{$J$: radius interval\\of length $\delta=2^{-n}$};
\node[align=left,anchor=west] at (6,1.35)
{Shaded intersection:\\$C\cap\{z\in\mathbb R^2:|x-z|\in J\}$};
\node[align=left,anchor=west] at (6,.25)
{Fine source cells\\have side $\delta$.};
\end{tikzpicture}
\caption{The occupied-parent annular cap. Supplying the coarse square
$C$ leaves a lower bound for the information in the radius. Coding
the radial distribution of $\frac{\mu|_C}{\mu(C)}$ cancels the
parent mass and bounds the full mass of the shaded intersection.
The annulus and fine cells are enlarged for visibility.
Their
drawn widths are not a scale comparison. In the finite selection
the parent side is $\eta=2^{-\lfloor a_{\mathrm{par}}n\rfloor}$, with
$a_{\mathrm{par}}$ slightly larger than the first hitting scale $a_*$
of the source profile.}
\label{fig:parent-annulus}
\end{figure}
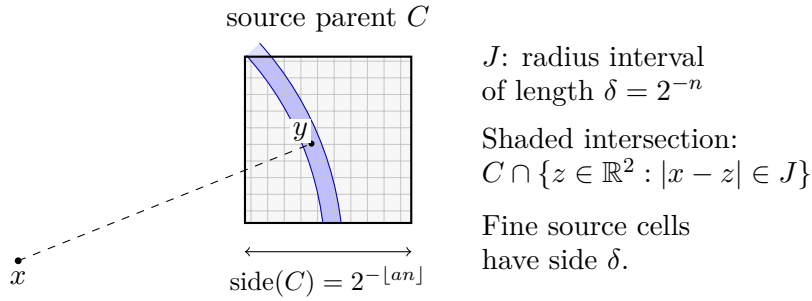

\part{Circle graphs, bisectors, and the pinned improvement}\label{part:geometry}

\section{Perpendicular bisectors and their size}\label{sec:bisector-geometry}

This section separates two notions of size that occur in the proof.
Hausdorff dimension measures a fixed family of lines through all
small scales. A finite count measures a line family at one
specified scale. The proposition below concerns Hausdorff
dimension.
The finite counts are the subject of
Sections~\ref{sec:circle-rows}--\ref{sec:rectangular-sampling}.
The identity after the proposition explains how repeated circle
incidences produce rich tubes. Both follow directly from
elementary geometry.

\begin{proposition}[Bisectors with one endpoint fixed]
\label{prop:bisector-dimension}
Let $y\in\mathbb R^2$ and $F\subset\mathbb R^2\setminus\{y\}$.
In the line-space metric of Section~\ref{sec:preliminaries},
\[
\dim_{\mathrm H}\{\mathcal B(y,z):z\in F\}
=\dim_{\mathrm H}F.
\]
Consequently, if $E$ has positive Hausdorff dimension and
$\mathcal B(E)=\{\mathcal B(y,z):y,z\in E, y\ne z\}$, then
\[
\dim_{\mathrm H}\mathcal B(E)\ge\dim_{\mathrm H}E.
\]
\end{proposition}
\begin{proof}
Write $z-y=rn$, with $r>0$ and $n\in S^1$. Its bisector has
normal $n$ and offset
\[
c=n\cdot y+\frac{r}{2}.
\]
Conversely, reflecting $y$ across $L_{n,c}$ recovers the other
endpoint:
\[
z=y+2(c-n\cdot y)n.
\]
The second formula is independent of the choice between $(n,c)$
and $(-n,-c)$. The first map and its inverse are locally Lipschitz
away from $r=0$. Cover the punctured plane by countably many
bounded angular charts with $r$ bounded above and bounded away
from zero. On sufficiently small closed subcharts the maps are
bi-Lipschitz. Hausdorff dimension is preserved there, and countable
stability proves the equality for $F$.

For the final assertion fix $y\in E$ and apply the equality to
$F=E\setminus\{y\}$. Removing one point does not change a
positive Hausdorff dimension, and this pinned bisector family is
contained in $\mathcal B(E)$.
\end{proof}

The approximate bisector relation follows from the identity
\[
|x-y|^2-|x-z|^2
=|y|^2-|z|^2-2x\cdot(y-z).
\]
If the two distances agree to error $O(\delta)$ and remain bounded,
then $x$ is within $O(\frac{\delta}{|y-z|})$ of
$\mathcal B(y,z)$. Source separation $|y-z|\ge c\eta$ makes this a
tube of width $O(\frac{\delta}{\eta})$, and separation at a scale
$\delta^{c}$ with $c\le a_{\mathrm{par}}$ results in the width
$O(\delta^{1-c})$ used in Sections~\ref{sec:circle-rows}
and~\ref{sec:finite-incidence}.
This calculation explains why the circle graph must produce
separated pairs rather than just distinct vertices.

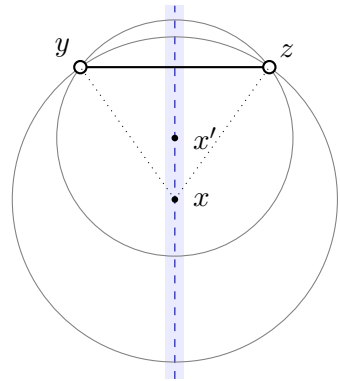
\begin{figure}[htbp]
\centering
\begin{tikzpicture}[scale=1.25]
\fill[blue!8] (-.1,-1.9) rectangle (.1,2.05);
\draw[dashed,blue!65!black] (0,-1.9)--(0,2.05);
\draw[gray] (0,0) circle (1.72047);
\draw[gray] (0,.65) circle (1.25);
\draw[thick] (-1,1.4)--(1,1.4);
\fill[white] (-1,1.4) circle (.065);
\fill[white] (1,1.4) circle (.065);
\draw[thick] (-1,1.4) circle (.065);
\draw[thick] (1,1.4) circle (.065);
\node[above left] at (-1,1.4) {$y$};
\node[above right] at (1,1.4) {$z$};
\fill (0,0) circle (.035) node[right=3pt] {$x$};
\fill (0,.65) circle (.035) node[right=3pt] {$x'$};
\draw[dotted] (0,0)--(-1,1.4);
\draw[dotted] (0,0)--(1,1.4);
\node[align=left,anchor=west] at (2.15,1.25)
{The same pair occurs on\\circles about several pins.};
\node[align=left,anchor=west] at (2.15,.15)
{Their centers lie near the\\perpendicular bisector.};
\node[align=left,anchor=west] at (2.15,-.95)
{Tube width: $O(\frac{\delta}{|y-z|})$\\For $|y-z|\ge c\eta$: $O(\frac{\delta}{\eta})$.};
\end{tikzpicture}
\caption{Repeated circle incidences produce a rich bisector tube.
The disks represent terminal source cells.
Their radii and the tube
width are enlarged for visibility. Actual incidences have error
$O(\delta)$, so the centers lie in a tube rather than necessarily on
the exact bisector.}
\label{fig:bisectors}
\end{figure}

\section{Drawing approximate circle incidences}\label{sec:circle-drawing}

In a multigraph, distinct edges may join the same pair of vertices.
These are parallel edges, and their number is the multiplicity of
the pair. A loop joins a vertex to itself. A simple graph has
neither loops nor parallel edges. A complete graph joins every
pair of distinct vertices. An independent set of vertices has no
edge between any two of its members. These are the graph
conventions used below.

Join consecutive assigned source vertices along each circle.
This produces a graph with many edges and few crossings, because
two distinct circles meet at most twice. The approximate incidences
require some care. A circle can enter a source disk without being
assigned to its vertex.

Inside each disk we route assigned strands through the vertex and
other strands past it. The latter create no new graph incidences.
The extra crossings are controlled by the number of circle strands
entering the disks. At our cardinalities, this cost is smaller than
the main circle-pair count.

A drawing below represents each edge by a simple arc joining
its two vertices. Its interior avoids every graph vertex.
Overlapping arcs are separated by arbitrarily small perturbations.
All interior intersections are transverse, and no three edge
interiors meet at one point. A proper intersection is a crossing.
Edges may share an endpoint, and that shared endpoint is not counted.
The crossing number $\operatorname{cr}(G)$ is the least number of
crossings over these drawings of $G$. We write $V(G)$ and $E(G)$
for its numbers of vertices and edges when a graph argument uses
these as numerical quantities.

\begin{lemma}[Drawing prescribed circle incidences]
\label{lem:eighth-robust-circle-drawing}
Fix $R\ge1$ and $0<d_0<D_0$.  Let $0<\delta\ll \frac{d_0}{R}$, and let
$V$ source vertices $p_v$ have mutual distances greater than $5R\delta$.
Let $\mathcal C$ be a finite set of distinct circles, all with radii
in $[d_0,D_0]$, and write $C=\#\mathcal C$.
Choose pairwise disjoint vertex disks
$D_v=B(p_v,r_v)$, with $R\delta<r_v<2R\delta$.
Specify an arbitrary set of incidences $(v,\gamma)$ such that
$\gamma$ meets $D_v$.  Let $I$ denote their number.

There is a loop-free multigraph $G$ on the source vertices, with at least
$I-C$ edges, obtained by joining consecutive prescribed visits on
each circle, and a drawing satisfying
\begin{equation}\label{eq:eighth-disk-local-crossing}
 \operatorname{cr}(G)
 \le 2\binom C2+4\sum_v\binom{k_v}2,\qquad
   k_v=\#\{\gamma\in\mathcal C:\gamma\cap D_v\ne\varnothing\}.
\end{equation}
The disk radii may be chosen generically inside the indicated ranges.
No prescribed visit is lost if its circle already meets the smaller
disk of radius $R\delta$.

If the circles are centered at $M$ distinct pins and, for each pin,
their radii are $\delta$-separated, then
\begin{equation}\label{eq:eighth-robust-crossing-upper}
 \operatorname{cr}(G)\le C_R(C^2+VM^2).
\end{equation}
One may take
\[
C_R=2(\lfloor4R\rfloor+2)^2.
\]
Every graph edge retains its original circle label and has both
endpoints among that circle's prescribed incidences.
\end{lemma}
\begin{proof}
Perturb each disk radius sufficiently slightly inside the indicated
range, preserving exactly the same circle visits, and avoid the
finitely many values that produce a circle tangency to a disk boundary
or put a circle--circle intersection on such a boundary.  Since $r_v<d_0$, a circle meeting
$D_v$ has a single proper arc inside it.  This follows also directly
by writing the disk condition on a circle as one inequality for the
cosine of its angular parameter.  There is therefore one local
strand for each of the $k_v$ circles.

Inside $D_v$, replace an incident strand by the two straight spokes
from its boundary points to $p_v$.  Replace a nonincident strand by
a straight chord, or by a generic two-segment polygonal path if a
chord would pass through $p_v$.  All these paths stay in the convex
disk.  A nonincident path is chosen to avoid $p_v$, overlaps, and
nonrequired multiple intersections.  Each pair of local strands
has at most four proper intersections, since each uses at most two
straight segments.  In particular, incident strands meeting at the
common graph vertex do not contribute proper crossings.

Outside the disks retain the original circle arcs.  Distinct circles
have at most two intersection points, so these portions contribute
at most $2\binom C2$ crossings.  If necessary, an arbitrarily small
generic perturbation of the circles, preserving all disk visits,
puts their exterior intersections in general position.
The same
bound holds.  Such a perturbation exists because the finite list of
boundary incidences was made transverse.  The local paths can then
be chosen in general position as above.

A circle with $k\ge2$ prescribed visits now contributes $k$ graph edges
between successive prescribed vertices in cyclic order.  When $k=2$
these are two parallel edges, which are allowed.  Discard circles
with zero or one prescribed visit.  The total number of retained
edges is at least $I-C$.  A nonprescribed disk passage never creates
a graph endpoint.  The preceding crossing count proves
\eqref{eq:eighth-disk-local-crossing}.

For a pin $x$, a circle of radius $r$ can meet $D_v$ only if
$||x-p_v|-r|\le r_v<2R\delta$.  A $\delta$-separated set of radii
has at most $\lfloor4R\rfloor+2$ members in this interval.  Thus
$k_v\le(\lfloor4R\rfloor+2)M$, proving
\eqref{eq:eighth-robust-crossing-upper} with the stated constant:
the two terms in \eqref{eq:eighth-disk-local-crossing} are at most
$C^2$ and $2V(\lfloor4R\rfloor+2)^2M^2$.
\end{proof}

The angular span of a row about a pin $x$ is the smallest interval
in its chosen angular chart containing the angles $\arg(v-x)$ of
its entries $v$. Its length is the angular diameter. An angular
window is an interval in that chart.

\begin{corollary}[Clique edges in disjoint circle rows]
\label{cor:eighth-circle-row-cliques}
Assume the hypotheses of
Lemma~\ref{lem:eighth-robust-circle-drawing}, including that the circles
are centered at $M$ distinct pins and have $\delta$-separated radii
at each pin. On each circle choose disjoint rows of at most $s$
prescribed source vertices, in circular order, such that their angular
spans have disjoint interiors. Draw all edges of the complete graph on
each row, retaining parallel edges from different rows or circles.  For each fixed $s$ this multigraph
has a drawing satisfying
\begin{equation}\label{eq:eighth-row-clique-crossing}
 \operatorname{cr}(G_s)\le C_{R,s}(C^2+VM^2).
\end{equation}
For example, the construction below permits
\[
C_{R,s}=32\binom{s}{2}^2(\lfloor4R\rfloor+2)^2.
\]
No incidence with an unselected vertex is added.
\end{corollary}
\begin{proof}
Within a row, label an edge by the ordered pair $(i,j)$ of the
positions of its endpoints.  For each circle there are at most
$L=\binom s2$ such labels.  For a fixed label, the supporting arcs
in different row spans have disjoint interiors.

Use $L$ arbitrarily close concentric circle tracks for each original
circle, one per label.  The tracks can be chosen to preserve all
disk visits, as in the proof of
Lemma~\ref{lem:eighth-robust-circle-drawing}. Support each selected edge
on its track between its two endpoint disks.
It may pass other disks
without being incident to their vertices.  Distinct tracks are
circles, hence have at most two exterior intersections, and at most
one selected edge occupies any given track arc.  Their total
exterior crossing count is therefore $O_s(C^2)$.

Fix one disk visit of one track. Its exterior support is a union
of arcs with disjoint interiors: the edge label is fixed and the
row spans on this circle have disjoint interiors. At each of
the two boundary points of the visit, at most one supported edge
arc can enter the disk. Thus the visit carries at most two
incident ends, or one passage, or a combination bounded by these
two boundary ports. Route each incident end to $p_v$ by a spoke,
and each passage by a chord or a generic two-segment path avoiding
$p_v$. Figure~\ref{fig:local-routing} illustrates these rules.
With $L$ tracks for each of the $k_v$ original circle visits,
there are at most $2Lk_v$ local paths. The more generous bound
$4Lk_v$ permits endpoint splitting and generic perturbations. The same two-segment construction produces
at most
\[
4\binom{4Lk_v}{2}=2(4Lk_v)(4Lk_v-1)\le32L^2k_v^2
\]
local crossings.
The exterior crossing count is at most $2\binom{LC}{2}\le L^2C^2$.
Since $k_v\le(\lfloor4R\rfloor+2)M$, summation proves
\eqref{eq:eighth-row-clique-crossing} with the displayed choice
of $C_{R,s}$.
All local replacements stay inside their respective disks.
If a replacement creates a self-intersection of an edge,
erase the loop.
This removes intersections and adds no vertex.
A final arbitrarily small perturbation produces simple edge arcs
in general position without increasing the displayed bound.
All graph endpoints were prescribed row vertices throughout.
\end{proof}

\begin{figure}[htbp]
\centering
\begin{tikzpicture}[x=1.1cm,y=1.1cm]
\begin{scope}
\fill[gray!5] (0,0) circle (1.25);
\draw[gray] (0,0) circle (1.25);
\draw[blue!65!black,thick] (-1.083,-.625) .. controls (-.55,-.15) and (.45,.18) .. (1.083,.625);
\draw[blue!65!black,thick] (-1.225,-.249) .. controls (-.55,.08) and (.52,.47) .. (.875,.893);
\draw[gray!80,dashed,thick] (-.75,-1) .. controls (-.18,-.4) and (.55,-.12) .. (1.225,.249);
\fill (0,0) circle (.04) node[below right] {$p_v$};
\node[above] at (0,1.5) {Before local routing};
\node[below,font=\small,align=center] at (0,-1.45)
{Prescribed disk visits};
\end{scope}
\draw[->,thick] (1.65,0)--(2.5,0);
\begin{scope}[xshift=4.7cm]
\fill[gray!5] (0,0) circle (1.25);
\draw[gray] (0,0) circle (1.25);
\draw[blue!65!black,thick] (-1.083,-.625)--(0,0)--(1.083,.625);
\draw[blue!65!black,thick] (-1.225,-.249)--(0,0)--(.875,.893);
\draw[gray!80,dashed,thick] (-.75,-1)--(1.225,.249);
\fill (0,0) circle (.045) node[below right] {$p_v$};
\node[above] at (0,1.5) {After local routing};
\node[below,font=\small,align=center] at (0,-1.45)
{Spokes and passages};
\end{scope}
\end{tikzpicture}
\caption{Routing inside a source disk. Each solid spoke ends at the
graph vertex.
A dashed path passes through the disk while avoiding
the vertex. The boundary ports are unchanged. Two assigned visits
are routed through the vertex, while the unassigned visit is
routed past it. In the
clique construction there are $\binom{s}{2}$ tracks per circle,
one for each pair of row positions. Each local path uses at most
two straight segments, so two paths have at most four proper
intersections. A shared graph endpoint is not a crossing.}
\label{fig:local-routing}
\end{figure}
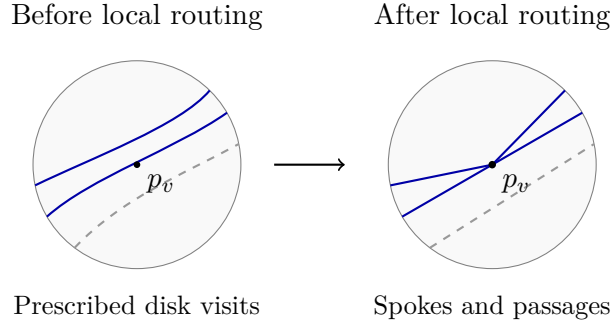

\begin{lemma}[Parallel edges give approximate bisector richness]
\label{lem:eighth-parallel-bisector-richness}
In the first lemma, if vertices $p\ne q$ have graph multiplicity
$m_G(p,q)$, the number of edges joining them, then at least $c_Rm_G(p,q)$ distinct pins lie within
distance
\[
 C_{R,D_0}\frac{\delta}{|p-q|}
\]
of the perpendicular bisector of $p,q$.
For the row-clique graph the same statement holds with constants
also depending on $s$.
\end{lemma}
\begin{proof}
A circle contributes at most two edges to a fixed pair in the
cyclic graph.  Each pin has at most $C_R$ circle radii meeting the
disk of $p$, so at least $c_Rm_G(p,q)$ distinct pins support the
parallel edges.  The row-clique graph has a similarly bounded
number of contributions per circle and pair, with a constant
depending only on $s$.

If a circle centered at $x$ meets both disks, then
\[
 \big||x-p|-|x-q|\big|\le r_p+r_q\le4R\delta.
\]
Both distances are at most $D_0+2R\delta$.  Consequently
$\big||x-p|^2-|x-q|^2\big|\le C_{R,D_0}\delta$.
The expression on the left equals twice $|p-q|$ times the distance
of $x$ from the bisector.  This proves the claim.  The labels in the
drawing refer to the original circles, so auxiliary track
perturbations introduce no change in this distance calculation.
\end{proof}

\begin{lemma}[Localized drawing]\label{lem:localized-drawing}
In the setting of Corollary~\ref{cor:eighth-circle-row-cliques},
suppose in addition that every row has angular diameter at most
$Cw$ about its pin, where $\delta\ll w\ll1$, and that each source
vertex belongs to at most $C$ rows at any one pin. For a square $Q$
of a grid of mesh $w$, let $V_Q$ be the number of source vertices in
a fixed enlargement of $Q$. Then the drawing can be chosen with
\begin{equation}\label{eq:localized-drawing}
 \operatorname{cr}(G_s)\le C_{R,s}M^2\Bigl(\sum_QV_Q^2+V\Bigr)
 \le C_{R,s}M^2V\Bigl(1+\max_QV_Q\Bigr).
\end{equation}
\end{lemma}
\begin{proof}
Use the drawing of Corollary~\ref{cor:eighth-circle-row-cliques}.
Every exterior edge is the short track arc between its endpoint
disks inside its row span.
Since the radii are bounded and the span
is $O(w)$, the arc lies within $O(w)$ of either endpoint, and the
local routings change this by $O(\delta)$. Charge each exterior
crossing to the mesh square containing it.
Both crossing edges then
have an endpoint in a fixed enlargement of that square. At most
$CMV_Q$ edge occurrences have an endpoint there, because a vertex
belongs to $O(1)$ rows at each of the $M$ pins and a row has
$\binom{s}{2}$ edges. Distinct tracks meet at most twice and arcs
of different rows on one track have disjoint interiors, so the
crossings charged to $Q$ number at most $C(MV_Q)^2$. Inside a
source disk the local routing costs $O_{R,s}(M^2)$ crossings, as in
the corollary. Summing proves the first bound in
\eqref{eq:localized-drawing}, and $\sum_QV_Q\lesssim V$ implies the
second.
\end{proof}

For source representatives in distinct terminal $\delta$-squares,
the required disk separation can be arranged by a fixed finite
coloring of the grid.  A color retaining a fixed fraction of the
prescribed incidences is sufficient.  In the configuration of
Section~\ref{sec:circle-rows}, $V\lesssim\frac{p}{b}$,
$M\lesssim\upsilon_f^{-1}$, and $C\lesssim\upsilon_f^{-1}\delta^{-q}$,
so the two terms in \eqref{eq:eighth-row-clique-crossing} are
$O(\upsilon_f^{-2}\delta^{-2q})$ and $O(p\upsilon_f^{-2}b^{-1})$.
The
second is the smaller one because $p\delta^{2q}\le b$.
Thus local detours do not introduce a new loss.

\section{Circle graphs produce rows with rich bisectors}\label{sec:circle-rows}

The reason for drawing the graph can be seen from the counts. After
a thinning of the source described below, the source vertices number
about $\frac{p}{b}$, the circles about $\upsilon_f^{-1}\delta^{-q}$,
and the rows about $\frac{p^2}{nb\upsilon_f}$, where the symbols are
introduced next. A simple graph with these many vertices and edges
would have more crossings than the drawing of
Section~\ref{sec:circle-drawing} allows. Hence many edges must be
repeated.

For a source pair, repetition means that many circles from different
pins visit both endpoints. Those pins lie near the perpendicular
bisector. We need this conclusion for every pair in a row of fixed
length, so we draw all pairs within each row and discard rows
containing a pair of low multiplicity. Two further steps sharpen the
conclusion. An ordered selection inside each row extracts shorter
rows whose pair separations lie between two consecutive scales.
The
drawing of these rows can then be localized, which increases the
multiplicity of their pairs. Finally the pins are thinned to the
resolution of the resulting bisector tubes.

\subsection{The finite inputs}

The finite configuration is produced from the measures in
Section~\ref{sec:discretization}. We list here the properties that
the circle argument uses, so that this section and the next two are
finite statements. The hypotheses below govern the pinned applications
through Section~\ref{sec:exceptional}.
Proposition~\ref{prop:endpoint-finite-geometry} states a separate
set of finite hypotheses and verifies the relevant counting steps
directly for the rigid endpoint configuration. In the present application,
$\Delta,\omega,\lambda,F,a_c,W_0,q$ are as in
\eqref{eq:level-notation}, and $\theta$ is a fixed template in the
sense of Section~\ref{sec:discretization}: a point
$(u_-,W,\varsigma,\chi,\mathfrak e)$ of the polytope
\eqref{eq:plateau-parameters}, with $u_+=u_-+a_c+\chi\omega$ and with
$u_0,\varpi,e_0$ given by \eqref{eq:u0-varpi}, together with a branch
index, where branch~1 requires $\mathfrak e\le\frac{1}{8}$ and
branch~2 requires $\mathfrak e\ge\frac{1}{8}$. Fixed
tolerances $\beta,\xi,\xi_{\mathrm{par}},\tau_0>0$ are chosen in
Section~\ref{sec:universal-pinned}. Put
\begin{equation}\label{eq:finite-scales}
 a_{\mathrm{par}}=a_c+\varpi\omega+\xi_{\mathrm{par}},\qquad 
 \delta=2^{-n},\qquad 
 \eta=2^{-\lfloor a_{\mathrm{par}}n\rfloor},\qquad 
 b=\delta^{\lambda-\beta},
\end{equation}
where $\xi_{\mathrm{par}}$ is a selectable positive tolerance in the
parent exponent. The integer rounding in $\eta$ changes its side
length by at most a factor two. Fix candidate exponents
\begin{equation}\label{eq:candidate-exponents}
\begin{gathered}
 a_0=a_c-\omega-\beta-\xi<a_1<\cdots<a_{\mathfrak n}=a_{\mathrm{par}},\\
 0<a_j-a_{j-1}\le\xi \quad \text{for } 1\le j\le\mathfrak n,
\end{gathered}
\end{equation}
with $\mathfrak n$ fixed, together with dyadic lengths
$\ell_j\asymp\delta^{a_j}$ such that $\ell_0\ge4\delta^{a_0}$ and
$\ell_{\mathfrak n}<c_0\eta$, where $c_0$ is the source-parent
separation constant of Proposition~\ref{prop:rows-common-richness}.
The candidate tube widths are $\rho_j=\frac{\delta}{\ell_j}\asymp\delta^{\zeta_j}$,
$\zeta_j=1-a_j$. All source and pin sets lie in fixed bounded regions
separated by a fixed positive distance, and all circle radii lie in a
fixed interval $[d_0,D_0]\subset(0,\infty)$. The numbers
$m,\upsilon_f,\upsilon_1,\ldots,\upsilon_{\mathfrak n}\in(0,1]$ are
cell masses of the source and pin measures $\mu,\nu$, and
$p\in(0,1]$ is a relation mass with $p\ge n^{-C}$ for a fixed $C$.
The notation $X\lesssim_{\mathrm{pol}}Y$ means
$X\le Cn^{C}p^{-C}Y$ with $C$ fixed. We use
$X\gtrsim_{\mathrm{pol}}Y$ for $Y\lesssim_{\mathrm{pol}}X$, and
$X\asymp_{\mathrm{pol}}Y$ when both inequalities hold.
A quantity denoted by $O_{\mathrm{pol}}(Y)$ has absolute value
$\lesssim_{\mathrm{pol}}Y$.
Such factors are
$\delta^{-o(1)}$.

Lemma~\ref{lem:seventh-measure-finite-selection}, applied in
Corollary~\ref{cor:finite-inputs}, supplies a common
$\delta$-separated source set $S$ of cell centers, a
$\delta$-separated set $X$ of actual pins, with at most one pin in each
dyadic $\delta$-cell and $N=\#X$, and marked
incidences $(x,v,J)$ between pins, source vertices, and radial bins
of length $\delta$, called assigned incidences in that lemma, with
the following properties.
\begin{enumerate}[label=(F\arabic*)]
\item\label{F:counts} $\#S\le m^{-1}$,
$\frac{cp}{\upsilon_f}\le N\le\upsilon_f^{-1}$, and the number $I$
of marked incidences is at least $\frac{cp}{m\upsilon_f}$. At each
pin at most $C\delta^{-q}$ radial bins are marked and their radii
are $\delta$-separated.
Every marked vertex is within $C\delta$ of
its circle, and a vertex has $O(1)$ marked radial labels at any one
pin.
\item\label{F:angular} The number $A_x(J)$ of source vertices marked
at $x$ with direction in an angular interval $J$, and the number
$A_x(C_\eta,J_\delta)$ marked at $x$ in a source square $C_\eta$ of
side $\eta$ with radial bin $J_\delta$, satisfy
\[
 A_x(J)\le\frac{C(|J|+\delta)}{pm},\qquad 
 A_x(C_\eta,J_\delta)\le\frac{Cb}{pm};
\]
the second bound holds also for annuli of width $O(\delta)$ about
$x$.
\item\label{F:source-mass} Every source $\delta$-cell containing a
vertex of $S$ has mass in $[m,2m)$, with $m\le b$. Every source
square of side $\eta$ containing a vertex of $S$ has mass at most
$\delta^{F-\tau_0}$. If $\theta$ has branch~1, then for every
$a\in[a_c-\omega,a_{\mathrm{par}}]$ every dyadic square of side
$\asymp\delta^{a}$ containing a vertex of $S$ has mass at most
$\delta^{a+a_c-(1+u_0+\mathfrak e)\omega-\tau_0}$.
\item\label{F:pin-mass} Every $\delta$-cell containing a pin of $X$
has mass in $[\upsilon_f,2\upsilon_f)$, and for each $j$ its
ancestor of side $\rho_j$ has mass in $[\upsilon_j,2\upsilon_j)$.
Every ancestor of side $2^{-k}$, $0\le k\le n$, of a pin of $X$ has
mass at most $\delta^{\frac{k}{n}+T_\theta(\frac{k}{n})-\tau_0}$, where
$T_\theta(u)=\max\{0,W-\operatorname{dist}(u,[u_-,u_+])\}$. Moreover
$\nu(B(z,r))\le Cr^{1+\varkappa}$ for all balls, with a fixed
$\varkappa>0$.
\end{enumerate}
The angular density bound $L$ of
Lemma~\ref{lem:seventh-measure-finite-selection} is a fixed constant
and has been absorbed into $C$. All cardinalities are polynomial in
$\delta^{-1}$. These bounds refer to the original measures.
No source
law is renormalized at different pins.

Fix also an angular cutoff
\begin{equation}\label{eq:angular-cutoff}
 \rho_{\min}=\delta^{1-a_c+\omega+\beta+2\xi}.
\end{equation}
We assume
\begin{equation}\label{eq:tolerance-conditions}
 2\beta+2\xi<a_c-\omega,\qquad 
 \xi_{\mathrm{par}}+3\tau_0\le\beta,\qquad 
 8\xi_{\mathrm{par}}\le\tau_0,\qquad
 a_{\mathrm{par}}<\frac{1}{2},
\end{equation}
which hold once the tolerances are small, since
$a_{\mathrm{par}}\le\frac{1}{2}-W_0+2\omega+\xi_{\mathrm{par}}$. Then
$\rho_{\min}\le\rho_j\le\delta^{\frac{1}{2}}$ for every $j$, $\frac{\delta}{\rho_{\min}}\to0$,
and, with
\[
\varepsilon_{\mathrm{ang}}=a_c+\Delta-\omega-2\beta-2\xi>0,
\]
we have $\frac{\rho_{\min}}{b}=2^{\varepsilon_{\mathrm{ang}}n}$.
The positivity follows from $2\beta+2\xi<a_c-\omega$ and
$\Delta\ge0$. Every tolerance is fixed before $n$ tends to infinity.

\subsection{One thinning of the source for all pins}

A thinning retains each point independently with a specified
probability. We write $\mathbb P$ for probability, $\mathbb E Z$
for expectation, and $\operatorname{Var}(Z)=\mathbb E(Z-\mathbb E Z)^2$
for variance. We use Markov's inequality
$\mathbb P\{Z\ge t\}\le\frac{\mathbb E Z}{t}$ for $Z\ge0$, and
Chebyshev's inequality
$\mathbb P\{|Z-\mathbb E Z|\ge t\}\le\frac{\operatorname{Var}(Z)}{t^2}$.
For a sum $Z$ of independent zero-one variables with mean $m$, the
Chernoff bounds used below are
\[
\mathbb P\{Z<\frac{1}{2}m\}\le\mathrm e^{-\frac{m}{8}},\qquad
\mathbb P\{Z\ge m+t\}\le
\exp\left(-\frac{t^2}{2(m+\frac{t}{3})}\right) \quad \text{for } t>0.
\]
See~\cite[Chapters~2, 4 and Appendix~A]{AlonSpencer}.
The union bound adds the failure probabilities over all events.
With at most $2^{Cn}$ events, the upper-tail estimate ensures
$Z\le C'(m+n)$ for all of them with probability tending to one,
by taking $C'$ sufficiently large. These are the simultaneous
sampling bounds used in this and the subsequent sections.

The simultaneous requirement is essential. A related probabilistic
selection appears in Shmerkin--Yavicoli~\cite[Lemmas~3.4 and~3.6]{SYUniversality},
where one small set must work for a whole family of sumsets.
Their discrete argument is adapted from Ruzsa, as explained there.
Here the selected vertices must preserve both the total incidence
count and the upper occupancy bounds at every retained pin. The
following proof establishes these properties for one source sample.

\begin{lemma}[One simultaneous source restriction]
\label{lem:source-thinning}
For every sufficiently large $n$, there is a subset $S'\subset S$
such that, retaining the marked incidences with vertices in $S'$,
the following hold.
\begin{align}
 \#S'&\le\frac{Cp}{b},&
 I'&\ge\frac{cp^{2}}{b\upsilon_f},
 \label{eq:thinned-counts}\\
 A'_x(C_\eta,J_\delta)&\le Cn,&
 A'_x(J)&\le\frac{C|J|}{b} \quad \text{for } |J|\ge\rho_{\min}.
 \label{eq:thinned-angular}
\end{align}
Every square $Q$ of side $\asymp\delta^{a}$ with
$a\in[a_0,a_{\mathrm{par}}]$, and every fixed enlargement of it,
satisfies
\begin{equation}\label{eq:thinned-occupancy}
 \#(S'\cap Q)\lesssim_{\mathrm{pol}}\frac{\mu_a}{b},\qquad 
 \mu_a=\delta^{F-\tau_0}\Bigl(\frac{\delta^{a}}{\eta}\Bigr)^{2},
\end{equation}
and, if $\theta$ has branch~1 and $a\ge a_c-\omega$, also with
$\mu_a=\delta^{a+a_c-(1+u_0+\mathfrak e)\omega-\tau_0}$.
The parent bound in \eqref{eq:thinned-angular} holds for annuli of
width $O(\delta)$ as well.
\end{lemma}
\begin{proof}
Keep each source vertex independently with probability
$t=\frac{c_0pm}{b}\le1$, with one choice serving every pin. Each
parent and radial-bin occupancy has expectation at most $Cc_0$. The binomial
tail bound and a union bound over the polynomially many parent and radial-bin
pairs and pins show that all these occupancies are at most $Cn$
with probability tending to one. For a dyadic angular interval of
length at least a fixed multiple of $\rho_{\min}$, the expected
count is at most $\frac{C|J|}{b}$, since
$\frac{\delta}{\rho_{\min}}\to0$. Chernoff's inequality and a
second union bound give $A'_x(J)\le C(\frac{|J|}{b}+n)$
simultaneously. A cover by a bounded number of dyadic intervals
of comparable length, in the fixed angular charts, extends this
bound to every $|J|\ge\rho_{\min}$. The additive term is absorbed
uniformly over all such intervals, because
\[
\sup_{|J|\ge\rho_{\min}}\frac{bn}{|J|}
\le\frac{bn}{\rho_{\min}}
=n2^{-\varepsilon_{\mathrm{ang}}n}\longrightarrow0.
\]
This proves the angular bound for every fixed choice of tolerances
satisfying \eqref{eq:tolerance-conditions}. The threshold for $n$
may depend on these tolerances, but is uniform in the pins and intervals.

A square $Q$ of side $\asymp\delta^{a}$ meets $O((\frac{\delta^{a}}{\eta})^2)$
squares of side $\eta$, each of mass at most $\delta^{F-\tau_0}$ if
it contains a vertex of $S$, and in branch~1 with $a\ge a_c-\omega$
it is covered by $O(1)$ dyadic squares of side $\asymp\delta^{a}$, each of
mass at most $\delta^{a+a_c-(1+u_0+\mathfrak e)\omega-\tau_0}$ if it
contains a vertex of $S$, both by \ref{F:source-mass}. Since the
$\delta$-cells of $S$ have mass at least $m$, $Q$ contains at most
$\frac{C\mu_a}{m}$ vertices of $S$, and the expected number retained
is at most $\frac{C\mu_a}{b}$. Chernoff's inequality and a union bound
over the dyadic scales in the stated range and their polynomially many
squares, followed by a bounded dyadic cover of each $Q$, give
\eqref{eq:thinned-occupancy} with an additive
$O(n)$, which is absorbed because $\frac{\mu_a}{b}$ grows faster than
every power of $n$: the exponent of $\delta$ in $\frac{\mu_a}{b}$ is
at most
\[
F+\xi_{\mathrm{par}}-\lambda+\beta
=-\frac{1}{2}+\xi_{\mathrm{par}}+\beta
<-\frac{a_c}{2}\le-\frac{1}{8}.
\]
Indeed $\xi_{\mathrm{par}}<\frac{1}{2}-a_c$ follows from
$a_{\mathrm{par}}<\frac{1}{2}$ and $\varpi\ge0$, while
$\beta<\frac{a_c}{2}$ follows from
\eqref{eq:tolerance-conditions}. For the covering bound,
$a\le a_{\mathrm{par}}$ implies $\delta^a\ge\frac{\eta}{2}$.
Dyadic rounding therefore changes only a fixed factor. For the
branch~1 bound, use
$a+a_c-(1+u_0+\mathfrak e)\omega\le F+\xi_{\mathrm{par}}$,
which follows from $a\le a_{\mathrm{par}}$ and $\varpi\le1+u_0$.

Let $d_v$ be the number of marked incidences at the vertex $v$.
Then $d_v\le CN$. Hence
\[
 \mathbb EI'=tI\ge\frac{cp^{2}}{b\upsilon_f},\qquad 
 \operatorname{Var}(I')\le CN\mathbb EI',\qquad 
 \frac{\operatorname{Var}(I')}{(\mathbb EI')^{2}}
 \le\frac{Cb}{p^{2}}\longrightarrow0,
\]
so the incidence bound holds with probability tending to one.
Finally $\mathbb E\#S'\le\frac{c_0p}{b}$, and Markov's inequality
bounds the probability of violating the vertex bound by
$\frac{1}{10}$ after choosing its constant. All requirements hold
simultaneously with positive probability. An annulus of width
$O(\delta)$ meets $O(1)$ radial bins, which proves the last
assertion.
\end{proof}

\subsection{Rows with rich bisectors}

The degree of a pin in a collection of rows means the number of
rows assigned to it. To regularize degrees dyadically, divide
positive degrees into intervals $[2^j,2^{j+1})$ and retain one
class. When there are $O(n)$ classes, one carries at least an
$O(n)^{-1}$ fraction of the total count being divided.

\begin{lemma}[Multigraph edge bound]\label{lem:multigraph-edge-bound}
A loop-free multigraph with $V$ vertices, $E$ edges, maximum edge
multiplicity at most $k_1\ge1$, and a drawing with $X$
crossings satisfies
\begin{equation}\label{eq:eighth-multigraph-edge-bound}
 E\le C\bigl(k_1V+
                  k_1^{\frac{1}{3}}V^{\frac{2}{3}}X^{\frac{1}{3}}\bigr).
\end{equation}
\end{lemma}
\begin{proof}
Independently for each bundle of $k\le k_1$ parallel edges,
retain the bundle with probability $\frac{k}{k_1}$ and then
choose one of its edges uniformly. The resulting random graph is
simple, with expected edge count $\frac{E}{k_1}$ and expected
crossing count at most $\frac{X}{k_1^2}$: each edge survives
with probability $\frac{1}{k_1}$, edges from distinct bundles
are independent, and two edges of one bundle cannot both survive.
Apply the crossing lemma~\cite[Theorem~1]{Szekely} in its global form
$\operatorname{cr}(G_{\mathrm{samp}})\ge c \frac{E(G_{\mathrm{samp}})^3}{V^2}-CV$,
then Jensen's inequality. Multiplication by $k_1^2$ and
rearrangement give \eqref{eq:eighth-multigraph-edge-bound}.
\end{proof}

\begin{proposition}[Circle graphs force rows with rich bisectors]
\label{prop:rows-common-richness}
Fix an integer $s_0\ge3$. Under the hypotheses of this section there
are rows of $s_0$ marked source vertices of $S'$, each row lying on
one original circle, with the following properties.
\begin{enumerate}
\item At least $\frac{cp}{n^{2}\upsilon_f}$ pins carry rows, and each
of them carries between $\mathfrak d$ and $2\mathfrak d$ rows, where
$\frac{cp^{2}}{nb}\le\mathfrak d\le\frac{Cp}{b}$.
\item On each circle the angular spans of its rows have disjoint
interiors, and every row lies in one member of a fixed angular
partition into windows of width
$w_0=\frac{C_{s_0}nb}{p^{2}\delta^{q}}$. The source angles in a
row are pairwise separated by at least $c_0\eta$, with $c_0>0$
fixed. Each source vertex belongs to at most $C_{s_0}$ rows at any
one pin.
At a retained pin, any angular interval of length
$r\ge\rho_{\min}$ contains at most $\frac{Cr}{b}$ row entries, and
each window contains at most $\frac{Cw_0}{b}$ rows.
\item Every pair of vertices in every row has at least
\begin{equation}\label{eq:row-richness}
 k_1=\frac{cp^{4}\delta^{2q}}{n^{3}b\upsilon_f}
\end{equation}
distinct pins $x\in X$ with $\bigl||x-v|-|x-z|\bigr|\le C\delta$.
\end{enumerate}
The constants depend only on $s_0$ and the fixed geometric bounds.
\end{proposition}
\begin{proof}
Apply Lemma~\ref{lem:source-thinning} and retain its marked
incidences. First color the source points in boundedly many colors
so that points of one color have pairwise distance greater than
$100(C_0+1)\delta$, where $C_0\delta$ bounds the incidence error.
This is possible by the bounded degree of the corresponding
proximity graph. Independently color the dyadic side-$\eta$
parents according to their two indices modulo $4$. Distinct parents
of one color have distance at least $\eta$. Keep a combined color
carrying a fixed fraction of the marked incidences. Thus all
retained source disks of radius $O(\delta)$ can be made disjoint,
and the bounds of Lemma~\ref{lem:source-thinning} persist.

On each circle keep one marked vertex from each occupied parent.
Since the parent occupancy is at most $Cn$, the number of
representatives is at least
\[
 T_0=\frac{cp^{2}}{nb\upsilon_f}.
\]
Representatives on a common circle are at distance at least
$\eta$, their radii about the pin differ by $O(\delta)$, and the
radii are bounded above and below.
So their angles are separated
by at least $c_0\eta$. A vertex has $O(1)$ marked radial labels at
a pin, hence occurs in $O_{s_0}(1)$ rows there.

The number of circles is at most
$\mathcal C_0=C\upsilon_f^{-1}\delta^{-q}$. Partition each circle into
angular windows of width $w_0$, order the representatives in each
window, and group consecutive representatives into disjoint rows
of $s_0$, discarding fewer than $s_0$ at the end of each window. The
discarded count is at most $C_{s_0}\mathcal C_0(1+w_0^{-1})$. Since
$\frac{nb}{p^{2}\delta^{q}}=\frac{n}{p^{2}}\delta^{a_c-\omega-\beta}\to0$,
the term $C_{s_0}\mathcal C_0$ is negligible against $T_0$ and
$w_0\to0$.
The term $C_{s_0}\mathcal C_0w_0^{-1}$ is a small fraction of
$T_0$ once the constant in $w_0$ is large. We obtain at least $cT_0$
rows.

Form the multigraph on the retained vertices with all
$\binom{s_0}{2}$ pairs of every row, one edge per row occurrence.
Within a circle, the angular spans of its rows have disjoint
interiors. Corollary~\ref{cor:eighth-circle-row-cliques} provides a
drawing with
\[
 \operatorname{cr}(G)\le C_{s_0}\bigl(\mathcal C_0^{2}+\#S'N^{2}\bigr)
 \le C_{s_0}X_0,\qquad  X_0=\upsilon_f^{-2}\delta^{-2q},
\]
where the last inequality uses $\#S'\le\frac{Cp}{b}$ and
$\frac{\delta^{2q}}{b}=\delta^{\frac{1}{2}+2\omega+\beta}\to0$. Let
$G_{\mathrm{lo}}$ consist of the edge occurrences of pairs whose
total multiplicity is less than $k_1$. Its crossing number is at
most that of $G$, so Lemma~\ref{lem:multigraph-edge-bound} implies
\[
 E(G_{\mathrm{lo}})\le C_{s_0}\bigl(k_1\#S'
 +k_1^{\frac{1}{3}}(\#S')^{\frac{2}{3}}X_0^{\frac{1}{3}}\bigr).
\]
Since $(\#S')^{2}X_0\le\frac{Cp^{2}}{b^{2}\upsilon_f^{2}\delta^{2q}}$,
the second term is at most $(Cc)^{\frac{1}{3}}T_0$, where $c$ is the
constant in \eqref{eq:row-richness}, and $k_1\#S'\le CT_0\frac{p^{3}\delta^{2q}}{n^{2}b}=o(T_0)$.
Choosing $c$ small makes $E(G_{\mathrm{lo}})$ a small fraction of
the row count. Every row containing a low-multiplicity pair
contributes a distinct edge occurrence to $G_{\mathrm{lo}}$: the
rows are disjoint on each circle, and occurrences from different
circles retain their separate circle labels. Consequently at least
$cT_0$ rows have every pair of multiplicity at least $k_1$. A fixed
unordered source pair contributes $O_{s_0}(1)$ edge copies at any one
pin, since only $O(1)$ of its $\delta$-separated radius bins can
meet the two source disks and each vertex is used at most once in a
row per circle. Thus every pair in a surviving row is nearly
equidistant, to error $C\delta$, from at least $c_{s_0}k_1$ distinct
pins.
Renaming $c$ proves item~3.

Keep exactly a fixed multiple of $T_0$ surviving rows. A pin has
at most $C_{s_0}\#S'\le\frac{C_{s_0}p}{b}$ rows. Delete the pins with
fewer than $\frac{cp^{2}}{nb}$ rows.
For a small constant they
account for fewer than half the rows, since $N\le\upsilon_f^{-1}$.
Among the $O(n)$ remaining dyadic degree classes, one carries at
least $\frac{cT_0}{n}$ rows.
Its pins have degree between
$\mathfrak d$ and $2\mathfrak d$ and number at least
$\frac{cT_0}{n\mathfrak d}\ge\frac{cp}{n^{2}\upsilon_f}$. The sampled
angular bound \eqref{eq:thinned-angular}, with the bounded row
participation of each vertex, bounds the number of row entries in
an interval of length $r\ge\rho_{\min}$ by $\frac{Cr}{b}$ and the
window population by $\frac{Cw_0}{b}$.
\end{proof}

A witness of a source pair is a pin supporting an assigned
occurrence of both vertices on one circle. The witnesses in
item~3 can be chosen among the assigned occurrences
of the pair on the original circles: this is how their multiplicity
was counted. We make that choice throughout, and retain the same
assignment when rows are restricted or pins are coarsened.
The witnesses of item~3 lie within $O(\frac{\delta}{|v-z|})$ of the
perpendicular bisector of the pair, by the squared-distance identity
of Section~\ref{sec:bisector-geometry}.
The width of this tube
depends on the separation of the pair. The next step selects rows
whose pair separations all lie between two consecutive candidate
scales.

\subsection{Rows at one pair-separation scale}

\begin{lemma}[Ordered selection between consecutive scales]
\label{lem:ordered-row-scales}
Let $s_1\ge2$ and let $\ell_0>\ell_1>\cdots>\ell_{\mathfrak n}>0$ be
lengths whose consecutive ratios are powers of two. Consider the
nested partitions of one interval of length $\ell_0$ into cells of
length $\ell_j$ at depth $j$. Let $F_0$ be a finite subset of that
interval with at most one point in each cell of depth $\mathfrak n$.
If $\#F_0>(2s_1-2)^{\mathfrak n}$, then for some
$j\in\{1,\ldots,\mathfrak n\}$ there are $s_1$ points of $F_0$ in one
cell of depth $j-1$ whose pairwise distances all lie between
$\ell_j$ and $\ell_{j-1}$.
\end{lemma}
\begin{proof}
If every occupied cell had at most $2s_1-2$ occupied children, there
would be at most $(2s_1-2)^{\mathfrak n}$ occupied cells of depth
$\mathfrak n$. Thus some cell of depth $j-1$ has at least $2s_1-1$
occupied children. Order these children from left to right and
choose one point from the first, third, fifth, and so on, through
the $(2s_1-1)$st. Between two chosen children lies a whole child
cell, so the chosen points are at least $\ell_j$ apart, and their
common parent bounds their separation by $\ell_{j-1}$.
\end{proof}

\begin{proposition}[Rows with a common pair-separation scale]
\label{prop:homogeneous-rows}
Fix $s_1\ge3$ and apply Proposition~\ref{prop:rows-common-richness}
with $s_0=(2s_1-2)^{\mathfrak n}+1$. Then there is an index
$1\le j\le\mathfrak n$, and there are rows of $s_1$ vertices of $S'$
on one original circle each, with the following properties.
Write
$a=a_{j-1}$, $c=a_j$, $\zeta=\zeta_j=1-c$, $\rho=\rho_j$, and
$w=\ell_{j-1}\asymp\delta^{a}$.
\begin{enumerate}
\item At least $\gtrsim_{\mathrm{pol}}\upsilon_f^{-1}$ pins carry
rows, each with degree between $\mathfrak d'$ and $2\mathfrak d'$,
where $\mathfrak d'\gtrsim_{\mathrm{pol}}b^{-1}$.
\item Every row lies in one cell of a fixed angular partition into
windows of width $w$.
Every pair in a row has angular separation
between $\ell_j$ and $\ell_{j-1}$.
On each circle the spans of the
rows have disjoint interiors.
A vertex belongs to $O_{s_1}(1)$ rows at
any one pin.
At a retained pin an angular interval of length
$r\ge\rho$ contains at most $\frac{Cr}{b}$ row entries and a window at
most $\frac{Cw}{b}$ rows.
\item Every pair $v,z$ in a row has the $k_1$ witnesses of
Proposition~\ref{prop:rows-common-richness}, and each witness lies
within $C\rho$ of the perpendicular bisector of $v,z$, whose direction
differs from the half-sum of the two source angles, modulo $\pi$ and
up to a fixed rotation, by $O(\rho)$.
\end{enumerate}
\end{proposition}
\begin{proof}
An $s_0$-row lies in a window of width
$w_0=\frac{C_{s_0}nb}{p^{2}\delta^{q}}=C_{s_0}\frac{n}{p^{2}}\delta^{a_c-\omega-\beta}
\le\frac{1}{4}\delta^{a_0}$ for small $\delta$, and its entries are
separated by at least $c_0\eta>\ell_{\mathfrak n}$. Use the three root
grids of mesh $\ell_0$ with origins $0,\frac{\ell_0}{3},\frac{2\ell_0}{3}$,
each with its nested dyadic descendants. An interval of length at
most $\frac{\ell_0}{4}$ meets the boundary of at most one of the
three grids, so every row lies in one root cell of some grid, and
each terminal cell contains at most one of its entries.
Lemma~\ref{lem:ordered-row-scales} provides an $s_1$-entry subrow, a
grid, and an index $j$ such that all pair separations lie in
$[\ell_j,\ell_{j-1}]$ and the subrow lies in one cell of length
$\ell_{j-1}$. Among the $3\mathfrak n$ choices of grid and index,
retain one carrying a fixed fraction of the rows. The depth-$(j-1)$
cells of the chosen grid are the asserted angular partition. Remove
the pins carrying fewer than a small fixed multiple of
$\frac{\mathfrak d}{n}$ subrows and regularize the remaining degrees
dyadically.
Item~1 follows from item~1 of
Proposition~\ref{prop:rows-common-richness}. Subrow spans lie inside
the original row spans, so their disjoint interiors, the bounded
participation of each vertex, and the angular counts persist.

For a pair $v,z$ in a subrow and one of its witnesses $x$,
$\bigl||x-v|^{2}-|x-z|^{2}\bigr|\le C\delta$ and $|v-z|\ge c\ell_j$,
so $x$ lies within $\frac{C\delta}{\ell_j}=C\rho$ of the bisector.
Replacing the two radii by their mean changes the chord by
$O(\delta)$ and its direction by $O(\rho)$, which proves the half-sum
formula.
\end{proof}

\subsection{Localized drawing and a second pruning}

The rows of Proposition~\ref{prop:homogeneous-rows} have angular
diameter at most $w\asymp\delta^{a}$, so their drawing can be
localized. This increases the multiplicity of their pairs when
$a$ is close to $a_{\mathrm{par}}$. Put
\begin{equation}\label{eq:gap-and-psi}
 \mathfrak h=a_{\mathrm{par}}-a\ge0,\qquad 
 \psi(\mathfrak h)=
 \begin{cases}
 \min\{2\mathfrak h, \mathfrak h+e_0\omega\},&\theta\text{ of branch 1 and }a\ge a_c-\omega,\\
 2\mathfrak h,&\text{otherwise}.
 \end{cases}
\end{equation}

\begin{proposition}[Richness from the shorter rows]
\label{prop:second-pruning}
A fraction $\gtrsim_{\mathrm{pol}}1$ of the rows of
Proposition~\ref{prop:homogeneous-rows} can be retained, and the
pin degrees regularized, so that items~1 and~2 there persist and
every pair $v,z$ in every retained row has at least
\begin{equation}\label{eq:fine-richness}
 k_f\gtrsim_{\mathrm{pol}}N\max\left\{\frac{\delta^{2q}}{b},
 b\delta^{-F+\tau_0+\psi(\mathfrak h)}\right\}
\end{equation}
distinct pins $x\in X$ with $\bigl||x-v|-|x-z|\bigr|\le C\delta$.
\end{proposition}
\begin{proof}
The retained rows number $T\gtrsim_{\mathrm{pol}}\frac{N}{b}$, by
item~1 of Proposition~\ref{prop:homogeneous-rows} and
$N\asymp_{\mathrm{pol}}\upsilon_f^{-1}$, and their vertices number
$V\le\#S'\lesssim_{\mathrm{pol}}b^{-1}$. Lemma~\ref{lem:localized-drawing}
with mesh $w$, together with \eqref{eq:thinned-occupancy} at the
exponent $a$, provides a drawing with
\[
 X\lesssim_{\mathrm{pol}}N^{2}V\max_QV_Q
 \lesssim_{\mathrm{pol}}\frac{N^{2}\mu_a}{b^{2}},
\]
and $\mu_a\le\delta^{F-\tau_0-\psi(\mathfrak h)}$:
indeed $\delta^{a}\le\eta\delta^{-\mathfrak h}$ implies
$\mu_a\le\delta^{F-\tau_0-2\mathfrak h}$, and in branch~1 with
$a\ge a_c-\omega$,
\[
 a+a_c-(1+u_0+\mathfrak e)\omega
 =F-\mathfrak h-(1+u_0+\mathfrak e-\varpi)\omega+\xi_{\mathrm{par}}
 \ge F-\mathfrak h-e_0\omega,
\]
by \eqref{eq:finite-scales}, $F=2a_c$, and \eqref{eq:u0-varpi}.
In the remaining cases $\psi(\mathfrak h)=2\mathfrak h$. The
same drawing also obeys the bound $X\lesssim N^{2}\delta^{-2q}$ of
Proposition~\ref{prop:rows-common-richness}. Let $k_f$ be the right
side of \eqref{eq:fine-richness} with a small constant. Then
\[
 \frac{T^{3}}{V^{2}X}\gtrsim_{\mathrm{pol}}
 N\max\left\{\frac{\delta^{2q}}{b}, b\delta^{-F+\tau_0+\psi(\mathfrak h)}\right\},\qquad
 k_fV\ll T,
\]
the last because $\frac{\delta^{2q}}{b}\to0$ and
$b\delta^{-F+\psi}=\delta^{\frac{1}{2}+\psi-\beta}\to0$. As in the
proof of Proposition~\ref{prop:rows-common-richness},
Lemma~\ref{lem:multigraph-edge-bound} bounds the edge occurrences of
pairs of multiplicity below $k_f$ by a small fraction of $T$.
Deleting every row containing such a pair, converting multiplicity
into distinct pins by the bounded participation at each pin, and
regularizing degrees proves the claim.
\end{proof}

At leading exponent order the two candidate thresholds in
\eqref{eq:fine-richness} cross at $\psi(\mathfrak h)=2\omega$.
With the tolerances retained, equality occurs at
$\psi(\mathfrak h)=2\omega+2\beta-\tau_0$. The first is the multiplicity available
at every scale.
The second uses the source occupancy of
\eqref{eq:thinned-occupancy} and is better when the row scale is
close to the parent scale. When $a<a_c-\omega$, the gap satisfies
$\mathfrak h>a_{\mathrm{par}}-a_c+\omega=(\varpi+1)\omega+\xi_{\mathrm{par}}\ge2\omega$,
so $\min\{2\omega,\psi(\mathfrak h)\}=2\omega$ whichever branch the
template has.
Only this minimum enters the exponents of
Section~\ref{sec:universal-pinned}.

\subsection{Pins at the resolution of the tubes}

The witnesses lie within $C\rho$ of the bisectors, but they are only
$\delta$-separated. We thin them to one pin per dyadic $\rho$-cell,
keeping the original pin measure, so that the resulting set is
spread like a set of dimension one at all scales between $\rho$ and
one.

\begin{lemma}[Pin selection at the tube resolution]
\label{lem:pin-coarsening}
Let $X'\subset X$ be the pins carrying rows in
Proposition~\ref{prop:second-pruning}, so that
$\#X'\gtrsim_{\mathrm{pol}}\upsilon_f^{-1}$, and let $\mathcal E$ be
the set of source pairs in the retained rows, each $e\in\mathcal E$
with its witness set $\mathcal W(e)\subset X$, $\#\mathcal W(e)\ge k_f$.
Write $\upsilon_\zeta=\upsilon_j$ for the pin ancestor mass at the
scale $\rho=\rho_j$. If
\begin{equation}\label{eq:coarsening-margin}
 W_0-4\omega-\beta-\xi_{\mathrm{par}}>0,
\end{equation}
then for all sufficiently small $\delta$ there is a set $P\subset X$
with at most one point in each dyadic $\rho$-cell such that
\begin{align}
 \#P&\le\upsilon_\zeta^{-1},\qquad 
 \#(P\cap X')\gtrsim_{\mathrm{pol}}\upsilon_\zeta^{-1},
 \label{eq:coarse-counts}\\
 \#(P\cap\mathcal W(e))&\gtrsim_{\mathrm{pol}}
 \frac{k_f\upsilon_f}{\upsilon_\zeta} \quad \text{for } e\in\mathcal E,
 \label{eq:coarse-richness}
\end{align}
and, for every $z$ and $r\ge\rho$,
\begin{equation}\label{eq:coarse-ball-count}
 \#(P\cap B(z,r))\le\frac{\nu(B(z,r+\sqrt2\rho))}{\upsilon_\zeta}.
\end{equation}
The rows at the pins of $P\cap X'$ are unchanged. Choose one of the nine
classes of $\rho$-cells with fixed indices modulo three which contains
at least a ninth of $P\cap X'$, and let $P'$ be the points of
$P\cap X'$ in that class. Then
\begin{equation}\label{eq:coarse-frostman}
 \frac{\#(P'\cap B(z,r))}{\#P'}\lesssim_{\mathrm{pol}}r^{1+\varkappa}\le r
  \quad \text{for } \rho\le r\le1.
\end{equation}
\end{lemma}
\begin{proof}
By \ref{F:pin-mass}, every occupied $\rho$-cell has mass below
$2\upsilon_\zeta$ and each of its occupied $\delta$-cells has mass at
least $\upsilon_f$, so it contains at most $L=\lceil\frac{2\upsilon_\zeta}{\upsilon_f}\rceil$
pins of $X$, and $2\upsilon_\zeta>\upsilon_f$ by nesting, so
$L\asymp\frac{\upsilon_\zeta}{\upsilon_f}$. Label the pins of each
$\rho$-cell by distinct slots in $\{1,\ldots,L\}$, choose one uniform
slot in every occupied cell independently, and retain its pin if the
slot is occupied. Each pin survives with probability $\frac{1}{L}$.
For a fixed pair $e$, the number of retained witnesses is a sum of
independent indicators, one per occupied $\rho$-cell, with mean
$\frac{\#\mathcal W(e)}{L}\gtrsim_{\mathrm{pol}}\frac{k_f\upsilon_f}{\upsilon_\zeta}$.
This mean is a positive power of $\delta^{-1}$: by
\eqref{eq:fine-richness} and $N\upsilon_f\gtrsim_{\mathrm{pol}}1$ it is
at least $\delta^{o(1)}\frac{\delta^{2q}}{b\upsilon_\zeta}$, and
$\upsilon_\zeta\le C\rho^{1+\varkappa}\le C\delta^{\zeta}$ with
$\zeta=1-c\ge1-a_{\mathrm{par}}$ gives the exponent
\[
 \zeta-(2q-\lambda+\beta)
 \ge1-a_c-\varpi\omega-\xi_{\mathrm{par}}-\frac{1}{2}-2\omega-\beta
 \ge W_0-4\omega-\beta-\xi_{\mathrm{par}}>0 .
\]
Chernoff's lower-tail bound
$\mathbb P\{Z<\frac{1}{2}\mathbb EZ\}\le\mathrm e^{-\frac{\mathbb E Z}{8}}$ and a
union bound over the polynomially many pairs show that all the
witness counts are at least half their means with probability
tending to one. The number of retained pins of $X'$ is a sum of the
same kind with mean $\frac{\#X'}{L}\gtrsim_{\mathrm{pol}}\upsilon_\zeta^{-1}$,
again a positive power of $\delta^{-1}$, and is at least half its
mean with probability tending to one. Fix a sample realizing these
events. There are at most $\upsilon_\zeta^{-1}$ occupied
$\rho$-cells, each having mass at least $\upsilon_\zeta$, which implies
$\#P\le\upsilon_\zeta^{-1}$.

The points stay at their original locations. The $\rho$-cells of
the points of $P\cap B(z,r)$ lie in $B(z,r+\sqrt2\rho)$ and have
mass at least $\upsilon_\zeta$, which proves
\eqref{eq:coarse-ball-count}. Each of the nine classes modulo three
is $\rho$-separated.
Choose such a class containing at least a ninth of $P\cap X'$,
and let $P'$ be its intersection with $P\cap X'$. Then
$\#P'\gtrsim_{\mathrm{pol}}\upsilon_\zeta^{-1}$.
Dividing \eqref{eq:coarse-ball-count} by this lower bound and using
$\nu(B(z,3r))\le Cr^{1+\varkappa}$ proves
\eqref{eq:coarse-frostman}.
\end{proof}

Only the pins of $P'$ are used to generate tubes.
The whole set $P$
remains the witness set. The pin ancestor masses at scales between
$\rho$ and one are those of \ref{F:pin-mass}.
Dividing them by
$\upsilon_\zeta$ bounds the number of points of $P$ in the
corresponding squares. This is the input to the rectangular bound of
Section~\ref{sec:rectangular-sampling}, while
\eqref{eq:coarse-frostman} is the input to the Furstenberg bound of
Section~\ref{sec:finite-incidence}.

\section{Counting the bisector directions}
\label{sec:finite-incidence}

The circle argument produces, at each pin of $P'$, rows in which every
pair has a rich bisector. We now choose one pair from every row so
that the chosen directions are spread at every scale between
$\rho$ and one. The restricted-sum estimate controls the number of
distinct directions that any set of rows can produce, and a
max-flow argument turns this into a simultaneous choice. The
Furstenberg estimate of Ren and Wang then provides a lower bound for
the number of distinct tubes. All results in this section are finite
and make no information-theoretic assumptions.

\subsection{The restricted-sum estimate}

For an $n\times s$ real matrix $A=(a_{ij})$ let
$S(A)=\{a_{ij}+a_{ik}:i,j,k\in\mathbb N,\ 1\le i\le n, 1\le j<k\le s\}$ be the set of
pairwise sums of entries from a common row, and let $f_s(n)$ be the
minimum of $|S(A)|$ over all such matrices whose $sn$ entries are
pairwise distinct. The defining requirement in \cite[\S1]{KT}
is \lq\lq all its $sn$ entries being pairwise distinct.\rq\rq{} In
this statement $n$ is the number of rows.

\begin{theorem}[Katz--Tardos, {\cite[Corollary~5]{KT}}]
For any $\epsilon>0$ there exists an $s>0$ such that for all $n>0$ we
have $f_s(n)\ge n^{\frac{10-3\mathrm e}{24-7\mathrm e}-\epsilon}$.
\end{theorem}

Write
\[
 \alpha_* = \frac{10-3\mathrm e}{24-7\mathrm e}=0.371107\ldots.
\]
Distinctness is required across the whole matrix. Katz and Tardos
point out that the bounds of~\cite{Tardos} hold whenever no two rows
share more than one entry, whereas their argument uses that all entries
are distinct~\cite[\S4]{KT}. The explicit cases stated after
\cite[Theorem~4]{KT} are
\[
 f_5(n)\ge n^{\frac{7}{19}},\qquad 
 f_7(n)\ge n^{\frac{33}{89}},\qquad 
 f_9(n)\ge n^{\frac{59}{159}}.
\]

Fix $\frac{1}{3}<\alpha<\alpha_*$ and choose an integer $s\ge3$ such
that $f_s(n)\ge c_\alpha n^\alpha$ for all $n$.
Katz--Tardos supply
one, and $s=5$ permits $\alpha=\frac{7}{19}$, $s=9$ permits
$\alpha=\frac{59}{159}$. All constants below can depend on $\alpha,s$
and fixed geometric bounds, but not on $\delta$.

\subsection{A simultaneous choice of one direction per row}

At each retained pin discard surplus rows so that exactly
$\mathfrak d'$ rows remain. All upper bounds and pair-witness
lower bounds persist, and the row count decreases by at most a
factor two. Fix a pin $x\in P'$ and these rows of
Proposition~\ref{prop:second-pruning}.
Recall $a\le c\le a+\xi$,
$\zeta=1-c$, $\rho\asymp\delta^{\zeta}$, $w\asymp\delta^{a}$, and
$2c<1$. The direction of the bisector of a pair is determined, up to
$O(\rho)$ and a fixed rotation, by the half-sum of the two source
angles. A direction bin is an interval of length $\rho$ in the
angular chart. The bins form a fixed partition of that chart.
When defining a measure on the outputs, we identify each bin with
its midpoint.
The output of a pair is the bin containing its half-sum.

This description also explains the connection with the near-additive-energy
estimate of~\cite[Lemma~19]{MIKRYSumsets}, which counts pairs of sums
that differ by at most a prescribed error. Such an estimate for the
angular rows could control clustering of their directions. The
lacunarity used in that lemma is unavailable here. Moreover, the
selection below requires bounds for every subfamily of rows and
at each relevant scale. A replacement would have to provide that
uniform control while retaining the number of rows needed by the
incidence argument.

\begin{lemma}[A simultaneous choice of directions]
\label{lem:flow}
Fix small $\xi_{\mathrm{cut}}>0$ and let $s_1$ be the
$\mathfrak c$-color Ramsey number of the complete graph on $s$
vertices, where $\mathfrak c=\lceil\frac{1}{\xi_{\mathrm{cut}}}\rceil+4$.
This is the least integer such that every coloring of the pairs
of that many vertices with $\mathfrak c$ colors has $s$ vertices
all of whose pairs have one color.
Ramsey's theorem~\cite{Ramsey} makes it finite. Then one can choose one pair from every row at $x$ so that the
probability measure $\sigma_x$ giving mass $\frac{1}{\mathfrak d'}$ to
the output of each chosen pair satisfies
\begin{equation}\label{eq:flow-measure}
 \sigma_x(J)\lesssim_{\mathrm{pol}}
 \begin{cases}
 |J|,& w\le|J|\le1,\\
 \delta^{-\tau_{\mathrm{dir}}}|J|^{\alpha}w^{1-\alpha},&\rho\le|J|\le w,
 \end{cases}
  \tau_{\mathrm{dir}}=\alpha\max\{\xi,\xi_{\mathrm{cut}}\},
\end{equation}
for every angular interval $J$. Consequently, with
\begin{equation}\label{eq:s-act}
 s_{\mathrm{act}}=\alpha+(1-\alpha)\frac{a}{\zeta}<1,
\end{equation}
one has $\sigma_x(J)\lesssim_{\mathrm{pol}}\delta^{-\tau_{\mathrm{dir}}}|J|^{s_{\mathrm{act}}}$
for $\rho\le|J|\le1$.
\end{lemma}
\begin{proof}
Work in one window, containing $R\le\frac{Cw}{b}$ rows. Unwrap the
angular chart, subtract the fixed rotation, and double the
directions, so that each pair $(i,j)$ of a row with angles
$\vartheta_i,\vartheta_j$ has a real output $o_{ij}$ with
$|o_{ij}-(\vartheta_i+\vartheta_j)|\le C\rho$.
Rounding the outputs to
a mesh comparable to $\rho$ changes this by a constant. All outputs
lie in an interval of length $O(w)$. Choose a dyadic root interval of
comparable length, using one of three translated grids, and the
finite dyadic tree below it, whose nodes are intervals and
whose children are their two halves. Its leaves are the terminal
intervals, of length comparable to $\rho$.
The root intervals of different windows have bounded overlap.
For a tree interval $I$ prescribe the capacity
\[
 c(I)=C_0\delta^{-\tau_{\mathrm{dir}}}|I|^{\alpha}w^{1-\alpha}b^{-1}.
\]

We first verify the cut inequality needed for the flow. We claim that if $U$ is a set of rows and
$\mathcal A$ is an antichain, a collection of tree intervals none
containing another, covering every output
of every row of $U$, then
$\sum_{I\in\mathcal A}c(I)\ge\#U$, provided $C_0$ is large. Put
$r_0=\frac{\ell_j}{100}$, so that entries of one row are separated
by at least $100r_0$. If $\mathcal A$ contains an interval of length
at least $r_0$, its capacity is at least
$c_1C_0\delta^{-\tau_{\mathrm{dir}}}\delta^{\alpha(c-a)}\frac{w}{b}\ge c_1C_0\frac{w}{b}\ge R\ge\#U$,
with an absolute constant $c_1>0$, because $c-a\le\xi$ and $\tau_{\mathrm{dir}}\ge\alpha\xi$. Otherwise
every interval of $\mathcal A$ is shorter than $r_0$. Put
$r_i=r_0\delta^{i\xi_{\mathrm{cut}}}$ and sort the intervals of
$\mathcal A$ into bins according to the $i$ with
$|I|\in(r_{i+1},r_i]$.
Since the leaves have length comparable to
$\rho=\delta^{\zeta}$, $r_0\asymp\delta^{c}$, and $\zeta-c<1$, at
most $\mathfrak c$ bins occur. Color every pair of every row of $U$
by the bin of the interval of $\mathcal A$ containing its output.
By Ramsey's theorem each row contains $s$ entries all of whose pairs
have one color, and one color is used by a subfamily $U'$ of at
least $\frac{\#U}{\mathfrak c}$ rows.
Let $r=r_i$ be its scale.
Round the entries of these $s$-subrows to multiples of $r$.
Entries
within a subrow remain distinct because $r\le r_0$. Every rounded
value belongs to at most $\frac{Cr}{b}$ subrows, by item~2 of
Proposition~\ref{prop:homogeneous-rows}, so the conflict graph on
$U'$, joining two subrows when they share a rounded entry, has an
independent set of at least $\frac{c\#U'b}{r}$ subrows
with globally distinct rounded entries, or a single subrow if this
number is below one. Katz--Tardos give at least
$c(\frac{\#Ub}{r})^{\alpha}$ distinct pair sums of rounded entries
in this independent set. A rounded pair sum differs by $O(r)$
from its original output, since every bin contains a tree interval
of length comparable to or greater than $\rho$, so $\rho\lesssim r$.
Thus only
$O(1)$ distinct rounded sums can correspond to any one output
bin at scale $r$. There are consequently at least a fixed
multiple of this many occupied output bins. An interval of the
color has length at most $r$ and meets $O(1)$ rounded outputs, so
at least $c(\frac{\#Ub}{r})^{\alpha}$ intervals of $\mathcal A$ carry
this color, each of length greater than $r\delta^{\xi_{\mathrm{cut}}}$.
Their capacities sum to at least
\[
 cC_0\delta^{-\tau_{\mathrm{dir}}+\alpha\xi_{\mathrm{cut}}}
 (\#U)^{\alpha}\Bigl(\frac{w}{b}\Bigr)^{1-\alpha}
 \ge cC_0\#U,
\]
by $\#U\le\frac{Cw}{b}$ and $\tau_{\mathrm{dir}}\ge\alpha\xi_{\mathrm{cut}}$.
This proves the claim.

A flow on a directed network assigns a nonnegative amount to each
arc, no larger than its capacity, and conserves total inflow and
outflow except at the source and sink. Its value is the net flow
out of the source. A cut separates the source from the sink.
Its
capacity is the sum of capacities of arcs leaving the source side.
The original max-flow min-cut theorem appears in \cite{FF}.
For the directed networks with integer capacities used here, start
with zero flow. Each arc has a forward residual arc of capacity
equal to its unused capacity and a reverse residual arc of capacity
equal to its flow. Retain only arcs of positive residual capacity.
Augment along a residual source--sink path by
its minimum residual capacity, increasing flow on forward arcs and
decreasing it on reverse arcs. Flows remain integral, and every
augmentation increases the value by at least one. The value is
bounded by the total capacity leaving the source, so the process
terminates. When no residual path remains, the vertices reachable
from the source determine a cut whose outgoing arcs are saturated
and whose incoming arcs carry zero flow. The flow value equals
this cut capacity, and every flow value is at most any cut capacity.
This proves maximality and integrality.

Build a network with a source joined to every row
by an arc of capacity one, every row joined to each of its output
leaves by an arc of capacity $R$, every tree interval $I$ joined to
its parent by an arc of capacity $\lceil c(I)\rceil$, and the root
joined to the sink by an arc of capacity $R$. Every cut of capacity
less than $R$ leaves some rows $U$ joined to the source and must
separate all their leaves from the root by tree arcs.
The highest
such arcs on the leaf--root paths form a covering antichain, so the
cut has capacity at least $R-\#U+\sum_{\mathcal A}\lceil c(I)\rceil\ge R$.
By the integral-flow argument above there is an
integral flow of value $R$, that is, one output per row, with at
most $\lceil c(I)\rceil$ chosen outputs in each tree interval $I$.
Since $w\gtrsim\rho\gg b$, every $c(I)$ is at least a fixed
constant, so the ceilings change only constants, and an arbitrary
interval is covered by $O(1)$ tree intervals of comparable length.
Dividing by $\mathfrak d'\gtrsim_{\mathrm{pol}}b^{-1}$ proves the
second line of \eqref{eq:flow-measure}, after accounting for the
bounded overlap of the windows. For $|J|\ge w$, every row with
chosen output in $J$ has its entries in a fixed enlargement of $J$,
so there are $O(\frac{|J|}{b})$ such rows, which proves the first
line.

To obtain the stated exponent, for $\rho\le|J|\le w$ write $|J|=\delta^{u}$
with $a\le u\le\zeta$.
Then
$\alpha u+(1-\alpha)a-s_{\mathrm{act}}u=(1-\alpha)a(1-\frac{u}{\zeta})\ge0$,
so $|J|^{\alpha}w^{1-\alpha}\le|J|^{s_{\mathrm{act}}}$.
For $|J|\ge w$,
$|J|\le|J|^{s_{\mathrm{act}}}$.
\end{proof}

\subsection{The Furstenberg estimate and the tube family}

We use the discretized Furstenberg estimate of Ren and Wang.
For a dyadic square $Q$ of side $\rho$ in slope--intercept parameter
space, the corresponding dyadic $\rho$-tube is
\[
\bigcup_{(a,b)\in Q}
\{(u,v)\in\mathbb R^2:v=au+b\},
\]
restricted to the fixed bounded spatial region under consideration.
Its central line has the parameters of the midpoint of $Q$.
We work in bounded-slope charts. In the Ren--Wang chart the slopes
lie in $[-1,1]$. These tubes have width comparable to $\rho$ on
the bounded regions used here \cite[Definition~2.3]{RW}.
For dyadic $\rho$, a finite family $\mathcal F$ of dyadic
$\rho$-squares or dyadic $\rho$-tubes is a $(\rho,s,C)$-set if
$\#\{\text{members of }\mathcal F\text{ meeting }B\}
\le Cr^s\#\mathcal F$ for every ball $B$ of radius
$r\in[\rho,1]$ in the respective parameter space
\cite[Definitions~1.3 and~2.3]{RW}.
For tubes, this count is taken over their defining squares in
slope--intercept parameter space. A pair $(\mathcal P,\mathcal T)$ is
a $(\rho,s,C,M)$-nice configuration if every $p\in\mathcal P$ is
assigned a $(\rho,s,C)$-set $\mathcal T(p)\subset\mathcal T$ with
$\#\mathcal T(p)\sim M$ and $T\cap p\ne\varnothing$ for all
$T\in\mathcal T(p)$ \cite[Definition~3.1]{RW}.

\begin{theorem}[Ren--Wang, {\cite[Theorem~4.1]{RW}}]\label{thm:RW}
Let $s\in(0,1]$, $t\in(0,2]$, and $\epsilon>0$. There is
$\eta_{\mathrm{RW}}>0$ such that, for all sufficiently small $\rho$,
every $(\rho,s,\rho^{-\eta_{\mathrm{RW}}},M)$-nice configuration
$(\mathcal P,\mathcal T)$ in which $\mathcal P$ is a
$(\rho,t,\rho^{-\eta_{\mathrm{RW}}})$-set satisfies
\[
 \#\mathcal T\gtrsim_\epsilon\rho^{-\min\{t,\frac{s+t}{2},1\}+\epsilon}M .
\]
\end{theorem}

We apply this with $t=1$ in the proof of
Theorem~\ref{thm:universal-main} and with $t=1+\varkappa$ in the
proof of Theorem~\ref{thm:exceptional-quantitative}.
For $t=1$ the
exponent is $\frac{1+s}{2}$.

\begin{proposition}[The tube family]\label{prop:tube-family}
Fix $\vartheta>0$ and let $s'$ be a point of a fixed finite grid in
$(0,1)$ with $\vartheta\le s_{\mathrm{act}}-s'\le2\vartheta$.
Suppose
$\tau_{\mathrm{dir}}<\frac{\vartheta\zeta}{2}$. For all sufficiently
large $n$ there exist a family
$\mathcal T$ of $\rho'$-tubes, $\rho'=2^{j_0}\rho$ with $j_0$ fixed,
and subfamilies $\mathcal T(x)$, $x\in P''$, for a subset
$P''\subset P'$ with $\#P''\ge c\#P'$, such that $(P'',\mathcal T)$
is a $(\rho',s',C_n\delta^{-\tau_{\mathrm{dir}}},M)$-nice
configuration with $M=\lceil\rho^{-s'}\rceil$ and
$C_n\lesssim_{\mathrm{pol}}1$, after the
identification of $P''$ with the set of dyadic $\rho'$-squares
containing its points.
Every tube of $\mathcal T$ contains a fixed
$C\rho$-neighborhood of the bisector of one of the chosen pairs at
one of its pins, and hence at least $k$ points of $P$, where
\begin{equation}\label{eq:k-from-kf}
 k\gtrsim_{\mathrm{pol}}\frac{k_f\upsilon_f}{\upsilon_\zeta};
\end{equation}
and $\mathcal T$ is contained in a parameter-separated family of
width-$C\rho$ tubes of the same cardinality up to a constant.
\end{proposition}
\begin{proof}
In this proof $C_n\lesssim_{\mathrm{pol}}1$ denotes a common
upper bound for the polynomial factors in Lemma~\ref{lem:flow}.
These factors cannot be treated as fixed constants. At each
$x\in P'$ consider the measure $M\sigma_x$, of total mass $M$. By
\eqref{eq:flow-measure}, a single direction bin then carries mass at
most $C_n\delta^{-\tau_{\mathrm{dir}}}\rho^{s_{\mathrm{act}}-s'}<1$, and an
interval of length $r\ge\rho$ carries mass at most
$C_n\delta^{-\tau_{\mathrm{dir}}}(\frac{r}{\rho})^{s'}$, since
$r^{s_{\mathrm{act}}}\le r^{s'}$. On the dyadic tree of direction bins,
join a source to every leaf by an arc of capacity one, every nonroot
tree interval $I$ to its parent by an arc of capacity
$\lceil M\sigma_x(I)\rceil$, and the root to a sink by an arc of
capacity $M$. The flow sending $M\sigma_x(\text{leaf})\le1$ through
every leaf is feasible and has value $M$. By integrality there is an
integral flow of value $M$. It selects exactly $M$ supported distinct
direction bins, and every tree interval $I$ contains at most
$\lceil M\sigma_x(I)\rceil$ of them. An interval of length $r\ge\rho$
meets at most two tree intervals of length less than $2r$, so it
contains at most $6C_n\delta^{-\tau_{\mathrm{dir}}}(\frac{r}{\rho})^{s'}$
of them. For each selected bin keep one chosen pair with output in
it and the line through $x$ in the bisector direction of that pair.
It lies within $C\rho$ of the bisector and of $x$.

The leaf bound is strictly smaller than one for large $n$:
the positive power margin
$\zeta(s_{\mathrm{act}}-s')-\tau_{\mathrm{dir}}
>\frac{\vartheta\zeta}{2}$ absorbs $C_n$.

Restrict to one of finitely many angular charts, in which direction
and slope are bi-Lipschitz.
At each pin one chart carries a fixed
fraction of the selected bins, and one chart is chosen by a fixed
fraction of the pins. Write lines as $y=ax+b$ in a fixed bounded
region. For a sufficiently large fixed $j_0$, in one of finitely
many translated dyadic grids in the $(a,b)$-plane, the dyadic
$\rho'$-square containing the parameter of a selected line contains
a $C\rho$-neighborhood of it, so the corresponding dyadic tube
contains the original width-$C\rho$ tube within the region: keep the
slope and vary the intercept over the error interval. The tube
therefore contains the pin, the bisector neighborhood, and all
witnesses. Choose a translation for each selected line for which this
containment holds. At each pin retain a translation carrying a
fixed fraction of its lines, then retain a fixed fraction of pins
with the same translation. Translating the slope and intercept
grid amounts to one fixed shear and translation in the plane,
which preserve the required incidence and spacing bounds up to
constants.
A dyadic $\rho'$-square in the parameter space absorbs $O(1)$
selected bins at a fixed pin, so the selected tubes at each pin
still number $\sim M$ and satisfy the $(\rho',s',C_n\delta^{-\tau_{\mathrm{dir}}})$
condition with a new constant. A dyadic $\rho'$-square of the plane
contains $O(1)$ points of $P'$.
Keep one of them per square. This
is the nice configuration.
\eqref{eq:k-from-kf} is
\eqref{eq:coarse-richness}. Finally, associate to each selected dyadic parameter square
its central line and a tube of sufficiently large fixed width
$C\rho$. In the bounded region this tube contains the entire
dyadic tube, hence all its witnesses. Distinct parameter squares
have separated centers. A bounded coloring accommodates any
additional fixed dilation. This constructs a family of comparable
cardinality while preserving the richness of every member.
\end{proof}

The lower bound for $\#\mathcal T$ now follows from
Theorem~\ref{thm:RW} with $\mathcal P=P''$, which is a
$(\rho',1,\delta^{-o(1)})$-set by \eqref{eq:coarse-frostman}, and
$s=s'$:
\begin{equation}\label{eq:RW-lower}
 \#\mathcal T\gtrsim\rho^{-\frac{1+s'}{2}+\epsilon_{\mathrm{RW}}}M
 \ge\rho^{-\frac{1+3s'}{2}+\epsilon_{\mathrm{RW}}}
 \ge\delta^{-\frac{\zeta+3s_{\mathrm{act}}\zeta}{2}+3\vartheta+\epsilon_{\mathrm{RW}}},
\end{equation}
where $\epsilon_{\mathrm{RW}}>0$ is fixed and the enlargement
$\rho'=2^{j_0}\rho$ and the constants $\delta^{-o(1)}$ change no
exponent. The tolerance $\eta_{\mathrm{RW}}$ depends on $s'$ and
$\epsilon_{\mathrm{RW}}$.
Since $s'$ ranges
over a fixed finite grid, the minimum of these tolerances is a fixed
positive number, and the losses $\delta^{-\tau_{\mathrm{dir}}}$ and
$\delta^{-o(1)}$ are below $\rho^{-\eta_{\mathrm{RW}}}$ once
$\tau_{\mathrm{dir}}$ is small and $\delta$ is small.

\section{Rectangular sampling and rich incidences}
\label{sec:rectangular-sampling}

The upper bound for the rich tubes comes from an incidence theorem
of Fu, Gan and Ren for balls that are spread on a rectangular grid.
In this section a tube is a fixed-length neighborhood of width
$C\rho$ of a line in a bounded region, and separation refers to
the line parameters of Section~\ref{sec:preliminaries} in a fixed
finite collection of charts. 
For $\rho>0$, a family of $\rho$-balls, or of $\rho$-tubes, is
essentially distinct if the intersection of two members has at most
half the area of either member.
$T_k(\mathcal B)$ denotes the set of
essentially distinct $\rho$-tubes meeting more than $k$ balls of
$\mathcal B$.

\begin{theorem}[Fu--Gan--Ren, {\cite[Theorem~3]{FGR}}]
Let $1\le U\le V\le\rho^{-1}$, and divide $[0,1]^2$ into
$U^{-1}\times V^{-1}$ rectangles. Let $\mathcal B$ be a family of
$\rho$-balls with at most one ball in each rectangle. For every
$\sigma>0$ and every $k>\max\{\rho^{1-2\sigma}UV,1\}$,
\begin{equation}\label{eq:FGR-input}
 \#T_k(\mathcal B)\lesssim_\sigma
 \rho^{-\sigma}\#\mathcal BUVk^{-2}\bigl(k^{-1}+U^{-1}\bigr).
\end{equation}
\end{theorem}

A parameter-separated family splits
into a bounded number of essentially distinct families, since overlap
exceeding half the area forces parameter distance $O(\rho)$, and
separation bounds the degree of this conflict graph. The same holds
for balls about a separated set. These bounded colorings, fixed
dilations, and fixed similarities affect only constants, and allow
\eqref{eq:FGR-input} to be used for the separated pins and tubes of
the circle construction.

\begin{lemma}[Rectangular sampling]\label{lem:rectangular-sampling}
Let $X$ be a $\rho$-separated set of $N$ points in a fixed bounded
region. Suppose $\rho\le c_0h\le c_0R\le c_0$, with $c_0>0$
sufficiently small, and
\begin{equation}\label{eq:reciprocal-hypothesis}
 \sum_{x'\in X, |x-x'|\le3R}
 \min\left\{1,\frac{h}{|x-x'|}\right\}\le H \quad \text{for } x\in X,
\end{equation}
where the diagonal term is one. Fix $0<\sigma<\frac{1}{4}$, and let
$\mathcal T$ be a parameter-separated family of $\rho$-tubes, each
containing at least $k$ points of $X$. If
\begin{equation}\label{eq:sampling-threshold}
 k\ge CH\max\left\{\frac{\rho^{1-2\sigma}}{Rh},1\right\}
\end{equation}
for a sufficiently large fixed $C$, then
\begin{align}
 \#\mathcal T&\lesssim_\sigma
 \frac{\rho^{-\sigma}N}{Rh}\left(H^{2}k^{-3}+RHk^{-2}\right),
 \label{eq:rich-tubes}\\
 I(X,\mathcal T)&\lesssim_\sigma
 \frac{\rho^{-\sigma}N}{Rh}\left(H^{2}k^{-2}+RHk^{-1}\right),
 \label{eq:rich-incidences}
\end{align}
where $I(X,\mathcal T)$ is the number of pairs $(x,T)$ with
$x\in X\cap T$.
\end{lemma}
\begin{proof}
It suffices to prove \eqref{eq:rich-tubes} for a family whose
members contain between $k$ and $2k$ points: the classes with
richness in $[2^{j}k,2^{j+1}k)$ satisfy the hypothesis at every
larger richness, summing their bounds proves
\eqref{eq:rich-tubes}, and multiplying each bound by $2^{j+1}k$
before summing yields \eqref{eq:rich-incidences}. The grid and the
sample below may depend on the class.

Round $R$ and $h$ to comparable dyadic lengths. Rotate the
$R\times h$ rectangular grid uniformly and translate it uniformly
modulo one cell. Two points $x,x'$ lie in one cell only if
$|x-x'|\le3R$ and their displacement perpendicular to the long
direction is at most $h$, which has probability
$O(\min\{1,\frac{h}{|x-x'|}\})$. By \eqref{eq:reciprocal-hypothesis}
the expected occupancy of the cell containing a given point is at
most $CH$. Delete every point within $C_1\rho$ of a grid boundary.
This has probability $O(\frac{\rho}{R}+\frac{\rho}{h})\le\frac{1}{64}$
after reducing $c_0$. Delete also all points in cells of occupancy
greater than $L=\lceil C_2H\rceil$.
Markov's inequality bounds this
probability by $\frac{1}{64}$ for $C_2$ large. The expected number
of deleted incidences is at most
$\frac{1}{32}I(X,\mathcal T)\le\frac{1}{16}k\#\mathcal T$. Fix a grid
attaining this bound. Fewer than $\frac{1}{8}\#\mathcal T$ tubes lose
more than $\frac{k}{2}$ points, so a subfamily $\mathcal T_0$ with
$\#\mathcal T_0\ge\frac{7}{8}\#\mathcal T$ retains at least
$\frac{k}{2}$ points per tube, and the $\rho$-balls about the
retained points lie inside their cells.

Label the retained points of each cell by distinct slots in
$\{1,\ldots,L\}$. Independently in every cell select a uniform
slot and keep the point occupying it, if any. The sample $X'$ has
at most one point per cell and $\mathbb E\#X'\le\frac{N}{L}$. For a
tube $T\in\mathcal T_0$, its number of sampled points is a sum of
independent indicators, one per cell, with mean $m_T\ge\frac{k}{2L}$
and variance at most $m_T$.
Hence
\[
 \mathbb P\left\{\#(X'\cap T)<\frac{k}{4L}\right\}
 \le\frac{4}{m_T}\le\frac{8L}{k}\le\frac{1}{16},
\]
by \eqref{eq:sampling-threshold}. Also
$\mathbb P\{\#X'>\frac{16N}{L}\}\le\frac{1}{16}$. By the union
bound, the expected number of tubes of $\mathcal T_0$ with at least
$\frac{k}{4L}$ sampled points on a sample with
$\#X'\le\frac{16N}{L}$ is at least
$\frac{7}{8}\#\mathcal T_0>\frac{1}{2}\#\mathcal T$. Fix such a
sample.

In coordinates parallel to the grid, a grid-aligned square of a
fixed dyadic side contains the configuration, and a fixed
similarity sends it to $[0,1]^2$.
Restriction to finitely many
charts and a bounded coloring preserve a fixed fraction of the
relevant incidence counts. Thus \eqref{eq:FGR-input} applies to the
balls about $X'$ with $U\asymp R^{-1}$, $V\asymp h^{-1}$, and
threshold $\frac{k}{8L}$. In the rescaled coordinates the balls are
contained in their $U^{-1}\times V^{-1}$ rectangles, with at most
one ball per rectangle and $1\le U\le V\le\rho^{-1}$, where $\rho$
is also rescaled by the same fixed similarity. The last inequality
uses $\rho\le c_0h$. Its hypothesis follows from
\eqref{eq:sampling-threshold}. With $\#X'\lesssim\frac{N}{L}$ and
$L\asymp H\ge1$ it gives
\[
 \#\mathcal T\lesssim_\sigma
 \frac{\rho^{-\sigma}N}{Rh}\left(H^{2}k^{-3}+RHk^{-2}\right),
\]
which is \eqref{eq:rich-tubes} for one richness class.
\end{proof}

\section{From measures to finite configurations}\label{sec:discretization}

We now return to the measures and construct the finite sets needed
by the incidence theorem. The main requirement is that all pins
see vertices from one common source set. Separate discretizations
for separate pins would not give a graph on which the crossing
argument could be used.

We keep the original measures and the original oracle throughout.
The pins are actual points, with their full oracles available.
Source vertices are centers of small dyadic squares, so their
incidences with circles have error $O(\delta)$. An assigned incidence
records a relation retained by the selection.
A circle merely
passing close to a source vertex need not be assigned to it.

\begin{lemma}[Measure-to-finite selection with occupied-cell bounds]
\label{lem:seventh-measure-finite-selection}
Let $\mu,\nu$ be probability measures on fixed bounded source and pin
regions at positive distance from one another.  Let $G$ be a Borel
relation, and suppose that for every relevant pin $x$ the angular
pushforward of $\mu|_{G_x}$ is bounded by $L$ times angular Lebesgue
measure.  Fix dyadic scales $\delta\le\rho\le\eta$ and a Borel
subrelation $R\subset G$ of pair mass $P>0$.
Assume the following bounds whenever the indicated cells or bins
are occupied by $R$ with positive source-fiber mass.
\begin{enumerate}
\item Every occupied source $\delta$-square $Q$ satisfies
$m_-\le\mu(Q)\le m_+$.
\item Every occupied pin $\rho$-square $S$ satisfies
$v_-\le\nu(S)\le v_+$.
For every dyadic scale $r=2^{-k}\ge\rho$, each ancestor containing
a relevant pin satisfies $\nu(Q_k(x))\le U_k$.
\item For every source $\eta$-square $C$ and terminal dyadic radial
interval $J$ of length $\delta$ which contains a pair from $R_x$,
\[
 \mu(R_x\cap C\cap\{y\in\mathbb R^2:|x-y|\in J\})\le b.
\]
\item At each pin, at most $M$ terminal radial intervals are occupied.
\end{enumerate}
Then there are a $\rho$-separated finite set $X$ of actual pins, a
common set $V$ of source $\delta$-square centers, and assigned
incidences $(x,v,J)$, with circles centered at $x$ and radii equal
to the midpoints of $J$, such that, writing $N=\#X$, $I$ for the
number of assigned incidences, and $\mathcal C$ for the circle family,
\begin{align}
 \#V&\le m_-^{-1},& N&\le v_-^{-1},&
 I&\ge \frac{cP}{m_+v_+},& N&\ge\frac{I}{C\#V},
 \label{eq:seventh-finite-counts}\\
 \#\mathcal C&\le NM,&
 \bigl||x-v|-\operatorname{mid}(J)\bigr|&\le C\delta.
 \label{eq:seventh-finite-circles}
\end{align}
Each source vertex has at most $C$ assigned radial labels at one pin.
At every selected pin, assigned source vertices in an angular
interval $A$ obey
\begin{equation}\label{eq:seventh-finite-angular}
 \#\{v\in V:\arg(v-x)\in A, v\text{ assigned at }x\}
 \le \frac{CL(|A|+\delta)}{P m_-}.
\end{equation}
For every source $\eta$-parent and every assigned radial bin, the
number of assigned vertices in that parent is at most
\begin{equation}\label{eq:seventh-finite-parent-cap}
 \frac{Cb}{P m_-}.
\end{equation}
The same conclusion holds, with another constant, for vertices in
any annulus of width $C_0\delta$ about the pin.  Finally,
\begin{equation}\label{eq:seventh-finite-pin-balls}
 \#(X\cap B(z,2^{-k}))\le\frac{C U_k}{v_-}
  \quad \text{for } 2^{-k}\ge\rho.
\end{equation}
The constants depend only on the fixed geometric bounds, including
the source--pin separation and $C_0$.
\end{lemma}
\begin{proof}
Write
\[
 \lambda_{x,j}(Q)=
 \mu\{y\in R_x\cap Q:|x-y|\in J_j\}.
\]
One source $\delta$-square meets at most a fixed number $B_0$ of
terminal radial bins about any pin, since distance is $1$-Lipschitz.
Choose $\tau=\frac{P}{4B_0}$ and retain a cell and radial-bin assignment only when
$\lambda_{x,j}(Q)\ge\tau\mu(Q)>0$.
The total discarded source mass at each pin is at most
$B_0\tau\sum_Q\mu(Q)=B_0\tau$.  Hence the retained relation has
pair mass at least $\frac{P}{2}$.

Let $F(x)$ be the number of retained cell and radial-bin assignments at $x$.
Each assignment has mass at most $m_+$, so
\[
 \int F(x) d\nu(x)\ge \frac{P}{2m_+}.
\]
For each pin $\rho$-square $S$ with positive integral, choose an
actual relevant pin $x_S\in S$ such that
\[
 F(x_S)\ge\frac{1}{2\nu(S)}\int_S F d\nu.
\]
Such a choice exists because the integral is positive. The pin
has a positive retained fiber in the already uniformized relation:
its assigned cells and bins satisfy the cell-mass and parent
bounds, its occupied ancestors satisfy the pin-mass bounds, and
the angular bound holds on the containing fiber $G_{x_S}$.
These properties refer to the same original measures $\mu,\nu$,
which are not replaced when $x_S$ is selected.  Its occupied square has
mass between $v_-$ and $v_+$.  Therefore
\[
 \sum_S F(x_S)\ge\frac{P}{4m_+v_+},\qquad 
 \#\{S\}\le v_-^{-1}.
\]
Colour the $\rho$-grid by its two coordinate indices modulo three.
One of the nine colours carries at least one ninth of the displayed
incidence sum.  Keep that colour.  The selected actual pins are
$\rho$-separated.  Take as $V$ the centers of all source squares
used by their retained assignments.  Disjointness and the lower
source-cell mass give $\#V\le m_-^{-1}$.  There are at most $B_0$
assignments per source vertex and pin, proving the last bound in
\eqref{eq:seventh-finite-counts}.  The radial count and approximate
circle incidence follow directly from the construction.

For the angular bound, every good point in a source square whose
center has direction in $A$ has direction in the $C\delta$-enlargement
of $A$.  This uses the fixed positive source--pin separation.
The active cell and radial-bin sets are disjoint and each has mass at least
$\tau m_-$.  Their total mass is at most $CL(|A|+\delta)$ by the
original angular bound on $\mu|_{G_x}$.  Division proves
\eqref{eq:seventh-finite-angular}.

For one parent and bin, sum the disjoint active masses and use the
restricted parent and radial-bin bound $b$.  Division by $\tau m_-$ proves
\eqref{eq:seventh-finite-parent-cap}.
An annulus of width $C_0\delta$ for the source centers can correspond
to only a bounded number of their actual radial bins, so the same
bound applies after changing the constant.

For \eqref{eq:seventh-finite-pin-balls}, a ball of radius $2^{-k}$
meets only a bounded number of dyadic squares of that side length.
Keep only those squares containing selected pins.  Each such square
has full $\nu$-mass at most $U_k$, witnessed by one of its selected
relevant pins.  Its selected $\rho$-squares are disjoint, nested
inside it, and each has mass at least $v_-$.  This proves the count.
No global mass bound on unoccupied ancestors is required.
\end{proof}

\subsection*{A common template for the plateau and the source parent}

Let $\Omega$ be the compact set of
$\theta=(u_-,W,\varsigma,\chi,\mathfrak e)$ satisfying the constraints
of \eqref{eq:plateau-parameters} that do not involve $f$:
\begin{gather*}
 -1\le\chi\le2,\qquad  \varsigma\ge\max\{0,-2\chi\},\qquad 
 3\varsigma+2\chi\le4,\qquad  0\le\mathfrak e\le\varsigma,\\
 W=W_0-(\varsigma+\chi-\mathfrak e)\omega,\qquad 
 0\le u_--W\le(2\varsigma-2\mathfrak e)\omega .
\end{gather*}
Put $u_+=u_-+a_c+\chi\omega$, define the tent
\[
 T_\theta(u)=\max\{0, W-\operatorname{dist}(u,[u_-,u_+])\}
  \quad \text{for } 0\le u\le1,
\]
and let $u_0,\varpi,e_0$ be given by \eqref{eq:u0-varpi}.
All of
these depend continuously on $\theta$. A template is a point of
$\Omega$ together with a branch index: branch~1 requires
$\mathfrak e\le\frac{1}{8}$ and branch~2 requires
$\mathfrak e\ge\frac{1}{8}$. By Proposition~\ref{prop:pin-plateau},
Lemma~\ref{lem:compatible-hit}, and the $1$-Lipschitz property of
$g(u)-u\ge0$, every limiting profile of a pair with the certificate
at level $\lambda$ and $d_x(1)\le q$ satisfies
\[
 g(u)\ge u+T_\theta(u) \quad \text{for } 0\le u\le1,\qquad 
 a_*\le a_c+\varpi\omega,
\]
for some $\theta\in\Omega$, and if $\mathfrak e\le\frac{1}{8}$ also
the envelope \eqref{eq:source-envelope}. The next lemma replaces
these limiting statements by finitely many tests at one precision.

\begin{lemma}[Finite selection of a common template]
\label{lem:joint-template}
Fix $\tau_0>0$ and $0<\xi_{\mathrm{par}}\le\frac{\tau_0}{4}$. There
is a finite set $\Theta$ of templates, independent of the pair, with
the following property. Let $(x,y)$ satisfy the profile hypotheses
relative to the oracle $A$, and suppose that every limiting profile
of the pair along precisions $n$ with $K_n^{A,x}(|x-y|)\le qn$ has
the certificate \eqref{eq:certificate} at level $\lambda$. Then all
sufficiently large $n$ with $K_n^{A,x}(|x-y|)\le qn$ admit a template
$\theta\in\Theta$ such that, writing
$a_{\mathrm{par}}=a_c+\varpi\omega+\xi_{\mathrm{par}}$ and
$C_n=Q_{\lfloor a_{\mathrm{par}}n\rfloor}(y)$,
\begin{align}
 K_n^A(y)&\ge(\lambda-\tau_0)n,
 \label{eq:template-terminal}\\
 \frac{K_k^A(x)}{n}&\ge\frac{k}{n}+T_\theta\left(\frac{k}{n}\right)-\tau_0
  \quad \text{for } 0\le k\le n,
 \label{eq:template-pin}\\
 K_{\lfloor a_{\mathrm{par}}n\rfloor}^A(y)&\ge(F-\tau_0)n,
 \label{eq:template-parent}\\
 K_n^{A,x}(|x-y|\mid C_n)
 &\ge\lambda n-K_{\lfloor a_{\mathrm{par}}n\rfloor}^A(y)-\tau_0n,
 \label{eq:template-conditional}
\end{align}
and, if $\theta$ has branch~1,
\begin{equation}\label{eq:template-source}
 \frac{K_k^A(y)}{n}\ge\frac{k}{n}+a_c-(1+u_0+\mathfrak e)\omega-\tau_0
  \quad \text{for } a_c-\omega\le\frac{k}{n}\le a_{\mathrm{par}}.
\end{equation}
On any fixed positive-measure Borel relation of such pairs, the
threshold and the coding errors of Lemma~\ref{lem:measure-coding}
can be made uniform on a fixed positive-measure subrelation, and at
each such $n$ one common template is valid on at least
$\frac{1}{\#\Theta}$ of the low-distance pair mass.
\end{lemma}
\begin{proof}
The tents, the number $\varpi$, and the coefficient
$(1+u_0+\mathfrak e)\omega$ are continuous on the compact branches
$\Omega\cap\{\mathfrak e\le\frac{1}{8}\}$ and
$\Omega\cap\{\mathfrak e\ge\frac{1}{8}\}$. Choose a finite net in
each branch such that every point of the branch has a net point with
tent and coefficient within $\frac{\tau_0}{3}$ and with
$\varpi\omega$ within $\frac{\xi_{\mathrm{par}}}{3}$.
Let $\Theta$ be
the union of the two nets.

Suppose a fixed pair failed all templates at infinitely many
low-distance precisions. Extract uniformly convergent source, pin,
and distance profiles along them, as in
Lemma~\ref{lem:joint-extraction}. The limiting profiles satisfy the
prefix and transfer estimates, $d_x(1)\le q$, and the certificate.
So they have the plateau, the bound $a_*\le a_c+\varpi_{\theta_0}\omega$,
and, if $\mathfrak e_{\theta_0}\le\frac{1}{8}$, the envelope
\eqref{eq:source-envelope}, for some $\theta_0\in\Omega$. Let
$\theta$ be a net point of the branch containing $\theta_0$ within
the prescribed errors. Its parent scale satisfies
$a_{\mathrm{par}}\ge a_c+\varpi_{\theta_0}\omega+\frac{2}{3}\xi_{\mathrm{par}}>a_*$,
hence $f(a_{\mathrm{par}})\ge F$, and
$H_{f,g}(a_{\mathrm{par}},1)\ge H_{f,g}(a_*,1)\ge\lambda$ because
$H_{f,g}(a,1)$ is nondecreasing in $a$. Also $f(1)\ge\lambda$.
Uniform convergence implies
\eqref{eq:template-terminal}--\eqref{eq:template-parent} at all
large precisions of the sequence, and
Lemma~\ref{lem:conditional-parent-transfer}, with
$a=a_{\mathrm{par}}$ and $b=1$, establishes \eqref{eq:template-conditional}.
Indeed, from the finer square $Q_{\lceil a_{\mathrm{par}}n\rceil}(y)$
and one extra bit one computes $C_n=Q_{\lfloor a_{\mathrm{par}}n\rfloor}(y)$.
Thus the complexity conditioned on $C_n$ is at least that conditioned
on the finer square minus $O(1)$. The source complexities at these
two precisions differ by $O(\log n)$, absorbed by the fixed tolerance.
For a template of branch~1, the envelope \eqref{eq:source-envelope}
holds on $[a_c-\omega,a_c+\varpi_{\theta_0}\omega]$.
On the remaining
part of the tested interval, of length at most
$\frac{4}{3}\xi_{\mathrm{par}}\le\frac{\tau_0}{3}$, monotonicity of $f$
loses at most this length. Including the net error proves
\eqref{eq:template-source} at all large precisions of the sequence.
This contradicts the assumed failures.

The eventual conditions are Borel, so the threshold and the
pair-dependent coding constants can be bounded on a
positive-measure subrelation, losing an arbitrarily small prescribed
amount of mass. Assign each pair to its first valid template in a
fixed ordering and retain a largest class at each precision.
\end{proof}

\begin{corollary}[The finite inputs]
\label{cor:finite-inputs}
In the setting of Lemma~\ref{lem:seventh-classical-recurrence}, let
$\lambda,F,a_c,W_0,q$ be as in \eqref{eq:level-notation}, suppose
that every limiting profile of every pair of $G$ along
low-distance precisions has the certificate \eqref{eq:certificate}
at level $\lambda$, and let the tolerances satisfy
\eqref{eq:tolerance-conditions}. After one fixed restriction of $G$
of positive pair mass, every sufficiently large low-distance
precision $n$ with mass $P_n$ admits a template $\theta\in\Theta$ and
masses $m,\upsilon_f,\upsilon_1,\ldots,\upsilon_{\mathfrak n}$ such
that the subrelation of $R_n$ on which $\theta$ is valid, the source
$\delta$-cell mass lies in $[m,2m)$, the pin $\delta$-cell mass lies
in $[\upsilon_f,2\upsilon_f)$, and the pin ancestor mass at side
$\rho_j$ lies in $[\upsilon_j,2\upsilon_j)$ for every $j$, has mass
\[
 p\ge\frac{cP_n}{n^{\mathfrak n+2}} .
\]
Lemma~\ref{lem:seventh-measure-finite-selection}, applied to this
subrelation with $\rho=\delta$, $m_-=m$, $m_+=2m$, $v_-=\upsilon_f$,
$v_+=2\upsilon_f$, the cap $b=\delta^{\lambda-\beta}$,
$M\le C\delta^{-q}$, and
$U_k=\delta^{\frac{k}{n}+T_\theta(\frac{k}{n})-\tau_0}$, supplies the finite
inputs \ref{F:counts}--\ref{F:pin-mass} of
Section~\ref{sec:circle-rows} with this $p$.
\end{corollary}
\begin{proof}
Apply Lemma~\ref{lem:joint-template} with tolerance
$\frac{\tau_0}{2}$. Its uniformization, together with the fixed
restrictions of Lemma~\ref{lem:seventh-classical-recurrence}, produces
the relation $G$ on which the tests hold for all $n\ge N$. At each
low-distance precision retain a largest template class.

The coding identities and the bound $K_k^A(z)\le2k+O(\log k)$
give cell masses at least $\delta^{2+\tau_0}$ at every occupied
cell and ancestor for large $n$, and all masses are at most one, so
each of the $\mathfrak n+2$ dyadic mass classes has $O(n)$ values.
One product class carries mass at least $\frac{cP_n}{n^{\mathfrak n+2}}$.

The tests of Lemma~\ref{lem:joint-template} and the coding
identities give the mass bounds of \ref{F:source-mass}
and~\ref{F:pin-mass}: at an occupied source cell, parent, or
tested ancestor, and at an occupied pin ancestor, one retained pair
witnesses the bound for the original measure, and the coding error
$O(\log n)$ is below $\frac{\tau_0}{2}n$. In particular
$m\le\delta^{\lambda-\tau_0}\le b$ by
\eqref{eq:template-terminal} and $3\tau_0\le\beta$, and the parents
have mass at most $\delta^{F-\tau_0}$ by \eqref{eq:template-parent}, and the ancestor bound is \eqref{eq:template-pin}. The Frostman
bound in \ref{F:pin-mass} is the exponent of $\nu$.
The conditional test \eqref{eq:template-conditional} implies the
parent--annulus cap: the coding argument in the proof of
Proposition~\ref{prop:conditional-parent-cap} uses only the
conditional bound and the source coding identity, and yields
$\mu\{z\in C_n:|x-z|\in J\}\le\delta^{\lambda-2\tau_0}\le b$ for every
occupied parent and terminal radius bin, as in
Corollary~\ref{cor:all-parents-kept}. At most $2^{qn+O(1)}$
programs have length at most $qn$, and each output determines a
bounded number of terminal radial bins, so at most $C\delta^{-q}$
bins are occupied at each pin.
Throughout this count the entire
original pin remains available as an oracle. Every pin selected by
Lemma~\ref{lem:seventh-measure-finite-selection} belongs to the
retained subrelation and therefore inherits the template and all
mass classes. The remaining assertions are the conclusions of that
lemma.
The angular bound $L$ is fixed and is absorbed into the
constants. The selections are finite measure-theoretic operations
and introduce no new oracle.
\end{proof}

\begin{lemma}[Classical bad pins give one recurrent product relation]
\label{lem:seventh-classical-recurrence}
Let $K\subset\mathbb R^2$ be compact with $\dim_{\mathrm H} K>1$, and let
$q>0$.  If
\[
 B_q=\{x\in\mathbb R^2:\dim_{\mathrm H}\Delta_xK<q\}
\]
has Hausdorff dimension greater than one, then there are separated
compact source and pin supports, Frostman probability measures
$\mu,\nu$ on them whose exponents $1+\kappa<\dim_{\mathrm H}K$ and
$1+\varkappa<\dim_{\mathrm H}B_q$ may be prescribed, with
$\kappa,\varkappa>0$, one
oracle $A$, and a positive-measure Borel relation $G$ with all the
fixed-pair hypotheses, such that
\[
 G\subseteq\limsup_{n\to\infty}
 \{(x,y)\in\mathbb R^2\times\mathbb R^2:K_n^{A,x}(|x-y|)\le qn\}.
\]
The angular-density and fixed-threshold restrictions used above can
be imposed on $G$ while preserving positive measure.  For the resulting
relations $R_n$ and masses $P_n$, one has $\sum_nP_n=\infty$ and
$P_n\ge n^{-2}$ for infinitely many $n$. Every limiting source
profile of a pair of $G$ satisfies $f(u)\ge(1+\kappa)u$, every pin $x$
of a pair of $G$ satisfies $\dim^Ax\ge1+\varkappa$, and
$\nu(B(z,r))\le Cr^{1+\varkappa}$ for every ball.
\end{lemma}
\begin{proof}
The map $x\mapsto\Delta_xK$ is continuous into the space of compact
sets with the Hausdorff metric.  The condition that a compact set
has dimension less than $q$ is Borel. Express it using a rational
exponent below $q$ and, for every integer $m$, the existence of a
finite rational open cover with diameters and total exponent cost
less than $\frac{1}{m}$.  Thus $B_q$ is Borel.
Frostman's lemma and compact extraction give probability measures
with the prescribed exponents on compact subsets of $K$ and of
$B_q$.  By restricting to two disjoint small balls carrying positive
mass, and then to compact positive-mass subsets, arrange separated
supports.
The exponents are unchanged.  Denote the compact source
support by $K'\subset K$.

Use the encoding described in Section~\ref{sec:classical} to choose
one oracle $A$ making the fixed compact supports effectively compact
and both measures computable, including their dyadic masses. The same
oracle is used for every pin and every precision. Lemma~\ref{lem:measure-coding}, first with
$B=(A,x)$ and then with the endpoints exchanged, establishes both
full-oracle independence estimates on a product-full Borel relation,
by Fubini. It also proves the base-oracle dyadic mass identities.
The Frostman bounds give both effective dimensions
greater than one.
More precisely, the source coding identity and
$\mu(Q_k(y))\le C2^{-(1+\kappa)k}$ give
$K_k^A(y)\ge(1+\kappa)k-O_y(\log(k+2))$, so every limiting source
profile satisfies $f(u)\ge(1+\kappa)u$. The same argument with $\nu$
gives $K_k^A(x)\ge(1+\varkappa)k-O_x(\log(k+2))$, hence
$\dim^Ax\ge1+\varkappa$, for $\nu$-almost every pin. Include this
condition in the product-full relation. Radial absolute continuity and the arclength-null
angular complexity deficiency sets give the directional hypotheses
for both choices of pin, as in Proposition~\ref{prop:fixed-pair}.
This establishes the fixed-pair hypotheses for arbitrary computable
Frostman measures, not only for normalized Hausdorff restrictions.
For each pin in the chosen support, $\Delta_xK'$ is effectively
compact relative to $(A,x)$.  The restricted point-to-set principle
therefore implies, for every $y\in K'$,
\[
 \dim^{A,x}|x-y|\le\dim_{\mathrm H}\Delta_xK'
 \le\dim_{\mathrm H}\Delta_xK<q.
\]
Consequently every such pair has infinitely many scales with
$K_n^{A,x}(|x-y|)\le qn$.  This uses one common oracle $A$.
No family of separately chosen pin-dependent point-to-set oracles
is being encoded.

Intersect the full-measure good-pair relation with the bounded radial
density restriction and then the fixed-threshold restrictions
described above.  The resulting fixed relation $G$ has positive pair
measure, and every pair still has infinitely many low-distance scales.
Tonelli's theorem implies
$\sum_nP_n=\int_G\sum_n1_{R_n} d(\nu\times\mu)=\infty$.
Since $\sum_n n^{-2}<\infty$, infinitely many $P_n$ are at least
$n^{-2}$.
\end{proof}

We have obtained a common finite source set, separated actual pins,
and assigned circle incidences with the required angular and parent
bounds. The preceding circle and bisector estimates now apply to this configuration.
We complete the parameter calculation in the next section.

\section{The universal pinned improvement}
\label{sec:universal-pinned}

All parts of the pinned argument are now available. We first
record the exponents of the configuration produced by
Corollary~\ref{cor:finite-inputs} and the selections of
Sections~\ref{sec:circle-rows} and~\ref{sec:finite-incidence}, then
show that they are incompatible with the two tube counts. The
parameters are chosen before the set and its measures. The measures
determine how far out in the precision sequence we must go, but do
not change the threshold.

\subsection{The exponents of the configuration}

Fix $\frac{1}{3}<\alpha<\alpha_*$ with a row length $s$ as in
Section~\ref{sec:finite-incidence}, and put
\begin{equation}\label{eq:c-alpha}
 c_\alpha=3\alpha-1\in\left(0,\frac{1}{4}\right),\qquad 
 r_\alpha=\frac{3\alpha}{40-12\alpha}<\frac{1}{24}.
\end{equation}
Let $0\le\Delta<\frac{1}{2}$ and $0<\omega<\frac{1-2\Delta}{24}$.
Then Section~\ref{sec:plateau} applies,
and $\lambda,F,a_c,W_0,q$ are given by \eqref{eq:level-notation}.
Let $\theta$ be the template, $a\le c\le a+\xi$ the row exponents,
$\zeta=1-c$, $\rho\asymp\delta^{\zeta}$, and
$\mathfrak h=a_{\mathrm{par}}-a$, as in
Propositions~\ref{prop:homogeneous-rows}
and~\ref{prop:second-pruning}.
Put $t=u_+-u_-$ and write the pin ancestor
mass at the scale $\rho$ as $\upsilon_\zeta=\delta^{\gamma}$. Let
$\tau>0$ absorb the template, mass-class, coding, and polynomial
losses: we take $\tau_0\le\frac{\tau}{2}$ and
$\xi,\xi_{\mathrm{par}}\le\tau$, and use that every factor
$\lesssim_{\mathrm{pol}}$ is at most $\delta^{-\tau}$ for large $n$
along the recurrent precisions. Let $P$ be the pin set of
Lemma~\ref{lem:pin-coarsening}, $\mathcal T$ the tube family of
Proposition~\ref{prop:tube-family}, and $k$ the number of points of
$P$ in each tube. For all sufficiently large recurrent $n$,
\begin{align}
 \#P&\le\delta^{-\gamma},&
 \gamma&\ge\max\{\zeta,u_++W\}-\tau,
 \label{eq:N-exponent}\\
 H&\le\delta^{-\iota_H},&
 \iota_H&=\gamma-u_+-W+\tau\ge0,
 \label{eq:H-exponent}\\
 k&\ge\delta^{-\iota_k},&
 \iota_k&=\gamma-\mathfrak r-\beta-\tau,\qquad 
 \mathfrak r=\frac{1}{2}+\min\{2\omega,\psi(\mathfrak h)\},
 \label{eq:k-exponent}\\
 1-a_{\mathrm{par}}&\le\zeta\le1-a_0,
 \label{eq:zeta-range}\\
 \#\mathcal T&\gtrsim\delta^{-\Lambda},&
 \Lambda&=\frac{(1+3s_{\mathrm{act}})\zeta}{2}-3\vartheta-\epsilon_{\mathrm{RW}},
 \label{eq:lower-count}
\end{align}
where $H$ is the reciprocal sum \eqref{eq:reciprocal-hypothesis} for
$P$ at $R=\delta^{u_-}$ and $h=\delta^{u_+}$, and $s_{\mathrm{act}}$ is
given by \eqref{eq:s-act}.

For \eqref{eq:N-exponent}, $\#P\le\upsilon_\zeta^{-1}$ is
\eqref{eq:coarse-counts}. The $\rho$-cell of a pin has mass at least
$\upsilon_\zeta$.
It lies in a ball of radius $2\rho$, so
$\upsilon_\zeta\le C\rho^{1+\varkappa}\le C\delta^{\zeta}$ by
\ref{F:pin-mass}, and it is nested in the ancestor of side
$2^{-\lfloor u_+n\rfloor}$, because
$1-a_{\mathrm{par}}-u_+\ge W_0-(\varpi+\varsigma-\mathfrak e)\omega-\xi_{\mathrm{par}}>0$
by \eqref{eq:plateau-consequences} and \eqref{eq:plateau-polytope},
and that ancestor has mass at most $\delta^{u_++W-\tau_0-o(1)}$ by
\ref{F:pin-mass}. For \eqref{eq:H-exponent}, let $z\in P$ and
$\rho\le r\le1$. The ball $B(z,r+\sqrt2\rho)$ is covered by $O(1)$
dyadic squares of side comparable to $r$, and those containing a
point of $P$ are ancestors of pins of $X$.
The $\rho$-cells of the counted points are disjoint, each has mass
at least $\upsilon_\zeta=\delta^\gamma$, and each lies in one of
these occupied ancestor squares, whose masses are at most
$C\delta^{u+T_\theta(u)-\tau_0}$ by \ref{F:pin-mass}. Summing only
their masses, we obtain
$\#(P\cap B(z,r))\le C\delta^{u+T_\theta(u)-\tau_0-\gamma}$ for
$r=\delta^{u}$, which is at most $Cr\delta^{W-\tau_0-\gamma}$ for
$h\le r\le3R$. Summing over the dyadic annuli
$2^{i}h\le|z-z'|<2^{i+1}h$, $0\le i\le n$, and the ball $B(z,h)$,
the reciprocal sum is at most $Cnh\delta^{W-\tau_0-\gamma}$,
which is \eqref{eq:H-exponent}.
$\iota_H\ge0$ by
\eqref{eq:N-exponent}. For \eqref{eq:k-exponent},
\eqref{eq:k-from-kf}, \eqref{eq:fine-richness}, and
$N\upsilon_f\ge cp$ from \ref{F:counts} give
\[
 k\gtrsim_{\mathrm{pol}}\frac{1}{\upsilon_\zeta}
 \max\Bigl\{\frac{\delta^{2q}}{b}, b\delta^{-F+\tau_0+\psi(\mathfrak h)}\Bigr\}
 =\delta^{-\gamma}\max\bigl\{\delta^{\frac{1}{2}+2\omega+\beta},
 \delta^{\frac{1}{2}+\psi(\mathfrak h)-\beta+\tau_0}\bigr\},
\]
since $2q-\lambda=\frac{1}{2}+2\omega$ and $\lambda-F=\frac{1}{2}$.
For $0<\delta<1$, the map $u\mapsto\delta^u$ is decreasing, so
\[
\max\{\delta^{u_1},\delta^{u_2}\}=\delta^{\min\{u_1,u_2\}}.
\]
Thus the maximum in the richness bound corresponds to the minimum in
$\mathfrak r$. Since $\tau_0\le\frac{\beta}{3}$ by
\eqref{eq:tolerance-conditions}, we have
\[
\min\{2\omega+\beta,\psi(\mathfrak h)-\beta+\tau_0\}
\le\min\{2\omega,\psi(\mathfrak h)\}+\beta.
\]
Consequently $k\gtrsim_{\mathrm{pol}}\delta^{-\gamma+\mathfrak r+\beta}$.
Absorbing the polynomial loss into $\delta^\tau$ yields
\eqref{eq:k-exponent}.
Estimate \eqref{eq:zeta-range} is $a_0<c\le a_{\mathrm{par}}$, and
\eqref{eq:lower-count} is \eqref{eq:RW-lower}.

\begin{proposition}[Strict exponent comparison]\label{prop:strict-budget}
The finite configuration is impossible for all sufficiently small
$\delta$ whenever the fixed positive parameters satisfy
\begin{align}
 \frac{3\alpha}{4}(1-2\Delta)-(10-3\alpha)\omega
 &>3\beta+7\tau+\sigma+3\vartheta+\epsilon_{\mathrm{RW}},
 \label{eq:strict-budget}\\
 W_0-6\omega&>\beta+4\tau+2\sigma.
 \label{eq:threshold-budget}
\end{align}
Both hold for every $\omega<r_\alpha(1-2\Delta)$ once
$\sigma,\beta,\tau,\vartheta,\epsilon_{\mathrm{RW}}>0$ are
sufficiently small.
\end{proposition}
\begin{proof}
We begin with the upper count. Split $P$ into the nine classes of
$\rho$-cells with fixed indices modulo three. Each class is
$\rho$-separated, every tube of $\mathcal T$ contains at least
$\frac{k}{9}$ points of one class, and one class does so for at least
a ninth of the tubes. Apply Lemma~\ref{lem:rectangular-sampling} to
that class, with $\rho$, $R=\delta^{u_-}$, $h=\delta^{u_+}$, richness
$\frac{k}{9}$, and the bound $H$, which dominates the reciprocal sum
of the class.
The parameter-separated family of width-$C\rho$ tubes
containing $\mathcal T$ has cardinality $\gtrsim\#\mathcal T$ by
Proposition~\ref{prop:tube-family}. The two exponents of
$\delta^{-1}$ in \eqref{eq:rich-tubes} are
\[
 \gamma+u_-+u_++2\iota_H-3\iota_k=U_1+3\beta+5\tau,\qquad 
 \gamma+u_++\iota_H-2\iota_k=U_2+2\beta+3\tau,
\]
where
\[
 U_1=3\mathfrak r-t-2W,\qquad  U_2=2\mathfrak r-W;
\]
the exponent $\gamma$ cancels from both. Moreover
$U_1-U_2=\mathfrak r-t-W=\min\{2\omega,\psi(\mathfrak h)\}+(\varsigma-\mathfrak e)\omega\ge0$
by \eqref{eq:plateau-parameters} and $a_c+W_0=\frac{1}{2}$. Therefore
\begin{equation}\label{eq:tube-upper}
 \#\mathcal T\lesssim_\sigma\delta^{-U_1-3\beta-5\tau-\sigma}.
\end{equation}
The hypotheses of the lemma have positive power margins. By
\eqref{eq:plateau-parameters}, \eqref{eq:plateau-consequences},
\eqref{eq:plateau-polytope}, \eqref{eq:zeta-range}, and
$\mathfrak r\le\frac{1}{2}+2\min\{\omega,\mathfrak h\}$,
\begin{align*}
 \zeta-u_+&\ge W_0-(\varpi+\varsigma-\mathfrak e)\omega+\mathfrak h
 -\xi_{\mathrm{par}}-\xi\ge W_0-4\omega-2\tau,\\
 \iota_k-\iota_H&=u_++W-\mathfrak r-\beta-2\tau
 \ge W_0-5\omega-\beta-2\tau,\\
 \iota_k-\iota_H-(u_-+u_+-\zeta)&=W-u_-+\zeta-\mathfrak r-\beta-2\tau
 \ge W_0-6\omega-\beta-4\tau,
\end{align*}
using $u_++W\ge t+2W$, $2\varsigma+\chi\le3$, $\varpi+\varsigma\le\frac{7}{2}$,
$2\varsigma+\varpi\le5$, and $\mathfrak h-2\min\{\omega,\mathfrak h\}\ge-\omega$.
The first implies $\frac{\rho}{h}\to0$, the second
$\frac{k}{H}\to\infty$, and the third with
\eqref{eq:threshold-budget} implies
$\frac{kRh}{H\rho^{1-2\sigma}}\to\infty$, which is
\eqref{eq:sampling-threshold} with room for the constants, the
factor $\frac{1}{9}$, and dyadic rounding.

For the lower count, by \eqref{eq:lower-count}, \eqref{eq:s-act},
$\zeta\ge1-a-\xi$, and $a=a_c+\varpi\omega+\xi_{\mathrm{par}}-\mathfrak h$,
\[
 \Lambda\ge\frac{1+3\alpha}{2}+(1-3\alpha)a-\frac{1+3\alpha}{2}\xi
 -3\vartheta-\epsilon_{\mathrm{RW}}
 \ge L-2\tau-3\vartheta-\epsilon_{\mathrm{RW}},
\]
where
\begin{equation}\label{eq:L}
 L=\frac{3}{4}(1+\alpha)+\Bigl(\frac{1}{2}-\frac{3\alpha}{2}\Bigr)\Delta
 -c_\alpha\varpi\omega+c_\alpha\mathfrak h .
\end{equation}

We compare the two exponents. Put $\mathfrak u=\frac{\mathfrak h}{\omega}\ge0$ and
$f_0(\mathfrak u)=\frac{1}{\omega}\min\{2\omega,\psi(\mathfrak h)\}$, so that
$f_0(\mathfrak u)=\min\{2,2\mathfrak u,\mathfrak u+e_0\}$ in branch~1 and $f_0(\mathfrak u)=\min\{2,2\mathfrak u\}$ in
branch~2.
In branch~1 with $a<a_c-\omega$ both expressions equal $2$,
since then $\mathfrak u>2$, as noted after Proposition~\ref{prop:second-pruning}.
By \eqref{eq:plateau-consequences},
\[
 U_1=\frac{3}{4}+\frac{\Delta}{2}+[3f_0(\mathfrak u)+2\varsigma+\chi-2\mathfrak e]\omega,\qquad
 L-U_1=\frac{3\alpha}{4}(1-2\Delta)-C_1(\mathfrak u)\omega,
\]
with
\[
 C_1(\mathfrak u)=3f_0(\mathfrak u)+2\varsigma+\chi-2\mathfrak e+c_\alpha(\varpi-\mathfrak u).
\]
We claim that $C_1(\mathfrak u)\le9-c_\alpha=10-3\alpha$ for all $\mathfrak u\ge0$. The
function $3f_0(\mathfrak u)-c_\alpha \mathfrak u$ increases until $f_0$ reaches $2$ and
decreases afterwards, so its maximum is at the first point $\mathfrak u_*$ with
$f_0(\mathfrak u_*)=2$, and $C_1(\mathfrak u)\le6+2\varsigma+\chi-2\mathfrak e+c_\alpha(\varpi-\mathfrak u_*)$.
In branch~1, $\mathfrak u_*=\max\{1,2-e_0\}$. If $e_0\le1$, then
$\varpi-\mathfrak u_*=\varpi-2+e_0=u_0+\mathfrak e-1$ by \eqref{eq:u0-varpi},
and $2\varsigma+\chi\le3-\frac{u_0}{4}$ yields
\[
 C_1(\mathfrak u)\le9-c_\alpha+\Bigl(c_\alpha-\frac{1}{4}\Bigr)u_0+(c_\alpha-2)\mathfrak e
 \le9-c_\alpha .
\]
If $e_0>1$, then $u_0>1$, $\varpi=2$, $\mathfrak u_*=1$, and
$u_0=e_0+1-\mathfrak e\ge2-\mathfrak e$, so
\[
 C_1(\mathfrak u)\le9-\frac{u_0}{4}-2\mathfrak e+c_\alpha
 \le9-\frac{1}{2}+c_\alpha\le9-c_\alpha ,
\]
because $2c_\alpha<\frac{1}{2}$. In branch~2, $\mathfrak u_*=1$ and
$\varpi-1=\min\{1,u_0\}$, so
\[
 C_1(\mathfrak u)\le9-\frac{u_0}{4}-2\mathfrak e+c_\alpha\min\{1,u_0\}
 \le9-2\mathfrak e\le9-\frac{1}{4}\le9-c_\alpha,
\]
because $c_\alpha\min\{1,u_0\}\le\frac{u_0}{4}$ and
$\mathfrak e\ge\frac{1}{8}$. This proves the claim, hence
\begin{equation}\label{eq:leading-margin}
 L-U_1\ge\frac{3\alpha}{4}(1-2\Delta)-(10-3\alpha)\omega .
\end{equation}
Comparing \eqref{eq:lower-count} with \eqref{eq:tube-upper}, the
exponent of $\delta^{-1}$ in the lower bound exceeds that in the
upper bound by at least the left side of \eqref{eq:strict-budget}
minus its right side. This is impossible as $\delta\to0$. Finally,
if $\omega<r_\alpha(1-2\Delta)$, the left side of
\eqref{eq:strict-budget} is positive, and $\omega<\frac{W_0}{6}$
makes $W_0-6\omega$ positive.
Choosing the loss parameters
successively small proves the last assertion.
\end{proof}

The coefficient $10-3\alpha$ is sharp for these scalar inequalities:
$\varsigma=2$, $\chi=-1$, $u_0=\mathfrak e=0$, $\varpi=1$, and $\mathfrak u=2$
give $C_1(\mathfrak u)=9-c_\alpha$. Whether a configuration with these
parameters exists is a different question, which we do not address.

\subsection{Proof of the theorem}

\begin{proof}[Proof of Theorem~\ref{thm:universal-main}]
We first prove the bound $q_*$. Fix $\frac{1}{3}<\alpha<\alpha_*$
with a row length $s$ such that $f_s(n)\ge cn^{\alpha}$, and
$0<\omega<r_\alpha$.
Put $\Delta=0$, so that $\lambda=1$,
$F=\frac{1}{2}$, $a_c=W_0=\frac{1}{4}$, and $q=\frac{3}{4}+\omega$.
Choose $\sigma,\vartheta,\epsilon_{\mathrm{RW}}>0$ small, then the
fixed finite direction grid of Proposition~\ref{prop:tube-family}
covering all $s_{\mathrm{act}}\in[\alpha,1)$, and let $\eta_*>0$ be
the minimum tolerance of Theorem~\ref{thm:RW} over this grid at $t=1$
and $\epsilon=\epsilon_{\mathrm{RW}}$. Choose $\beta,\tau>0$ satisfying
\eqref{eq:strict-budget}--\eqref{eq:threshold-budget}, then
$\tau_0\le\frac{\tau}{2}$ and $\xi,\xi_{\mathrm{par}},\xi_{\mathrm{cut}}\le\tau$
satisfying \eqref{eq:tolerance-conditions} and
$\alpha\max\{\xi,\xi_{\mathrm{cut}}\}<\min\{\frac{\vartheta}{4},\frac{\eta_*}{4}\}$, the
candidate exponents \eqref{eq:candidate-exponents} with their number
$\mathfrak n$, the net $\Theta$ of Lemma~\ref{lem:joint-template},
the row lengths $s_1$ of Lemma~\ref{lem:flow} and $s_0$ of
Proposition~\ref{prop:homogeneous-rows}. All these choices precede
the set and its measures.

Let $K$ be compact with $\dim_{\mathrm H}K>1$, and suppose that its
bad-pin set at threshold $q$ has dimension greater than one.
Lemma~\ref{lem:seventh-classical-recurrence} supplies separated Frostman
measures $\mu,\nu$ with exponents greater than one, one oracle $A$,
and one positive-measure relation $G$ whose low-distance relations
have masses $P_n\ge n^{-2}$ at arbitrarily large precisions. Every
pair of $G$ has the certificate at level $1$
(Section~\ref{sec:plateau}). Corollary~\ref{cor:finite-inputs}
supplies, at each such precision, a template, mass classes with
$p\ge cn^{-\mathfrak n-4}$, and the finite inputs
\ref{F:counts}--\ref{F:pin-mass}. The scale conditions of
Section~\ref{sec:circle-rows} hold: $\frac{\rho_{\min}}{bn}\to\infty$
and $\frac{\delta^{2q}}{b}\to0$ by \eqref{eq:tolerance-conditions},
and \eqref{eq:coarsening-margin} follows from
\eqref{eq:threshold-budget}. Propositions~\ref{prop:rows-common-richness},
\ref{prop:homogeneous-rows}, and~\ref{prop:second-pruning},
Lemma~\ref{lem:pin-coarsening}, Lemma~\ref{lem:flow}, and
Proposition~\ref{prop:tube-family} produce the pins $P$ and the tubes
$\mathcal T$, and Proposition~\ref{prop:strict-budget} excludes the
configuration for all sufficiently small $\delta$. Since the
recurrent precisions are unbounded, this is a contradiction, and the
bad-pin set of $K$ has dimension at most one. All finite choices are
made after the oracle is fixed.
They need not be computable or
consistent at different precisions.

For analytic $E$ with $\dim_{\mathrm H}E>1$, choose a compact
$K\subset E$ of dimension greater than one. Since
$\Delta_xK\subset\Delta_xE$ for every pin, the bad-pin set of $E$ at
threshold $q$ is contained in that of $K$. Choose
$\alpha_j\nearrow\alpha_*$ in $(\frac{1}{3},\alpha_*)$, each with a fixed
row length, and $0<\omega_j<r_{\alpha_j}$ with
$\omega_j\nearrow r_{\alpha_*}$. The bad-pin set of $E$ at threshold
$q_*$ is the countable union over $j$ of its bad-pin sets at the
thresholds $\frac{3}{4}+\omega_j$, the $j$th treated with $\alpha=\alpha_j$.
Countable stability of Hausdorff dimension implies
$\dim_{\mathrm H}\{x\in\mathbb R^2:\dim_{\mathrm H}\Delta_xE<q_*\}\le1$. Each
threshold uses one fixed row length.
No estimate at zero margin is
used. With $\alpha=\frac{7}{19}$ and $s=5$ the same argument proves
the threshold $\frac{3}{4}+r_\alpha=\frac{132}{169}$, and
$\alpha=\frac{59}{159}$, $s=9$, yields $\frac{368}{471}$.

Now suppose $\dim_{\mathrm H}E>1$. For $0<r<\frac{1}{8}$ and
$0\le\kappa\le\frac{1}{2}$ put
\begin{equation}\label{eq:E-r}
 E_r(\kappa)=\frac{r+(\frac{1}{2}-4r)\kappa}{1-8r\kappa},
\end{equation}
so that $Q(\dim_{\mathrm H}E)=\frac{3}{4}+E_{r}(\dim_{\mathrm H}E-1)$ for $1\le \dim_{\mathrm H}E\le\frac{3}{2}$, with
$r=q_*-\frac{3}{4}=r_{\alpha_*}$. Fix $\alpha$ as above, put
$r=r_\alpha$, fix $0<\kappa<\min\{\dim_{\mathrm H}E-1,\frac{1}{2}\}$, and let
$\varepsilon<E_r(\kappa)$ be close to $E_r(\kappa)$. Put
\[
 q=\frac{3}{4}+\varepsilon,\qquad 
 b=b_{\mathrm{pin}}(\kappa)=\frac{1-2\kappa}{4(1-\kappa)},\qquad 
 \Delta_*=\frac{\kappa(1-4\varepsilon)}{1-2\kappa},\qquad 
 \omega=\varepsilon-\frac{\Delta_*}{2},
\]
so that $\Delta_*$ is the level of Lemma~\ref{lem:bootstrap} with
this $b$, by \eqref{eq:D-pin-sym}. Multiplying by $1-2\kappa$ shows
that $\omega<r_\alpha(1-2\Delta_*)$ is equivalent to
$\varepsilon<E_r(\kappa)$. At $\varepsilon=E_r(\kappa)$, with
$D=1-8r\kappa$,
\begin{gather*}
 \Delta_*=\frac{(1-4r)\kappa}{D}<\frac{1}{2},\qquad 
 \omega=\frac{r(1-2\kappa)}{D}>0,\\
 \varepsilon-\frac{\kappa}{2}=\frac{r(1-2\kappa)^{2}}{D}>0,\qquad 
 \frac{1+2\kappa}{8}-\varepsilon=\frac{(1-2\kappa)(1-8r+8r\kappa)}{8D}>0,
\end{gather*}
and the last two inequalities are the range
\eqref{eq:bootstrap-range} for $b=b_{\mathrm{pin}}(\kappa)$, which
reads $\frac{\kappa}{2}\le\varepsilon<\frac{1+2\kappa}{8}$. So for
$\varepsilon$ close to $E_r(\kappa)$ we have $0\le\Delta_*<\frac{1}{2}$,
$0<\omega<r_\alpha(1-2\Delta_*)$, and the range. Choose the
parameters as above with $\Delta=\Delta_*$ and this $\omega$. Apply
Lemma~\ref{lem:seventh-classical-recurrence} to a compact $K\subset E$
with $\dim_{\mathrm H}K>1+\kappa$, with a source measure
of Frostman exponent $1+\kappa$.
Then every limiting source profile
satisfies $f(u)\ge(1+\kappa)u$. The pins of $G$ have effective
dimension larger than one, so Corollary~\ref{cor:prefix-rate},
applied with $1<p\le s<1+\kappa$ and then $p\searrow1$,
$s\nearrow1+\kappa$, yields the prefix estimate \eqref{eq:general-prefix}
with $b=b_{\mathrm{pin}}(\kappa)$ for every limiting profile, and
Lemma~\ref{lem:bootstrap} supplies the certificate at level
$\lambda_*=1+\Delta_*$ for every limiting profile along low-distance
precisions. The finite argument applies verbatim and proves
$\dim_{\mathrm H}\{x\in\mathbb R^2:\dim_{\mathrm H}\Delta_xE<q\}\le1$. Since
\[
 \frac{\partial E_r}{\partial\kappa}=\frac{(1-4r)^2}{2D^{2}}>0,\qquad 
 \frac{\partial E_r}{\partial r}=\frac{(1-2\kappa)^{2}}{D^{2}}>0,
\]
letting $\varepsilon\nearrow E_r(\kappa)$, then
$\kappa\nearrow\min\{\dim_{\mathrm H}E-1,\frac{1}{2}\}$, and then $\alpha\nearrow\alpha_*$
through countably many thresholds, countable stability proves the
bound $\frac{3}{4}+E_{r_{\alpha_*}}(\dim_{\mathrm H}E-1)=Q(\dim_{\mathrm H}E)$ for $\dim_{\mathrm H}E\le\frac{3}{2}$.
For $\dim_{\mathrm H}E\ge\frac{3}{2}$ the thresholds increase to
$\frac{3}{4}+E_r(\frac{1}{2})=1$. The monotonicity of $E_r$ in
$\kappa$ shows that $Q$ is increasing, and $Q(1)=q_*$. With
$r=\frac{21}{676}$ and $r=\frac{59}{1884}$ the same argument yields
the fixed-row formulas stated after the theorem.
\end{proof}

\begin{remark}[Dependence of the threshold on the measures]
\label{rem:threshold-dependence}
The constants suppressed in $\lesssim_{\mathrm{pol}}$, the angular
density bound $L$, the uniformization thresholds of
Lemma~\ref{lem:joint-template}, and the coding constants of
Lemma~\ref{lem:seventh-classical-recurrence} depend on the fixed
measures $\mu,\nu$ and on the relation $G$: through the Frostman
constants and support diameters, the source--pin separation, and
the fixed constants of the drawing lemmas. The proof uses only that
they are independent of the precision $n$. They are not bounded in
terms of the Frostman exponents or of $\dim_{\mathrm H}E$ alone: the
uniformization thresholds come from Borel--Cantelli arguments
without a rate, and $L$ is chosen so that a restriction retains
positive mass. Consequently the scale below which the finite
configuration is excluded depends on the measures, whereas the
incidence parameters and the threshold are fixed before the set and
the measures are chosen. The source dimension enters only through
the level of the certificate. The theorem is a dimension bound, for
which the existence of arbitrarily large admissible scales is all
that is needed.
No effective threshold in terms of the dimension is
claimed. This dependence already enters the information-to-measure
step, through the almost-sure coding lemma, before the finite
geometric argument is applied.
\end{remark}

\section{Two-sided angular bounds and bisector witnesses}\label{sec:relation-witnesses}

We retain the notation of Sections~\ref{sec:circle-rows} and
\ref{sec:finite-incidence}. The source and pin sets are related by
assigned incidences. Previously we controlled the directions of the
source vertices from each pin. We now also control the directions
of the assigned pins from each source vertex. Near a perpendicular
bisector, this second angular estimate becomes a bound for the
number of witnesses in a small square.

All source and pin locations are in fixed bounded sets at positive distance
from each other. Constants may depend on this separation, on the diameter
bounds, and on the fixed incidence error. Directions are in $S^1$.
Using finitely many angular charts changes constants only.

Let $P$ be a pin set and let $I\subset P\times V$ be the assigned
incidence relation with a common finite source set $V$. Assume, for each $v\in V$ and every angular interval $J$ of length at least $\rho$, that
\begin{equation}\label{rel-reverse}
 \#\{x\in P:(x,v)\in I, \arg(x-v)\in J\}
 \le C_0 N_0(|J|+\rho).
\end{equation}
Here $N_0$ is a cardinality normalization, comparable to $\# P$ up to
the allowed polynomial factors. The rest of this section explains how to supply
\eqref{rel-reverse} in the construction of Section~\ref{sec:discretization}.

For distinct $v,z\in V$ put $\ell=|v-z|$ and
$\rho_\ell=\frac{\delta}{\ell}$. Let $\mathcal W(v,z)$ consist of assigned
common-incidence witnesses satisfying
\[
 (x,v),(x,z)\in I,\qquad 
 \bigl||x-v|-|x-z|\bigr|\le C\delta.
\]
Any subcollection, in particular the witnesses from actual row occurrences,
may be used. It is essential to keep this assigned-incidence condition.
Whenever the grid $\mathcal D_D$ is used in the following count,
$D$ is a dyadic side length.

\begin{theorem}\label{rel-main}
Suppose \eqref{rel-reverse} holds with $\rho\le C\rho_\ell$, and
$\rho_\ell\le c\ell$. For every square $Q$ of side $D\ge\rho_\ell$,
\begin{equation}\label{rel-parentcap}
 \#(\mathcal W(v,z)\cap Q)
 \le C C_0N_0(\ell D+\rho_\ell).
\end{equation}
Consequently, if $\#\mathcal W(v,z)\ge k$, then
\begin{equation}\label{rel-spread}
 \#\{Q\in\mathcal D_D:Q\cap\mathcal W(v,z)\ne\varnothing\}
 \ge \frac{k}{CC_0N_0(\ell D+\rho_\ell)}.
\end{equation}
Without the comparison $\rho\le C\rho_\ell$, the same conclusions
hold with $\ell D+\rho_\ell+\rho$ in place of $\ell D+\rho_\ell$.
\end{theorem}
\begin{proof}
The squared-distance identity places the witnesses in the
$C\rho_\ell$-neighborhood of the perpendicular bisector $L$ of $v,z$.
The distance from $v$ to $L$ is $\frac{\ell}{2}$. A segment of $L$ of length
$O(D)$, at distance bounded above and below from $v$, has radial image
as viewed from $v$ of angular length $O(\ell D)$: the derivative of
its angle with respect to arclength is
\[
 \frac{\operatorname{dist}(v,L)}{|x-v|^2}=O(\ell).
\]
Thickening the segment by $O(\rho_\ell)$ enlarges its radial image by
$O(\rho_\ell)$. Intersecting with the pin region, and splitting into a
bounded number of segments if necessary, shows that the radial image of
$Q\cap L^{(C\rho_\ell)}$ is covered by boundedly many arcs of total
length $O(\ell D+\rho_\ell)$. Apply \eqref{rel-reverse} to these arcs,
enlarging short arcs to the resolution $\rho$. This proves
\eqref{rel-parentcap}.
Summing over occupied $D$-squares proves \eqref{rel-spread}.
\end{proof}

\subsection{Preserving both angular bounds}
The additional condition is available from the common relation, but it
requires changing the deterministic choice of finite pin representatives.

\begin{lemma}[Two-sided angular restriction]\label{rel-two-sided}
Let $\mu,\nu$ be separated compactly supported planar Frostman probability
measures, both of exponent strictly greater than one. Any Borel relation
$G$ of positive $\nu\times\mu$ measure has a positive-measure restriction
$G'$ and a finite $L$ such that
\[
 (\pi_x)_*(\mu|_{G'_x})\le L\mathcal H^1,\qquad
 (\pi_y)_*(\nu|_{(G')^y})\le L\mathcal H^1
\]
for almost every relevant $x,y$, respectively.
\end{lemma}
\begin{proof}
Orponen's radial absolute continuity theorem \cite{O}, applied in both
orientations, establishes absolute continuity for almost every center with
respect to the other measure: the exceptional dimension is strictly less
than one, so the other Frostman measure assigns it mass zero.
Choose measurable versions of the two radial densities. Restrict $G$ to
pairs where both densities evaluated at the pair's direction are at most
$L$. These restrictions increase, as $L\to\infty$, to a product-full
subset of $G$. Choose finite $L$ retaining positive mass. Each restricted
fiber has angular pushforward bounded by its truncated original density.
\end{proof}

The restriction is fixed before the recurrent low-distance precisions
are selected. Every subsequent subrelation inherits these bounds.

\begin{lemma}[Random representatives retain both angular bounds]
\label{rel-random}
Use Lemma~\ref{lem:seventh-measure-finite-selection}, with source
$\delta$-cell masses in $[m,2m)$, fine pin-cell masses in $[v,2v)$,
and relation mass $p\ge n^{-C}$, where $\delta=2^{-n}$.
If the relation has both angular bounds of Lemma \ref{rel-two-sided}, its
finite representatives can be chosen so that all original forward
estimates hold and, simultaneously for every source-cell center $v_Q$,
\begin{equation}\label{rel-reversefine}
 \#\{x\in\mathbb R^2:(x,v_Q)\text{ is assigned}, \arg(x-v_Q)\in J\}
 \lesssim_{\mathrm{pol}} v^{-1}(|J|+\delta)+O(n)
\end{equation}
at every dyadic angular scale. Polynomial factors depend on $n,p^{-1}$
and the fixed measure constants.
\end{lemma}
\begin{proof}
First perform the active cell and radial-bin selection in the proof of Lemma~\ref{lem:seventh-measure-finite-selection}.
An assigned source cell $Q$ has active fiber mass at least
$\tau\mu(Q)$, with $\tau\gtrsim_{\mathrm{pol}}1$, and the surviving active
relation has mass $\gtrsim_{\mathrm{pol}} p$. Choose one independent point
$x_B$ from each relevant fine pin cell $B$, with law
$\frac{\nu|_B}{\nu(B)}$, and discard samples with no active assignments.

For a source cell $Q$ and an angular interval $J$, the indicator that
$x$ assigns $Q$ and $\arg(x-v_Q)\in J$ is bounded by
\[
 \frac{1}{\tau\mu(Q)}\int_Q
  1_{G'}(x,y)1_{\{\arg(x-y)\in J^{(C\delta)}\}}d\mu(y).
\]
This uses the active-mass lower bound and the fixed source--pin separation.
Integrate in $x$ over the pin cells and use $\nu(B)\ge v$.
Fubini and the reverse angular bound give the expectation at most
\[
 \frac{CL(|J|+\delta)}{\tau v}.
\]
The count is a sum of independent indicators. A Chernoff upper tail,
with an additive sufficiently large multiple of $n$, holds simultaneously
for every source cell and every dyadic angular interval: there are only
polynomially many in $\delta^{-1}$. Arbitrary intervals follow by a
bounded dyadic cover at comparable scale.

The original incidence lower bound also survives this
random choice. If $F(x)$ counts active assignments, then
$F(x)\lesssim m^{-1}$, while the expected total is
$\gtrsim_{\mathrm{pol}} \frac{p}{mv}$. Independence implies
\[
 \operatorname{Var}\Big(\sum_B F(x_B)\Big)
 \lesssim m^{-1}\mathbb E\sum_BF(x_B).
\]
The ratio of variance to squared expectation tends to zero, since it is
at most a polynomial factor times $\frac{v}{p}$, and $v$ has a positive-power
upper bound from the pin Frostman estimate. All forward angular and
parent and radial-bin upper bounds hold pointwise for any chosen relevant pin.
Thus a realization has the required incidence lower bound and all reverse
upper bounds. Fixed residue-class selections and subsequent deletions
preserve them, at the permitted polynomial cost.
\end{proof}

\subsection{Coarsening the pins}
At the chosen tube scale $\rho=\rho_j$, for the application in
Section~\ref{sec:relation-upper}, first use
Lemma~\ref{rel-lem:new-uniform-parents}. Every occupied $\rho$-cell
then contains exactly
$L_0=L_j\asymp_{\mathrm{pol}}\frac{\upsilon_\zeta}{v}$ fine pins,
where $\upsilon_\zeta$ is the common ancestor mass. Use exactly
$L_0$ occupied slots in the selection of
Lemma~\ref{lem:pin-coarsening}, so every occupied $\rho$-cell
retains exactly one actual pin. Each fine pin has inclusion
probability $L_0^{-1}$, independently across different coarse cells. For each fixed source
vertex and angular interval, the new count is again a sum of independent
indicators, with expectation bounded by \eqref{rel-reversefine} divided
by $L_0$. Chernoff and the same union bound give
\[
 \#\{x\in P:(x,v_Q)\text{ assigned}, \arg(x-v_Q)\in J\}
 \lesssim_{\mathrm{pol}} \upsilon_\zeta^{-1}(|J|+\rho),\qquad  |J|\ge\rho.
\]
The additive logarithmic errors are absorbed because
$\upsilon_\zeta\lesssim\rho^{1+\varkappa}$ with fixed
$\varkappa>0$. The same random experiment preserves every rich-pair
witness lower bound, by the positive lower-tail margin in Lemma~\ref{lem:pin-coarsening} and a union bound over its polynomially many source pairs.
Apply that lower-tail argument to the assigned row-occurrence
witnesses. Their cardinality has already been established, and each
retained witness still assigns both source vertices. Hence \eqref{rel-reverse} is obtained with
$N_0=\upsilon_\zeta^{-1}$ and $C_0\lesssim_{\mathrm{pol}}1$.
For the pair range in Proposition~\ref{prop:homogeneous-rows} $\ell_j\lesssim\ell\lesssim w$
and common tube width $\rho\asymp\frac{\delta}{\ell_j}$, use the last
variant of Theorem \ref{rel-main}. The resulting uniform bound is
$\#(\mathcal W(v,z)\cap Q_D)\lesssim_{\mathrm{pol}} N_0(wD+\rho)$.

\section{A second upper bound for rich bisectors}\label{sec:relation-upper}
We use the witnesses supplied by the graph: each one records an
assigned occurrence of the source pair on a circle. The bounds of
Section~\ref{sec:relation-witnesses} apply to these witnesses because
both source vertices remain assigned to the pin.

The common finite construction determines the following parameters.
Write $w=\delta^a$, $\rho=\delta^\zeta$, and
$N_0=\upsilon_\zeta^{-1}$, allowing fixed comparison constants.
By Proposition~\ref{prop:homogeneous-rows},
$\zeta=1-a+O(\xi)$. By \eqref{eq:fine-richness} and
\eqref{eq:k-from-kf}, we may use the lower richness bound
\[
k\gtrsim_{\mathrm{pol}}N_0\delta^{\mathfrak r+\beta},\qquad
\mathfrak r=\frac{1}{2}+\min\{2\omega,2\mathfrak h\},\qquad
  \mathfrak h=a_{\mathrm{par}}-a.
\]
This follows from the two-dimensional covering of source parents
in \eqref{eq:thinned-occupancy}, so it holds for both template branches.
The minimum and the tolerance $\beta$ have the same origin as in
\eqref{eq:k-exponent}. The additional branch~1 source bound can
increase the richness.
We use the covering bound common to both branches.
The parameter $\beta$ can be chosen arbitrarily small after the
strict exponent margins are fixed.
Put $\mathfrak u=\frac{\mathfrak h}{\omega}$.
The finite scales in \eqref{eq:finite-scales} give the exact bounds
\[
a=a_c+(\varpi-\mathfrak u)\omega+\xi_{\mathrm{par}},
\qquad a\ge a_c-\omega-\beta-\xi.
\]
Here $0\le\varpi\le2$ and $\mathfrak u\ge0$. In particular,
$(\varpi-\mathfrak u)\omega\ge-\omega-\beta-\xi-\xi_{\mathrm{par}}$.
In the leading exponent calculations below we set these tolerances
to zero, obtaining $\varpi-\mathfrak u\ge-1$. Every expression used
in the comparisons is Lipschitz in these parameters on the fixed
range. The error is therefore $O(\beta+\xi+\xi_{\mathrm{par}})$ and
is chosen smaller than the strict margins. This also covers the
first row scale, which may lie slightly below $a_c-\omega$.
The angular probability at each generating pin satisfies
\eqref{eq:flow-measure}. Therefore \eqref{eq:RW-lower} yields, after
making its fixed tolerances as small as required,
\begin{equation}\label{rel-eq:interface-lower}
\#\mathcal T\gtrsim\delta^{-\Lambda+o(1)},\qquad
\Lambda=\frac{1+3\alpha}{2}+(1-3\alpha)a.
\end{equation}
Here and below the exponent notation permits any prescribed small
fixed loss.
All strict margins are fixed before that loss is chosen.

Put $N_0=\upsilon_\zeta^{-1}$, so that $\#P\asymp_{\rm pol}N_0$.
The reverse angular estimate reads
\[
 \#(\mathcal W(e)\cap Q_D)
 \lesssim_{\rm pol}N_0(wD+\rho),\qquad  D\ge\rho.
\]
We shall use a candidate parent side comparable to
$D=\min\{w,\frac{\rho}{w}\}$.  These finitely many candidate parent scales
are frozen before source thinning and row construction.

\begin{lemma}[Preparing uniform pin parents]\label{rel-lem:new-uniform-parents}
Fix finitely many candidate parent widths $D_j$ and tube widths $\rho_j$,
and fix a positive integer $L_{\mathrm{reg}}$, with
$\epsilon=\frac{1}{L_{\mathrm{reg}}}$.  Before the source thinning and the construction
of rows, the finite pin set can be restricted at inverse-polynomial cost
so that the following hold for every candidate $j$:
\begin{enumerate}
\item distinct occupied $D_j$-parents are at distance at least $D_j$.
\item the numbers of retained fine pins in occupied cells are exactly
constant at every scale in
$\{D_j^{\epsilon},D_j^{2\epsilon},\ldots,D_j,\rho_j\}$.
\item after choosing one actual pin independently and uniformly in each
occupied $\rho_j$-cell, the numbers of selected pins in occupied cells
are exactly constant at the displayed coarser scales.
\end{enumerate}
The finite mass bounds and both angular bounds persist with adjusted
polynomial constants.  The resulting marked incidence count and pin
cardinality satisfy the original lower bounds with adjusted polynomial
losses.  Consequently the row construction and its richness conclusions
can be performed on this already restricted pin set.
\end{lemma}
\begin{proof}
First retain a dyadic class of marked pin degrees carrying at least an
inverse-polynomial fraction of the marked incidences.  The degrees are
then within a factor two, and the class contains an inverse-polynomial
fraction of $\upsilon_f^{-1}$ pins: each degree is at most $\frac{C}{m}$, while
the initial incidence count is at least an inverse-polynomial multiple
of $(m\upsilon_f)^{-1}$.  Any subsequent inverse-polynomial fraction of
this class preserves an inverse-polynomial fraction of its incidences.

At each candidate $D_j$, retain one of the nine residue classes of
parent indices modulo three.  This loses only a fixed factor at each
of finitely many scales and ensures the asserted separation.  Apply the
usual branching regularization to the finite ancestor tree at the
union of the displayed levels.  Working from the finest frozen level towards the root,
pigeonhole the number of occupied children into dyadic ranges and
retain a range carrying an inverse-polynomial fraction of the surviving
leaves.  The already regularized descendant subtrees have equal size,
so this weighted pigeonhole preserves that fraction of fine pins.  There are $O(n)$ possible ranges.  Prune each
surviving parent to exactly the same power-of-two number of children,
and continue recursively.  This loses at most an inverse polynomial
in $n$ and leaves an exactly uniform tree at every frozen level.
In particular every $\rho_j$-cell has the same number $L_j$ of fine
pins, and every coarser frozen parent has the same number of occupied
$\rho_j$-cells.  The latter count is unchanged by selecting the actual
representative within each cell.

Every retained fine pin has original fine-cell mass at least
$\upsilon_f$, and every occupied $\rho_j$-cell has original mass in
$[\upsilon_j,2\upsilon_j)$.  Thus
\[
 L_j\lesssim\frac{\upsilon_j}{\upsilon_f}.
\]
Conversely, there are at most $\upsilon_j^{-1}$ occupied $\rho_j$-cells,
whereas the number of surviving fine pins is
$\gtrsim_{\rm pol}\upsilon_f^{-1}$.  Therefore
$L_j\asymp_{\rm pol}\frac{\upsilon_j}{\upsilon_f}$.
The random representatives consequently have the same inclusion
probability, to polynomial factors, as in the original pin-coarsening
lemma.  Its lower-tail argument preserves all rich-pair witness counts
and an inverse-polynomial fraction of the generating pins.  The
reverse angular upper-tail argument is valid in the same experiment.
Its expectation is the fine angular count divided by $L_j$.
All mass bounds refer to the original measure and survive restriction.
The forward angular and source-parent estimates hold pointwise, and
reverse angular counts can only decrease before the random selection.
This proves all claims.  Each candidate parent side is replaced by a comparable dyadic side
whose logarithmic exponent is divisible by $L_{\mathrm{reg}}$.  The comparison
factor depends only on $L_{\mathrm{reg}}$.
Its logarithmic error is $O(\frac{L_{\mathrm{reg}}}{n})$,
and the displayed $L_{\mathrm{reg}}$-th roots are then dyadic.  This is absorbed
in the already selectable scale tolerances.
\end{proof}

\subsection{A linear dependence on coarse tube occupancy}

\begin{lemma}[Quantitative coarse occupancy]\label{rel-lem:RW-linear-C}
Let $0<\delta<\frac{\Delta}{32}$, $0<\Delta\le1$ be dyadic scales and let $\epsilon>0$
be sufficiently small with $\epsilon^{-1}\in\mathbb N$.
Let $\mathcal P$ be a multiset of dyadic $\delta$-squares in $[0,1)^2$, meaning
that each square may occur with a positive integer multiplicity.
Its total multiplicity is $N$.
All counts involving $\mathcal P$
include these repetitions. Suppose each occupied square of side $\Delta^{j\epsilon}$ contains
the same number of fine squares, counted with multiplicity, for
$1\le j\le\epsilon^{-1}$. Assume these scales are dyadic and that
distinct occupied $\Delta$-squares have distance at least $\Delta$.
Write $\mathcal D_r(\mathcal P)$ for the occupied dyadic $r$-squares.
Here we use dyadic tubes in one bounded-slope chart. A dyadic tube
is the union of the lines $y=ux+v$ whose parameters $(u,v)$ belong
to one dyadic square in a fixed nested parameter grid. For $r\ge\delta$,
$T^r$ is the tube corresponding to the dyadic $r$-ancestor of that
parameter square. If $U$ is a dyadic $r$-tube, $4U$ and $6U$ denote
its $4r$- and $6r$-neighborhoods, respectively. In particular,
$(T^s)^r=T^r$ whenever $\delta\le s\le r$. This ancestor convention
will be needed in the induction.
These are the uniformity and separation conditions in
\cite[Lemma~4.8]{RW}.  Suppose also
that, for all dyadic $r\in[\Delta,1]$ and all $r$-squares $Q$,
\[
 \#(\mathcal P\cap Q)\le K rN,\qquad  K\le\Delta^{-\epsilon}.
\]
Fix $C\ge1$ and $b\ge1$. Let $\mathcal T_{A,b}$ be a family of
distinct dyadic $\delta$-tubes. Suppose that, for each $T\in\mathcal T_{A,b}$,
there are at least $A\ge2$ occupied $\Delta$-parents $Q$ such that
at least $b$ members of $\mathcal P$, counted with multiplicity,
are contained in $Q\cap4T$. Assume
\[
 \#\{Q\in\mathcal D_r(\mathcal P):Q\subset6T^r\}
 \le C r\#\mathcal D_r(\mathcal P)
  \quad \text{for } r=\Delta^\epsilon,\ldots,\Delta,\qquad
   Ab>C\delta^{1-2\epsilon}N.
\]
Then, with an absolute fixed constant $c_0$ (for example $c_0=10$),
\begin{equation}\label{rel-eq:RW-linear-C}
 \#\mathcal T_{A,b}\lesssim_\epsilon
 C\delta^{-c_0\epsilon}\frac{N^2}{A^3b^2}.
\end{equation}
The parameter grid may have a fixed translation. It must be the
same nested grid at every scale. All implicit constants are independent
of $C$.
\end{lemma}
\begin{proof}
We track the constant in~\cite[Lemma~5.1]{RW}, using
\cite[Lemma~5.2 and Proposition~2.6]{RW} unchanged. Their proof is an induction on
the fine width $\delta$, for all multisets satisfying the hypotheses,
with $\Delta,C,K$ frozen. The bound on $K$ is expressed in terms of the frozen scale $\Delta$
and therefore persists through the induction.

The global Frostman hypothesis and Lemma~5.2 give
\[
 \#\mathcal T_{A,b}\lesssim K\log(\frac{2}{\Delta})\frac{N^2}{(Ab)^2}.
\]
Thus if $A\lesssim C\Delta^{-2\epsilon}$, this already proves
\eqref{rel-eq:RW-linear-C}, after absorbing a logarithm. If $\Delta$ is
bounded below in terms of $\epsilon$, the same weak estimate suffices,
as $A\lesssim\Delta^{-1}$.

The first induction base case is $\delta>\Delta^{1+\epsilon}$.
Uniformity and the coarse occupancy bound at $\Delta$ imply
$Ab\lesssim C\Delta N$. This contradicts
$Ab>C\delta^{1-2\epsilon}N$ when $\Delta$ is sufficiently small:
the fixed geometric constants are absorbed by the positive power gap
$\frac{\Delta^{1-\epsilon-2\epsilon^2}}{\Delta}\to\infty$.

In the remaining case $A$ is larger than a sufficiently large constant
times $C\Delta^{-2\epsilon}$. Put
$n_r=\#\mathcal D_r(\mathcal P)$ and choose the smallest frozen scale
$r=\Delta^{j\epsilon}$, equivalently the largest $j$, for which
\[
 A>c_1 C\Delta^{-\epsilon}r n_r.
\]
Such a scale exists by the planar packing bound at $\Delta^\epsilon$,
and it is strictly larger than $\Delta$ by the coarse occupancy
hypothesis. The next level and planar child packing imply
\begin{equation}\label{rel-eq:C-cut}
 c_1C\Delta^{-\epsilon}r n_r<A
 \lesssim C\Delta^{-2\epsilon}r n_r.
\end{equation}
For every finer frozen scale $t<r$, failure of the cut inequality implies
$r n_r<t n_t$.
The constant $C$ cancels. Uniformity therefore
shows that, for $Q\in\mathcal D_r(\mathcal P)$ and an occupied frozen
$t$-square $R\subset Q$,
\[
 \#(\mathcal P\cap R)=\frac{N}{n_t}
 \le\frac{t}{r}\frac{N}{n_r}
 =\frac{t}{r}\#(\mathcal P\cap Q).
\]
Rounding an intermediate dyadic side up to the next frozen side
costs at most $\Delta^{-\epsilon}$. Thus the rescaled configuration
inside $Q$ has local Frostman constant $\Delta^{-\epsilon}$,
independent of $C$, at every scale between $\frac{\Delta}{r}$ and $1$.

A tubelet is the part of a fine tube in an $r$-square. We use the
essentially distinct tubelets $\mathbb U_Q$ and the covering tubes
of width comparable to $\frac{\delta}{r}$ from
\cite[Lemma~5.1, Step~4]{RW}. For a tubelet $u$ based in $Q$, let
$u^+$ be the fixed enlargement in that construction, restricted
to $Q$, and give it weight
\[
w(u)=\#\{R\in\mathcal D_\Delta(\mathcal P):
           R\subset Q,\ \#(\mathcal P\cap R\cap u^+)\ge b\}.
\]
Assigning a $\Delta$-parent only to its containing $r$-square
keeps every counted fine square in the configuration to which the
local Frostman estimate applies. The tubelet cover ensures that
the sum of the weights incident to each original tube is at least
a fixed multiple of $A$. At most a bounded number of tubelets
in any one $Q$ are incident to that tube. Since every such $Q$
lies in $6T^r$, tubelets of weight one contribute at most
a constant times $Crn_r$ per tube. Choose $c_1$ in
\eqref{rel-eq:C-cut} larger than this absolute constant. After
discarding these tubelets, the weighted incidence count is still
bounded below by a fixed multiple of $A$ for each tube. Only
tubelets with $w(u)\ge2$ remain.

For a covering tube $\mathbf T$, write $\mathbb U_{\mathbf T}$
for its associated tubelets and $\mathcal T_{\mathbf T}$ for its
fine tubes. Each fine tube and each tubelet belongs to only
boundedly many such families. Straighten $\mathbf T$ and dilate
its transverse coordinate by a factor comparable to
$\frac{r}{\delta}$. Its tubelets then have both dimensions
comparable to $r$, and the fine tubes have width comparable to
$r$. This is the rescaling in Step~4 of the cited proof. The
high--low estimate is therefore applied at width $r$.
Let $S$ be the least power of two not smaller than
$\Delta^{-\frac{\epsilon}{100}}$. Because
$\Delta\le r\le\Delta^\epsilon$, for sufficiently small $\Delta$
we have
\[
r^{-\frac{\epsilon}{100}}\le S\le r^{-1}.
\]
These are exactly the restrictions on the enlargement parameter
in \cite[Proposition~2.6]{RW} after rescaling. Its enlarged
fine tubes have width comparable to $S\delta$ in the original
coordinates.

Put
\[
W_2=\sum_{Q\in\mathcal D_r(\mathcal P)}
        \sum_{u\in\mathbb U_Q:\,w(u)\ge2}w(u)^2.
\]
Inside each $Q$, apply Lemma~5.2 of \cite{RW} after scaling $Q$
to the unit square. The two relevant widths are
$\frac{\delta}{r}<\frac{\frac{\Delta}{r}}{32}$, the parent separation is
$\frac{\Delta}{r}$, and the local Frostman constant is at most
$\Delta^{-\epsilon}$. Since $Q$ contains exactly $\frac{N}{n_r}$
fine squares, counted with multiplicity, this proves
\begin{equation}\label{rel-eq:tubelet-second-moment}
W_2\lesssim
n_r\Delta^{-\epsilon}\log(\frac{2r}{\Delta})
       (\frac{N}{bn_r})^2
\lesssim_\epsilon
n_r\Delta^{-2\epsilon}(\frac{N}{bn_r})^2.
\end{equation}
The fixed enlargements and the bounded overlap of the tubelet
cover change only absolute constants. In particular, $C$ does
not occur in this estimate.

For a fine tube $T$, define its enlarged weighted count by
\[
F(T)=\sum_{\mathbf T:\,T\in\mathcal T_{\mathbf T}}
       \ \sum_{\substack{u\in\mathbb U_{\mathbf T}\\
                         u\subset6(T^{S\delta})}}w(u).
\]
The rescaled high--low estimate and Cauchy--Schwarz imply, for
every subfamily $\mathcal V\subseteq\mathcal T_{A,b}$,
\begin{equation}\label{rel-eq:subfamily-high-low}
A\#\mathcal V\lesssim_\epsilon
       (\frac{S}{r}\#\mathcal V\,W_2)^{\frac{1}{2}}
       +S^{-1+\epsilon}\sum_{T\in\mathcal V}F(T).
\end{equation}
The same tubelets, weights, and covering tubes may be used for
every $\mathcal V$. The lower bound per tube survives restriction,
and the sum of squared weights can only decrease. This explains
why the estimate is available for subfamilies as well as for the
whole family.

If the first term controls the left side of
\eqref{rel-eq:subfamily-high-low} for at least half the tubes,
\eqref{rel-eq:tubelet-second-moment} implies
\[
 \#\mathcal T_{A,b}
 \lesssim_\epsilon
 \frac{S\Delta^{-2\epsilon}}{r n_r}
 \frac{N^2}{A^2b^2}
 \lesssim_\epsilon
 C\Delta^{-5\epsilon}\frac{N^2}{A^3b^2},
\]
where the last step uses the right side of \eqref{rel-eq:C-cut}.
This is the sole nontrivial occurrence of $C$ in this part of the proof.

To pass to the other case, call $T$ good if
$F(T)\ge cS^{1-\epsilon}A$, where $c>0$ is a sufficiently small
constant depending only on $\epsilon$. Apply
\eqref{rel-eq:subfamily-high-low} to the bad tubes. Their second
term is at most a constant times $cA$ times their number, and is
absorbed into the left side. Thus, if at least half the tubes are
bad, the preceding bound proves the lemma. Otherwise at least
half are good, and we retain them. This argument justifies the
pointwise lower bound needed below.

We describe the parent count explicitly. In the tubelet decomposition
of \cite[Lemma~5.1, Step~4]{RW}, the tubelets associated to one
covering tube have directions within $O(\frac{\delta}{r})$ of
one another. Their fixed enlargements have bounded overlap. Each
fine tube belongs to only boundedly many covering tubes. Thus for
each good $T$ we can choose one covering tube whose contribution
is at least a fixed fraction of $F(T)$.

For a $\Delta$-parent $Q$, let $g_T(Q)$ count the tubelets in this
chosen family that contribute to the enlarged incidence of $T$ and
for which $Q$ is one of the parents counted by $w(u)$. Counting the same
tubelet--parent pairs in two orders gives
\[
 \sum_Qg_T(Q)\gtrsim S^{1-\epsilon}A.
\]
The tubelets are nearly parallel and essentially distinct. A
width-$O(S\delta)$ strip can contain only $O(S)$ of them through
one $\Delta$-parent, since $r\ge\Delta$. Hence $g_T(Q)\lesssim S$.
The overlap bound also holds for the fixed enlargements $u^+$:
within one covering tube, their slopes vary by $O(\frac{\delta}{r})$
and their transverse positions are separated at scale $\delta$,
up to a bounded number of classes. If $M_0$ is the resulting fixed
bound, their incidences in
$Q$ account for at least $\frac{g_T(Q)b}{M_0}$ fine-square copies.
This remains true for a multiset, since each copy is counted at most
$M_0$ times.

Choose a dyadic value $k_T\ge1$ for which the parents with
$k_T\le g_T(Q)<2k_T$ carry at least a fraction
$c(\log S)^{-1}$ of the last displayed sum. There are at least
\[
 c\frac{S^{1-\epsilon}A}{k_T\log S}
\]
such parents. Pigeonhole once more over the tubes to fix one
value $k=k_T$, with $1\le k\lesssim S$, losing a further factor
$O(\log S)$ in the number of tubes. Set $\delta'=8S\delta$ and
replace each retained tube by its dyadic $\delta'$-ancestor.
The incidence containment checked below shows that every such
ancestor is rich in the selected parents, with
\[
 A'\gtrsim\frac{S^{1-\epsilon}A}{k\log S},\qquad
 b'=\max\{1,\lfloor\frac{kb}{M_0}\rfloor\}\asymp kb.
\]
If the expression inside the floor is less than one, a selected
parent still contains at least one fine-square copy, which explains
the maximum with one. Each ancestor has at most $64S^2$ dyadic
fine descendants. After deleting duplicate ancestors we therefore
obtain a family $\mathcal T'_{A',b'}$ with
\[
 \#\mathcal T_{A,b}\lesssim S^2\log S\,
                  \#\mathcal T'_{A',b'}.
\]
The two logarithmic losses have different roles. One enters $A'$,
and the other enters this cardinality bound. All overlap constants
are absolute and independent of $C$.

The fixed dilation $8$ accommodates our neighborhood convention:
the tubelets counted in the low-frequency term lie in
$6(T^{S\delta})$, their contributing fine squares lie within a further
$c_2\delta$ for an absolute constant $c_2$, and replacing a fine
square by its $\delta'$-ancestor adds at most $\sqrt{2}\delta'$.
Since $T^{S\delta}\subset T^{\delta'}$, the total distance from
$T^{\delta'}$ is at most
\[
6S\delta+c_2\delta+\sqrt{2}\delta'
=\bigl(\frac{3}{4}+\frac{c_2}{8S}+\sqrt{2}\bigr)\delta'
<4\delta'
\]
for sufficiently large $S$. Thus these ancestors lie in
$4T^{\delta'}$. Once $\delta'<\Delta$, their dyadic ancestors
also remain in the same $\Delta$-parent as the original fine
squares.
Thus the enlarged incidences satisfy the same $4T$ condition as the
original ones. The dilation changes no exponent or dependence on $C$.

For sufficiently small $\Delta$, we have $\delta'<\frac{\Delta}{32}$,
since $\delta\le\Delta^{1+\epsilon}$. Let $\mathcal P^{\delta'}$
be the multiset obtained by replacing every fine square of
$\mathcal P$ by its $\delta'$-ancestor, retaining multiplicity. The total count $N$,
all frozen parent counts, their uniformity and separation, and the
global Frostman bound are unchanged. For every
frozen $t\ge\Delta$,
\[
 (T^{\delta'})^t=T^t,
 \qquad
 \mathcal D_t(\mathcal P^{\delta'})=\mathcal D_t(\mathcal P).
\]
The coarse occupancy inequality therefore holds with exactly the
same $C$. A comparison of widths alone would not establish this.
The richness condition is also preserved. More explicitly,
\[
\frac{A'b'}{C(\delta')^{1-2\epsilon}N}
\gtrsim_\epsilon
\frac{S^\epsilon}{\log S}
\frac{Ab}{C\delta^{1-2\epsilon}N}>1
\]
for sufficiently small $\Delta$, because $\epsilon$ is fixed and
$S\to\infty$ as $\Delta\to0$. Also, $k\lesssim S$ and the original
small-$A$ exclusion imply
\[
A'\gtrsim_\epsilon\frac{S^{-\epsilon}}{\log S}A
\gtrsim_\epsilon
\frac{C\Delta^{-2\epsilon+\frac{\epsilon^2}{100}}}{\log(\frac{2}{\Delta})}\ge2
\]
after decreasing the same threshold for $\Delta$. Thus all the
inductive hypotheses hold. Applying the induction hypothesis retains the same single
factor $C$. The powers in the recurrence can be checked before absorbing
constants. Substitution of the lower bound for $A'$ and of $b'\asymp kb$
produces
\[
 \frac{S^2\log S}{(A')^3(b')^2}
 \lesssim
 \frac{kS^{-1+3\epsilon}(\log S)^4}{A^3b^2}
 \lesssim\frac{S^{3\epsilon}(\log S)^4}{A^3b^2},
\]
since $k\lesssim S$. The replacement of
$\delta^{-c_0\epsilon}$ by $(8S\delta)^{-c_0\epsilon}$ contributes
$(8S)^{-c_0\epsilon}$. With $c_0=10$, the recurrence coefficient
is at most a constant depending only on $\epsilon$ times
$S^{-7\epsilon}(\log S)^4$. Choose $\Delta$ small enough that this
coefficient is less than $\frac{1}{2}$. This choice is independent
of $C$. Choose the implicit constant in the asserted bound larger
than those in the weak estimate and the high-frequency estimate.
The low-frequency step then uses at most half this constant at
the larger width $\delta'$. Induction on the dyadic fine width
closes, with $\Delta,C,K$ unchanged and with one factor of $C$.
\end{proof}

\subsection{Parents below the row width}

Let $w=\delta^a$, $\rho=\delta^\zeta$, $\zeta=1-a$ up to the selectable
row-scale tolerance. At each pin $x$ generating rows, retain the
direction probability $\sigma_x$ of Lemma~\ref{lem:flow},
with its two angular bounds. Write $T^r(x,\theta)$ for the
width-$r$ tube through $x$ in direction $\theta$, restricted to
the fixed bounded region containing the pins. For any chosen
parent side $D\in[\rho,w]$, averaging
against the global pin Frostman bound yields
\[
 \int\#(P\cap T^{Cr}(x,\theta))d\sigma_x(\theta)
 \lesssim_{\rm pol}N_0\{r+\delta^{-\tau_{\mathrm{dir}}}w^{1-\alpha}r^\alpha\}
 \lesssim_{\rm pol}\delta^{-\tau_{\mathrm{dir}}}N_0r(\frac{w}{D})^{1-\alpha},\qquad  r\ge D.
\]
Indeed the angular interval associated to a point at distance $t$ is
$O(\frac{r}{t})$.
Use the linear direction bound when $\frac{r}{t}\ge w$, the
$\delta^{-\tau_{\mathrm{dir}}}w^{1-\alpha}(\frac{r}{t})^\alpha$ bound otherwise.
Points within distance $r$ contribute at most $CN_0r^{1+\varkappa}$.
For the remaining points, sum over distance annuli with $t>r\ge\rho$.
The truncated energies at exponents $1$ and $\alpha$ are bounded by
the pin Frostman estimate. Uniformity at the finitely many frozen parent scales,
Markov's inequality, and restriction of the directional probabilities
therefore supply the coarse occupancy hypothesis of
Lemma~\ref{rel-lem:RW-linear-C} with
\[
 C\le\delta^{-o(1)-\eta_{\mathrm{occ}}}(\frac{w}{D})^{1-\alpha},
\]
where $\eta_{\mathrm{occ}}>0$ is an arbitrarily small preset tolerance. Choose
$\tau_{\mathrm{dir}}<\frac{\eta_{\mathrm{occ}}}{2}$ and all remaining exponent losses
with sum less than $\frac{\eta_{\mathrm{occ}}}{2}$. This deletes
a fixed fraction of directional mass, not witnesses. Renormalize the
retained directional probabilities and apply
Proposition~\ref{prop:tube-family} to them. Uniformity,
separation, and the forward and reverse angular estimates are prepared
before rows are constructed as above.  The normalized global pin
Frostman constant is at most $\delta^{-\tau}$ times polynomial
factors, with $\tau>0$ selectable.  The candidate exponent
$\frac{\log D}{\log\delta}$ stays bounded away from zero.  Choosing the
input losses sufficiently small therefore ensures $K\le D^{-\epsilon}$,
as required by Lemma~\ref{rel-lem:RW-linear-C}.
This is a bound in terms
of the frozen parent scale, and is not replaced by a bound in terms
of the fine width.

Choose $D=\min\{w,\frac{\rho}{w}\}$, up to the fixed dyadic comparison
factor from the preceding lemma.  At leading exponent order this is
$D=\frac{\rho}{w}$ when $a<\frac{1}{3}$, and $D=w$ otherwise.
At this scale the assigned-witness cap is $O_{\rm pol}(N_0\rho)$.
The exponent paid for $C$ is
$c_C=(1-\alpha)[1-3a]_+$. Consequently put
\[
 B_{\rm eff}=\zeta-c_C
 =\min\{\zeta, \alpha+(2-3\alpha)a\}.
\]
If $B_{\rm eff}>\mathfrak r$ with a fixed positive margin, then the
richness condition of Lemma~\ref{rel-lem:RW-linear-C} holds once
$\beta$ and the other tolerances are chosen smaller than this margin,
as does the requirement of at least two occupied parents.
For each tube choose one original source pair and a subset of exactly
$\lfloor k\rfloor$ of its assigned witnesses, with
$k\gtrsim_{\rm pol}N_0\delta^{\mathfrak r+\beta}$.  Dyadic pigeonholing
by the witness count in a parent supplies $A$ parents with between
$b$ and $2b$ witnesses each, where
\[
 Ab\gtrsim \frac{k}{\log(\frac{1}{\delta})},\qquad  b\lesssim_{\rm pol}N_0\rho.
\]
The condition $B_{\rm eff}>\mathfrak r$ ensures both $A\ge2$ and the
strengthened richness inequality $Ab>C\rho^{1-2\epsilon}\#P$.
Every chosen witness is an actual point of the same ambient set $P$.
Therefore the parent-richness hypothesis is valid even though the
chosen witness subset varies from tube to tube.  For each of the
$O(n^2)$ possible classes the lemma yields
\[
 \#\mathcal T_{A,b}\lesssim_\epsilon
 C\rho^{-c_0\epsilon}\frac{(\#P)^2b}{(Ab)^3}
 \lesssim_{\rm pol}C\rho^{-c_0\epsilon}
                 \frac{N_0^3\rho}{k^3}.
\]
For this application use the dyadic parameter tubes constructed
above, in each fixed parameter grid and bounded-slope chart separately.
Every frozen ancestor and its $6r$-neighborhood lie in a tube of
width $O(r)$ about the original line. The angular estimate was imposed
with this fixed enlargement, so it supplies the occupancy hypothesis
for these ancestors. Any fixed factor is included in $C$ before the
richness inequality is checked.

Choose a comparable larger dyadic fine side if necessary and push
$P$ to its fine cells with multiplicity. The $4T$ convention then
contains the cells of all assigned witnesses. The fine side remains
smaller than $D$, so all frozen parent counts are preserved.
Parameter separation bounds the number of original tubes represented
by each new parameter cell by a fixed constant. Treat these finitely
many classes separately. No witness points are discarded. The two
bounded-slope charts can be obtained by interchanging coordinate axes.
Summing the logarithmic classes and choosing the fixed tolerances
and $\epsilon$ sufficiently small yields the following estimate.
In particular, the factor $k^{-3}$ contributes a loss of $3\beta$.
Choose this and all other exponent losses to have sum less than $\eta$. Then
\begin{equation}\label{rel-eq:refined-common-upper}
 \#\mathcal T\lesssim_\eta
 \delta^{-3\mathfrak r+B_{\rm eff}-\eta}
  \quad \text{for } \eta>0.
\end{equation}
In the case $a<\frac{1}{3}$ this improves the effective parent-cap exponent
$2a$ by $\alpha(1-3a)$. In the other case it retains the indispensable
fine-resolution floor $\zeta$.

\subsection{The exponent comparison}

Write $a=a_c+(\varpi-\mathfrak u)\omega$, where $a_c=\frac{1}{4}+\frac{\Delta}{2}$,
$\varpi-\mathfrak u\ge-1$, and $\mathfrak r=\frac{1}{2}+f_0(\mathfrak u)\omega$ with
$f_0(\mathfrak u)=\min\{2,2\mathfrak u\}$, using the global drawing and elementary
parent-covering richness bound specified above. These are leading
exponents, with the finite-scale errors accounted for as above.
Since $\frac{1}{3}<\alpha<\frac{1}{2}$, comparison with the linear
branch of $B_{\rm eff}$ yields
\begin{align*}
 \Lambda-\{3\mathfrak r-\alpha-(2-3\alpha)a\}
 &= -1+\frac{5}{2}\alpha+(3-6\alpha)a-3f_0(\mathfrak u)\omega\\
 &\ge -\frac{1}{4}+\alpha+(\frac{3}{2}-3\alpha)\Delta
       -(9-6\alpha)\omega.
\end{align*}
Thus this branch is positive if
\begin{equation}\label{rel-eq:refined-main-budget}
\omega<\frac{-1+4\alpha+6(1-2\alpha)\Delta}{12(3-2\alpha)}.
\end{equation}
The inequality used here is
$3f_0(\mathfrak u)-(3-6\alpha)(\varpi-\mathfrak u)\le9-6\alpha$,
which follows from $f_0\le2$ and $\varpi-\mathfrak u\ge-1$.
For the floor branch use
$f_0(\mathfrak u)-\alpha\mathfrak u\le2-\alpha$ and $\varpi\le2$.
Its margin is at least
$\frac{3\alpha}{4}(1-2\Delta)-(6+3\alpha)\omega$, and is positive if
\[
 \omega<\frac{\alpha(1-2\Delta)}{4(2+\alpha)}.
\]
A sufficient uniform condition for $B_{\rm eff}>\mathfrak r$ is
\[
\omega<\frac{\frac{\alpha}{4}+(1-\frac{3\alpha}{2})\Delta}{4-3\alpha}
\]
combined with the floor richness condition $\zeta>\mathfrak r$.
Indeed its leading margin is
$\zeta-\mathfrak r=W_0-(\varpi-\mathfrak u+f_0(\mathfrak u))\omega
\ge W_0-3\omega>0$, using $f_0(\mathfrak u)-\mathfrak u\le1$
and the structural condition $W_0>6\omega$.
The next section verifies these inequalities and the range required
by the certificate iteration.

\section{The dimension-dependent pinned improvement}\label{sec:relation-proof}
Fix $\frac{1}{3}<\alpha<\alpha_*$.  For $0<\kappa<\frac{1}{2}$, the certificate implies
$\Delta=\frac{4\kappa(1-q)}{1-2\kappa}$ and
$\omega=q-\frac{3}{4}-\frac{\Delta}{2}$.  The two exponent inequalities are
respectively equivalent to $q<Q_b$ and $q<Q_f$, where
\begin{align}
 Q_b(\kappa;\alpha)
 &=1-\frac{5(1-\alpha)(1-2\kappa)}{6\{3-2\alpha+2(1-2\alpha)\kappa\}},\label{rel-eq:strong-Qb}\\
 Q_f(\kappa;\alpha)
 &=1-\frac{1-2\kappa}{4+2\alpha-4\alpha\kappa}.
 \label{rel-eq:strong-Qf}
\end{align}
Write $\mathcal Q_\alpha(\kappa)=\min\{Q_b(\kappa;\alpha),Q_f(\kappa;\alpha)\}$.
For comparison, denote the earlier pinned bound with parameter
$\alpha$ by
\[
D_0=10-3\alpha-6\alpha\kappa,\qquad
Q_{\mathrm{old}}=1-\frac{(5-3\alpha)(1-2\kappa)}{2D_0}.
\]
Thus $Q_{\mathrm{old}}$ is $Q(1+\kappa)$ with $\alpha_*$ replaced
by $\alpha$.

\subsection{The required scale inequalities}
Write
\[
 H_b(\Delta)=\frac{\alpha-\frac{1}{4}+(\frac{3}{2}-3\alpha)\Delta}{9-6\alpha},\qquad
 H_g(\Delta)=\frac{\frac{\alpha}{4}+(1-\frac{3\alpha}{2})\Delta}{4-3\alpha}.
\]
Then
\[
 H_g(\Delta)-H_b(\Delta)
 =\frac{(\alpha-1)(-6\Delta+3\alpha-2)}{6(2\alpha-3)(3\alpha-4)}>0
  \quad \text{for } \frac{1}{3}<\alpha<\frac{2}{5}, \Delta\ge0.
\]
Thus the first exponent inequality implies the effective parent-spread and
strengthened richness condition. The floor comparison implies the original
structural restriction, since
\[
 \frac{1}{24}-\frac{\alpha}{4(2+\alpha)}
 =\frac{2-5\alpha}{24(2+\alpha)}>0.
\]

We next check the range required by the certificate iteration.
For $0\le\kappa<\frac{1}{2}$, put
\[
 H=1-2\kappa,\qquad 
 D_b=3-2\alpha+2(1-2\alpha)\kappa,\qquad 
 A=4\alpha-1+6(1-2\alpha)\kappa,\qquad
   J=2+\alpha-2\alpha\kappa.
\]
At $q=Q_b$ the bootstrap parameters are
\[
 \Delta_b=\frac{10(1-\alpha)\kappa}{3D_b},\qquad
   \omega_b=\frac{A}{12D_b}>0,
\]
and
\[
 Q_b-\left(\frac{3}{4}+\frac{\kappa}{2}\right)=\frac{HA}{12D_b}>0,
\]
\[
 \left(\frac{7}{8}+\frac{\kappa}{4}\right)-Q_b
 =\frac{H\{11-14\alpha-6(1-2\alpha)\kappa\}}{24D_b}>0.
\]
For the last inequality, the bracket is at least $8(1-\alpha)>0$.
At $q=Q_f$ one has
\[
 \Delta_f=\frac{2\kappa}{J},\qquad 
 \omega_f=\frac{\alpha H}{4J}>0,
\]
\[
 Q_f-\left(\frac{3}{4}+\frac{\kappa}{2}\right)=\frac{\alpha H^2}{4J}>0,
\]
\[
 \left(\frac{7}{8}+\frac{\kappa}{4}\right)-Q_f
 =\frac{H\{2+\alpha(2\kappa-1)\}}{8J}>0.
\]
The lower certificate comparisons also give $\Delta_b,\Delta_f<\kappa$.
Consequently $\min\{Q_b,Q_f\}$ lies in the required bootstrap interval,
with $0<\Delta<\frac{1}{2}$ for $\kappa>0$. The strict inequalities persist for
$q$ below and sufficiently close to this minimum. The two main geometric
comparisons then become strict by the positive linear coefficient of $1-q$.

\subsection{The useful interval and its branch switch}
Define
\begin{align*}
 \kappa_{\mathrm{rel}}&=\frac{5-8\alpha-3\alpha^2}{6\alpha^2-48\alpha+30},\\
 \kappa_{bf}&=\frac{5\alpha^2-\alpha-1}{10\alpha^2+2\alpha-6},\\
 \kappa_f&=\frac{6+10\alpha-
 \sqrt{(6+10\alpha)^2-48\alpha(1+2\alpha)}}{24\alpha}.
\end{align*}
For $\frac{1}{3}\le\alpha\le\frac{3}{8}$ one has the strict ordering
\begin{equation}\label{rel-eq:strong-order}
 0<\kappa_{\mathrm{rel}}<\frac{3}{20}<\kappa_{bf}<\frac{3}{16}<\kappa_f<\frac{1}{5}.
\end{equation}
Here is an elementary verification.  The numerator controlling
$Q_b-Q_{\mathrm{old}}$ is the increasing linear polynomial
\[
 P_0(k)=(6\alpha^2-48\alpha+30)k+3\alpha^2+8\alpha-5.
\]
Its value at zero is negative and
$P_0(\frac{3}{20})=\frac{39\alpha^2+8\alpha-5}{10}>0$.
The numerator controlling the branch switch, before multiplication
by $1-2k$, is the decreasing linear polynomial
\[
 P_{bf}(k)=(10\alpha^2+2\alpha-6)k-5\alpha^2+\alpha+1.
\]
Its values at $\frac{3}{20}$ and $\frac{3}{16}$ are, respectively,
$-\frac{7\alpha^2}{2}+\frac{13\alpha}{10}+\frac{1}{10}>0$ and
$-\frac{25\alpha^2}{8}+\frac{11\alpha}{8}-\frac{1}{8}<0$.
Finally, the numerator controlling $Q_f-(\frac{2}{3}+k)$ is
\[
 P_f(k)=1+2\alpha-(6+10\alpha)k+12\alpha k^2.
\]
It is strictly decreasing on $[0,\frac{1}{5}]$, and
$P_f(\frac{3}{16})=\frac{35\alpha}{64}-\frac{1}{8}>0$ whereas
$P_f(\frac{1}{5})=\frac{12\alpha}{25}-\frac{1}{5}<0$.
This proves \eqref{rel-eq:strong-order}.

Put $D_b=3-2\alpha+2(1-2\alpha)\kappa$ and
$J=2+\alpha-2\alpha\kappa$.  Direct subtraction yields
\[
 Q_b-Q_f=-\frac{(1-2\kappa)P_{bf}(\kappa)}{6JD_b}.
\]
Thus $Q_b$ is the active branch for $\kappa\le\kappa_{bf}$ and
$Q_f$ for $\kappa\ge\kappa_{bf}$.
The second branch exceeds $\frac{2}{3}+\kappa$ exactly up to $\kappa_f$
in the interval considered.  The numerator of
$Q_b-(\frac{2}{3}+\kappa)$ is
\[
 P_b(\kappa)=1+\alpha-(4+6\alpha)\kappa
               +12(2\alpha-1)\kappa^2,
\]
over the positive denominator $6D_b$.
This polynomial decreases strictly for $\kappa\ge0$.
At $\kappa_{bf}$ the two branches agree and exceed $\frac{2}{3}+\kappa_{bf}$,
by \eqref{rel-eq:strong-order}.  Therefore the first branch also exceeds
that comparison throughout $[\kappa_{\mathrm{rel}},\kappa_{bf}]$.

\subsection{Comparison with the previous bounds}
The exact differences are
\[
 Q_b-Q_{\mathrm{old}}
 =\frac{(1-2\kappa)P_0(\kappa)}{6D_0D_b},\qquad
 Q_f-Q_{\mathrm{old}}
 =\frac{\alpha(2-3\alpha)(1-2\kappa)^2}{2JD_0}>0.
\]
Thus the minimum strictly exceeds the retained-prefix bound
for $\kappa_{\mathrm{rel}}<\kappa<\kappa_f$, and agrees at the lower endpoint.
The bootstrap comparisons above show that both branches also exceed
$\frac{3}{4}+\frac{\kappa}{2}$.
After $1+\kappa\nearrow\dim_{\mathrm H}E$, this becomes $\frac{1}{4}+\frac{\dim_{\mathrm H}E}{2}$.

\subsection{Closing the dimension argument}
Let $\kappa\in[\kappa_{\mathrm{rel}}(\alpha),\kappa_f(\alpha)]$ be a strict
source Frostman excess, and take $q<\mathcal Q_\alpha(\kappa)$ sufficiently
close to this minimum.  The bootstrap range and both geometric comparisons
are then strict.  The main comparison implies the effective parent-spread
and strengthened richness condition.
The floor comparison implies the original
profile restriction.  If the bad pin set at this threshold had dimension
greater than one, the common finite construction, the retained
certificate, and the reverse angular restriction would produce the
configuration above.  Its tube lower exponent strictly exceeds the
refined upper exponent, a contradiction for small recurrent $\delta$.

Now take admissible $\alpha\nearrow\alpha_*$, source exponents
$1+\kappa\nearrow \dim_{\mathrm H}E$, and a countable sequence of thresholds approaching
$Q_{\rm rel}(\dim_{\mathrm H}E)$.  For an analytic source one can first pass to compact
subsets with dimensions approaching $\dim_{\mathrm H}E$.
Bad pins for $E$ are bad for
each such subset.  The lower endpoint $S_{\mathrm{rel}}$ also follows from the
previous retained-prefix estimate, with which the new formula agrees.
The upper endpoint follows from strict source exponents and continuity.
The countable union of threshold exceptional sets has Hausdorff dimension
at most one.  This proves Theorem~\ref{rel-new-pinned-theorem}, including the
stated strict comparisons and endpoint values.

\section{Exceptional sets above the universal threshold}
\label{sec:exceptional}

The preceding proof used the pin measure through the ancestor bounds
of the plateau and through the Frostman bound in
\eqref{eq:coarse-frostman}, where the exponent $1+\varkappa$ was
replaced by $1$. Keeping $\varkappa$ in the Furstenberg estimate
yields explicit bounds for the exceptional set above $q_*$.
Combined
with the certificate level $\lambda_*$ of Section~\ref{sec:bootstrap}
it proves part~(b) of the theorem.

\begin{proposition}[The pin exponent in the comparison]\label{prop:pin-budget}
In the setting of Section~\ref{sec:universal-pinned}, suppose that
the pin measure has Frostman exponent $1+\varkappa$ with
$0\le\varkappa<\frac{5}{2}-6\alpha$, and that
\begin{equation}\label{eq:pin-margin}
 \frac{3\alpha}{4}(1-2\Delta)+\frac{\varkappa}{8}(3-2\Delta)
 -\Bigl(10-3\alpha-\frac{\varkappa}{2}\Bigr)\omega>0 .
\end{equation}
Then the finite configuration is impossible for all sufficiently
small $\delta$, once the loss parameters are small compared with the
left side of \eqref{eq:pin-margin}, with $(1-2\Delta)\frac{40\alpha-11}{96}$,
and with $W_0-6\omega$.
\end{proposition}
\begin{proof}
By \eqref{eq:coarse-frostman}, $P'$ is a
$(\rho,1+\varkappa,\delta^{-o(1)})$-set. Choose the direction grid
before $\xi$ and $\xi_{\mathrm{cut}}$, and let $\eta_*>0$ be the
minimum tolerance of Theorem~\ref{thm:RW} over this grid at
$t=1+\varkappa$ and $\epsilon=\epsilon_{\mathrm{RW}}$. Take
$\alpha\max\{\xi,\xi_{\mathrm{cut}}\}<\min\{\frac{\vartheta}{4},\frac{\eta_*}{4}\}$.
The theorem then applies with $t=1+\varkappa$. Since $s'<1\le t$, the
exponent is $\min\{\frac{s'+t}{2},1\}$, and \eqref{eq:RW-lower}
becomes $\#\mathcal T\gtrsim\delta^{-\Lambda'}$ with
\[
 \Lambda'=\min\Bigl\{\frac{(1+3s_{\mathrm{act}}+\varkappa)\zeta}{2},
 (1+s_{\mathrm{act}})\zeta\Bigr\}-3\vartheta-\epsilon_{\mathrm{RW}}.
\]
Everything else is unchanged, and \eqref{eq:tube-upper} holds. We
show that both entries of the minimum exceed $U_1$ by the stated
margins.

For the first entry, the derivation of \eqref{eq:L}, with the
additional term
$\frac{\varkappa\zeta}{2}\ge\frac{\varkappa}{2}\bigl(\frac{3}{4}-\frac{\Delta}{2}-\varpi\omega+\mathfrak h-\xi_{\mathrm{par}}-\xi\bigr)$,
gives
\begin{gather*}
 \frac{(1+3s_{\mathrm{act}}+\varkappa)\zeta}{2}\ge L'-3\tau,\qquad 
 c=c_\alpha+\frac{\varkappa}{2}<\frac{1}{4},\\
 L'=\frac{3}{4}(1+\alpha)+\Bigl(\frac{1}{2}-\frac{3\alpha}{2}\Bigr)\Delta
 +\frac{\varkappa}{8}(3-2\Delta)-c\varpi\omega+c\mathfrak h,
\end{gather*}
the last inequality being $\varkappa<\frac{5}{2}-6\alpha$. The proof
of \eqref{eq:leading-margin} used only $0<c_\alpha<\frac{1}{4}$, so
with $c$ in place of $c_\alpha$ it yields
\[
 L'-U_1\ge\frac{3\alpha}{4}(1-2\Delta)+\frac{\varkappa}{8}(3-2\Delta)-(9-c)\omega,
\]
which is the left side of \eqref{eq:pin-margin}.

For the second entry, $\zeta\ge1-a-\xi$ and
$a=a_c+\varpi\omega+\xi_{\mathrm{par}}-\mathfrak h$ give
\[
 (1+s_{\mathrm{act}})\zeta=(1+\alpha)\zeta+(1-\alpha)a
 \ge1+\frac{\alpha}{2}-\alpha\Delta-2\alpha\varpi\omega+2\alpha\mathfrak h-3\tau .
\]
With $\mathfrak u=\frac{\mathfrak h}{\omega}$ as in
Section~\ref{sec:universal-pinned}, we have
$f_0(\mathfrak u)\le\min\{2,2\mathfrak u\}$ in both branches,
$3f_0(\mathfrak u)-2\alpha\mathfrak u\le6-2\alpha$, and \eqref{eq:plateau-polytope}
implies
$6+2\varsigma+\chi-2\mathfrak e+2\alpha(\varpi-1)\le9-\frac{u_0}{4}+2\alpha\min\{1,u_0\}\le\frac{35}{4}+2\alpha$.
Hence
\[
 (1+s_{\mathrm{act}})\zeta-U_1
 \ge\Bigl(\frac{1}{4}+\frac{\alpha}{2}\Bigr)(1-2\Delta)-\Bigl(\frac{35}{4}+2\alpha\Bigr)\omega-3\tau
 \ge(1-2\Delta)\frac{40\alpha-11}{96}-3\tau,
\]
using $\omega<\frac{1-2\Delta}{24}$.
The leading term is positive
because $\alpha>\frac{1}{3}$. The conditions \eqref{eq:threshold-budget}
are as before. Choosing $\beta,\tau,\sigma,\vartheta,\epsilon_{\mathrm{RW}}$
small compared with the three margins, the argument of
Proposition~\ref{prop:strict-budget} leads to a contradiction.
\end{proof}

\begin{proof}[Proof of Theorem~\ref{thm:exceptional-quantitative}]
Fix $\frac{1}{3}<\alpha<\alpha_*$ with its row length.

For part (a), put $\Delta=0$ and let
\[
 q_\alpha=\frac{3}{4}+r_\alpha=\frac{15-3\alpha}{20-6\alpha}
 \le q<\frac{19}{24},\qquad  \omega=q-\frac{3}{4}<\frac{1}{24}.
\]
Define
\begin{equation}\label{eq:kappa-alpha}
 \varkappa_\alpha(q)=20-6\alpha-\frac{15-3\alpha}{q}
 =\frac{(80-24\alpha)\omega-6\alpha}{3+4\omega}.
\end{equation}
It is nonnegative for $q\ge q_\alpha$, increasing in $q$, and at most
$\frac{20-42\alpha}{19}$, which is less than $\frac{5}{2}-6\alpha$
because $\alpha<\frac{55}{144}$. Let $K$ be compact with
$\dim_{\mathrm H}K>1$, and suppose that its bad-pin set $B_q$ has
dimension greater than $1+\varkappa_\alpha(q)$. Fix $\varkappa$ with
\[
 \varkappa_\alpha(q)<\varkappa<\min\Bigl\{\frac{5}{2}-6\alpha,
 \dim_{\mathrm H}B_q-1\Bigr\},
\]
and apply Lemma~\ref{lem:seventh-classical-recurrence} with a pin
measure of Frostman exponent $1+\varkappa$. With $\Delta=0$, the
left side of \eqref{eq:pin-margin} is
$\frac{3\alpha}{4}+\frac{3\varkappa}{8}-(10-3\alpha-\frac{\varkappa}{2})\omega$,
which is positive exactly when $\varkappa>\varkappa_\alpha(q)$, and $W_0-6\omega=\frac{1}{4}-6\omega>0$. Proposition~\ref{prop:pin-budget}
and the proof of Theorem~\ref{thm:universal-main} give a
contradiction. Hence $\dim_{\mathrm H}B_q\le1+\varkappa_\alpha(q)$,
which is the asserted bound for $q<\frac{19}{24}$. The bad-pin set
at $\frac{19}{24}$ is the union of the bad-pin sets at thresholds
$q_j\nearrow\frac{19}{24}$, so countable stability and the
continuity of the bound in $q$ give the endpoint. The passage to
analytic $E$ is as before. Finally, the right side of (a) decreases
as $\alpha$ increases, its derivative in $\alpha$ being
$-6+\frac{3}{q}<0$.
Letting $\alpha\nearrow\alpha_*$ through fixed
row lengths proves the bound with $\alpha_*$, and $q_\alpha\nearrow q_*$.

For part (b), let $0<k\le\frac{1}{6}$ and put
\begin{equation}\label{eq:D-B-k}
 \mathcal D=\frac{k(4-k)}{1-3k+k^{2}},\qquad 
 B=1-b_{\mathrm{sym}}(k)(1-k)=\frac{(1+k)(3-k)}{4-k},\qquad 
 c=c_\alpha+\frac{k}{2}<\frac{1}{8}+\frac{1}{12}<\frac{1}{4},
\end{equation}
and, for $q=\frac{3}{4}+\varepsilon$, $\Delta=\mathcal D(1-q)$,
$\omega=\varepsilon-\frac{\Delta}{2}$: these are the level and the
margin supplied by Lemma~\ref{lem:bootstrap} with $b=b_{\mathrm{sym}}(k)$
and \eqref{eq:D-pin-sym}. With the pin exponent $\varkappa=k$, the
left side of \eqref{eq:pin-margin} is affine in $q$, with slope
$-\frac{1}{2}\bigl(\mathcal D(10-6\alpha-k)+20-6\alpha-k\bigr)<0$, and it
vanishes at
\begin{equation}\label{eq:P-alpha}
 P_\alpha(1+k)
 =\frac{\frac{15-3\alpha}{2}+\bigl(5-3\alpha-\frac{k}{2}\bigr)\mathcal D}{10-3\alpha-\frac{k}{2}+\bigl(5-3\alpha-\frac{k}{2}\bigr)\mathcal D}
 =\frac{k^{3}+(1+3\alpha)k^{2}-(5+15\alpha)k+15-3\alpha}{9k^{2}-(21+6\alpha)k+20-6\alpha};
\end{equation}
so the margin is positive exactly when $q<P_\alpha(1+k)$. Both
denominators are positive for $0\le k\le\frac{1}{6}$, and $P_\alpha(t)$
is the function $P(t)$ of the theorem with $\alpha$ in place of
$\alpha_*$.

If $P_\alpha(1+k)\le B$, the conclusion follows from
Theorem~\ref{thm:pinnedmain} with $p=1+k<\dim_{\mathrm H}E$, because
$B(1+k,\dim_{\mathrm H}E)\ge B(1+k,1+k)=B$, using that
$B(p,t)$ is nondecreasing in its source parameter $t$. So let $P_\alpha(1+k)>B$ and let
$B<q<P_\alpha(1+k)$ be close to $P_\alpha(1+k)$. We check the remaining
hypotheses at $q=P_\alpha(1+k)$.
They are strict, so they persist.
The range \eqref{eq:bootstrap-range} for $b=b_{\mathrm{sym}}(k)$ is
$B\le q<\frac{1+B}{2}$, and
\[
 \frac{1+B}{2}-P_\alpha(1+k)
 =\frac{(1-3k+k^{2})\bigl(20+3k-7k^{2}-18\alpha+12\alpha k\bigr)}{2(4-k)\bigl(20-6\alpha-(21+6\alpha)k+9k^{2}\bigr)}>0 .
\]
Since $q>B$ and $\mathcal D(1-B)=k$, we have $0<\Delta<k\le\frac{1}{6}$.
The margin $\omega$ increases with $q$ and vanishes at
$q_0=\frac{3+2\mathcal D}{4+2\mathcal D}$, where $\Delta=\frac{\mathcal D}{4+2\mathcal D}<\frac{1}{2}$
and the left side of \eqref{eq:pin-margin} equals
$\frac{3\alpha}{4}(1-2\Delta)+\frac{k}{8}(3-2\Delta)>0$.
Hence
$q_0<P_\alpha(1+k)$ and $\omega>0$. Finally,
$\omega<\frac{1-2\Delta}{24}$ is equivalent to
$q<\frac{19+10\mathcal D}{24+10\mathcal D}$, and
\begin{gather*}
 \frac{19+10\mathcal D}{24+10\mathcal D}-P_\alpha(1+k)
 =\frac{(1-3k+k^{2})G_\alpha(k)}{2(12-16k+7k^{2})\bigl(20-6\alpha-(21+6\alpha)k+9k^{2}\bigr)},\\
 G_\alpha(k)=20-79k+57k^{2}-14k^{3}-42\alpha(1-3k+k^{2}),
\end{gather*}
where $G_\alpha$ decreases in $\alpha$ and in $k$ on
$[0,\frac{1}{6}]$ for $\alpha\le\frac{3}{8}$, since
$\partial_\alpha G_\alpha=-42(1-3k+k^{2})<0$ and
$\partial_kG_\alpha=-79+114k-42k^{2}+42\alpha(3-2k)<0$ there.
So
$G_\alpha(k)\ge G_{\frac{3}{8}}(\frac{1}{6})=\frac{17}{432}>0$. Hence
$\omega<\frac{1-2\Delta}{24}$, so that Section~\ref{sec:plateau}
and all the conditions of Section~\ref{sec:universal-pinned} apply, and
$\varkappa=k<\frac{5}{2}-6\alpha$.

Let $K$ be compact with $\dim_{\mathrm H}K>1+k$, and let
$q<P_\alpha(1+k)$, $q>B$, be so close to $P_\alpha(1+k)$ that the
checks above hold. Apply Lemma~\ref{lem:seventh-classical-recurrence}
with a source measure of Frostman exponent $1+k$.
Then every
limiting source profile satisfies $f(u)\ge(1+k)u$. If the bad-pin
set $B_q$ of $K$ had dimension greater than $1+k$, the lemma would
also supply a pin measure of Frostman exponent $1+k$, and every pin of
$G$ would have $\dim^Ax\ge1+k$.
Corollary~\ref{cor:prefix-rate}, applied to the pairs of $G$ with
$1<p\le s<1+k$ and then $p,s\nearrow1+k$, yields the prefix estimate
\eqref{eq:general-prefix} with $b=b_{\mathrm{sym}}(k)$ for every
limiting profile.
Lemma~\ref{lem:bootstrap} supplies the certificate at
level $\lambda=1+\Delta$ along low-distance precisions, and Proposition~\ref{prop:pin-budget} with the argument of
Theorem~\ref{thm:universal-main} would give a contradiction. Hence
$\dim_{\mathrm H}B_q\le1+k$.
By inclusion this holds for every
$q<P_\alpha(1+k)$, and by countable stability at $q=P_\alpha(1+k)$.
The first expression in \eqref{eq:P-alpha} increases with
$\mathcal D$ and with $k$, its partial derivatives having the positive
numerators $(5-3\alpha-k)(10-6\alpha-k)$ and $15-3\alpha+(5-3\alpha)\mathcal D$, and $P_\alpha(1+k)$ increases with $\alpha$, the numerator of its
$\alpha$-derivative being $3(1-3k+k^{2})(10-29k+11k^{2})>0$. Letting
$\alpha\nearrow\alpha_*$ and using countable stability once more
proves the bound with $P(1+k)$. For analytic $E$ with
$\dim_{\mathrm H}E>1+k$, choose a compact $K\subset E$ with
$\dim_{\mathrm H}K>1+k$.
Since $\Delta_xK\subset\Delta_xE$, the bad-pin
set of $E$ is contained in that of $K$.

For the last assertion of the theorem, $P$ is increasing on
$[1,\frac{7}{6}]$, because $\mathcal D$ increases with $k$ and the
first expression in \eqref{eq:P-alpha} increases with $\mathcal D$ and
with $k$.
So for $1<\dim_{\mathrm H}E\le\frac{7}{6}$ and $q<P(\dim_{\mathrm H}E)$ there is
$k<\dim_{\mathrm H}E-1$ with $q<P(1+k)$, and the set of $x\in E$ with
$\dim_{\mathrm H}\Delta_xE<q$ has dimension at most $1+k<\dim_{\mathrm H}E$. Its
complement in $E$ therefore has dimension $\dim_{\mathrm H}E$.
\end{proof}

The comparison $P(\dim_{\mathrm H}E)>Q(\dim_{\mathrm H}E)$ for $1<\dim_{\mathrm H}E\le\frac{7}{6}$ follows from the
same monotonicity: writing
$Q_\alpha(1+k)=\frac{3}{4}+E_{r_\alpha}(k)$ for the earlier pinned
bound with parameter $\alpha$, this is the
first expression in \eqref{eq:P-alpha} with $\frac{4k}{1-2k}$ in
place of $\mathcal D$ and $0$ in place of the pin exponent $k$, and
$\mathcal D-\frac{4k}{1-2k}=\frac{k^{2}(3-2k)}{(1-2k)(1-3k+k^{2})}>0$.

\section{Uniform finite-precision estimates and selection}
\label{sec:new-finite}

The lower-box proof uses a new pair at each terminal precision.
We therefore need uniform estimates for such changing pairs.
This section supplies that uniformity, together with the
selection and covering arguments; an asymptotic statement for
one fixed pair would not suffice.

\begin{lemma}[Limits of finite generic pairs]\label{lem:new-finite}
Fix an oracle $A$, $1<\sigma\le2$, $C>0$, and a bounded region.
Suppose $n_j\to\infty$ and pairs $x_j,y_j$ in that region have
$|x_j-y_j|\ge n_j^{-C}$. For each $n=n_j$, write $x=x_j,y=y_j$,
$e=\frac{y-x}{|y-x|}$, and suppose for every $1\le m\le2n$ that
\begin{align}
 K_m^A(x),K_m^A(y)&\ge\sigma m-C\log_2(n+2),\label{newfin:eq:F1}\\
 K_m^{A,x}(e),K_m^{A,y}(e)&\ge m-C\log_2(n+2),\label{newfin:eq:F2}\\
 K_m^{A,x}(y)&\ge K_m^A(y)-C\log_2(n+2),\label{newfin:eq:F3a}\\
 K_m^{A,y}(x)&\ge K_m^A(x)-C\log_2(n+2).\label{newfin:eq:F3b}
\end{align}
Every joint subsequential limit of the normalized endpoint and
conditioned distance profiles satisfies \eqref{eq:new-system}
in both orientations, with $f(u),g(u)\ge\sigma u$.
The conditional-parent transfer of
Lemma~\ref{lem:conditional-parent-transfer} also holds in the
finite approximation, with arbitrarily small fixed linear error.
\end{lemma}
\subsection{Uniform estimates for changing pairs}
\begin{proof}[Proof of Lemma~\ref{lem:new-finite}]
We first explain precisely why the retained information estimates survive
when $(x,y)$ varies with $n$. Set
\[
 \ell_n=\lceil(\log_2(n+2))^2\rceil,\qquad
 k_n=\lceil n^{\frac{1}{4}}\rceil,\qquad h_n=\lceil\sqrt n\rceil,
\]
and fix $1<d<\min\{\sigma,\frac{5}{4}\}$. For all large $n$,
\begin{equation}\label{newfin:eq:floor}
 K_m^A(x),K_m^A(y)\ge dm\qquad(\ell_n\le m\le2n).
\end{equation}
Thus the point-dependent initial cutoff supplied by effective dimension
is replaced by an explicit logarithmic-squared cutoff. On a local
projection piece with right endpoint $b\ge k_n$, a deficiency
$C\log(n+2)$ is at most $5C\log(b+2)$ for all large $n$.

The local projection estimates of Stull~\cite[Lemmas 14--15]{S} are
finite statements. Their tolerances, a logarithmic directional
deficiency, and a bound on the norm of the point control the threshold.
The remark following Lemma~14 states this uniformity explicitly, and
the proof of Lemma~15 has the same dependence. Their oracle simulations
are uniform. In particular, no threshold depending on an unrecorded
property of the selected direction is used here.

For the closed projection estimate, the proof of Lemma~\ref{lem:projection} retains
\begin{equation}\label{newfin:eq:projection}
 K_R^A(x\mid p_vx,v)
 \le\max\{K_R^A(x)-R,\frac{K_R^A(x)-u}{2}\}+o(n)
\end{equation}
when $R$ is a fixed positive fraction of $n$ and $\frac{u}{R}$ is bounded below.
In fact, for fixed $c_0,\tau,\xi>0$, the error is at most $\xi R$
uniformly for $c_0h_n\le R\le n$ and $\tau R\le u\le R$, once $n$
is large enough. Its partition has a bounded number of pieces once
$\tau$ is fixed. Discard projection pieces below $k_n$, at cost
$O(k_n)=o(R)$ uniformly in this range. Every remaining piece has a
common directional-deficiency constant, and $k_n\gg\ell_n$ ensures
that \eqref{newfin:eq:floor} supplies the lower dimension cutoff throughout
the remaining argument. Sum the bounded number of local estimates
with a fixed sufficiently small tolerance. This proves the required
uniform version of \eqref{newfin:eq:projection}, including the smaller
projection precisions used inside the prefix partition.

The circle geometry has a second possible dependence on the pair. If
$|x-z|=|x-y|$ and $v=\frac{y-z}{|y-z|}$, then
\begin{equation}\label{newfin:eq:circle}
 \left|e\cdot v\right|=\frac{|y-z|}{2|x-y|}.
\end{equation}
Since $|x-y|\ge n^{-C}$, approximating $v$ by a perpendicular to $e$
loses at most $O_C(\log n)$ bits of precision. The normalization of
$y-z$ and the perpendicular-bisector formula have the same uniform
logarithmic costs on bounded sets. The tangent transition also widens
by only $O_C(\log n)$. The distance-transfer search and the
complexity-lowering oracle in Fiedler--Stull~\cite[Lemmas 14--15]{FS}
are uniform finite algorithms; their parameter descriptions cost
$O_C(\log n)$ here. Supplying a bounded number of additional precision
indices does not introduce a point-dependent constant.

These observations account for the pair-dependent quantities, namely the
initial lower-complexity cutoff, the directional deficiency, the
separation of the endpoints, and the finite parameter descriptions.
We next carry out the limit argument using these bounds.

\subsection{Limits of a sequence of changing pairs}
Consider a sequence $n_j\to\infty$ and pairs $(x_j,y_j)$ satisfying the hypotheses of Lemma~\ref{lem:new-finite}. Write $\rho_j=|x_j-y_j|$.
Take a common uniformly convergent subsequence of the normalized
complexity functions
\[
 \begin{array}{ll}
 f_j(u)=n_j^{-1}K_{\lfloor un_j\rfloor}^A(y_j),&
 g_j(u)=n_j^{-1}K_{\lfloor un_j\rfloor}^A(x_j),\\[2pt]
 d_{x,j}(u)=n_j^{-1}K_{\lfloor un_j\rfloor}^{A,x_j}(\rho_j),&
 d_{y,j}(u)=n_j^{-1}K_{\lfloor un_j\rfloor}^{A,y_j}(\rho_j).
 \end{array}
\]
Set their values at zero to zero and interpolate between integer
precisions. Standard precision-increment estimates are uniform for
bounded points and for all oracles, so this extraction does not require
the points to be fixed. The limits satisfy
\begin{equation}\label{newfin:eq:profiles}
 \sigma u\le f(u),g(u)\le2u,
\end{equation}
with $f,g$ nondecreasing and $2$-Lipschitz, and $d_x,d_y$ nondecreasing
and $1$-Lipschitz, vanishing at zero. Every limiting upper bound on the terminal distance complexities is inherited by $d_x(1)$ and $d_y(1)$.

\subsection{The transfers remain valid}
Recall the closed transfer quantities
\[
 p_g(u;b)=\min\{b-u,\frac{g(b-u)+u}{2}\},\qquad
 H_{f,g}(a,b)=\min\{f(b),\inf_{a\le u\le b}(f(u)+p_g(u;b))\}.
\]
Fix $0<a<b\le1$ and a cap $c<H_{f,g}(a,b)$ with $c>f(a)$.
For a circle competitor at separation $|y_j-z|=2^{-un_j}$, the
uniform collision estimate, after the source complexity is capped
at $cn_j$, reads
\begin{equation}\label{newfin:eq:capped}
 K_{bn_j}^{A,G_j}(z)
 \ge n_j\min\{c,f(u)\}+n_jp_g(u;b)-o(n_j).
\end{equation}
Integer subscripts are rounded. For $u<\frac{b}{2}$ this follows from
\eqref{newfin:eq:projection} and the bisector identity. For $u\ge \frac{b}{2}$, the
tangent estimate has $p_g(u;b)=b-u$; the logarithmic transition is
absorbed by the cap's fixed strict margin.

The lowering oracle preserves information conditional on the source.
Consequently \eqref{newfin:eq:F3a}--\eqref{newfin:eq:F3b} preserve the positive pin
term, while the negative projection residual can only decrease. This
uses the original directional randomness, without requiring new
randomness relative to $G_j$. The strict gap between $c$ and $H$
verifies the competitor hypothesis of the finite distance-transfer
lemma. After the chain rule, removal of $G_j$, and letting the local
tolerance tend to zero, we obtain
\begin{equation}\label{newfin:eq:closed-transfer}
 d_x(b)-d_x(a)\ge[H_{f,g}(a,b)-f(a)]_+.
\end{equation}
The same reasoning applies with the endpoints exchanged.

We also need the full partition transfer. Put $h_g(u)=g(u)-u$. In a
yellow interval, $h_g$ has its minimum at the left endpoint; in a teal
interval it has its minimum at the right endpoint. Define
\[
 \mathcal L_u(g;R)=\inf_{\mathcal P}\sum_{[v,w]\in\mathcal P\,\mathrm{yellow}}
       (h_g(w)-h_g(v)),\qquad P_g(u;R)=g(R)-\mathcal L_u(g;R),
\]
where the infimum is over finite labelled yellow--teal partitions of
$[0,R]$ with piece lengths at most $u$. Set
\[
 \widehat p_g(u;b)=\max\{p_g(u;b),P_g(u;b-u)\},\quad
 \widehat H_{f,g}(a,b)
 =\min\{f(b),\inf_{a\le u\le b}(f(u)+\widehat p_g(u;b))\}.
\]
For a cap below $\widehat H$, Lemma~\ref{lem:uniform-partition-witnesses} supplies finitely many templates, a
bounded piece count, and a common strict margin for all
$u\in[a,\frac{b}{2}]$. The local Stull estimates apply to their approximate
yellow and teal conditions. Discard $[0,\kappa n_j]$, charge its full
complexity, and encode the rounded endpoints of each template in
$O(\log n_j)$ bits. Every remaining local estimate is at a
macroscopic precision and has uniform constants. First let $j\to\infty$
and then $\kappa\downarrow0$. Substituting these gains into
\eqref{newfin:eq:capped} proves
\begin{equation}\label{newfin:eq:full-transfer}
 d_x(b)-d_x(a)\ge[\widehat H_{f,g}(a,b)-f(a)]_+,
\end{equation}
and its reverse orientation. In particular, the number of witnesses
does not increase with $j$ or with the competitor.

\subsection{The prefixes remain valid}
Appendix~\ref{app:retained-prefix} proves its prefix bound using yellow blocks
and selected teal blocks. We check how its errors depend on the terminal precision. Use the fixed $d$ chosen above. For a selected teal block
$[a,b]$ with $b\gg h_n$, its root identities and \eqref{newfin:eq:floor} give
\[
 \frac{a}{b}\ge\frac{d}{5-d}-o(1)>0,\qquad a\le \frac{b}{2},
\]
and, for $a\le u\le b$,
\[
 F(u)\ge F(b)-\frac{d}{2} b+\frac{d-1}{2}u,
 \qquad F(u)\ge F(b)-(b-u),
\]
where $F(m)=K_m^A(y)$, with interpolation. The closed collision gain
is at least
\[
 \min\{b-u,\frac{d}{2}b-\frac{d-1}{2}u\}-o(b).
\]
Each branch, added to $F(u)$, is at least $F(b)-o(b)$.
The lowering and transfer lemmas therefore pass the full teal source
increment to the distance. The source increment per unit length is at
least
\[
 c_d=\min\{1,\frac{d(3-d)}{5-2d}\}\ge\frac{2}{3}.
\]
On a yellow block, the distance increment is at least its length,
up to the local error, while the source increment is at most twice
its length. The at-most-doubling yellow subpartitions use the same
uniform finite local estimates.

Every selected teal block at least halves the right endpoint; two
outer yellow blocks are not consecutive. Consequently the sum of
right endpoints is $O(n)$ and there are $O(\log n)$ relevant blocks.
Stop when the right endpoint is a fixed multiple of $h_n$, so the
discarded range costs $O(h_n)$. For fixed local tolerance $\xi$, the
total error is
\begin{equation}\label{newfin:eq:prefix-error}
 O_d(\xi n)+O_{C,d,\xi}((\log n)^2)+O_d(\sqrt n).
\end{equation}
The precision ratios ensure that all needed lower bounds remain above
the discarded cutoff. No estimate dividing by $2-D$ at $D=2$ is used.

If $L_n$ is the total yellow length, the two sums are therefore
\[
 K_n^{A,x}(|x-y|)\ge\frac{2n+L_n}{3}-o(n),\qquad
 K_n^{A,x}(|x-y|)\ge K_n^A(y)-L_n-o(n).
\]
Weighting them by $\frac{3}{4}$ and $\frac{1}{4}$, and applying the same argument at
each fixed positive proportion of $n_j$, yields
\begin{equation}\label{newfin:eq:prefix-limit}
 d_x(u)\ge \frac{u}{2}+\frac{f(u)}{4},\qquad
 d_y(u)\ge \frac{u}{2}+\frac{g(u)}{4}.
\end{equation}
In \eqref{newfin:eq:prefix-error}, $n\to\infty$ precedes $\xi\downarrow0$.

The same strict-cap argument with a source parent supplied as
finite side information proves the last assertion. In
Lemma~\ref{lem:conditional-parent-transfer}, the parent precision
is a fixed positive fraction of $n$. Describing the parent precision costs $O(\log n)$; the parent
itself is supplied as side information, and the projection, tangent, and complexity-lowering
steps just considered have the same uniform bounds. Thus its error
is at most $\varepsilon n$ for all sufficiently large $n$, with
$\varepsilon>0$ fixed before taking the limit.
This proves Lemma~\ref{lem:new-finite}.
\end{proof}

\section{Selecting generic pairs at a fixed precision}\label{newsel:sec:selection}
After a fixed similarity, assume $E\subset[0,1]^2$. Fix an oracle $A$
which computes the dyadic masses of $\mu$ and finite rational
Hausdorff approximations to $E$. It is fixed before any bad set or
pair is selected.

\begin{lemma}\label{newsel:lem:selection}
Assume the Frostman upper bound $\mu(B(z,r))\le Cr^s$, $s>1$, and fix $M>0$ and $1<\sigma<s$.
There is $C<\infty$ such that for all sufficiently large $n$ and
every Borel set $B\subset E$ with $\mu(B)>n^{-M}$, there exist
$x,y\in B$ satisfying \eqref{newfin:eq:F1}--\eqref{newfin:eq:F3b} for $m\le2n$
and $|x-y|\ge n^{-C}$.
\end{lemma}

\subsection{Quantitative description tails}
Let $Q_m(z)$ be the dyadic square of side $2^{-m}$ containing $z$.
The Kraft--Chaitin code for the computable masses satisfies, uniformly in
$z$,
\begin{equation}\label{newsel:eq:uppercode}
 K_m^A(z)\le-\log_2\mu(Q_m(z))+c_A\log_2(m+2).
\end{equation}
For any oracle $W$, Kraft's inequality implies
\[
 \mu\{z:K^W(Q_m(z))<-\log_2\mu(Q_m(z))-h\}\le2^{-h}.
\]
Point and square descriptions differ by $O(\log(m+2))$, uniformly
on the bounded support. Indeed, a precision-$m$ approximation
identifies a bounded list of squares; the appropriate member costs a
bounded index. In the other direction, a bounded descendant index
selects a center within $2^{-m}$ of the point.

Apply this estimate with $W=(A,x)$ for each fixed $x$, integrate in
$x$, and then do the same in the other orientation. The Frostman
upper bound for $\mu(Q_m(z))$ and \eqref{newsel:eq:uppercode} show that,
outside a subset of $\mu\times\mu$ measure at most
\begin{equation}\label{newsel:eq:description-tail}
 O(n^{1-D}),
\end{equation}
both independence inequalities and both lower-complexity floors hold
for all $m\le2n$, with a deficiency at most
$C_D\log(n+2)$. Here $D>0$ is arbitrary and $C_D$ is fixed once $D$
is chosen. The estimates are uniform in the real oracle being
conditioned on; no intersection over all pins of unspecified
almost-everywhere sets is taken.

\subsection{A radial estimate uniform over separated restrictions}
Choose $1<a<s$ with $a<2$, choose $2-a<b<1$, and choose
\[
 1<p<\min\{2-b,\frac{b}{2-a}\},\qquad \tau=1-\frac{1}{p}>0.
\]
All the corresponding energies of $\mu$ are finite. The quantitative
estimate in Orponen's proof~\cite[Section 3, equation (3.6)]{O} has
the following consequence. If $\nu,\eta\le\mu$ have disjoint compact
supports, then
\begin{equation}\label{newsel:eq:radial}
 \int\|\pi_x\eta\|_{L^p(S^1)}^p\,d\nu(x)\le C_\mu,
 \qquad \pi_x(y)=\frac{y-x}{|y-x|}.
\end{equation}
The constant is uniform over these restrictions. To see the dependence,
use the weighted source $|x-y|^{-1}\,d\eta(y)$ in Orponen's estimate;
its bound uses only fixed powers of the $a$-energy of $\eta$ and the
$b$-energy of $\nu$. Both are bounded by those of $\mu$. On the fixed
bounded support, the unweighted radial measure is dominated by a
constant times this weighted one. The support separation is used for
the approximation argument, and its reciprocal does not enter this
energy bound. Taking strict inequalities in the choice of $p$ avoids
the endpoint in the angular energy estimate.

For every fixed pin $x$, the directions whose precision-$m$
complexity relative to $(A,x)$ is below $m-D\log_2(n+2)$ lie in at
most $2^m(n+2)^{-D}$ arcs of length $O(2^{-m})$. The union over
$m\le2n$ has arclength at most $O(n^{1-D})$. By H\"older and
\eqref{newsel:eq:radial}, the measure of pairs in $\operatorname{supp}\nu\times\operatorname{supp}\eta$
failing this directional condition is at most
\begin{equation}\label{newsel:eq:angular-tail}
 O_\mu(n^{-(D-1)\tau}).
\end{equation}
This remains true for pin-dependent unions of arcs. Interchanging
$\nu$ and $\eta$ gives the other orientation. Reversing a unit vector
costs only $O(1)$ bits.

\subsection{Two separated portions of the bad set}
Choose a fixed integer $J$ with $Js>M+1$. Partition the unit square
into a grid of side length $h=\frac{1}{\lceil n^J\rceil}$. There are
$O(n^{2J})$ cells. One cell $Q$ has
\[
 \mu(B\cap Q)\gtrsim n^{-(M+2J)}.
\]
The bounded collection of cells within distance $2h$ of $Q$ is
contained in a ball of radius $O(h)$, so its mass is
$O(n^{-Js})<\frac{\mu(B)}{2}$ for large $n$. Among the remaining cells
there is an $R$ with
\[
 \mu(B\cap R)\gtrsim n^{-(M+2J)},\qquad \operatorname{dist}(Q,R)\ge h.
\]
Consequently
\begin{equation}\label{newsel:eq:product-mass}
 (\mu\times\mu)((B\cap Q)\times(B\cap R))
 \gtrsim n^{-(2M+4J)}.
\end{equation}

Use \eqref{newsel:eq:radial} with $\nu=\mu|_Q$ and $\eta=\mu|_R$.
Their supports are contained in separated closed cells; the cells'
boundaries are null because $s>1$. Choose $D$ so large that
\[
 D-1>2M+4J+2,\qquad (D-1)\tau>2M+4J+2.
\]
The sum of the description and angular exceptional measures is then
smaller than \eqref{newsel:eq:product-mass}. A remaining pair belongs to
$B\times B$ and has all the required finite-precision properties.
Since $h\gtrsim n^{-J}$, increasing the fixed constant $C$ also gives
$|x-y|\ge n^{-C}$. This proves Lemma~\ref{newsel:lem:selection}.

\begin{remark}
Neither $B$ nor its restricted measures are added to the oracle.
The mass code always uses the original $\mu$. This avoids a possible
circularity between choosing a bad set and selecting incompressible
points from it. Shmerkin's lower box argument~\cite[Section 6.3]{Sh}
uses the same broad strategy of selecting a pair from a set of mass
at least a negative power of the precision. The grid and Kraft
estimates above give the version needed for the present complexity
hypotheses.
\end{remark}

\subsection{Coding a small pinned cover}\label{newsel:sec:coding}
\begin{lemma}\label{newsel:lem:cover}
For the fixed oracle $A$ above, there is a constant $c$ such that
for every $x\in E$, $r\in\Delta_xE$, and $n\ge1$,
\begin{equation}\label{newsel:eq:covercode}
 K_n^{A,x}(r)\le\log_2N_{2^{-n}}(\Delta_xE)+c\log_2(n+2).
\end{equation}
\end{lemma}
\begin{proof}
Using $A$ and the oracle $x$, compute a finite rational approximation
to $\Delta_xE$ with Hausdorff error at most $\frac{2^{-n}}{100}$. Round its
members to a grid of mesh $\frac{2^{-n}}{4}$ and remove repetitions. The
resulting list is a $2^{-n}$-net for $\Delta_xE$. It has at most a fixed
constant times $N_{2^{-n}}(\Delta_xE)$ members, since an interval of radius
$2^{-n}$, enlarged by the approximation error, meets only a bounded
number of the grid cells. The index of a suitable list member,
together with $n$, has length at most the right side of
\eqref{newsel:eq:covercode}. The algorithm is the same for every pin.
\end{proof}

The cover bound is an upper bound for the description of every
distance at that pin. That is why two bad pins force the common
distance to be short in both orientations.

\section{Matching endpoints at a fixed precision}\label{sec:finite-matching}
Fix an oracle $A$ encoding the original Frostman measure and the
compact set, as in Section~\ref{newsel:sec:selection}. Given $n$, let
\[
 m_n=\lfloor\sqrt{\log_2(n+2)}\rfloor
\]
(increase small values to one). For each point $z$, record the
$m_n$ values
\[
 \left\lfloor\frac{m_n}{n}
 K^A_{\lfloor j\frac{n}{m_n}\rfloor}(z)\right\rfloor,
 \qquad j=1,\ldots,m_n.
\]
Each value has $O(m_n)$ possibilities on the fixed bounded support.
Thus the number of possible records is
\begin{equation}\label{newmatch:eq:records}
 (O(m_n))^{m_n}=2^{O(\sqrt{\log n}\log\log n)}=n^{o(1)}.
\end{equation}

If two points have the same record, the uniform precision-increment
bound gives
\begin{equation}\label{newmatch:eq:matching}
 \sup_{1\le r\le n}|K_r^A(x)-K_r^A(y)|
 \le O(\frac{n}{m_n})+O(\log n)=o(n).
\end{equation}
Indeed, the discrepancy at a recorded precision is at most $\frac{n}{m_n}$;
between records, increasing precision costs at most twice the gap,
plus a logarithmic error. The initial interval is handled by
$0\le K_r^A(z)\le2r+O(\log n)$.

Let $B\subset E$ be Borel with $\mu(B)>n^{-M}$. Partition $B$ by these
records. For large $n$, one class has mass at least
$n^{-(M+1)}$. Apply Lemma~\ref{newsel:lem:selection} inside that class. We obtain $x_n,y_n\in B$, separated by
at least $n^{-C}$, such that for $r\le2n$ both points have the
prescribed dimension floor, both directional deficiencies are
$O(\log n)$, and both full-oracle independence deficiencies are
$O(\log n)$. In addition, \eqref{newmatch:eq:matching} holds.

No record or selected class is put into the oracle. The selection
lemma works inside an arbitrary Borel set of the stated mass.
Thus all its constants remain uniform after replacing $M$ by
$M+1$. Along every convergent subsequence of the selected
normalized profiles, \eqref{newmatch:eq:matching} forces $f=g$.

This finite-scale matching does not assert that there is one fixed
pair with matching profiles at every scale. The endpoints may
change with $n$, which is permitted in the lower-box proof.

\subsection{The lower-box bound and its exceptional sets}
\begin{proof}[The dimension-dependent part of Theorem~\ref{thm:new-box}]
Fix $q<L(s)$ and $M>0$. Choose
$1<\sigma<\min\{s,\frac{9}{8}\}$ with $q<\frac{3\sigma+1}{5}$.
Suppose that along infinitely many $n$ the set
\[
 B_n=\{x\in E:N_{2^{-n}}(\Delta_xE)<2^{qn}\}
\]
has mass greater than $n^{-M}$. Select generic matching pairs from
$B_n$ by the preceding construction, using lower complexity floor
$\sigma$. Lemma~\ref{lem:new-finite} supplies limiting
profiles satisfying the full scalar system. They have $f=g$ and
$f(u)\ge\sigma u$.

Lemma~\ref{newsel:lem:cover} implies
\[
 K_n^{A,x_n}(|x_n-y_n|),\ K_n^{A,y_n}(|x_n-y_n|)
 \le qn+O(\log n).
\]
Both terminal distance profiles are at most $q$.
Proposition~\ref{newmatch:prop:matched} instead makes each of them at least
$\frac{3\sigma+1}{5}>q$. This contradiction proves \eqref{eq:new-tail}.
Take $M=2$, apply Borel--Cantelli, and then let $q$ increase to
$L(s)$ through a countable sequence to obtain the bound $L(s)$ in \eqref{eq:new-box}. The universal part is proved in Section~\ref{sec:new-universal}.
\end{proof}

\begin{proof}[The spectral assertions in Corollary~\ref{cor:new-exceptional}]
First assume $E$ is compact. The map $x\mapsto\Delta_xE$ is
continuous for the Hausdorff metric; the covering conditions above
and hence the lower-box sublevel sets are Borel.

For $\frac{4}{5}\le q<\frac{7}{8}$, put $p=\frac{5q-1}{3}$.
If the Borel set $F=\{x\in E:\underline{\dim}_{\mathrm B}\Delta_xE\le q\}$ had Hausdorff
dimension greater than $p$, choose an $r$-Frostman probability
measure carried by $F$, with $p<r<\dim_{\mathrm H} F$. Since $L(r)>q$,
Theorem~\ref{thm:new-box} is a contradiction.
At $q=\frac{7}{8}$, the same reasoning applies to the set with
$\underline{\dim}_{\mathrm B}\Delta_xE<\frac{7}{8}$, using any exponent $r>\frac{9}{8}$ below its alleged
dimension. This proves the endpoint bound.

For a bounded set $E$, apply these results to its compact closure.
The identity $\Delta_x\overline E=\overline{\Delta_xE}$ and invariance of
lower box dimension under closure prove the last assertion.
\end{proof}

The proof does not give an exceptional-set bound of one for
$\underline{\dim}_{\mathrm B}\Delta_xE<L(\dim_{\mathrm H} E)$. The Frostman measure on a putative
exceptional set can have exponent smaller than $\dim_{\mathrm H} E$.
The statements of Corollary~\ref{cor:new-exceptional} give the precise
conclusions with the correct quantifiers.

\section{The geometric contradiction near the endpoint}\label{sec:new-geometry}

The rigid plateau has a concrete geometric meaning. The source
parent at exponent $\frac{3}{10}$ admits very little mass in any
one annulus, while the pin measure is spread over the scales
between $\frac{1}{5}$ and $\frac{3}{5}$. We show that these two
facts are incompatible with only $\delta^{-\frac{4}{5}}$ radii at each pin.
The finite statement below records the hypotheses separately
from the one-sided plateau parameters used earlier.

\begin{proposition}[Finite configurations near the rigid plateau]
\label{prop:endpoint-finite-geometry}
There is an absolute $\varepsilon_*>0$ with the following property.
Fix $C>0$ and a pin Frostman exponent $1+\varkappa>1$.
Let $0<\varepsilon<\varepsilon_*$ and $\delta=2^{-n}$.
Suppose a common source set, a set of actual pins, and assigned
radial incidences satisfy \ref{F:counts}--\ref{F:pin-mass}, with
the following replacements for their parameters.
\[
q=\frac{4}{5}+\varepsilon,\qquad
b=\delta^{1-\varepsilon},\qquad
\eta\asymp\delta^{\frac{3}{10}},\qquad F=\frac{1}{2}.
\]
The source-parent mass bound is $\delta^{\frac{1}{2}-\varepsilon}$.
No branch-specific source bound is assumed. For the pin ancestors,
replace $T_\theta$ by
\[
T(u)=\max\left\{0,\frac{1}{5}
-\operatorname{dist}\left(u,\left[\frac{1}{5},\frac{3}{5}\right]\right)\right\}
\]
and replace $\tau_0$ by $\varepsilon$.
Use a fixed mesh of candidate pair-separation exponents from
$\frac{1}{5}-3\varepsilon$ to $\frac{3}{10}$, with consecutive
gaps at most $\varepsilon$, and corresponding pin ancestor mass
classes. Assume $p\ge n^{-C}$.
Polynomial factors in $n$ are allowed in the angular bounds and
in the fixed geometric constants, and the source and pin supports
may have separation at least $n^{-C}$.
Then no such configuration exists for all sufficiently large $n$.
The threshold may depend on the fixed measures, $C$, $\varkappa$
and $\varepsilon$.
\end{proposition}

\begin{proof}
We first check that the common thinning and the global circle
drawing remain applicable under these hypotheses.
Retain each source vertex with probability a fixed small multiple
of $p\frac{m}{b}$, as in Lemma~\ref{lem:source-thinning}. This is at most
one because $m\le b$. The expected occupancy of a source parent
in a radial bin is polynomially bounded, so a Chernoff estimate
and a union bound make every such occupancy polynomially bounded.
The total source count is at most $b^{-1}$ up to polynomial
factors, and the retained incidence count is at least
$(b\upsilon_f)^{-1}$ up to polynomial factors.
For angular intervals of length at least
$\rho_{\min}=\delta^{\frac{4}{5}+4\varepsilon}$, the additive
logarithmic sampling error is absorbed by $\frac{|J|}{b}$, since
\[
\frac{\rho_{\min}}{b}=\delta^{-\frac{1}{5}+5\varepsilon}
\]
tends to infinity by a fixed power. The source-parent occupancy
estimate, if needed, has the positive margin
$1-\varepsilon-(\frac{1}{2}-\varepsilon)=\frac{1}{2}$.
Thus the thinning uses the same source vertices at every pin.

The proof of Proposition~\ref{prop:rows-common-richness} now
applies to this thinned set. Its two vanishing ratios are
\[
b\delta^{-q}=\delta^{\frac{1}{5}-2\varepsilon},\qquad
\frac{\delta^{2q}}{b}=\delta^{\frac{3}{5}+3\varepsilon}.
\]
These are the only numerical uses of the earlier plateau
parameters in the global drawing argument. It produces rows
whose pairs have at least
\[
k_1\gtrsim_{\mathrm{pol}}
\frac{\delta^{2q}}{b\upsilon_f}
\]
witnessing pins. The argument of
Proposition~\ref{prop:homogeneous-rows} selects rows at a common
pair-separation scale. Writing that scale as $\delta^a$ and the
tube width as $\rho=\delta^\zeta$, we have
\[
\frac{1}{5}-O(\varepsilon)\le a\le\frac{3}{10},\qquad
\zeta=1-a+O(\varepsilon).
\]
Only the global richness is used; no second pruning or
branch-specific source estimate is required.

Let $\upsilon_\zeta=\delta^\gamma$ be the common pin ancestor
mass. The sampling proof of Lemma~\ref{lem:pin-coarsening}
retains one pin per $\rho$-cell and preserves every witness set.
Indeed, its expected witness count is at least
\[
\frac{k_1\upsilon_f}{\upsilon_\zeta}
\gtrsim_{\mathrm{pol}}
\delta^{-\gamma+\frac{3}{5}+3\varepsilon}.
\]
The Frostman bound implies $\gamma\ge\zeta-o(1)$, and
$\zeta\ge\frac{7}{10}-O(\varepsilon)$, leaving a strict
margin of $\frac{1}{10}-O(\varepsilon)$. Chernoff's estimate
therefore applies simultaneously to all source pairs.
The resulting witness set $P$ and generating subset $P'$ satisfy
\begin{equation}\label{eq:endpoint-pin-counts}
\#P\le\delta^{-\gamma},\qquad
\#P'\gtrsim_{\mathrm{pol}}\delta^{-\gamma},\qquad
k_{\mathrm{rich}}\ge
\delta^{-\gamma+\frac{3}{5}+O(\varepsilon)+o(1)}.
\end{equation}
The normalized ball counts of $P'$ have exponent one, with
polynomial constants, by the original Frostman bound.
The terminal mass $\upsilon_f$ has canceled.

The direction selection of Lemma~\ref{lem:flow} uses only the
row count, its angular upper bounds and its pair-separation
range. Those properties hold here. Use the Katz--Tardos
five-entry exponent $\alpha=\frac{7}{19}$. With the fixed
direction-selection tolerances included in $O(\varepsilon)$,
the direction exponent satisfies
\[
m=\sigma\zeta=\alpha+(1-2\alpha)a+O(\varepsilon).
\]
Choose the exponent in the Ren--Wang theorem slightly below
$\sigma$, from a fixed finite grid, exactly as in
Proposition~\ref{prop:tube-family}. Since $0<\sigma<1$,
the pin exponent one and the middle branch of that theorem imply
\begin{equation}\label{eq:endpoint-tube-lower}
\#\mathcal T\ge
\delta^{-\Lambda+O(\varepsilon)+o(1)},\qquad
\Lambda=\frac{1}{2}+\frac{3\alpha}{2}+(1-3\alpha)a.
\end{equation}
Each tube retains the witnesses counted in
\eqref{eq:endpoint-pin-counts}. The finite grid makes the
Ren--Wang tolerances uniform in the possible direction exponents.

For the upper count take $R=\delta^{\frac{1}{5}}$ and
$h=\delta^{\frac{3}{5}}$. The pin ancestor bound implies
\[
\#(P\cap B(z,r))\le
\delta^{-\gamma+\frac{1}{5}-O(\varepsilon)-o(1)}r
\quad(h\le r\le3R).
\]
Summing over dyadic annuli and the ball of radius $h$ proves
the reciprocal-occupancy estimate of
Lemma~\ref{lem:rectangular-sampling} with
\[
H\le\delta^{-\gamma+\frac{4}{5}-O(\varepsilon)-o(1)}.
\]
The diagonal is covered as well, since the ancestor at exponent
$\frac{3}{5}$ has mass at most $\delta^{\frac{4}{5}-\varepsilon}$
and contains the occupied $\rho$-cell.
At zero error the three strict sampling conditions have margins
\[
\frac{3}{5}+\frac{1}{5}-\frac{3}{5}=\frac{1}{5},\qquad
\zeta+\frac{1}{5}-\frac{1}{5}-\frac{3}{5}\ge\frac{1}{10},\qquad
\zeta-\frac{3}{5}\ge\frac{1}{10}.
\]
They remain strict when all tolerances are sufficiently small.
Substitution in \eqref{eq:rich-tubes} yields
\begin{equation}\label{newgeo:eq:upper}
\#\mathcal T\le\delta^{-U-O(\varepsilon)-o(1)},\qquad
U=\max\left\{3\frac{3}{5}-\frac{2}{5}-\frac{2}{5},
2\frac{3}{5}-\frac{1}{5}\right\}=1.
\end{equation}

Because $1-3\alpha<0$, the least lower exponent occurs at
$a=\frac{3}{10}$. Consequently
\begin{equation}\label{newgeo:eq:gap}
\Lambda-U\ge\frac{3\alpha}{5}-\frac{1}{5}
=\frac{2}{95}>0.
\end{equation}
Fix all exponent tolerances so that their total cost is less than
half this gap. Then take $n$ large enough to absorb the
polynomial factors. The lower and upper tube counts contradict
each other.

If the support separation is only $n^{-C}$, angular rounding
and the circle-to-bisector identities introduce fixed powers
of $n$. Enlarging the fine scale by such a factor changes its
precision by $O(\log n)$. The mass caps and exponent inequalities
above change by $o(1)$ only. This proves the stated uniformity.
\end{proof}

\subsection{Realizing the finite hypotheses on a common relation}

Fix a Frostman measure, an oracle encoding it and its compact
support, a lower dimension exponent greater than one, and a
polynomial relation-mass loss $n^{-M}$.
The description tails in Section~\ref{newsel:sec:selection}
can be imposed on a relation of this size. Choose their
deficiency constant large enough that the discarded product
mass is $o(n^{-M})$. Simultaneously at both endpoints and for
all $m\le2n$, the retained pairs satisfy
\[
K_m^A(z)=-\log_2\mu(Q_m(z))+O(\log n).
\]
The finite generic conditions of Lemma~\ref{lem:new-finite}
are imposed in the same way. Two fixed restrictions of the
original measure may be used for pins and sources.

Let $F_x$ be the density of the radial pushforward of the
source restriction. The radial $L^p$ estimate
\eqref{newsel:eq:radial} implies
\[
\int\int_{\{F_x>n^D\}}F_x(\theta)\,d\theta\,d\nu(x)
\le n^{-D(p-1)}\int\|F_x\|_p^p\,d\nu(x).
\]
For sufficiently large fixed $D$, removing these pairs costs
$o(n^{-M})$. Every retained fiber has angular mass at most
$n^D|I|$ on an arc $I$, and this property survives all further
restrictions.

Record both endpoint profiles and both distance profiles on
the mesh of \eqref{newmatch:eq:records}, with the same rounded
heights. There are $n^{o(1)}$ joint records, so one class retains
polynomial pair mass. Its profiles differ from their
representatives by $o(1)$ uniformly on $[0,1]$.
The two endpoint records need not agree.
Bin the finitely many source, parent and pin-ancestor masses
needed at the candidate scales. These restrictions cost further
polynomial factors. The oracle remains the original one.

We check the annular cap carefully. At parent exponent
$a_{\mathrm{par}}=\frac{3}{10}$ let
$C=Q_{\lfloor a_{\mathrm{par}}n\rfloor}(y)$.
The stable tail certificate and the conditional transfer in
Lemma~\ref{lem:new-finite} imply
\[
K_n^{A,x}(|x-y|\mid C)
\ge(1-\varepsilon)n-K_{\lfloor a_{\mathrm{par}}n\rfloor}^A(y).
\]
For an occupied radial bin $J$, the conditional mass code from
Proposition~\ref{prop:conditional-parent-cap} bounds the same
quantity from above by
\[
-\log_2\frac{\mu(C\cap\{z:|x-z|\in J\})}{\mu(C)}
+O(\log n),
\]
allowing a bounded enlargement of $J$.
The base mass code at the parent cancels $\mu(C)$. Thus the
original source measure satisfies
\[
\mu(C\cap\{z:|x-z|\in J\})\le\delta^{1-O(\varepsilon)}.
\]
The source floor at the parent is
$f(\frac{3}{10})\ge\frac{1}{2}-O(\varepsilon)$.
The stable plateau, its Lipschitz extension, and the base floor
produce the pin-ancestor bound with the function $T$ in
Proposition~\ref{prop:endpoint-finite-geometry}.
Lemma~\ref{lem:seventh-measure-finite-selection} therefore
supplies one common source set and actual pins with assigned
radial labels satisfying that proposition, after slightly
enlarging the fixed error tolerance.
The source law has not been normalized separately at different pins.

Choose the stability tolerance small enough for
Proposition~\ref{prop:endpoint-finite-geometry}, and then choose
\[
0<\eta_0<\min\left\{\eta(\varepsilon),\varepsilon,
\frac{3}{40}\right\}
\]
with any further fixed reduction needed to absorb the errors
above. Lemma~\ref{lem:new-stability} and compactness imply that
finite generic profiles with both terminal values at most
$\frac{4}{5}+\eta_0$ satisfy those stable inequalities once
$n$ is sufficiently large. Otherwise a sequence violating them
would converge, by Lemma~\ref{lem:new-finite}, to a scalar
system contradicting the stability lemma.
The resulting finite configurations are impossible.

The constant $\eta_0$ is absolute. The geometric gap and the
scalar stability modulus are independent of the measure.
Frostman constants, radial energies, oracle constants and
polynomial relation losses affect only the starting precision.
They remain fixed while $n$ tends to infinity.

\section{The universal unpinned and lower-box bounds}\label{sec:new-universal}

\subsection{General unpinned Hausdorff dimension}

Suppose a compact $E$ with $\dim_{\mathrm H} E>1$ satisfies $\dim_{\mathrm H}\Delta E<\frac{4}{5}+\eta_0$. Choose $q$ strictly between these two numbers.
Choose a Frostman measure of exponent greater than one and restrict
it to separated source and pin portions, denoted by $\mu$ and $\nu$. Encode the measure and
the compact set using a base oracle $A$ as in Section~\ref{sec:classical}.
The effectively compact point-to-set principle implies
\[
 \liminf_n \frac{K_n^A(|x-y|)}{n}<q
\]
for every pair of points in these portions. Work with a slightly
smaller intermediate threshold so that logarithmic coding costs
can be absorbed in $qn$.

On the fixed positive-measure good-pair relation $G$, both
directional hypotheses, full-oracle independence estimates,
angular-density restrictions, and base mass codes hold. Let $R_n$
retain pairs with low unconditioned distance complexity.
Each pair belongs to infinitely many such $R_n$. The Tonelli argument in Lemma~\ref{lem:seventh-classical-recurrence} gives $(\nu\times\mu)(R_n)\ge n^{-2}$ at infinitely
many scales. Low unconditioned complexity implies low complexity
in both orientations:
\[
 K_n^{A,x}(|x-y|),\ K_n^{A,y}(|x-y|)\le qn+O(\log n).
\]
Bin the two endpoint profiles and the distance profiles jointly;
there are $n^{o(1)}$ records on the usual slowly growing scale mesh.
A common retained relation still has polynomial mass.

The resulting bilateral profiles and common finite configurations are excluded by Section~\ref{sec:new-geometry}, since $q<\frac{4}{5}+\eta_0$. This proves Theorem~\ref{thm:main} for compact sets.
For analytic sets, first choose a compact subset with Hausdorff
dimension greater than one; its distance set is contained in the
original distance set.

A recurrent relation of low
distance complexity need not contain many pairs in a single profile
class; a dense relation can connect two different classes. Therefore
finite matching alone does not justify the Hausdorff conclusion. Proposition~\ref{newrig:prop:rigidity} removes precisely that
restriction. It does not improve the abstract scalar four-fifths
bound by itself. The geometric contradiction is still essential.

\subsection{Internal pinned lower box dimension}

For $B_n=\{x\in E:N_{2^{-n}}(\Delta_xE)<2^{qn}\}$ we prove
\[
                    \mu(B_n)\le n^{-M}
\]
for every fixed $M>0$ at all sufficiently large $n$, whenever
$q<\frac{4}{5}+\eta_0$. We use Lemmas~\ref{lem:new-finite}, \ref{newsel:lem:selection}, and~\ref{newsel:lem:cover}. If this conclusion fails, choose two
polynomially separated portions of $B_n$ and retain a polynomial-mass
product relation. Both endpoints are bad pins at the same original
precision, so both conditioned distance complexities are low.

Joint binning of arbitrary endpoint profiles suffices, because Proposition~\ref{newrig:prop:rigidity} allows $f\ne g$. The finite
geometric contradiction above then proves the same quantitative
tail bound. Borel--Cantelli proves the almost-everywhere lower-box
claim. A Frostman measure on a hypothetical internal exceptional
set of dimension greater than one contradicts this conclusion; closure
extends that conclusion from compact to arbitrary bounded $E$,
since lower box dimension is unchanged by closure.

Together with Section~\ref{sec:finite-matching}, this proves
Theorem~\ref{thm:new-box}. For $q<L_*(s)$, either $q<L(s)$ or
$q<\frac{4}{5}+\eta_0$, so one of the two tail estimates applies.
The same exceptional-set argument proves \eqref{eq:new-exceptional}.

\section{The gains and the remaining limitations}\label{sec:gain}

The universal unpinned estimate improves $\frac{3}{4}$ by
$\frac{1}{20}+\eta_0$. The gain has two parts. Reciprocal
partition estimates account for $\frac{1}{20}$; endpoint rigidity
and the common-source geometry account for $\eta_0$.
The universal external pinned gain is
$q_*-\frac{3}{4}=0.03131\ldots$.
The distinctions between these arguments matter when considering
how far the methods can be pushed.

\subsection{What can be improved, and the limits of the methods}
\label{sec:method-limits}

There is one proved obstruction to a uniform pinned improvement
using only the pointwise inequalities of this paper.
Proposition~\ref{prop:pinned-obstruction} constructs abstract
profiles satisfying both prefix estimates and both transfer
estimates strengthened by the full partition functional, with
$f(u),g(u)\ge(1+\theta)u$ but
$d_x(1)=\frac{3}{4}+\theta$ and $d_y(1)=1$ for every sufficiently
small $\theta>0$. Consequently, further optimization of those
same constraints cannot force a fixed positive improvement over
$\frac{3}{4}$ at a prescribed pin. The common-source mass bounds
supply an additional constraint, and it is this additional
information that permits our pinned improvement. The obstruction
concerns the stated scalar inequalities. We do not prove that
its profiles come from a pair of points satisfying all the
geometric hypotheses, or that every possible pointwise argument
has the same limitation.

\subsection{The scalar obstruction at four fifths}
\label{sec:direct-scalar-obstruction}

The reciprocal argument reaches $\frac{4}{5}$. The next proposition
shows that the scalar inequalities used in that argument cannot force
a larger universal constant. In the construction, the excess
$f(u)-u$ first rises, stays constant for an interval of length
$\frac{t}{5}$, and then falls. A partition piece of length at most
$\frac{t}{5}$ cannot join any of that rise to the later fall. We use
this separation to compute the partition loss directly.

\begin{proposition}[The terminal-flat profiles]
\label{prop:direct-terminal-flat}
For $1<t\le\frac{5}{4}$, put, on $[0,1]$,
\[
 f(u)=g(u)=\min\{2u,u+\frac{2t}{5},t\},\qquad
 d_x(u)=d_y(u)=\min\{u,\frac{4t}{5}\}.
\]
The functions $f,g$ are nondecreasing and $2$-Lipschitz, and
$d_x,d_y$ are nondecreasing and $1$-Lipschitz. All four vanish at
zero. These profiles satisfy $f(u),g(u)\ge tu$, the ordinary prefix
inequalities, and both full transfer inequalities
\[
 d_x(b)-d_x(a)\ge[\widehat H_{f,g}(a,b)-f(a)]_+,
 \qquad
 d_y(b)-d_y(a)\ge[\widehat H_{g,f}(a,b)-g(a)]_+
\]
for $0<a\le b\le1$. They also satisfy every strengthened prefix
\[
 d_x(u),d_y(u)\ge(1-2b(p,s))u+b(p,s)f(u),
 \qquad 1\le p\le s\le t,
\]
where
\[
 c(p,s)=\frac{p(2-s)+s}{5-2s},\qquad
 b(p,s)=\frac{1-c(p,s)}{2-c(p,s)}.
\]
Consequently, these scalar inequalities cannot force a universal
unpinned constant larger than $\frac{4}{5}$, even when the lower
dimensions of both endpoint profiles are strictly greater than one.
\end{proposition}

\begin{proof}
The slopes of $f=g$ are $2,1,0$, in that order. The slopes of
$d_x=d_y$ are $1,0$. These observations prove the monotonicity
and Lipschitz assertions. Also $2u\ge tu$, $t\ge tu$, and
$u+\frac{2t}{5}\ge tu$ for $0\le u\le1$ and $1<t\le\frac{5}{4}$.

Write $h=f-u$. It increases with slope one up to $\frac{2t}{5}$, is
constant on $[\frac{2t}{5},\frac{3t}{5}]$, and decreases with slope minus one
afterwards. Fix $0<v\le \frac{t}{5}$. An interval of length at most $v$
cannot meet both the increasing and decreasing parts of $h$ in
positive length. Thus every admissible yellow--teal partition of
$[0,R]$ pays for the entire increase of $h$ on that interval.
Conversely, partitioning at the two corners and subdividing the
monotone pieces realizes precisely that cost. Therefore
\[
 \mathcal L_v(f;R)=\min\{R,\frac{2t}{5}\},\qquad
 P_f(v;R)=\min\{R,\frac{3t}{5}\}.
\]
The closed gain is also at most $\frac{3t}{5}$: for $R\le\frac{3t}{5}$ use its
bound by $R$, and for $R\ge\frac{3t}{5}$ use
$\frac{f(R)+v}{2}=\frac{t+v}{2}\le\frac{3t}{5}$. It follows that
\begin{equation}\label{eq:direct-small-start}
 \widehat p_f(v;b)\le\frac{3t}{5}\qquad(0<v\le \frac{t}{5}).
\end{equation}

We now check the transfer. If $b\le\frac{4t}{5}$, the gain cap
$\widehat p_f(a;b)\le b-a$ gives the assertion immediately.
If $a\ge\frac{4t}{5}$, then $f(a)=f(b)=t$ and the right hand side is
zero. Suppose, then, that $a<\frac{4t}{5}<b$. When $a\ge \frac{t}{5}$, we have
$h(a)\ge \frac{t}{5}$, and hence
\[
 f(a)+\frac{4t}{5}-a=h(a)+\frac{4t}{5}\ge t\ge\widehat H_{f,f}(a,b).
\]
When $a<\frac{t}{5}$, choose the starting scale $a$ in the infimum defining
$\widehat H$ and use \eqref{eq:direct-small-start}. Consequently
\[
 \widehat H_{f,f}(a,b)-f(a)\le\frac{3t}{5}\le\frac{4t}{5}-a.
\]
The two orientations agree because $f=g$.

It remains to check the prefixes. The function $c(p,s)$ increases
in both variables on the stated range, and $b$ decreases as a
function of $c$. Since $f(u)\le2u$, the strongest displayed
prefix occurs at $p=s=t$. With $k=t-1$, its coefficient is
\[
 b(t,t)=\frac{1-3k+k^2}{(1-k)(4-k)}.
\]
For $u\le\frac{4t}{5}$, the prefix follows again from $f(u)\le2u$.
For $u\ge\frac{4t}{5}$, we have $f(u)=t$. The right side of
the prefix inequality increases with $u$, so it suffices to check
$u=1$. The difference is
\[
 \frac{4t}{5}-\{1-b(t,t)(2-t)\}
 =\frac{t^2}{5(5-t)}>0.
\]
The ordinary prefix corresponds to coefficient $\frac{1}{4}$, which is
no stronger. Finally, $d_x(1)=d_y(1)=\frac{4t}{5}$ and $t$ may decrease
to one. This proves the asserted scalar obstruction. It does not
assert that these profiles are realized by a geometric configuration.
\end{proof}

The proposition and Proposition~\ref{prop:reciprocal-scalar} identify
$\frac{4}{5}$ as the exact universal constant for the stated full
scalar system. This does not establish geometric sharpness. We have
not realized these profiles by a Euclidean pair with all the required
independence properties and a common source measure.

\subsection{The two geometric comparisons}

The first geometric upper count has a sharper local limitation.
For a fixed admissible restricted-sum exponent $\alpha$, the
comparison in \eqref{eq:leading-margin} is
\[
L-U_1\ge\frac{3\alpha}{4}(1-2\Delta)
                 -(10-3\alpha)\omega.
\]
The coefficient $10-3\alpha$ is attained in the abstract
parameter range of Proposition~\ref{prop:strict-budget} by
$\varsigma=2$, $\chi=-1$, $u_0=\mathfrak e=0$, $\varpi=1$, and
$\mathfrak u=2$. At transfer level one, where $\Delta=0$ and
$q=\frac{3}{4}+\omega$, this uniform comparison therefore requires
\[
q<\frac{3}{4}+\frac{3\alpha}{40-12\alpha}.
\]
Letting $\alpha\nearrow\alpha_*$ yields $q_*$. Optimizing the
same coefficient over the same parameter range cannot improve
this bound. It could improve if a stronger restricted-sum
estimate supplied a larger admissible $\alpha$, if one of the
geometric counts improved, or if additional geometric information
excluded the parameters attaining the coefficient. We make no
claim that those parameters are realized by the configurations
arising from a common source measure. Nor do we claim that
$\alpha_*$ is the optimal restricted-sum exponent. Substituting
a stronger input would also require checking the other scale
and richness inequalities in the proof.

The second upper count uses angular information from both
endpoints. Its effective parent exponent is
\[
B_{\mathrm{eff}}=\min\{\zeta,\alpha+(2-3\alpha)a\},
\]
and Section~\ref{sec:relation-upper} compares the lower count
with both entries. The resulting sufficient conditions are
\[
\omega<\frac{-1+4\alpha+6(1-2\alpha)\Delta}{12(3-2\alpha)},\qquad
\omega<\frac{\alpha(1-2\Delta)}{4(2+\alpha)}.
\]
After the certificate iteration these become $q<Q_b$ and
$q<Q_f$. Their minimum is the bound proved in
Section~\ref{sec:relation-proof}.
The controlling branch changes
at $S_{\mathrm{switch}}$ after $\alpha\nearrow\alpha_*$. At a
fixed source dimension, improving only the larger branch leaves
the minimum unchanged. Improving the smaller branch requires a
better count, a stronger angular bound, or a further restriction
on the allowable row and parent scales. The verified branch
switch is exact for the two displayed formulas, but the paper
does not prove that these formulas describe the best possible
use of the common relation. In particular, the present second
comparison does not raise the universal pinned constant as the
source dimension decreases to one. Proposition~\ref{prop:joint-count-obstruction}
in Appendix~\ref{sec:joint-count-obstruction} also shows that
combining the two displayed counts before optimizing does not
improve their threshold within the stated parameter set. Two
explicit choices of parameters suffice to prove this limitation.
They need not arise from a common source measure.

The constant continuation $Y(t)=\frac{7}{8}$ for $t\ge S_0$
comes from the range $\Delta<\frac{1}{8}$ of the reciprocal
estimate. In the source-dimension iteration, the limiting value
of $\Delta$ reaches $\frac{1}{8}$ at $t=S_0$.
The matching-profile argument reaches the same boundary when
the source floor is $\frac{9}{8}$.
Extending these calculations requires further control at and
beyond this transfer level. The value $\frac{7}{8}$ is not an
upper bound for the other methods of the paper; $R(t)$ exceeds it
for larger $t$.

The fixed gain $\eta_0$ uses a strict geometric exponent gap of
$\frac{2}{95}$ at the limiting configuration. This number is
not the value of $\eta_0$. A numerical dimension gain would
require quantitative control of the stability lemma and all
subsequent tolerances. The qualitative argument suffices for
the stated theorem.

Several losses affect the scale at which a contradiction occurs
but not its limiting exponent. Fixed multiplicative constants,
dyadic rounding, and polynomial factors in $n$ change the
logarithms of the counts by at most $O(\log n)$. Coding errors
have the same order. These contributions vanish after division
by $n$. The fixed exponent tolerances are chosen arbitrarily
small after a positive leading margin has been established.
Removing these losses alone cannot turn a zero leading margin
into a positive one or raise the constants above. Likewise, the
exceptional-set threshold one uses radial projection input with
source and pin exponents greater than one. Extending that part
of the proof to smaller pin exponents needs an appropriate
replacement for that input, not only a numerical reoptimization
of the final count.

Finally, every distance set is a subset of $\mathbb R$, so its
Hausdorff dimension is at most one. The conclusions giving full
pinned dimension when $\dim_{\mathrm H}E\ge\frac{3}{2}$ have
already reached this absolute ceiling. Below that range, the
paper proves the obstructions and limits of individual
comparisons described above.
It does not determine the best bound
obtainable by combining all the information in a common-source
configuration. The optimal unpinned constant and the optimal use of the two
pinned geometric counts remain open.

\begin{question}
Can the stability statement in Lemma~\ref{lem:new-stability} be made
quantitative enough to give a useful numerical value of $\eta_0$?
\end{question}

\begin{question}
Can the two-sided angular information in the common relation raise
the universal pinned constant $q_*$, rather than only the bound at
larger source dimensions?
\end{question}

\begin{question}
What additional compatibility across scales would turn the finite
rich-bisector counts into a Hausdorff-dimension estimate for one
fixed limiting family of rich bisectors?
\end{question}

\appendix

\section{The two prefix bounds with source information retained}
\label{app:retained-prefix}

References to equations and lemmas of Stull~\cite{S} and
Fiedler--Stull~\cite{FS} in this appendix use arXiv:2207.12501v2
and arXiv:2309.11701v1, respectively.

We give the partition calculation behind Lemma~\ref{lem:prefix},
including its endpoint at planar packing dimension two. The inputs
are the local yellow-interval estimate of \cite[Lemma~27]{FS}, the
circle-collision inequality of \cite[Lemma~30]{FS}, and the
distance-transfer lemma of \cite[Lemma~15]{FS}. The latter two are
used in Section~\ref{sec:transfer}. We retain the source complexity
through the summation instead of replacing it by a dimension bound.

Fix
\[
1<d<\min\{\dim^A x,\dim^A y,\frac{5}{4}\}.
\]
The upper cutoff $\frac{5}{4}$ is a convenient choice for this
auxiliary parameter because it keeps the rate
$\frac{d(3-d)}{5-2d}$ used below between $\frac{2}{3}$ and one,
with positive denominators. Such a $d$ always exists by (H1),
so this choice imposes no upper bound on either endpoint dimension.
The final inequalities in Proposition~\ref{prop:appendix-prefix}
are independent of $d$.
Let $F$ be the piecewise linear function on $[1,\infty)$ with
$F(m)=K_m^A(y)$ at every positive integer $m$, as in
\cite[\S4.3]{FS}.
Logarithmic interpolation and rounding errors
will be included in the errors below. Thus
$ds-o(s)\le F(s)\le2s+O(\log(s+2))$. Write
\[
\Delta_y[a,b]=K_b^{A,x}(y)-K_a^{A,x}(y),\qquad 
\Delta_d[a,b]=K_b^{A,x}(|x-y|)-K_a^{A,x}(|x-y|).
\]
Finite-precision chain rules identify these differences with the
corresponding conditional increments up to $O(\log b)$.

\begin{lemma}[The selected teal intervals]\label{lem:appendix-teals}
For every sufficiently large terminal precision $r$, the partition
of \cite[\S4.3]{FS}, with parameters $d$ and $D=2$, has the following
properties, up to errors that sum to $o(r)$.
On a yellow block $[a,b]$,
\[
\Delta_d[a,b]\ge b-a-o(b),\qquad 
\Delta_y[a,b]\le2(b-a)+O(\log b).
\]
On a selected teal block $[a,b]$,
\[
\Delta_d[a,b]\ge\Delta_y[a,b]-o(b),\qquad 
\Delta_y[a,b]\ge c_d(b-a)-o(b),\qquad
  c_d=\min\{1,\frac{d(3-d)}{5-2d}\}\ge\frac{2}{3}.
\]
The same assertions hold with the two endpoints exchanged.
\end{lemma}
\begin{proof}
Here is the construction and the information it preserves. Starting
from a right endpoint $b$, first take the maximal terminal block
that is a union of yellow intervals $[v,w]$ with $w\le2v$.
The local yellow estimate, summed on its greedy at-most-doubling
partition, gives $\Delta_d[a,b]\ge b-a-o(b)$.
The source upper bound follows from the planar precision-increment
bound. If there is no such terminal block, let $t<b$ be the largest
solution of
\[
F(t)=F(b)-\frac{d}{2}b+\frac{d-1}{2}t,
\]
and let $t'<b$ be the largest solution of
$F(t')=F(b)-(b-t')$. Set $a=\max\{t,t'\}$.
These are equations~(10)--(11) of \cite{FS} with $D=2$.
The lower bound $F(b)>db-o(b)$ and continuity provide the first
root.
The absence of a terminal yellow interval ensures the teal
alternative. Above the last roots we have
\begin{equation}\label{eq:appendix-teal-lines}
F(s)\ge F(b)-\frac{d}{2}b+\frac{d-1}{2}s,\qquad
  F(s)\ge F(b)-(b-s),\qquad  a\le s\le b.
\end{equation}
The block is teal. The halving argument in
\cite[Lemma~25]{FS} implies $a\le\frac{b}{2}$.
The first inequality in \eqref{eq:appendix-teal-lines}, together
with $F(a)\le2a+o(b)$ and $F(b)\ge db-o(b)$, yields
\[
\frac{a}{b}\ge\frac{d}{5-d}-o(1).
\]
In particular, the block endpoints have a bounded positive ratio.

If $a=t'$, its source increment is $b-a+o(b)$.
Otherwise the defining identity at $a=t$ implies
\[
\frac{F(b)-F(a)}{b-a}
=\frac{d}{2}+\frac{a}{2(b-a)}
\ge\frac{d}{2}+\frac{d}{2(5-2d)}-o(1)
=\frac{d(3-d)}{5-2d}-o(1).
\]
Independence (H3) at the two endpoints transfers this increment
bound to $\Delta_y$. The coefficient is at least $\frac{2}{3}$,
because
\[
\frac{d(3-d)}{5-2d}-\frac{2}{3}
=\frac{(d-1)(10-3d)}{3(5-2d)}\ge0.
\]

It remains to justify the transfer of the whole teal source
increment. Cap the source at $m=K_b^A(y)-\theta b$, where
$\theta>0$ is fixed and small. For a circle competitor at separation
$2^{-s}$, $a\le s\le\frac{b}{2}$, the capped collision inequality
of Section~\ref{sec:transfer} gives
\[
K_b^{A,G}(z)\ge\min\{m,F(s)\}
+\min\left\{b-s,\frac{d}{2}b-\frac{d-1}{2}s\right\}-o(b).
\]
Both alternatives in the second minimum, added to $F(s)$, are at
least $F(b)$ by \eqref{eq:appendix-teal-lines}. If the source term
is capped, the remaining gain is at least $\frac{b}{2}$.
For $s\ge\frac{b}{2}$, the tangent bound instead yields
\[
K_b^{A,G}(z)\ge\min\{m,F(s)\}+b-s-o(b).
\]
The teal inequality in \eqref{eq:appendix-teal-lines} handles
the uncapped case.
The capped case has the required terminal
allowance $b-s$. Choose the local error tolerance sufficiently
smaller than $\theta$. These estimates verify the competitor
hypothesis of \cite[Lemma~15]{FS}, in exactly the form used in
Lemma~\ref{lem:transfer}. Lemma~14(iii) of \cite{FS} and (H3)
preserve the capped source increment relative to the pin. Removing
the cap and then letting $\theta$ decrease to zero proves
$\Delta_d[a,b]\ge\Delta_y[a,b]-o(b)$.

Finally, these local errors can be summed. Two outer yellow blocks
cannot be consecutive, and every teal block at least halves the
precision. Hence the sum of their right endpoints is $O(r)$ and
there are $O(\log r)$ blocks above a fixed initial cutoff. The
greedy subpartitions of yellow blocks have the same geometric
summability property. For a fixed local tolerance $\xi$, their
total error is $O(\xi r)+O((\log r)^2)$.
The discarded bounded
initial range contributes $O(1)$. Take $r\to\infty$ first and
then $\xi\searrow0$. This proves the asserted $o(r)$ error.
\end{proof}

\begin{proposition}[The retained prefix inequalities]
\label{prop:appendix-prefix}
Let $L_r$ be the total length of the yellow blocks in the preceding
partition. Under (H1)--(H4),
\[
K_r^{A,x}(|x-y|)\ge\frac{2r+L_r}{3}-o(r),\qquad 
K_r^{A,x}(|x-y|)\ge K_r^A(y)-L_r-o(r).
\]
\end{proposition}
\begin{proof}
Summing the yellow distance increments and the teal lower bounds
from Lemma~\ref{lem:appendix-teals} results in
\[
K_r^{A,x}(|x-y|)\ge L_r+c_d(r-L_r)-o(r)
\ge\frac{2r+L_r}{3}-o(r).
\]
For the second inequality, keep the full source complexity.
The sums indexed by $Y$ and $T$ run over the yellow and selected
teal blocks, respectively, with each increment evaluated at that
block's endpoints:
\begin{align*}
K_r^{A,x}(y)
&=\sum_Y\Delta_y+\sum_T\Delta_y+o(r)\\
&\le2L_r+\sum_T\Delta_d+o(r)\\
&\le L_r+K_r^{A,x}(|x-y|)+o(r).
\end{align*}
The only step removing the pin from this source quantity is (H3):
\[
K_r^A(y)-O(\log r)\le K_r^{A,x}(y).
\]
Combining proves the second inequality. Multiplying the first by
$\frac{3}{4}$ and the second by $\frac{1}{4}$ cancels $L_r$ and
proves Lemma~\ref{lem:prefix}.
\end{proof}

The endpoint calculation involves three places in the cited
argument. The first ratio bound of \cite[Lemma~26]{FS} survives at
$D=2$.
Its weaker bound proportional to $2-D$ is replaced by
$\frac{a}{b}\ge\frac{d}{5-d}-o(1)$. Equations~(22)--(23) in the
proof of its Lemma~31 use that weaker bound only to justify the
projection theorem. Here Lemma~\ref{lem:projection} applies because
$\frac{s}{b-s}$ is bounded below. Finally, the bound
$S<\frac{2}{d(2-D)}$ on the number of red--green--blue sequences
belongs to the proof of \cite[Theorem~16]{FS}.
We
use the bounded-ratio partition count proved in
Lemma~\ref{lem:projection} instead. The root identities, halving
argument, source-increment calculation, and summation above do not
divide by $2-D$. Thus no limiting assertion in $D$ is needed.

\section{The first scalar improvement}
\label{sec:thirteen-seventeen}

This appendix and Appendix~\ref{sec:refined} record preliminary
scalar estimates superseded by Proposition~\ref{newlev:prop:level}.
The first proof uses a single exchange of the endpoint estimates.
Suppose both
descriptions of the distance are short. With $y$ as the pin, we
deduce that $x$ must have a certain amount of complexity at every
late scale. We then use $x$ as the pin, where this lower bound
improves the projection estimate. The resulting contradiction proves the bound
$\frac{13}{17}$.

\begin{proposition}[Scalar consequence]\label{prop:scalar}
Let $f,g,d_x,d_y$ be scalar profiles as in
Section~\ref{sec:profiles}, satisfying \eqref{eq:profileclass},
\eqref{eq:limitprefix}, and Lemma~\ref{lem:transfer} for both
choices of pin. Then
\begin{equation}\label{eq:scalar}
\max\{d_x(1),d_y(1)\}\ge\frac{13}{17}.
\end{equation}
\end{proposition}
\begin{proof}
Set $q=\max\{d_x(1),d_y(1)\}$. The baseline implies $q\ge\frac{3}{4}$. If $q\ge\frac{7}{8}$ there is nothing to prove, so assume $\frac{3}{4}\le q<\frac{7}{8}$.

For every $v\in[\frac{1}{2},1]$, the tangent range and $g(s)\ge s$ imply
\[
H_{g,f}(v,1)\ge1.
\]
Combining the prefix estimate and tail transfer with $y$ as the pin yields
\[
q\ge\frac{v}{2}+\frac{g(v)}{4}+\left[1-g(v)\right]_+
\ge1+\frac{v}{2}-\frac{3g(v)}{4}.
\]
Consequently
\begin{equation}\label{eq:pinenvelope}
g(v)\ge\frac{4(1-q)+2v}{3} \quad \text{for } \frac{1}{2}\le v\le1.
\end{equation}
For $s\le\frac{1}{2}$, substitute $v=1-s$ into \eqref{eq:pinenvelope} to obtain
\begin{equation}\label{eq:pgq}
p_g(s;1)\ge\min\left\{1-s, 1-\frac{2q}{3}+\frac{s}{6}\right\}.
\end{equation}

Choose the unique $a>0$ such that
\begin{equation}\label{eq:cutoff}
f(a)+\frac{a}{6}=\frac{2q}{3}.
\end{equation}
Existence and uniqueness follow because $f(u)+\frac{u}{6}$ is continuous and strictly increasing. Since $f(\frac{1}{2})\ge\frac{1}{2}$ and $q<\frac{7}{8}$, we have $a<\frac{1}{2}$. Since $f(a)\le2a$,
\begin{equation}\label{eq:alower}
a\ge\frac{4q}{13}>0.
\end{equation}
For $a\le s\le\frac{1}{2}$, monotonicity implies $f(s)+\frac{s}{6}\ge\frac{2q}{3}$. Equations \eqref{eq:pgq} and $f(s)\ge s$ therefore imply $f(s)+p_g(s;1)\ge1$. For $s\ge\frac{1}{2}$, the same conclusion follows from $p_g(s;1)=1-s$. Since $f(1)\ge1$, we conclude that
\[
H_{f,g}(a,1)\ge1.
\]
Now use the prefix estimate and tail transfer with $x$ as the pin:
\begin{align*}
q&\ge\frac{a}{2}+\frac{f(a)}{4}+\left[1-f(a)\right]_+\\
&\ge1+\frac{a}{2}-\frac{3f(a)}{4}\\
&=1-\frac{q}{2}+\frac{5a}{8}
\ge1-\frac{q}{2}+\frac{5}{8}\frac{4q}{13}
=1-\frac{4q}{13}.
\end{align*}
Thus $17q\ge13$, proving \eqref{eq:scalar}.
\end{proof}

The envelope \eqref{eq:pinenvelope} says that if the distance has a
short description given $y$, then $x$ must be fairly complicated at
every scale in $[\frac{1}{2},1]$. Otherwise the transfer with $y$ as
the pin would force a longer description of the distance. We now
regard $x$ as the pin. Its additional complexity improves the
projection gain $p_g$ and hence the transfer estimate for $d_x$.
This is the interaction between the two choices of pin. An upper
bound for $d_x(1)$ alone does not provide the upper bound for $d_y(1)$
used at the start.
Section~\ref{sec:pinned-obstruction} makes the
resulting limitation precise.

\section{An elementary iteration of the scalar bounds}\label{sec:refined}

The first argument used only a lower bound on the tail of each
profile. The interval estimates also give information at intermediate
scales. We use that information twice, and then return to the tail.
This produces an affine update for the cutoff from which the full
tail transfers. Its fixed point determines the exponent $\frac{2725}{3521}$.

The calculation below shows how iterating the cutoff improves the
initial bound. The reciprocal argument in
Section~\ref{sec:reciprocal} develops this idea further.

\begin{proposition}[Two intermediate envelopes]\label{prop:refinedscalar}
Under the scalar hypotheses of Proposition~\ref{prop:scalar},
\[
\max\{d_x(1),d_y(1)\}\ge\frac{2725}{3521}=0.7739278614\ldots.
\]
\end{proposition}
\begin{proof}
Write $q=\max\{d_x(1),d_y(1)\}$ and $k=1-q$.
We will repeatedly move a common tail cutoff to the left. Every
valid cutoff must remain at least $k$, whereas the update has a
fixed point below $k$. This will give the contradiction.
By Proposition~\ref{prop:scalar}, we may assume
\begin{equation}\label{eq:refinedrange}
\frac{13}{17}\le q<\frac{2725}{3521}<\frac{31}{40},\qquad
  \frac{796}{3521}<k\le\frac{4}{17}.
\end{equation}
In particular $k>\frac{9}{40}$. We retain the sharper bound
$k>\frac{796}{3521}$ for the final cutoff update. Set
\[
L(u)=\frac{4k+2u}{3},\qquad  A_0=\frac{4-12k}{5}.
\]
We will decrease a common tail cutoff $A$, while preserving the following conditions for both choices of pin.
\begin{align}
k\le A\le A_0\le\frac{13}{50}, 
&H_{f,g}(A,1),H_{g,f}(A,1)\ge1,\label{eq:refinedinvariant}\\
\beta=\frac{4}{3}(A-k), 
&f(u),g(u)\ge
\begin{cases}
2u-\beta,&0\le u\le A,\\
L(u),&A\le u\le1.
\end{cases}\label{eq:refinedenvelope}
\end{align}
To initialize, apply the proof of Proposition~\ref{prop:scalar} first with $x$ as the pin and then with $y$ as the pin. The two resulting cutoffs are each at most $A_0=\frac{12q-8}{5}$. At those cutoffs the respective thresholds $H_{f,g}$ and $H_{g,f}$ are at least one. Hence both thresholds remain at least one at $A_0$. Prefix plus tail then gives $f(u),g(u)\ge L(u)$ for $u\ge A_0$. The functions $2u-f(u)$ and $2u-g(u)$ are nondecreasing, so their values at $A_0$ supply the early bounds in \eqref{eq:refinedenvelope}. Finally, $A_0\ge k$ follows from $q\ge\frac{13}{17}$.

We first improve the profiles on an intermediate interval.
For
\[
\frac{8}{15}\le b\le4k,\qquad 
c_1=\frac{4k+2b+6\beta}{13}
=\frac{8A-4k+2b}{13},
\]
we claim
\begin{equation}\label{eq:firstintermediateH}
H_{f,g}(c_1,b)\ge L(b).
\end{equation}
Here $b\ge2A$ and $0<c_1\le A$: the latter upper bound is equivalent to $b\le2k+\frac{5A}{2}$, which follows from $b\le4k$ and $A\ge k$.

For $s\ge \frac{b}{2}$, the projection gain is $b-s$, and
\[
f(s)+b-s\ge L(s)+b-s\ge L(b).
\]
For $c_1\le s\le \frac{b}{2}$, the pin argument $b-s$ is at least $A$. If $s\le A$, the second projection branch satisfies
\[
f(s)+\frac{g(b-s)+s}{2}
\ge\frac{2k}{3}+\frac{b}{3}+\frac{13s}{6}-\beta
\ge L(b).
\]
If $s\ge A$, that branch is instead at least $2k+\frac{b}{3}+\frac{5s}{6}$, which is at least $L(b)$ because
$4k-2b+5s\ge k>0$.
The competing tangent branch for $s\le A$ is at least $b+c_1-\beta$. The required inequality $b+c_1-\beta\ge L(b)$ is equivalent to
\[
19b\ge40k+21\beta=28A+12k.
\]
It follows from
\[
19b\ge\frac{152}{15},\qquad 
28A+12k\le28A_0+12k
=\frac{112-276k}{5}\le\frac{499}{50}<\frac{152}{15}.
\]
For $s\ge A$, the tangent branch was already checked, and $f(b)\ge L(b)$. This proves \eqref{eq:firstintermediateH}.

Prefix plus transfer, using $f(c_1)\le2c_1$, yields
\[
d_x(b)\ge L(b)-c_1.
\]
Since $b\ge A$, the tail at $b$ has height at least one. Consequently
\begin{equation}\label{eq:firstaffine}
f(b),g(b)\ge M_A(b):=\frac{103k+20b-24A}{39}
 \quad \text{for } \frac{8}{15}\le b\le4k.
\end{equation}
Indeed, $q\ge d_x(b)+1-f(b)$ implies $f(b)\ge k+L(b)-c_1=M_A(b)$.
The calculation with $y$ as the pin proves the same conclusion for $g$.

We apply the new envelope once more.
For
\[
\frac{37}{50}\le b\le4k,\qquad 
c_2=\frac{128A-103k+32b}{175},
\]
we claim
\begin{equation}\label{eq:secondintermediateH}
H_{f,g}(c_2,b)\ge L(b).
\end{equation}
The previous certificate covers $s\ge c_1$, so it suffices to check $c_2\le s\le c_1$. All pin arguments on this new interval lie in the range of \eqref{eq:firstaffine}. More explicitly, $c_2>0$, and
\[
c_1-c_2=\frac{3(213k-88A-22b)}{2275}>0,
\]
since $b\le4k$, $k\ge\frac{9}{40}$, and $A\le\frac{13}{50}$ imply
$213k-88A-22b\ge125k-88A>0$. Also
\[
b-c_1=\frac{11b+4k-8A}{13}
\ge\frac{174}{325}>\frac{8}{15}.
\]
Thus $b-s\in[\frac{8}{15},4k]$ whenever $c_2\le s\le c_1$, while $s\le c_1\le A$. By \eqref{eq:refinedenvelope} and \eqref{eq:firstaffine},
\[
f(s)+\frac{g(b-s)+s}{2}
\ge2s-\beta+\frac{M_A(b-s)+s}{2}\ge L(b).
\]
The last expression increases with $s$ and equals $L(b)$ at $s=c_2$, by the defining formula for $c_2$. The competing tangent branch is at least $b+c_2-\beta$.
The condition that this reaches $L(b)$ is
\[
271b\ge316A+309k.
\]
It follows from $b\ge\frac{37}{50}$, $A\le\frac{13}{50}$, and
$k\le\frac{1}{4}$. This proves \eqref{eq:secondintermediateH}.

Applying prefix and terminal tail exactly as before gives
\begin{equation}\label{eq:secondaffine}
f(b),g(b)\ge N_A(b):=\frac{1534k+254b-384A}{525}
 \quad \text{for } \frac{37}{50}\le b\le4k.
\end{equation}

The two intermediate estimates now improve the full tail.
Set
\[
\gamma=\frac{271}{1050},\qquad 
T=\frac{796-1534k+384A}{1050},
\]
and let $a$ solve $f(a)+\gamma a=T$. The solution exists uniquely in $(0,A]$. Indeed, $T>0$ follows from $A\ge k$ and $k\le\frac{4}{17}$, and
\[
f(A)+\gamma A-T
\ge\frac{2934k+587A-796}{1050}
\ge\frac{3521k-796}{1050}>0
\]
by the sharper lower bound $k>\frac{796}{3521}$ in \eqref{eq:refinedrange}.
This is the one place where $k>\frac{9}{40}$ would not suffice. Also $f(a)\le2a$ implies
\[
a\ge\frac{796-1534k+384A}{2371}
\ge\frac{796-1150k}{2371}\ge1-4k;
\]
the last step only requires $8334k\ge1575$, which follows from $k\ge\frac{9}{40}$.

For $a\le s\le A$, therefore, $1-s\in[\frac{37}{50},4k]$. From the new envelope \eqref{eq:secondaffine} we obtain
\[
f(s)+\frac{g(1-s)+s}{2}
\ge f(s)+\frac{1534k+254-384A+271s}{1050}\ge1.
\]
The competing branch $f(s)+1-s$ is also at least one. For $s\ge A$, invariant \eqref{eq:refinedinvariant} already supplies the required inequality. Thus $H_{f,g}(a,1)\ge1$.

Prefix plus this full tail implies
\[
q\ge1+\frac{a}{2}-\frac{3f(a)}{4}
=1-\frac{3T}{4}+\left(\frac{1}{2}+\frac{3\gamma}{4}\right)a,
\]
and hence
\begin{equation}\label{eq:cutoffupdate}
a\le A':=\frac{796-2934k+384A}{971}.
\end{equation}
The calculation with $y$ as the pin yields its own full-tail cutoff, bounded by the same $A'$. Every such cutoff is at least $k$, since prefix plus tail and $f(a)\le2a$ give $q\ge1-a$. If $A'<k$ we have a contradiction immediately.

Otherwise both tails at $A'$ have height at least one. Prefix plus tail implies $f(u),g(u)\ge L(u)$ for $u\ge A'$, and monotonicity of $2u-f(u)$ and $2u-g(u)$ supplies the early bounds with $\beta'=\frac{4(A'-k)}{3}$. Moreover,
\[
A'-A=\frac{796-2934k-587A}{971}
\le\frac{796-3521k}{971}<0.
\]
Thus \eqref{eq:refinedinvariant}--\eqref{eq:refinedenvelope} are preserved until a contradiction occurs.

The affine update \eqref{eq:cutoffupdate} is a contraction, with fixed point
\[
A_* =\frac{796-2934k}{587}<k.
\]
Starting at $A_0$, finitely many iterations must therefore give $A_n<k$, a contradiction. This proves the proposition.
\end{proof}

\begin{remark}[Finiteness and the role of the transfer lemma]
For every fixed $q<\frac{2725}{3521}$, only finitely many updates are needed. Within each update all interval endpoints are fixed positive proportions of the terminal precision, and Lemma~\ref{lem:transfer} holds for every such pair of endpoints. The proof does not sum a precision-dependent number of error terms. Proposition~\ref{prop:refinedscalar} improves the scalar consequence of the same analytic inputs.
Its conversion to the distance-set statement uses the same fixed-pair and classical arguments as the $\frac{13}{17}$ bound.
\end{remark}

\section{Pinned consequences of matching profiles}
\label{sec:matching-realization}

Finite-scale matching allowed us to choose a different pair at each
precision in the lower-box argument. A common dyadic mass law
ensures matching for almost every fixed pair. We can then use
the same scalar estimates to obtain a pinned Hausdorff bound.
The geometric argument supplies the fixed gain near dimension one.

\begin{proposition}[A matched-profile pinned consequence]
\label{prop:matched-profile}
Suppose $x\ne y$ satisfy (H1)--(H4) relative to $A$, and
\begin{equation}\label{eq:matched-profile}
\sup_{1\le m\le r}|K_m^A(x)-K_m^A(y)|=o(r).
\end{equation}
If $\dim^Ax,\dim^Ay>1+\kappa$ with $0<\kappa<\frac{1}{8}$, then
\[
\liminf_{r\to\infty}\frac{K_r^{A,x}(|x-y|)}{r}
\ge\frac{4}{5}+\frac{3\kappa}{5}.
\]
\end{proposition}
\begin{proof}
Take any sequence of terminal precisions and extract uniformly
convergent endpoint profiles and the distance profile $d_x$.
Condition~\eqref{eq:matched-profile} implies $f=g$.
The prefix and both transfer expressions therefore agree in the
two orientations. Hence $(f,f,d_x,d_x)$ is a scalar system, with
$f(u)\ge(1+\kappa)u$. Proposition~\ref{newmatch:prop:matched}
proves the lower bound at the terminal scale.
The initial sequence was arbitrary, which proves the assertion.
The two actual conditioned distance profiles need not agree.
\end{proof}

\begin{theorem}[Pinned distances under a common dyadic mass law]
\label{thm:uniform-main}
Let $\mu$ be a compactly supported planar $s$-Frostman probability
measure, with $s>1$, and put $E=\operatorname{supp}\mu$.
Suppose there are real numbers $H(n)$ such that
\[
-\log_2\mu(Q_n(z))=H(n)+o_z(n)
\quad\text{for $\mu$-almost every }z.
\]
Then $\dim_{\mathrm H}\Delta_xE\ge L_*(s)$ for $\mu$-almost every $x$.
The quotient $\frac{H(n)}{n}$ need not converge.
\end{theorem}

The mass law says that the linear-scale part of the information in
a dyadic cell is described by one function $H(n)$ at almost every
point. It permits the rate of information growth to vary with the scale.

\begin{proof}
Choose an oracle $A$ encoding $E$ and all dyadic masses of $\mu$.
Lemma~\ref{lem:measure-coding} implies
$K_n^A(z)=H(n)+o_z(n)$ almost everywhere. For almost every pair,
the difference of the two complexities is $o(n)$. This implies
\eqref{eq:matched-profile} as well. Given a positive tolerance,
the errors at all sufficiently large $m\le r$ are at most that
tolerance times $r$; the finitely many smaller errors are bounded
independently of $r$.

The same coding lemma with the other endpoint as an additional
oracle, followed by Fubini, proves full-oracle independence for
almost every pair. The Frostman bound ensures effective dimension
at least $s$. To obtain the directional hypotheses, fix $1<s'<s$
and apply Orponen's theorem~\cite[Theorem~1.11]{O} to countably
many source restrictions in balls separated from the pin.
The exceptional centers have dimension at most $2-s'<1$ and
are $\mu$-null. Directions failing the logarithmic complexity
bound form an arclength-null set. Absolute continuity and Fubini
therefore supply both directional estimates.

For almost every $x$, Proposition~\ref{prop:matched-profile}
applies to almost every $y$, for each
$0<\kappa<\min\{s-1,\frac{1}{8}\}$.
Letting $\kappa$ increase to this upper endpoint proves
$\dim^{A,x}|x-y|\ge L(s)$. The pinned image is effectively
compact relative to $(A,x)$, so the effectively compact
point-to-set principle proves the Hausdorff bound $L(s)$.

It remains to prove the fixed gain. Suppose a positive measure
set of pins has pinned Hausdorff dimension below some
$q<\frac{4}{5}+\eta_0$. By a countable restriction, choose
separated source and pin portions of positive measure. The
common mass law and the preceding generic properties can be
made uniform on a positive product-measure relation. Restrictions
to positive-measure sets preserve the law, since the dyadic density
theorem makes their cell-mass ratio tend to a positive constant
at almost every retained point.

Lemma~\ref{lem:seventh-classical-recurrence} supplies infinitely
many precisions with a relation of polynomial mass and low
conditioned distance complexity in the prescribed orientation.
The endpoint profiles agree uniformly in the limit. The distance
profile in that orientation can therefore be duplicated to form
$(f,f,d_x,d_x)$. Lemma~\ref{lem:new-stability} now forces the
plateau and parent certificate used in
Proposition~\ref{prop:endpoint-finite-geometry}.
That proposition needs the small radial cover in only the actual
prescribed orientation. The duplicated scalar profile is used
solely to obtain its mass and plateau hypotheses.
Its contradiction proves the bound $\frac{4}{5}+\eta_0$.
Taking the maximum with $L(s)$ completes the proof.
\end{proof}

\begin{corollary}[Uniform branching]\label{cor:uniform-branching}
Start with a dyadic square. At level $n$, retain
$m_n\in\{1,2,3,4\}$ children of every retained square, using
the same number for all parents at that level. Their locations
may vary with the parent. Let $E$ be the limiting compact set
and let $\mu$ give equal mass to the retained cells. If
\[
s_0=\liminf_{n\to\infty}\frac{1}{n}\log_2\prod_{k=1}^n m_k>1,
\]
then $\dim_{\mathrm H}\Delta_xE\ge L_*(s_0)$ for
$\mu$-almost every $x$.
\end{corollary}
\begin{proof}
Put $H(n)=\sum_{k=1}^n\log_2m_k$. Each retained cell has
mass $2^{-H(n)}$. Dyadic boundary lines are null. A bounded
vertical or horizontal segment meets at most $C2^n$ cells, whose
total mass is at most $C2^{n-H(n)}\to0$. A ball of radius
$2^{-n}$ meets a bounded number of cells, so $\mu$ is
$t$-Frostman for some $t>1$. Apply the preceding theorem at every Frostman exponent below $s_0$
and use continuity of $L_*$. The dyadic mass identity is exact.
\end{proof}

Alternating two-child and four-child blocks of lengths $2^j$, $j\ge1$,
gives $\liminf_{n\to\infty}\frac{H(n)}{n}=\frac{4}{3}$ and
$\limsup_{n\to\infty}\frac{H(n)}{n}=\frac{5}{3}$. The ratio is monotone
within each block. This illustrates why the hypothesis concerns an entire
information profile rather than a single limit of the information per scale.
For the Hausdorff conclusion, the common law supplies one fixed
pair along all sequences of precisions. The lower-box proof needed
only the finite-scale matching of Section~\ref{sec:finite-matching}.

\section{Joint optimization of the two geometric counts}
\label{sec:joint-count-obstruction}

Both upper counts apply to the same family of rich bisectors. One
may therefore compare their minimum with the lower count before
optimizing over the parameters. We show that this does not improve
the threshold under the constraints used in the paper. The assertion
concerns these inequalities only. A further restriction imposed by
the common source measure could change the conclusion.

Fix $\frac{1}{3}<\alpha<\alpha_*$ and $0\le\Delta<\frac{1}{2}$. Recall that
$a_c=\frac{1}{4}+\frac{\Delta}{2}$ and $W_0=\frac{1-2\Delta}{4}$, and put
\begin{align*}
 \omega_o&=\frac{3\alpha(1-2\Delta)}{4(10-3\alpha)},\\
 \omega_b&=\frac{-1+4\alpha+6(1-2\alpha)\Delta}{12(3-2\alpha)},\\
 \omega_f&=\frac{\alpha(1-2\Delta)}{4(2+\alpha)}.
\end{align*}
Fix also $0<\omega<\frac{1-2\Delta}{24}$. We work with the leading
exponents, with the arbitrarily small tolerances set equal to zero.
The precise parameter set used here is
\[
 -1\le\chi\le2,\qquad
 \varsigma\ge\max\{0,-2\chi\},\qquad
 3\varsigma+2\chi\le4,\qquad 0\le\mathfrak e\le\varsigma.
\]
Set
\[
 u_0=\varsigma+2\chi,\qquad
 \varpi=\min\{2,1+u_0\},\qquad
 e_0=\mathfrak e+[u_0-1]_+,
\]
and take $0\le\mathfrak u\le\varpi+1$. In branch~1,
$\mathfrak e\le\frac{1}{8}$ and
$f_0(\mathfrak u)=\min\{2,2\mathfrak u,\mathfrak u+e_0\}$.
In branch~2, $\mathfrak e\ge\frac{1}{8}$ and
$f_0(\mathfrak u)=\min\{2,2\mathfrak u\}$. Thus we keep the
additional source bound in branch~1. Put
\[
 \mathfrak r=\frac{1}{2}+f_0(\mathfrak u)\omega,\qquad
 t=a_c+\chi\omega,\qquad
 W=W_0-(\varsigma+\chi-\mathfrak e)\omega,
\]
and retain
\[
 a=a_c+(\varpi-\mathfrak u)\omega,\qquad
 L=\frac{1+3\alpha}{2}+(1-3\alpha)a.
\]
The first upper exponent is $U_1=3\mathfrak r-t-2W$. The second
is the larger of
\[
 U_b=3\mathfrak r-\alpha-(2-3\alpha)a,
 \qquad U_f=3\mathfrak r-(1-a).
\]
Thus the combined comparison is contradictory at a template
precisely when
\[
 L>U_1\quad\hbox{or}\quad (L>U_b\ \hbox{and}\ L>U_f).
\]

\begin{proposition}[The joint comparison of the displayed exponents]
\label{prop:joint-count-obstruction}
The assertion that at least one of the two strict leading-exponent
comparisons succeeds for every parameter choice just specified is
equivalent to
\[
 \omega<\max\{\omega_o,\min\{\omega_b,\omega_f\}\}.
\]
The quantifier is uniform. Above this threshold there is an
admissible parameter choice at which both comparisons fail. Thus
joint optimization over this parameter set has the same threshold
as the better of the two separate optimizations.
\end{proposition}

\begin{proof}
The sufficiency follows from the separate strict comparisons in
Proposition~\ref{prop:strict-budget} and
Section~\ref{sec:relation-upper}. The latter remains valid with
the improved branch~1 richness, since decreasing $\mathfrak r$
can only decrease both upper exponents.
For the converse, only two templates are needed. Both are in branch
one, and both have $\mathfrak r=\frac{1}{2}+2\omega$, including the improved
source bound used in that branch.

First take
\[
 (\varsigma,\chi,\mathfrak e,\varpi,\mathfrak u)
 =(2,-1,0,1,2).
\]
Here $u_0=e_0=0$, $a=a_c-\omega$, $t=a_c-\omega$, and
$W=W_0-\omega$. Direct substitution yields
\begin{align*}
 L-U_1&=(10-3\alpha)(\omega_o-\omega),\\
 L-U_b&=3(3-2\alpha)(\omega_b-\omega).
\end{align*}
Consequently this template defeats the combined comparison whenever
$\omega\ge\max\{\omega_o,\omega_b\}$.

Next take
\[
 (\varsigma,\chi,\mathfrak e,\varpi,\mathfrak u)
 =(1,\frac{1}{2},0,2,1).
\]
Now $u_0=2$, $e_0=1$, and $a=a_c+\omega$. The branch-one
minimum is still $\min\{2,2\mathfrak u,\mathfrak u+e_0\}=2$.
For this template,
\begin{align*}
 L-U_f&=3(2+\alpha)(\omega_f-\omega),\\
 L-U_1&=\frac{3\alpha}{4}(1-2\Delta)
                  -(\frac{15}{2}+3\alpha)\omega.
\end{align*}
At $\omega=\omega_f$, the latter expression equals
\[
 -\frac{3\alpha(1-2\Delta)}{8(2+\alpha)}<0,
\]
and it decreases thereafter. This template therefore defeats the
combined comparison for $\omega\ge\omega_f$.

Since $\omega_f>\omega_o$, the two obstruction ranges start at
\[
 \min\{\max\{\omega_o,\omega_b\},\omega_f\}
 =\max\{\omega_o,\min\{\omega_b,\omega_f\}\}.
\]
All these thresholds lie in the stated structural range whenever
they control the maximum, since $\alpha<\alpha_*<\frac{2}{5}$ implies
$\omega_f<\frac{1-2\Delta}{24}$.
All the displayed template constraints, including
$3\varsigma+2\chi\le4$, $\varsigma\ge\max\{0,-2\chi\}$, and
$0\le\mathfrak e\le\varsigma$, hold for the two choices.
For both choices one may take $u_-=W$ and $u_+=W+t$, so the
remaining plateau-location inequalities hold as well. This
completes the calculation.
\end{proof}

\section*{Acknowledgments}

Alex Iosevich acknowledges partial support from National Science
Foundation grant DMS-2154232 and a 2026 Simons Fellowship in
Mathematics. Alexia Yavicoli acknowledges partial support from
the Natural Sciences and Engineering Research Council of Canada
(NSERC) through grants GR030571 and GR030540.

The authors thank Institut Mittag-Leffler and the program
organizers for their hospitality during the visit in which
part of this paper was written.

The authors used ChatGPT (OpenAI) to assist with literature
searches, calculations, the development and presentation of
arguments, and revision of the manuscript. The authors take
full responsibility for all statements, proofs, calculations,
and references in the paper.

\end{document}